\documentclass[10pt]{amsart}
\usepackage{amsfonts}
\usepackage{mathrsfs}
\usepackage{amscd}
\usepackage{amsmath}
\usepackage{amssymb}
\usepackage{latexsym}
\usepackage{lscape}
\usepackage{xypic}
\usepackage{comment}
\usepackage{amscd}
\usepackage{wasysym}  
\usepackage{tikz} 
\usetikzlibrary{matrix,arrows}
\usepackage{appendix}
\usepackage{geometry}
\usepackage[nice]{nicefrac}
\usepackage{dsfont}

\usepackage[pagebackref,hyperindex,linktocpage=true]{hyperref}
\hypersetup{
    colorlinks,
    linkcolor={red!50!black},
    citecolor={blue!50!black},
    urlcolor={blue!80!black}
}

\usepackage{color}

\newtheorem{thm}{Theorem}[section]

\newtheorem{prop}[thm]{Proposition}
\newtheorem{lem}[thm]{Lemma}

\newtheorem{lem-def}[thm]{Lemma-Definition}
\newtheorem{cor}[thm]{Corollary}
\newtheorem{conject}[thm]{Conjecture}

\theoremstyle{definition}

\newtheorem{ex}[thm]{Example}

\newtheorem{rmk}[thm]{Remark}
\newtheorem{dfn}[thm]{Definition}

\numberwithin{equation}{section}

\newcommand{\nc}{\newcommand}

\nc{\on}{\operatorname}
\nc{\fraka}{{\mathfrak a}} \nc{\bba}{{\mathbf a}}
\nc{\frakb}{{\mathfrak b}}
\nc{\frakc}{{\mathfrak c}}
\nc{\frakd}{{\mathfrak d}}
\nc{\frake}{{\mathfrak e}}
\nc{\frakf}{{\mathfrak f}}
\nc{\frakg}{{\mathfrak g}}
\nc{\frakh}{{\mathfrak h}}
\nc{\fraki}{{\mathfrak i}}
\nc{\frakj}{{\mathfrak j}}
\nc{\frakk}{{\mathfrak k}}
\nc{\frakl}{{\mathfrak l}}
\nc{\frakm}{{\mathfrak m}}
\nc{\frakn}{{\mathfrak n}}
\nc{\frako}{{\mathfrak o}}
\nc{\frakp}{{\mathfrak p}}
\nc{\frakq}{{\mathfrak q}}
\nc{\frakr}{{\mathfrak r}}
\nc{\fraks}{{\mathfrak s}}
\nc{\frakt}{{\mathfrak t}}
\nc{\fraku}{{\mathfrak u}}
\nc{\frakv}{{\mathfrak v}}
\nc{\frakw}{{\mathfrak w}}
\nc{\frakx}{{\mathfrak x}}
\nc{\fraky}{{\mathfrak y}}
\nc{\frakz}{{\mathfrak z}}
\nc{\frakA}{{\mathfrak A}}
\nc{\frakB}{{\mathfrak B}}
\nc{\frakC}{{\mathfrak C}}
\nc{\frakD}{{\mathfrak D}}
\nc{\frakE}{{\mathfrak E}}
\nc{\frakF}{{\mathfrak F}}
\nc{\frakG}{{\mathfrak G}}
\nc{\frakH}{{\mathfrak H}}
\nc{\frakI}{{\mathfrak I}}
\nc{\frakJ}{{\mathfrak J}}
\nc{\frakK}{{\mathfrak K}}
\nc{\frakL}{{\mathfrak L}}
\nc{\frakM}{{\mathfrak M}}
\nc{\frakN}{{\mathfrak N}}
\nc{\frakO}{{\mathfrak O}}
\nc{\frakP}{{\mathfrak P}}
\nc{\frakQ}{{\mathfrak Q}}
\nc{\frakR}{{\mathfrak R}}
\nc{\frakS}{{\mathfrak S}}
\nc{\frakT}{{\mathfrak T}}
\nc{\frakU}{{\mathfrak U}}
\nc{\frakV}{{\mathfrak V}}
\nc{\frakW}{{\mathfrak W}}
\nc{\frakX}{{\mathfrak X}}
\nc{\frakY}{{\mathfrak Y}}
\nc{\frakZ}{{\mathfrak Z}}
\nc{\bbA}{{\mathbb A}}
\nc{\bbB}{{\mathbb B}}
\nc{\bbC}{{\mathbb C}}
\nc{\bbD}{{\mathbb D}}
\nc{\bbE}{{\mathbb E}}
\nc{\bbF}{{\mathbb F}} \nc{\bbf}{{\mathbf f}}
\nc{\bbG}{{\mathbb G}}
\nc{\bbH}{{\mathbb H}}
\nc{\bbI}{{\mathbb I}}
\nc{\bbJ}{{\mathbb J}}
\nc{\bbK}{{\mathbb K}}
\nc{\bbL}{{\mathbb L}}
\nc{\bbM}{{\mathbb M}}
\nc{\bbN}{{\mathbb N}}
\nc{\bbO}{{\mathbb O}}
\nc{\bbP}{{\mathbb P}}
\nc{\bbQ}{{\mathbb Q}}
\nc{\bbR}{{\mathbb R}}
\nc{\bbS}{{\mathbb S}}
\nc{\bbT}{{\mathbb T}}
\nc{\bbU}{{\mathbb U}}
\nc{\bbV}{{\mathbb V}}
\nc{\bbW}{{\mathbb W}}
\nc{\bbX}{{\mathbb X}}
\nc{\bbY}{{\mathbb Y}}
\nc{\bbZ}{{\mathbb Z}}
\nc{\calA}{{\mathcal A}}
\nc{\calB}{{\mathcal B}}
\nc{\calC}{{\mathcal C}}
\nc{\calD}{{\mathcal D}}
\nc{\calE}{{\mathcal E}}
\nc{\calF}{{\mathcal F}}
\nc{\calG}{{\mathcal G}}
\nc{\calH}{{\mathcal H}}
\nc{\calI}{{\mathcal I}}
\nc{\calJ}{{\mathcal J}}
\nc{\calK}{{\mathcal K}}
\nc{\calL}{{\mathcal L}}
\nc{\calM}{{\mathcal M}}
\nc{\calN}{{\mathcal N}}
\nc{\calO}{{\mathcal O}}
\nc{\calP}{{\mathcal P}}
\nc{\calQ}{{\mathcal Q}}
\nc{\calR}{{\mathcal R}}
\nc{\calS}{{\mathcal S}}
\nc{\calT}{{\mathcal T}}
\nc{\calU}{{\mathcal U}}
\nc{\calV}{{\mathcal V}}
\nc{\calW}{{\mathcal W}}
\nc{\calX}{{\mathcal X}}
\nc{\calY}{{\mathcal Y}}
\nc{\calZ}{{\mathcal Z}}

\nc{\scrA}{{\mathscr A}}
\nc{\scrB}{{\mathscr B}}
\nc{\scrR}{{\mathscr R}}

\nc{\bnu}{{\bar{ \nu}}}

\nc{\olO}{\bar{\calO}}

\nc{\al}{{\alpha}} 
\nc{\be}{{\beta}}
\nc{\ga}{{\gamma}} \nc{\Ga}{{\Gamma}}
 \nc{\hGa}{\hat{\Gamma}}
\nc{\ve}{{\varepsilon}} 
\nc{\la}{{\lambda}} \nc{\La}{{\Lambda}}
\nc{\om}{\omega} \nc{\Om}{\Omega} 
\nc{\sig}{{\sigma}} \nc{\Sig}{{\Sigma}}

\nc{\tnb}{\psi_{\rm tame}}
\nc{\oM}{\overline{{M}}}
\nc{\op}{{\on{op}}}
\nc{\ad}{{\on{ad}}}
\nc{\alg}{{\on{alg}}}
\nc{\Ad}{{\on{Ad}}}
\nc{\Adm}{{\on{Adm}}} \nc{\aff}{{\on{af}}}
\nc{\Aut}{{\on{Aut}}}
\nc{\Bun}{{\on{Bun}}}
\nc{\cha}{{\on{char}}}
\nc{\der}{{\on{der}}}
\nc{\Der}{{\on{Der}}}
\nc{\diag}{{\on{diag}}}
\nc{\End}{{\on{End}}}
\nc{\Fl}{{\calF\!\ell}}
\nc{\Tr}{{\on{Transp}}}
\nc{\TR}{{\calT\!\calR}}
\nc{\Gal}{{\on{Gal}}}
\nc{\Gr}{{\on{Gr}}}
\nc{\rH}{{\on{H}}}
\nc{\Hom}{{\on{Hom}}}
\nc{\IC}{{\on{IC}}}
\nc{\id}{{\on{id}}}
\nc{\Id}{{\on{Id}}}
\nc{\ind}{{\on{ind}}}
\nc{\Ind}{{\on{Ind}}}
\nc{\Lie}{{\on{Lie}}}
\nc{\Pic}{{\on{Pic}}}
\nc{\pr}{{\on{pr}}}
\nc{\Res}{{\on{Res}}}
\nc{\res}{{\on{res}}} \nc{\Sat}{{\on{Sat}}}
\nc{\s}{{\on{sc}}}
\nc{\drv}{{\on{der}}}
\nc{\sgn}{{\on{sgn}}}
\nc{\Spec}{{\on{Spec}}}\nc{\Spf}{\on{Spf}} 
\nc{\Sph}{\on{Sph}}
\nc{\St}{{\on{St}}}
\nc{\tr}{{\on{tr}}}
\nc{\Mod}{{\mathrm{-Mod}}}
\nc{\Hilb}{{\on{Hilb}}} 
\nc{\Ext}{{\on{Ext}}} 
\nc{\vs}{{\on{Vec}}}
\nc{\ev}{{\on{ev}}}
\nc{\nO}{{\breve{\calO}}}
\nc{\tS}{{\tilde{S}}}
\nc{\spe}{{\on{sp}}}
\nc{\loc}{{\on{loc}}}

\nc{\nscrR}{{\mathscr{R}^{\on{nr}}}}

\nc{\GL}{{\on{GL}}}
\nc{\U}{{\on{U}}}
\nc{\Gl}{\on{Gl}} 
\nc{\GSp}{{\on{GSp}}}
\nc{\gl}{{\frakg\frakl}}
\nc{\SL}{{\on{SL}}} 
\nc{\SU}{{\on{SU}}} 
\nc{\SO}{{\on{SO}}}
\nc{\PGL}{{\on{PGL}}}

\nc{\Conv}{{\on{Conv}}}
\nc{\Rep}{{\on{Rep}}}
\nc{\Dom}{{\on{Dom}}}
\nc{\red}{{\on{red}}}
\nc{\act}{{\on{act}}}
\nc{\nr}{{\on{nr}}}
\nc{\ctf}{{\on{ctf}}}

\nc{\str}{{\on{-}}} 
\nc{\os}{{\bar{s}}}
\nc{\oeta}{{\bar{\eta}}}

\nc{\hookto}{\hookrightarrow}
\nc{\longto}{\longrightarrow}
\nc{\leftto}{\leftarrow}
\nc{\onto}{\twoheadrightarrow}
\nc{\lonto}{\twoheadleftarrow}

\nc{\uG}{{\underline{G}}}
\nc{\uA}{{\underline{A}}}
\nc{\uS}{{\underline{S}}}
\nc{\uT}{{\underline{T}}}
\nc{\uM}{{\underline{M}}}
\nc{\uP}{{\underline{P}}}
\nc{\uB}{{\underline{B}}}
\nc{\uN}{{\underline{N}}}

\nc{\ucG}{{\underline{\calG}}}
\nc{\ucA}{{\underline{\calA}}}
\nc{\ucS}{{\underline{\calS}}}
\nc{\ucT}{{\underline{\calT}}}
\nc{\ucM}{{\underline{\calM}}}
\nc{\ucP}{{\underline{\calP}}}
\nc{\ucN}{{\underline{\calN}}}

\nc{\bF}{{\breve{F}}}

\nc{\oFl}{{\overline{\Fl}}} 
\nc{\bU}{{\overline{U}}}
\nc{\tGr}{{\tilde{\Gr}}}
\nc{\cGr}{\calG\! r}
\nc{\oGr}{\overline{\on{Gr}}} 
\nc{\ocGr}{\overline{\calG\! r}}
\nc{\co}{{\colon}}
\nc{\sch}[1]{(Sch/{#1})}
\nc{\HypLoc}[1]{HypLoc({#1})}

\nc{\ohtimes}{\stackrel{!}{\otimes}}
\nc{\boxtilde}{\widetilde{\boxtimes}}
\nc{\vstar}{{\varhexstar}}

\nc{\Div}{\on{Div}}

\nc{\bslash}{\backslash}
\nc{\algQl}{{\bar{\bbQ}_\ell}}
\nc{\sF}{{\bar{F}}}
\nc{\nF}{{\breve{F}}}
\nc{\nW}{{W^{\on{nr}}}}
\nc{\sk}{{\bar{k}}}
\nc{\cont}{\on{c}}
\nc{\Supp}{\on{Supp}}
\nc{\blt}{\bullet}  
\nc{\dom}{\on{dom}}
\nc{\scon}{{\on{sc}}} 
\nc{\Affine}{\on{Aff}} 
\nc{\nscrA}{\mathscr{A}^{\on{nr}}} 
\nc{\nfraka}{{\bbf^{\on{nr}}}}
\nc{\ran}{{\rangle}}
\nc{\lan}{{\langle}}
\nc{\bk}{{\bar{k}}}
\nc{\tF}{{\tilde{F}}}
\nc{\sS}{{\bar{S}}}
\nc{\LG}{{^\text{L}\hspace{-0.04cm}G}}
\nc{\LL}{{^\text{L}\hspace{-0.07cm}L}}

\nc{\pot}[1]{ [\hspace{-0,5mm}[ {#1} ]\hspace{-0,5mm}] }
\nc{\rpot}[1]{ (\hspace{-0,7mm}( {#1} )\hspace{-0,7mm}) }

\nc{\defined}{\hspace{0.1cm}\stackrel{\text{\tiny \rm def}}{=}\hspace{0.1cm}}

\begin{document}

\title[The test function conjecture for parahoric local models]{The test function conjecture for \\ parahoric local models}
\author[T.\,J.\,Haines and T.\,Richarz]{by Thomas J. Haines and Timo Richarz}

\address{Department of Mathematics, University of Maryland, College Park, MD 20742-4015, DC, USA}
\email{tjh@math.umd.edu}

\address{Fachbereich Mathematik, TU Darmstadt, Schlossgartenstrasse 7, 64289 Darmstadt, Germany}
\email{richarz@mathematik.tu-darmstadt.de}

%\address{Faculty of Mathematics, University of Duisburg-Essen, Thea-Leymann-Str. 9, 45127 Essen, Germany}
%\email{timo.richarz@uni-due.de}

\thanks{Research of T.H.~partially supported by NSF DMS-1406787 and by Simons Fellowship 399424, and research of T.R.~ partially funded by the Deutsche Forschungsgemeinschaft (DFG, German Research Foundation) - 394587809.}

\maketitle

\begin{abstract}
 We prove the test function conjecture of Kottwitz and the first named author for local models of Shimura varieties with parahoric level structure, and their analogues in equal characteristic.
\end{abstract}

\setcounter{tocdepth}{1}
\tableofcontents
\setcounter{section}{0}

%\pagebreak

\thispagestyle{empty}

\section{Introduction}

\medskip
A prototypical example of a Shimura variety is the $\bbQ$-space of isomorphism classes of $g$-dimensional principally polarized abelian varieties with level structure. If the level at a prime $p$ is {\it parahoric}, one can construct an integral model of this space over $\bbZ_{(p)}$ by considering isogeny chains of abelian schemes having the same shape as the lattice chain which determines the parahoric subgroup of ${\rm GSp}_{2g}(\mathbb Q_p)$. Typically the resulting schemes have bad reduction. The {\it local models} serve as a tool to understand the singularities arising in the reduction modulo $p$. They are projective schemes over $\bbZ_{(p)}$ defined in terms of linear algebra -thus are easier to handle- and are \'etale locally isomorphic to the integral model for the Shimura variety. The study of local models for these and more general Shimura varieties started with the work of Deligne and Pappas \cite{DP94}, Chai and Norman \cite{CN92}, de Jong \cite{dJ93} and was formalized to some degree in the book of Rapoport and Zink \cite{RZ96}, building on their earlier results. Much work has been done in the intervening years, and we refer to the survey article of Pappas, Rapoport and Smithling \cite{PRS13} for more information.

Recently, Kisin and Pappas \cite{KP} constructed integral models for many Shimura varieties of abelian type with a parahoric level structure whenever the underlying group $G$ splits over a tamely ramified extension (assuming $p > 2$). One application of such models is to Langlands' program \cite{La76} to compute the Hasse-Weil zeta function of a Shimura variety in terms of automorphic $L$-functions. The zeta function has a local factor at $p$ which is determined by the points in the reduction modulo $p$ of the integral model, as well as its \'etale local structure - more precisely the sheaf of nearby cycles. In pursuing the Langlands-Kottwitz approach to the calculation of the (semi-simple) Lefschetz number, one needs to identify the \emph{test function} that is plugged into the counting points formula which resembles the geometric side of the Arthur-Selberg trace formula. This is the purpose of the test function conjecture promulgated in \cite[Conj 6.1.1]{Hai14}.

\'Etale locally the integral models of Kisin and Pappas are isomorphic to the local models constructed by Pappas and Zhu \cite{PZ13}. The determination of the nearby cycles reduces to that of the local model. The aim of the present manuscript is to prove the test function conjecture for these local models and their equal characteristic analogues, i.e.,\,to express the (semi-simple) trace of Frobenius function on the sheaf of nearby cycles on the local model in terms of automorphic data as predicted by the conjecture. We refer to the survey articles of Rapoport \cite{Ra90}, \cite{Ra05} and of the first named author \cite{Hai05}, \cite{Hai14} for how local models may be used in the Langlands-Kottwitz method.

\subsection{Formulation of the main result} \label{Form_Result} Let $p$ be a prime number. Let $F$ be a non-archimedean local field with ring of integers $\calO_F$ and finite residue field $k_F$ of characteristic $p$ and cardinality $q$, i.e. either $F/\bbQ_p$ is a finite extension or $F\simeq \bbF_q\rpot{t}$ is a local function field. Let $\sF/F$ be a separable closure, and denote by $\Ga_F$ the Galois group with inertia subgroup $I_F$ and fixed geometric Frobenius lift $\Phi_F\in\Ga_F$.

We fix a triple $(G,\{\mu\},\calG)$ where $G$ is a connected reductive $F$-group, $\{\mu\}$ a (not necessarily minuscule) conjugacy class of geometric cocharacters defined over a finite separable extension $E/F$, and $\calG$ is a parahoric $\calO_F$-group scheme in the sense of Bruhat-Tits \cite{BT84} with generic fiber $G$. If $F/\bbQ_p$, we assume that $G$ splits after a tamely ramified extension. Attached to these data is the (flat) \emph{local model}
\[
M_{\{\mu\}}=M_{(G,\{\mu\},\calG)},
\]
which is a flat projective $\calO_E$-scheme, cf.\,Definition \ref{localmodel} (we are using the definitions of local model given in \cite{PZ13} if $F/\bbQ_p$ and in \cite{Zhu14}, \cite{Ri16a} if $F\simeq\bbF_q\rpot{t}$, which, unlike the prototypical definitions tied to Shimura varieties, are not explicitly moduli schemes). The generic fiber $M_{\{\mu\}, E}$ is naturally the Schubert variety in the affine Grassmannian of $G/F$ associated with the class $\{\mu\}$. The special fiber $M_{\{\mu\},k_E}$ is equidimensional, but not irreducible in general. For a detailed treatment of local models and the problem of finding moduli descriptions, we refer to the survey article \cite{PRS13}.

Fix a prime number $\ell\not= p$, and fix a square root of the $\ell$-adic cyclotomic character (cf.\,$\S$\ref{gctcoho}). Let $d_\mu$ be the dimension of the generic fiber $M_{\{\mu\}, E}$, and denote the normalized intersection complex by
\[
\IC_{\{\mu\}}\defined j_{!*}\algQl[d_\mu](\textstyle{d_\mu\over 2})\in D_c^b(M_{\{\mu\}, E},\algQl)
\]  
cf. \S \ref{GeoSatSec} below. Under the geometric Satake equivalence \cite{Gi, Lu81, BD, MV07, Ri14a, RZ15, Zhu}, the complex $\IC_{\{\mu\}}$ corresponds to the $^LG_E=\widehat{G}\rtimes \Ga_{E}$-representation $V_{\{\mu\}}$ of highest weight $\{\mu\}$ defined in \cite[6.1]{Hai14}, cf. Corollary \ref{GeoSat1_cor} below. Both $\widehat{G}$ and $V_{\{\mu\}}$ are taken over $\bar{\mathbb Q}_\ell$.

Let $E_0/F$ be the maximal unramified subextension of $E/F$, and let $\Phi_E=\Phi_{E_0}=\Phi_F^{[E_0:F]}$ and $q_E = q_{E_0} = q^{[E_0:F]}$. The semi-simple trace of Frobenius function on the sheaf of nearby cycles 
\[
\tau^{\on{ss}}_{\{\mu\}}\co M_{\{\mu\}}(k_{E})\to \algQl, \;\;\; x\mapsto (-1)^{d_\mu}\on{tr}^{\on{ss}}(\Phi_{E}\,|\, \Psi_{M_{\{\mu\}}}(\IC_{\{\mu\}})_{\bar{x}}),
\]
is naturally a function in the center $\calZ(G({E_0}),\calG({\calO_{E_0}}))$ of the parahoric Hecke algebra. This is of course well-known when $G$ is unramified (cf.\,\cite{PZ13} if $F/\bbQ_p$, and \cite{Ga01}, \cite{Zhu14} if $F\simeq\bbF_q\rpot{t}$). The general case is more subtle and is explained in section \ref{TFC_stmt_sec}.

The values of the function $\tau^{\on{ss}}_{\{\mu\}}$ on each Kottwitz-Rapoport stratum were determined in the Drinfeld case in \cite{Hai01} by matching Hecke algebra computations with earlier geometric calculations of Rapoport \cite{Ra90}. This inspired Gaitsgory \cite{Ga01} to prove the centrality of $\tau^{\on{ss}}_{\{\mu\}}$ for all $\{\mu\}$ and all split groups in the function-field setting; he also characterized the functions using the Bernstein isomorphism between the center of the Iwahori Hecke algebra and the spherical Hecke algebra. Translating Gaitsgory's method to the $p$-adic setting using explicit lattice chains, Ng\^{o} and the first named author handled in \cite{HN02} the cases $G=\Gl_n, \on{GSp}_{2n}$, again proving centrality and expressing the functions $\tau^{\rm ss}_{\{\mu\}}$  in terms of the Bernstein presentation for $\calH(G(E_0),\calG(\calO_{E_0}))$. Further explicit calculations of the values of $\tau^{\on{ss}}_{\{\mu\}}$ on each Kottwitz-Rapoport stratum are due to G\"ortz \cite{Goe04} for $G=\Gl_4, \Gl_5$, and to Kr\"amer \cite{Kr03} for ramified unitary groups. Rostami treated in his 2011 thesis \cite{Ro17} the case of unramified unitary groups. In \cite{Zhu15}, Zhu determined the values of $\tau^{\on{ss}}_{\{\mu\}}$ for quasi-split ramified unitary groups with a \emph{very} special level structure, i.e., $\calG$ is special parahoric and stays special parahoric when passing to the maximal unramified extension (e.g.~ if $G$ is unramified, i.e.\,quasi-split and split after an unramified extension, then very special is equivalent to hyperspecial). For general unramified $F$-groups $G$, the semi-simple trace of Frobenius on the nearby cycles is determined in \cite{PZ13} in terms of the Bernstein isomorphism. For general quasi-split groups $G$, but now very special level $\calG$, our main theorem below reduces to \cite[10.4.5]{PZ13}, \cite[\S6]{Zhu15}. Let us point out that our main result holds for general groups $G$ and general parahoric levels $\calG$ under the hypotheses in the beginning of \S \ref{Form_Result}; every connected reductive group over a local field admits a parahoric subgroup by the work of Bruhat-Tits. Our result is the first which is valid for not necessarily quasi-split groups and arbitrary parahoric level. Further, being ``spectral'' in nature, our characterization is tailor-made to building the connection with automorphic forms, cf.\, \S\ref{Relation_to_sec}.\bigskip

\noindent\textbf{Main Theorem} (The test function conjecture for parahoric local models).
{\it  Let $(G,\{\mu\},\calG)$ be a general triple as above. Let $E/F$ be a finite separable extension over which $\{\mu\}$ is defined, and let $E_0/F$ be the maximal unramified subextension. Then
$$
\tau^{\on{ss}}_{\{\mu\}} =  z^{\rm ss}_{\{\mu\}}
$$
where $z^{\rm ss}_{\{\mu\}} = z^{\rm ss}_{\calG, {\{\mu\}}} \in \calZ(G(E_0),\calG(\calO_{E_0}))$ is the unique function which acts on any $\calG(\calO_{E_0})$-spherical smooth irreducible $\bar{\mathbb Q}_\ell$-representation $\pi$ by the scalar
\[
 \on{tr}\Bigl(s(\pi)\;\big|\; \on{Ind}_{^LG_E}^{^LG_{E_0}}(V_{\{\mu\}})^{1\rtimes I_{E_0}}\Bigr),
\]
where $s(\pi)\in [\widehat{G}^{I_{E_0}}\rtimes \Phi_{E_0}]_{\on{ss}}/\widehat{G}^{I_{E_0}}$ is the Satake parameter for $\pi$ \cite{Hai15}. The function $q^{d_\mu/2}_{E_0}\tau^{\on{ss}}_{\{\mu\}}$ takes values in $\mathbb Z$ and is independent of $\ell \neq p$ and $q^{1/2} \in \bar{\mathbb Q}_\ell$.}

\bigskip

The construction of $s(\pi)$ is reviewed in section \ref{Sat_trans_sec}, and the values of $q^{d_\mu/2}_{E_0}\tau^{\rm ss}_{\{\mu\}}$ are studied in section \ref{values_sec}. We remark that in the mixed characteristic case, $\tau^{\rm ss}_{\{\mu\}}$ lives in the center $\mathcal Z(G'(E'_0),\calG'(\mathcal O_{E'_0}))$ of the $\bar{\mathbb Q}_\ell$-valued Hecke algebra attached to {\em function field} analogues $(G', \calG', E'_0)$ of $(G_{E_0}, \calG_{\mathcal O_{E_0}}, E_0)$; we are implicitly identifying this with $\calZ(G(E_0), \calG(\calO_{E_0}))$ via Lemma \ref{hecke_identification}. The definition of the local models $M_{\{\mu\}}$ depends on certain auxiliary choices (cf.\,Remark \ref{BDrmk} and \ref{PZBDGrass_rmk}), but the function $\tau^{\rm ss}_{\{\mu\}}$ depends canonically only on the data $(G,\{\mu\},\calG)$. In mixed characteristic, this is related to \cite[Remark 3.2]{PZ13} and is not at all obvious, but it follows from our main theorem.

%In the above mentioned approaches, the function $\tau^{\on{ss}}_{\{\mu\}}$ is either calculated explicitly on each Kottwitz-Rapoport stratum or is described in terms of the Iwahori-Matsumoto presentation of $\calH(G(E_0),\calG(\calO_{E_0}))$, except the formulation for very special level structures given in \cite[10.4.5]{PZ13}, \cite[\S6]{Zhu15} which is close to ours. For applications to Shimura varieties and moduli spaces of shtukas our formulation is well-suited to building the connection with automorphic forms, cf.\, \S\ref{Relation_to_sec}.

\begin{rmk}
The Main Theorem can also be proved for the mixed characteristic local models constructed by B.\,Levin \cite{Lev16}, where $G = {\rm Res}_{F'/F}G'$ for a tamely ramified $F'$-group $G'$ and a possibly wildly ramified extension $F'/F$, cf.~\cite{HaRi}. In particular, this settles the test function conjecture for all mixed characteristic local models and primes $p\geq 5$.
The remaining cases are effectively reduced to four series of wildly ramified groups in residue characteristics $p=2,3$: ramified unitary, orthogonal and exceptional type $E_6$ groups if $p=2$, and the ramified ``triality'' if $p=3$.  
We hope that combining our techniques with the wildly ramified mixed characteristic local models constructed by Louren\c{c}o \cite{Lou} will yield the test function conjecture in these cases as well. In Conjecture 21.4.1 of \cite{SW}, Scholze predicts the existence of local models which are characterized using his theory of diamonds. We expect our results to apply to those local models as well.
\end{rmk}

\subsection{Relation to the Test Function Conjecture for Shimura varieties}\label{Relation_to_sec}

The test function conjecture makes sense for all levels, but here we consider only the case of parahoric level.    Consider the Shimura data $({\bf G}, X, K^pK_p)$, where $K_p \subset {\bf G}(\mathbb Q_p)$ is a parahoric subgroup, with corresponding parahoric group scheme $\mathcal G/\mathbb Z_p$. Let $\{\mu\} = \{\mu_h\}$ for $\mu_h \in X_*({\bf G}_{\bar{\mathbb Q}_p})$ corresponding to $h \in X$. Let $E \subset \mathbb Q_p$ be the field of definition of $\{\mu\} \subset X_*({\bf G}_{\bar{\mathbb Q}_p})$, with $E_j/E$ the degree $j$ unramified extension. Let $E_0/\mathbb Q_p$ (resp.\,$E_{j0}/\mathbb Q_p$) be the maximal unramified subextension of $E/\mathbb Q_p$ (resp.\,$E_j/\mathbb Q_p$). Note $E_{j0} = \mathbb Q_{p^r}$ for $r = j[E_0:\mathbb Q_p]$.

%We assume $G := {\bf G}_{\mathbb Q_p}$ is a tamely ramified group, and consider the integral model $\mathscr %S_{K_p}/{\rm Spec}(\mathcal O_E)$ of Kisin-Pappas \cite{KP15} of the Shimura $E$-variety ${\rm Sh}_{K_p}({\bf G}, %X)$, and the corresponding local models $M_{\{\mu\}}$ of Pappas-Zhu \cite{PZ13}.

We assume $G := {\bf G}_{\mathbb Q_p}$ is a tamely ramified group, and consider the local models $M_{\{\mu\}}$ of Pappas-Zhu \cite{PZ13}, cf.\,Def.\,\ref{localmodel}. In order to compute the (semi-simple) local Hasse-Weil zeta function of ${\rm Sh}_{K_pK^p}({\bf G}, X)$ at a prime ideal of $E$, the Test Function Conjecture \cite{Hai14} specifies a function in $\mathcal Z(G(E_{j0}), \mathcal G(\mathcal O_{E_{j0}})) = \mathcal Z(G(\mathbb Q_{p^r}), \mathcal G(\mathbb Z_{p^r}))$ which should be plugged into the twisted orbital integrals in the counting points formula for ${\rm tr}^{\rm ss}(\Phi_{E}^j \, | \, H^*({\rm Sh}_{K_pK^p}({\bf G}, X)_{\bar{E}}\,, \,\bar{\mathbb Q}_\ell))$. Setting $I(V_{\{\mu\}}) := {\rm Ind}^{\,^LG_{E_{j0}}}_{\,^LG_{E_j}} (V_{\{\mu\}})$, the test function is predicted to be $q^{jd_\mu/2}_{E_0} \,Z_{I(V_{\{\mu\}})} \star 1_{\calG(\mathcal O_{E_{j0}})}$ cf. section \ref{Sat_trans_sec}, where $Z_{I(V_{\{\mu\}})}$ is a distribution which is associated, assuming the local Langlands correspondence, to an element in the stable Bernstein center. The function $z^{\rm ss}_{\calG_{\mathcal O_{E_{j0}}}, {\{\mu\}}}$ is an unconditional version of $Z_{I(V_{\{\mu\}})} \star 1_{\calG_{\mathcal O_{E_{j0}}}}$. By our Main Theorem 
\[
z^{\rm ss}_{\calG_{\mathcal O_{E_{j0}}}, {\{\mu\}}} = (-1)^{d_\mu}{\rm tr}^{\rm ss}\big(\Phi_{E_j} \, | \,\Psi_{M_{\{\mu\}} \otimes \mathcal O_{E_j}}({\rm IC}_{\{\mu\}})\big), 
\]
and therefore the test function is expressed in terms of the geometry of the local model $M_{\{\mu\}}$.

Similar statements are in force in the function-field setting, where $G$ is any group over $\mathbb F_q(\!(t)\!)$ and where ${\rm Sh}_{K_pK^p}({\bf G}, X)$ is replaced by a moduli space of global $\ucG$-shtukas over a connected smooth projective curve $X/\mathbb F_q$.  

In light of recent progress on the Langlands-Rapoport conjecture for the Shimura varieties ${\rm Sh}_{K_pK^p}({\bf G}, X)$ \cite{Zhou} and for moduli stacks of global $\ucG$-shtukas \cite{ArHa}, our Main Theorem allows one to pursue the Langlands-Kottwitz approach to the description of the cohomology of these spaces in terms of automorphic representations. In particular, this sort of local information could be helpful in situations where knowledge of global objects is lacking (for example in cases where the existence and properties of global Arthur/Langlands parameters have not yet been established). This will be addressed in future work.

\subsection{Strategy of proof}
The local model $M_{\{\mu\}}$ is not semistable in general (cf. \cite{PRS13}), and it is difficult to determine the value of $\tau^{\on{ss}}_{\{\mu\}}$ at a given point in the special fiber. On the other hand, the global cohomology of the nearby cycles is by construction the intersection cohomology of the generic fiber $M_{\{\mu\},E}$ which is well understood due to the geometric Satake isomorphism. The idea of the proof is to take advantage of the latter, and proceeds in three steps as follows. \smallskip\\
{\it a\textup{)} Reduction to minimal Levi subgroups of $G$.}\smallskip\\
{\it b\textup{)} Reduction from anisotropic modulo center groups to quasi-split groups.}\smallskip\\
{\it c\textup{)} Proof for quasi-split groups.}\smallskip\\
Let us comment on steps b) and c) first. If $G$ is quasi-split, then a) reduces the proof of the Main Theorem to the case of tori which is easy, and implies c), cf.\,\S \ref{quasi_split_sec}. If $G$ is anisotropic modulo center, then the local model $M_{\{\mu\}}$ has a single $k_E$-point in the special fiber, cf.\,\S \ref{reduction_aniso_mod_center}. An application of the Grothendieck-Lefschetz trace formula expresses $\tau^{\on{ss}}_{\{\mu\}}$ in terms of the cohomology of the generic fiber. Using {\rm c) and properties of the normalized transfer homomorphisms \cite{Hai14}, we can express $z^{\on{ss}}_{\{\mu\}}$ in terms of the same cohomology groups, which implies b). The main part of the manuscript is concerned with the proof of a) which is summarized as follows. 

Let $M\subset G$ be a minimal Levi which is in good position relative to $\calG$. As we already know that $\tau^{\on{ss}}_{\{\mu\}}$ is a central function, it is uniquely determined by its image under the \emph{injective} constant term map
\[
c_M\co \calZ(G(E_0),\calG(\calO_{E_0}))\hookto \calZ(M(E_0),M(E_0)\cap \calG(\calO_{E_0})),
\]
cf.\,section \ref{HA_CT_sec} for a review. Our aim is to show $c_M(\tau_{\{\mu\}}^{\on{ss}})=c_M(z_{\{\mu\}}^{\on{ss}})$.

\begin{rmk}[Other approaches]
If $G$ is an unramified reductive group, i.e., quasi-split and split over an unramified extension, then it is possible to determine $\tau_{\{\mu\}}^{\on{ss}}$ using the projection to the affine Grassmannian, cf. \cite[Thm 10.16]{PZ13}. Further, if $G$ is quasi-split but $\calG$ {\it very} special, then one may use the ramified geometric Satake equivalence \cite{Zhu15, Ri16a} to deduce the Main Theorem, cf. \cite[Thm 10.23; 10.4.5]{PZ13}. In our general set-up, both techniques are not available.
\end{rmk}

\subsubsection{Geometric constant terms} By our choice of Levi $M$, we can find a cocharacter $\chi\co \bbG_{m,\calO_F}\to \calG$ whose centralizer $\calM$ is a parahoric $\calO_F$-group scheme with generic fiber $M$. Attached to $\chi$ is by the dynamic method promulgated in \cite{CGP10} also the smooth $\calO_F$-subgroup scheme $\calP\subset \calG$ whose generic fiber is a minimal parabolic subgroup $P\subset G$ with Levi subgroup $M$. The natural maps $\calM\leftarrow\calP\to\calG$ give rise to the diagram of Beilinson-Drinfeld Grassmannians
\begin{equation}\label{BDGrassIntro}
\Gr_\calM\overset{q}{\leftarrow}\Gr_\calP\overset{p}{\to}\Gr_\calG,
\end{equation}
which are $\calO_F$-ind-schemes that degenerate the affine Grassmannian into the (twisted) affine flag variety, cf.\,\cite{PZ13} if $F/\bbQ_p$, and \cite{Zhu14} (tamely ramified) and \cite{Ri16a} if $F\simeq\bbF_q\rpot{t}$. In the equal characteristic case, similar families in global situations where considered by Gaitsgory \cite{Ga01} and Heinloth \cite{He10} before. The generic fiber of \eqref{BDGrassIntro} is the diagram of affine Grassmannians denoted by
\[
\Gr_M\overset{q_\eta}{\leftarrow}\Gr_P\overset{p_\eta}{\to}\Gr_G,
\]
and the special fiber of \eqref{BDGrassIntro} is the diagram on affine flag varieties denoted by
\[
\Fl_\calM\overset{q_s}{\leftarrow}\Fl_\calP\overset{p_s}{\to}\Fl_\calG,
\]
cf.\,\S \ref{gctBDgrass} below. Associated with these data are two pairs of functors: nearby cycles ($\Psi_\calG$,$\Psi_\calM$)  and pull-push functors $(\on{CT}_M,\on{CT}_\calM)$ - the geometric constant terms. The nearby cycles 
\[
\Psi_\calG\co D_c^b(\Gr_G)\longto  D_c^b(\Fl_\calG\times_s\eta),
\] 
take as in [SGA 7 XIII] values in the constructible bounded derived category of $\algQl$-complexes on $\Fl_{\bar{s}}$ compatible with a continuous action of $\Ga_F$. Replacing $\calG$ by $\calM$ (resp. $G$ by $M$), we also have $\Psi_\calM$. The (normalized) geometric constant term in the generic (resp. special) fiber is
\[
\on{CT}_M\co D_c^b(\Gr_G)\to D_c^b(\Gr_M) \;\;\;\text{(resp. $\on{CT}_\calM\co D_c^b(\Fl_\calG\times_s\eta)\to D_c^b(\Fl_\calM\times_s\eta)$)}
\]
given by $\on{CT}_M= (q_\eta)_!\circ (p_\eta)^*\lan\chi\ran$ (resp. $\on{CT}_\calM= (q_s)_!\circ (p_s)^*\lan\chi\ran$) where $\lan\chi\ran$ denotes a certain shift and twist associated with the cocharacter $\chi$, cf. Definition \ref{NormConsTermGrass}, \ref{gctnormal} below. The functor $\on{CT}_M$ is well-known in Geometric Langlands \cite{BD, BG02, MV07} whereas the functor $\on{CT}_\calM$ only appears implicitly in the literature, cf. \cite[Thm 4]{AB09}, \cite[\S 9]{HNY13}.

Under the sheaf function dictionary, the nearby cycles $\Psi_\calG$  are a geometrization of the Bernstein isomorphism identifying the spherical Hecke algebra as the center of the parahoric Hecke algebra (cf. \cite{Ga01}), and the geometric constant term $\on{CT}_\calM$ is a geometrization of the map $c_M$ in the following sense: by definition the local model $M_{\{\mu\}}$ is a closed reduced subscheme of $\Gr_{\calG}\otimes_{\calO_F}\calO_E$, and, up to a sign, the function $c_M(\tau^{\on{ss}}_{\{\mu\}})\in \calZ(M(E_0),\calM(\calO_{E_0}))$ is the function associated with the complex
\[
\on{CT}_\calM\circ \Psi_{\calG, \calO_E}(\IC_{\{\mu\}}),
\]
cf.\;\S \ref{testfunctions} below. The following result is the geometric analogue of the compatibility of the Bernstein isomorphism with the constant term map \cite[(11.11.1)]{Hai14}. \medskip

\noindent {\bf Theorem A.} {\it The usual functorialities of nearby cycles give a natural transformation of functors $D_c^b(\Gr_G)\to D_c^b(\Fl_\calM\times_s \eta)$ as
\[
\on{CT}_\calM \circ \Psi_{\calG} \;{\longto}\; \Psi_\calM \circ \on{CT}_M,
\]
which is an isomorphism when restricted to $\bbG_m$-equivariant complexes. Here $\bbG_m$-equivariant means with respect to the $\bbG_m$-action induced by the cocharacter $\chi$ on $\Gr_G$.}\medskip

When the group $G$ is split, $\calG$ is an Iwahori and $M$ a maximal torus (i.e.\,$\chi$ is regular), then --forgetting the Galois action-- Theorem A may be seen as a version of the result of Arkhipov and Bezrukavnikov \cite[Thm 4]{AB09} (cf. also \cite[\S 7]{Zhu14} for tamely ramified groups). Another interesting application of this result is given in the work of Heinloth, Ng\^{o} and Yun \cite[\S 9]{HNY13}. Let us remark that our proof is based on a general commutation result for nearby cycles with hyperbolic localization \cite[Thm 3.3]{Ri19} (cf.\,also \cite[Prop 5.4.1(2)]{Na16} for complex manifolds), and does not use Wakimoto filtrations. Along with \cite{Ri19} (adapted to ind-schemes in Theorem \ref{nearbyhlthm}), the main geometric input is the following result (cf.\,Theorem \ref{gctgeoBD}, \ref{PZgctgeoBD} below): the cocharacter $\chi$ induces a $\bbG_m$-action on $\Gr_\calG$ trivial on $\calO_F$. Let $(\Gr_\calG)^0$ denote the fixed points, and let $(\Gr_\calG)^+$ be the attractor ind-scheme in the sense of Drinfeld \cite{Dr} (cf. \S \ref{CollectGm}). There are natural maps 
\[
(\Gr_\calG)^0{\leftarrow}(\Gr_\calG)^+{\to}\Gr_\calG,
\]
which relate to \eqref{BDGrassIntro} as follows.\medskip

\noindent {\bf Theorem B.} {\it The maps \eqref{BDGrassIntro} induce a commutative diagram of $\calO_F$-ind-schemes
\[
\begin{tikzpicture}[baseline=(current  bounding  box.center)]
\matrix(a)[matrix of math nodes, 
row sep=1.5em, column sep=2em, 
text height=1.5ex, text depth=0.45ex] 
{\Gr_\calM & \Gr_{\calP} & \Gr_\calG \\ 
(\Gr_{\calG})^0& (\Gr_{\calG})^+& \Gr_{\calG}, \\}; 
\path[->](a-1-2) edge node[above] {}  (a-1-1);
\path[->](a-1-2) edge node[above] {}  (a-1-3);
\path[->](a-2-2) edge node[below] {}  (a-2-1);
\path[->](a-2-2) edge node[below] {} (a-2-3);
\path[->](a-1-1) edge node[left] {$\iota^0$} (a-2-1);
\path[->](a-1-2) edge node[left] {$\iota^+$} (a-2-2);
\path[->](a-1-3) edge node[left] {$\id$} (a-2-3);
\end{tikzpicture}
\]
where the maps $\iota^0$ and $\iota^+$ are closed immersions which are open immersions on reduced loci. There are the following properties:\smallskip\\
i\textup{)} In the generic fiber, the maps $\iota^0_F$ and $\iota^+_F$ are isomorphisms.\smallskip\\
ii\textup{)} If $F\simeq\bbF_q\rpot{t}$ and $G=G_0\otimes_{\bbF_q} F$ is constant, then $\iota^0$ \textup{(}resp. $\iota^+$\textup{)} identifies $\Gr_\calM$ \textup{(}resp. $\Gr_{\calP}$\textup{)} as the flat closure of the generic fiber in $(\Gr_{\calG})^0$ \textup{(}resp. in $(\Gr_{\calG})^+$\textup{)}.}\medskip

In down to earth terms, the $\bbG_m$-action equips $\Gr_\calG$ with the stratification $(\Gr_\calG)^+\to \Gr_\calG$ where the image of each connected component of $(\Gr_\calG)^+$ defines a locally closed stratum (cf.\,Remarks \ref{locally_closed_stratum_rem}, \ref{locally_closed_stratum_rem2}). The flat strata belong to $\Gr_\calP$, and all other strata are concentrated in the special fiber and belong to $(\Gr_\calG)^+\bslash \iota^+(\Gr_\calP)$. Theorem A says that the cohomology of the nearby cycles on the strata belonging to $(\Gr_\calG)^+\bslash \iota^+(\Gr_\calP)$ vanishes.

\subsubsection{Translation to the local model} \label{local_model_intro_sec}
Let us explain what Theorem A means when translated to the local model (cf.\,Remarks \ref{locally_closed_stratum_rem}, \ref{locally_closed_stratum_rem2}). The local model $M_{\{\mu\}}$ is a closed subscheme of $\Gr_\calG$ which is stable under the $\bbG_m$-action. It induces a stratification of $M_{\{\mu\}}$ into locally closed strata 
\[
M_{\{\mu\}}=\coprod_{w}(M_{\{\mu\}})^+_w,
\]
where $w$ runs over the connected components of the fixed points $(M_{\{\mu\}})^0$, and the underlying topological space of $(M_{\{\mu\}})^+_w$ is the subspace of points $x$ such that the limit $\on{lim}_{\la\to 0}\chi(\la)\cdot x$ lies in the component for $w\in \pi_0((M_{\{\mu\}})^0)$. If $G$ is quasi-split, and $\chi$ is regular, then $\pi_0((M_{\{\mu\}})^0)$ can be identified with the $\{\mu\}$-admissible set in the sense of Kottwitz-Rapoport, cf.\,Corollary of Theorem C below. \medskip

\noindent {\bf Example.} For $G=\Gl_{2,\bbQ_p}$, $\{\mu\}$ minuscule, and $\calG$ an Iwahori, the local model $M_{\{\mu\}}$ is the blow up of $\bbP^1_{\bbZ_p}$ in the point $\{0\}_{\bbF_p}$ of the special fiber. The generic fiber is $\bbP^1_{\bbQ_p}$, and the special fiber consists of two $\bbP^1_{\bbF_p}$'s meeting transversally at a single point $\{e\}_{\bbF_p}$. Choose $\chi\co \bbG_{m,\bbZ_p}\to \calG, \la \mapsto \diag(\la,1)$. Then $\calM$ is the maximal split diagonal torus in $\calG$. The $\bbG_m$-action on $\bbP^1_{\bbZ_p}$, given in coordinates by $\la*[x_0:x_1]=[\la\cdot x_0:x_1]$, lifts uniquely to $M_{\{\mu\}}$, and agrees with the action constructed from $\chi$. Then the $\bbG_m$-fixed points are
\[
(M_{\{\mu\}})^{\bbG_m}=\{0\}_{\bbZ_p} \amalg\{\infty\}_{\bbZ_p}\amalg \{e\}_{\bbF_p},
\]
and $M_{\{\mu\}}$ decomposes into the three strata $\bbA^1_{\bbZ_p}\amalg \{\infty\}_{\bbZ_p}$ and $\bbA^1_{\bbF_p}$. The first pair of strata are flat, and contained in the smooth locus of $M_{\{\mu\}}$. Up to our choice of normalization, we get on compact cohomology
\[
\bbH^*_c(\bbA^1_{\bar{\bbF}_p}\amalg \{\infty\}_{\bar{\bbF}_p},\Psi_\calG(\algQl))=\algQl[-2](-1)\oplus\algQl.
\]
The non-flat stratum passes through the singularity $\{e\}_{\bbF_p}$, and a calculation shows
\[
\bbH^*_{c}(\bbA^1_{\bar{\bbF}_p},\Psi_\calG(\algQl))=0,
\]
which is in accordance with Theorem A. \medskip

For general split groups with Iwahori level and minuscule $\{\mu\}$ the situation is similar (cf.~Remark \ref{geometry_minuscule_rmk}): the flat $\bbG_m$-strata are contained in the smooth locus of the local model, and the compact cohomology of the nearby cycles on the non-flat $\bbG_m$-strata vanishes by virtue of Theorem A. In particular, the sheaf $\Psi_\calG(\IC_{\{\mu\}})$ is (a posteriori) uniquely determined by its values on the $\bbG_m$-strata lying in the smooth locus, and we do not need to bother about the type of the singularities in the special fiber of $M_{\{\mu\}}$. 

For general groups with parahoric level and general $\{\mu\}$ we make use of the fact that the functor in the generic fiber $\on{CT}_M$ corresponds under the geometric Satake equivalence to the restriction of ${^LG}$-representations $V\mapsto V|_{{^LM}}$ where ${^LM}\subset {^LG}$ is the closed subgroup associated with $M\subset G$. Hence, we know that the complex 
\[
\on{CT}_M(\IC_{\{\mu\}}),
\]
decomposes according to the irreducible ${^LM}$-representations appearing in $V_{\{\mu\}}|_{^LM}$ with strictly positive multiplicities. Hence, $c_M(\tau^{\on{ss}}_{\{\mu\}})$ decomposes accordingly, cf.\,\,Lemma \ref{pC_lem}. As $c_M(z^{\on{ss}}_{\{\mu\}})$ behaves similarly by Lemma \ref{C_lem}, we conclude $c_M(\tau^{\on{ss}}_{\{\mu\}})=c_M(z^{\on{ss}}_{\{\mu\}})$ by steps b), c) above, and hence $\tau^{\on{ss}}_{\{\mu\}}= z^{\on{ss}}_{\{\mu\}}$.

\subsection{Other results} Let us mention other results in the paper which we think are of independent interest. We use the following version of the geometric Satake equivalence
\[
\om_G\co \Sat_G\overset{\simeq}{\longto} \Rep_{\algQl}(^LG), \;\calA\mapsto \bigoplus_{i\in\bbZ}\bbH^i(\Gr_G,\calA)(\nicefrac{i}{2}),
\]
where $\Sat_G$ is the full subcategory of $\on{Perv}_{L^+G}(\Gr_G)$ generated by the intersection complexes on the $L^+G$-orbits, and $\algQl$-local systems on $\Spec(F)$ which are trivial after a finite extension, cf. \S \ref{gctcoho} for details. We consider the composition of functors
\[
\on{CT}_\calM\circ\Psi_\calG\co \Sat_G\to D_c^b(\Fl_\calG\times_s\eta)\to D_c^b(\Fl_\calM\times_s\eta).
\]
Let us specialize to the case where $\calM$ is a \emph{very} special parahoric group scheme, i.e. $\calM\otimes \calO_{\bF}$ is special in the sense of Bruhat-Tits where $\bF/F$ is the completion of the maximal unramified extension. In this case, there is the ramified version of the geometric Satake equivalence \cite{Zhu14, Ri16a} which gives an equivalence of Tannakian categories
\[
\om_\calM\co \Sat_\calM\overset{\simeq}{\longto}\Rep_{\algQl}({^LM}_r), \;\calA\mapsto \bigoplus_{i\in\bbZ}\bbH^i(\Gr_G,\calA)(\nicefrac{i}{2}),
\]
where $\Sat_\calM$ is a certain subcategory of $\on{Perv}_{L^+\calM}(\Fl_\calM\times_s\eta)$, and ${^LM}_r=\widehat{M}^{I_F}\rtimes \Ga_F$ considered as a closed subgroup of ${^LG}=\widehat{G}\rtimes\Ga_F$. Note that the group of invariants $\widehat{M}^{I_F}$ is a possibly disconnected reductive group over $\algQl$. The following result is a generalization of \cite[Thm 4]{AB09} to not necessarily split reductive groups. \medskip

\noindent{\bf Theorem C.} {\it Let $\calM$ be very special, and hence $M$ is quasi-split. For every $\calA\in \Sat_G$, one has $\on{CT}_\calM\circ \Psi_\calG(\calA)\in \Sat_\calM$, and there is a commutative diagram of Tannakian categories 
\[
\begin{tikzpicture}[baseline=(current  bounding  box.center)]
\matrix(a)[matrix of math nodes, 
row sep=1.5em, column sep=2em, 
text height=1.5ex, text depth=0.45ex] 
{\Sat_G&\Sat_{\calM} \\ 
\Rep_{\algQl}(^LG)&\Rep_{\algQl}({^LM}_r), \\}; 
\path[->](a-1-1) edge node[above] {$\on{CT}_\calM\circ\Psi_\calG$} (a-1-2);
\path[->](a-2-1) edge node[above] {$\res$} (a-2-2);
\path[->](a-1-1) edge node[right] {$\om_G$} (a-2-1);
\path[->](a-1-2) edge node[right] {$\om_\calM$} (a-2-2);
\end{tikzpicture}
\] 
where $\on{res}\co V\mapsto V|_{{^LM}_r}$ is the restriction of representations.  \medskip
}

The theorem allows us to calculate the support of the nearby cycles $\on{Supp}(\Psi_\calG(\IC_{\{\mu\}}))$, and we obtain the following result.\medskip

\noindent{\bf Corollary.} {\it The smooth locus $(M_{\{\mu\}})^{\on{sm}}$ is fiberwise dense in $M_{\{\mu\}}$, and on reduced subschemes
\[
(M_{\{\mu\},k})_\red=\on{Supp}(\Psi_\calG(\IC_{\{\mu\}}))=\bigcup_{w\in\Adm_{\{\mu\}}^\bbf}\Fl_\calG^{\leq w},
\]
where $\Adm_{\{\mu\}}^\bbf$ denotes the admissible set in the sense of Kottwitz-Rapoport, cf. \S \ref{AppLocMod} for details. \medskip
}

If $G$ splits over a tamely ramified extension and $p\nmid |\pi_1(G_\der)|$, then the corollary is a weaker form of \cite[Thm 9.3]{PZ13} (if $F/\bbQ_p$) and \cite[Thm 3.8]{Zhu14} (if $F\simeq\bbF_q\rpot{t}$). Hence, the result is new when either $p \mid |\pi_1(G_\der)|$ or $F\simeq\bbF_q\rpot{t}$ and $G$ splits over a wildly ramified extension. Let us point out that these are still classes containing infinite families: the first case happens e.g. for $\on{PGl}_n$ with $p\mid n$, the second case e.g. for unitary groups in characteristic $2$ as follows. Let $q=2$ and $F=\bbF_2\rpot{t}$, and assume $n > 2$. Outer forms of $\on{Sl}_{n,F}$ are classified by the Galois cohomology set $H^1(F,\bbZ/2)$ which by Artin-Schreier theory is equal to 
\[
\bbF_2\rpot{t}/\mathscr{P}\bbF_2\rpot{t}=\bbF_2[t^{-1}]/\mathscr{P}\bbF_2[t^{-1}],
\]
where $\mathscr{P}:=(\str)^2-\id$ is considered as an additive map. For distinct odd integers $a,b\in \bbZ_{<0}$, the classes $[t^a]$ and $[t^b]$ are different, and hence give rise to non-isomorphic special unitary groups. Explicitly, if $F_a$ (resp. $F_b$) denotes the quadratic field extension determined by the equation $X^2-X-t^a$ (resp. $X^2-X-t^b$), then the associated quasi-split ramified special unitary groups $\on{SU}_n(F_a/F)$ and $\on{SU}_n(F_b/F)$ are non-isomorphic. It would be interesting to give a moduli theoretic description of the local models for these cases.

\subsection{Acknowledgements} It is a pleasure for us to thank the following people for inspiration, guidance, and useful conversations, large and small, about this project: Patrick Brosnan, Ulrich G\"ortz, Jochen Heinloth, Eugen Hellmann, George Pappas, Michael Rapoport, Larry Washington and Torsten Wedhorn. 

The authors are grateful to the Simons Foundation, to Michael Rapoport and the University of Bonn, and to Benoit Stroh and the University of Paris VI, for providing financial and logistical support which made this research possible. Finally, the authors thank the referees for their very helpful comments and suggestions.

\subsection{Overview} In \S \ref{CollectGm} we collect some basic facts about $\bbG_m$-actions on (ind-)schemes. These are applied in \S \ref{gctaffgrass} and \S \ref{gctaffflag} to study $\bbG_m$-actions on affine Grassmannians and affine flag varieties, laying the foundations towards proving Theorem B, which is needed to deduce our Theorem A from Theorem \ref{nearbyhlthm}. In \S \ref{gctBDgrass} we study the degeneration of the $\bbG_m$-stratification on affine Grassmannians to affine flag varieties and prove Theorem B. Our geometric study is applied in \S \ref{constermflag} to the construction of geometric constant term functors on affine flag varieties and the proofs of Theorems A and C from the introduction. In the last section \S \ref{testfunctions}, we explain the reduction steps a)-c) in order to prove our Main Theorem.
 
\subsection{Conventions on ind-schemes}  Let $R$ be a ring. An $R$-space $X$ is a fpqc sheaf on the category of $R$-algebras, i.e. $X\co R{\text{-Alg}}\to \text{Sets}$ is a (covariant) functor that respects finite products, and such that, for every $R'\to R''$ faithfully flat, the set $X(R')$ is the equalizer of $X(R'')\rightrightarrows X(R''\otimes_{R'} R'')$. Let $\text{Sp}_R$ denote the category of $R$-spaces. It contains the category $\text{Sch}_R$ of $R$-schemes as a full subcategory. An $R$-ind-scheme is a (covariant) functor
\[
X\co R\text{-Alg}\to \text{Sets} 
\]
such that there exists a presentation as presheaves $X=\text{colim}_iX_i$ where $\{X_i\}_{i\in I}$ is a filtered system of $R$-schemes $X_i$ with transition maps being closed immersions. Note that filtered colimits exist in the category $\on{Sp}_R$, and agree with the colimits as presheaves. Hence, every $R$-ind-scheme is an $R$-space. The category of $R$-ind-schemes $\on{IndSch}_R$ is the full subcategory of $\on{Sp}_R$ whose objects are $R$-ind-schemes. If $X=\on{colim}_iX_i$ and $Y=\text{colim}_jY_j$ are presentations of ind-schemes and all $X_i$ are quasi-compact, then as sets
\[
\Hom_{\on{IndSch}_R}(X,Y)= \on{lim}_i\on{colim}_j\Hom_{\on{Sch}_R}(X_i,Y_j),
\]
because every map $X_i\to Y$ factors over some $Y_j$ (by quasi-compactness of $X_i$, cf.\,e.g.\,\cite[Lem 2.4]{Goe08}). The category $\text{IndSch}_R$ is closed under fiber products, i.e.,\,$\on{colim_{(i,j)}}(X_i\times_R Y_j)$ is a presentation of $X\times_R Y$. If $\calP$ is a property of schemes, then an $R$-ind-scheme $X$ is said to have $\text{ind-}\calP$ if there exists a presentation $X=\text{colim}_iX_i$ where each $X_i$ has property $\calP$. Similarly, if $\calP$ is a property of morphisms, then a map $f\co X\to Y$ of ind-schemes is said to have property $\calP$ if $f$ is schematic and for all $Y$-schemes $T$, the pullback $f\times_{Y}T$ has property $\calP$. 

\subsection{Notation} Let $F$ be a non-archimedean local field\footnote{Sometimes we consider more general fields which we indicate at the beginning of the section.} with ring of integers $\calO_F$ and finite residue field $k_F\simeq \bbF_q$ of characteristic $p>0$. Let $\sF$ be a fixed separable closure with ring of integral elements $\calO_{\sF}$, and residue field $\bar{k}_F$ (an algebraic closure of $k_F$). The field $\breve{F}$ is the completion of the maximal unramified subextension with ring of integers $\calO_{\breve{F}}$. Let $\Ga_F=\Gal(\sF/F)$ the Galois group with inertia subgroup $I_F\simeq \Ga_{\breve{F}}$ and quotient $\Sig_F=\Ga_F/I_F\simeq \Gal(\bar{k}_F/k_F)$. Let $\sig_F\in \Sig_F$ denote the arithmetic Frobenius, and write $\Phi_F = \sigma_F^{-1}$ for the geometric Frobenius. We normalize the valuation $|\str|_F\co F\to \bbQ$ so that an uniformizer in $\calO_F$ has valuation $q^{-1}$. If the field $F$ is fixed, we sometimes drop the subscript $F$ from the notation.  

We fix a prime $\ell\not= p$, and an algebraic closure $\algQl$ of $\bbQ_\ell$. We fix once and for all $q^{1/2}\in \algQl$, and the square root of the cyclotomic character $\Ga_F\to \bbZ_\ell^{\times}$ which maps any lift of $\Phi_F$ to $q^{-1/2}$, cf.\,$\S$\ref{gctcoho}. If $X$ is an $\bbF_q$-scheme and $x\in X(\bbF_q)$, then the geometric Frobenius $\Phi_F$ acts on $\algQl(-{1/2})_{\bar{x}}$ by $q^{1/2}$.

For a connected reductive $F$-group $G$, we denote by $\widehat{G}$ the Langlands dual group viewed as an algebraic group over $\algQl$. The action of the Galois group $\Ga_F$ on $\widehat{G}$ via outer automorphism is trivial restricted to $\Ga_E$ for any finite separable $E/F$ which splits $G$. Throughout the paper, we denote by $^LG=\widehat{G}\rtimes \Ga_F$ the Galois form of the $L$-group which we view via $\widehat{G}\rtimes \Ga_F=\varprojlim_E\, \widehat{G}\rtimes \Gal(E/F)$ as an algebraic group over $\algQl$. 

Our Hecke algebras $\mathcal H(G(F), \calG(\mathcal O_F))$ consist of functions taking values in the field $\bar{\mathbb Q}_\ell$. Convolution is always defined using the Haar measure giving volume 1 to $\calG(\mathcal O_F)$.  We let $\mathcal Z(G(F), \calG(\mathcal O_F))$ denote the center of $\mathcal H(G(F), \calG(\mathcal O_F))$.

%Recollection on G_m-actions
\section{Recollections on ${\mathbb G}_m$-actions}\label{CollectGm}
We recall some set-up and notation from \cite{Dr} and \cite{Ri19}. Let $R$ be a ring, and let $X$ be an $R$-scheme (or $R$-ind-scheme) equipped with an action of $\bbG_m$ which is trivial on $R$. There are three functors on the category of $R$-algebras
\begin{equation}\label{flow}
\begin{aligned}
\hspace{1cm} X^0\co&\; R'\longmapsto \Hom^{\bbG_{m}}_R(\Spec(R'), X)\\
\hspace{1cm} X^+\co& \; R'\longmapsto \Hom^{\bbG_{m}}_R((\bbA_{R'}^1)^+, X)\\
\hspace{1cm}X^-\co& \; R'\longmapsto \Hom^{\bbG_{m}}_R((\bbA^1_{R'})^-, X),
\end{aligned}
\end{equation}
where $(\bbA_R^1)^+$ (resp. $(\bbA_R^1)^-$) is $\bbA^1_R$ with the usual (resp. opposite) $\bbG_m$-action. The functor $X^0$ is the functor of $\bbG_m$-fixed points in $X$, and $X^+$ (resp. $X^-$) is called the attractor (resp. repeller). Informally speaking, $X^+$ (resp. $X^-$) is the space of points $x$ such that the limit $\lim_{\la\to 0}\la\cdot x$ (resp. $\lim_{\la\to \infty}\la\cdot x$) exists. The functors \eqref{flow} come equipped with natural maps
\begin{equation}\label{hyperloc}
X^0\leftarrow X^\pm \to X,
\end{equation}
where $X^\pm\to X^0$ (resp. $X^\pm\to X$) is given by evaluating a morphism at the zero section (resp. at the unit section). If $R$ is a field and if $X$ is a quasi-separated $R$-scheme of finite type, the functors $X^0$ and $X^\pm$ are representable by quasi-separated $R$-schemes of finite type, cf.\,\cite[Thm 1.4.2]{Dr} and \cite[Thm 2.24]{AHR}. If $X=\on{colim}_iX_i$ is an ind-scheme, and if each $X_i$ is $\bbG_m$-stable, quasi-separated and of finite type, then $X^0=\on{colim}_iX_i^0$ and $X^\pm=\on{colim}_iX_i^\pm$ are ind-schemes.

More generally, whenever the $\bbG_m$-action on a scheme $X$ is \'etale locally linearizable, i.e., the $\bbG_m$-action lifts - necessarily uniquely - to an \'etale cover which is affine over $R$ (cf.~\cite[Def.~1.6]{Ri19}), then the functors $X^0$ and $X^\pm$ are representable by \cite[Thm 1.8]{Ri19}. 
The property of being \'etale locally linearizable comes from the generalization of Sumihiro's theorem \cite[\S 2.3]{AHR}. 
The attribute ``linearizable'' refers to the fact that an affine $R$-scheme of finite presentation equipped with a $\bbG_m$-action can be  embedded equivariantly into affine space on which $\bbG_m$ acts linearly, cf.~\cite[Lem 2.21]{Ri19}.
In the present manuscript, the $\bbG_m$-actions are even Zariski locally linearizable, i.e. there exists a $\bbG_m$-invariant open affine cover (cf. Lemma \ref{loclinaffine}, \ref{loclinBD} below). Let us explain how \cite[Thm 1.8]{Ri19} generalizes to ind-schemes. We say a $\bbG_m$-action on an $R$-ind-scheme $X$ is \'etale (resp. Zariski) locally linearizable if there is a $\bbG_m$-stable presentation $X=\on{colim}_iX_i$ where the $\bbG_m$-action on each $X_i$ is \'etale (resp. Zariski) locally linearizable.

\begin{thm}\label{Gmthm} Let $X=\on{colim}_iX_i$ be an $R$-ind-scheme equipped with an \'etale locally linearizable $\bbG_m$-action. \smallskip\\
i\textup{)} The subfunctor $X^0$ of $X$ is representable by a closed sub-ind-scheme, and $X^0=\on{colim}_iX_i^0$. \smallskip\\
ii\textup{)} The functor $X^\pm$ is representable by an ind-scheme, and $X^\pm=\on{colim}_iX_i^\pm$. In particular, the map $X^\pm \to X$ is schematic. The map $X^\pm\to X^0$ is ind-affine with geometrically connected fibers.\smallskip\\
iii\textup{)} If $X=\on{colim}_iX_i$ is of ind-finite presentation \textup{(}resp.\,separated\textup{)}, so are $X^0$ and $X^\pm$.
\end{thm}  
\begin{proof} By definition of an ind-scheme, each $X_i$ is quasi-compact, and hence the schemes $X^0_i$ and $X_i^\pm$ are again quasi-compact by \cite[Thm 1.8 iii)]{Ri19}. Further, if $X_i\hookto X_j$ is a closed immersion, then $X_i^0=X_i\times_{X_j}X_j^0\hookto X_j^0$ is a closed immersion. This implies i). If $X_i$ is affine, then $X_i^\pm\subset X_i$ is a closed immersion by \cite[Lem 1.9 ii)]{Ri19}. It follows that if both $X_i$ and $X_j$ are affine, then $X_i^\pm\hookto X_j^\pm$ is a closed immersion. For the general case, choose an affine \'etale $\bbG_m$-equivariant cover $U\to X_j$. Then the following diagram of $R$-schemes 
\[
\begin{tikzpicture}[baseline=(current  bounding  box.center)]
\matrix(a)[matrix of math nodes, 
row sep=1.5em, column sep=2em, 
text height=1.5ex, text depth=0.45ex] 
{(U\times_{X_j}X_i)^\pm&X_i^\pm \\ 
U^\pm&X_j^\pm \\}; 
\path[->](a-1-1) edge node[above] {} (a-1-2);
\path[->](a-2-1) edge node[above] {} (a-2-2);
\path[->](a-1-1) edge node[right] {} (a-2-1);
\path[->](a-1-2) edge node[right] {} (a-2-2);
\end{tikzpicture}
\]
is cartesian which immediately follows from the definition. As $X_i\hookto X_j$ is a closed immersion, the map $U\times_{X_j}X_i\hookto U$ is a closed immersion of affine schemes. Hence, $(U\times_{X_j}X_i)^\pm\hookto U^\pm$ is a closed immersion by the affine case. As being a closed immersion is \'etale local on the target we conclude $X_i^\pm \hookto X^\pm_j$ is a closed immersion (note that $U^\pm \to X_j^\pm$ is \'etale surjective by \cite[Lem 1.10, 1.11]{Ri19}). Hence, $X^\pm=\on{colim}_iX_i^\pm$ is an ind-scheme, and $X^\pm\to X$ is schematic because for any quasi-compact test scheme $T\to X$ we have $X^\pm\times_XT =X_i^\pm\times_{X_i}T$ for $i>\!\!>0$. Since all maps $X_i^\pm\to X^0_i$ are affine (resp. geometrically connected) by \cite[Prop 1.17 ii)]{Ri19}, the map $X^\pm\to X^0$ is ind-affine (resp. geometrically connected) as well. Part ii) follows. If each $X_i$ is of finite presentation (resp. separated), so are $X_i^0$ and $X_i^\pm$ because the properties locally of finite presentation, quasi-compact and quasi-separated (resp. separated) are preserved by \cite[Thm 1.8 iii)]{Ri19}. Hence, $X^0=\on{colim}_iX_i^0$ and $X^\pm=\on{colim}_iX_i^\pm$ is of ind-finite presentation. This implies iii), and the theorem follows. 
\end{proof}

We shall use the following fact. %{\cc Part i) of the lemma is slightly weaker than in the previous version due to the Zariski hypothesis on $Y$. However, I suggest to stick with the weaker version since it unifies some reduction steps. Otherwise we would have to go through similar reductions in the $X^+$ case again. We could put a remark. What do you think?}

\begin{lem} \label{0-perm} Let $\Spec(R)$ be connected. Let $X$ \textup{(}resp. $Y$\textup{)} be an $R$-ind-scheme, endowed with an \'{e}tale \textup{(}resp. Zariski\textup{)} locally linearizable $\mathbb G_m$-action. Let $Y$ be separated, and let $X \rightarrow Y$ be a schematic $\mathbb G_m$-equivariant morphism. \smallskip\\
i\textup{)} If $X\to Y$ is locally of finite presentation \textup{(}resp.\,quasi-compact; resp.\,quasi-separated; resp.\,separated; resp.\,smooth; resp. proper\textup{)}, so is the morphism $X^0 \to Y^0$. \smallskip\\
ii\textup{)} If $X\to Y$ is locally of finite presentation \textup{(}resp.\,quasi-compact; resp.\,quasi-separated; resp.\,separated; resp.\, smooth\textup{)}, so is the morphism $X^\pm \to Y^\pm$.
\end{lem}
\begin{proof} The statement for $X^-\to Y^-$ follows from the statement for $X^+\to Y^+$ by inverting the $\bbG_m$-action, and it is enough to treat the latter. Further, all properties listed are stable under base change and fpqc local on the base, which will be used throughout the proof without explicit mentioning. 

Let $Y=\on{colim}_jY_j$ be a $\bbG_m$-stable presentation. Using $(X\times_YY_j)^0=X^0\times_{Y^0}Y_j^0$ (resp. $(X\times_YY_j)^+=X^+\times_{Y^+}Y_j^+$) and noting that the property of being \'etale locally linearizable is preserved under closed immersions, we reduce to the case that $Y=Y_j$ is a (quasi-compact) separated scheme. Hence, $X$ is a scheme as well (because $X\to Y$ is schematic). 

Let $U\to X$ (resp. $V\to Y$) be an $\bbG_m$-equivariant \'etale (resp. Zariski) cover with $U$ (resp. $V$) being a (resp. finite) disjoint union of affine schemes. As $Y$ is separated, the map $V\to Y$ is affine, i.e. the intersection of two open affines is again affine. The cartesian diagram of $R$-schemes
\[
\begin{tikzpicture}[baseline=(current  bounding  box.center)]
\matrix(a)[matrix of math nodes, 
row sep=1.5em, column sep=2em, 
text height=1.5ex, text depth=0.45ex] 
{U\times_YV & X\times_YV & V \\ 
U&X & Y, \\}; 
\path[->](a-1-1) edge node[above] {}  (a-1-2);
\path[->](a-1-2) edge node[above] {}  (a-1-3);
\path[->](a-2-1) edge node[below] {}  (a-2-2);
\path[->](a-2-2) edge node[below] {} (a-2-3);
\path[->](a-1-1) edge node[left] {} (a-2-1);
\path[->](a-1-2) edge node[left] {} (a-2-2);
\path[->](a-1-3) edge node[left] {} (a-2-3);
\end{tikzpicture}
\]
shows that the map $U\times_YV\to U$ is affine (because affine morphisms are stable under base change). Hence $U\times_YV$ is a disjoint union of affine schemes as well, and the $\bbG_m$-action on $X\times_YV$ is \'etale locally linearizable. By \cite[Thm 1.8 i), ii)]{Ri19} the map $V^0\to Y^0$ (resp. $V^+\to Y^+$) is \'etale surjective, and we reduce to the case that $Y$ is affine.

As $\Spec(R)$ is connected a $\bbG_m$-action on $Y$ is the same as a $\bbZ$-grading on its ring of global functions, and by \cite[Lem 1.9]{Ri19} both $Y^0\subset Y$ and $Y^+\subset Y$ are closed (affine) subschemes. Using $X^0=(X \times_Y Y^0)^0$ (resp. $X^+=(X \times_Y Y^+)^+$) and noting again that the property of being \'etale locally linearizable is preserved under closed immersions, we reduce in part i) (resp. in part ii)) to the case $Y=Y^0$ (resp. $Y=Y^+$). Now by \cite[Thm 1.8 iii)]{Ri19}, $X^0 \to Y$ satisfies each of the properties listed which $X\to Y$ satisfies (``proper'' is not listed there but this follows using that $X^0\subset X$ is a closed immersion). This shows i). 

For ii) note that the property of being locally of finite presentation is equivalent to the property of being limit preserving \cite[Tag 04AK]{StaPro}, and the latter is immediate from the definition of $X^+$. Now consider the map $Y=Y^+\to Y^0$ of affine schemes. The map $X^+\to Y$ factors as $X^+\to X^0\times_{Y^0}Y\to Y$, and using i), the map $X^0\times_{Y^0}Y\to Y$ has each of the properties listed. The map $X^+\to X^0$ being affine (cf. \cite[Cor 1.12]{Ri19}) implies that $X^+\to X^0\times_{Y^0}Y$ is affine, hence quasi-compact and separated, and ii) for the properties ``quasi-compact'' and ``(quasi)-separated'' follows.

It remains to treat the property ``smooth''. Consider the cover $U\to X$ again. Using \cite[Thm 1.8 ii)]{Ri19}, the map $U^+ \to X^+$ is \'etale surjective. Applying \cite[Tag 02K5]{StaPro} to the commutative diagram of $R$-schemes  
\[
\begin{tikzpicture}[baseline=(current  bounding  box.center)]
\matrix(a)[matrix of math nodes, 
row sep=1.5em, column sep=2em, 
text height=1.5ex, text depth=0.45ex] 
{ U^+& & X^+ \\ 
&Y&, \\}; 
\path[->](a-1-1) edge node[above] {}  (a-1-3);
\path[->](a-1-1) edge node[above] {}  (a-2-2);
\path[->](a-1-3) edge node[below] {}  (a-2-2);
\end{tikzpicture}
\]
we reduce to the case where $X$ is affine. By a standard reduction (cf.\,\cite[Lem 3.1]{Mar15}), we may further assume $Y$ (hence $X$) is Noetherian. Following the arguments in \cite[Lem 3.2]{Mar15} we proceed in two steps.\smallskip\\ 
{\it The map $X^+\to Y$ is smooth at all points in $X^0$.} Let $x\in X^0$, and denote by $y\in Y^0$ its image. We first consider the case where $\kappa(y)\simeq \kappa(x)$. Let $A$ (resp. $B$; resp. $C=B/J$) denote the coordinate rings of $Y$ (resp. $X$; resp. $X^+$). Let $A=\oplus_{i\in\bbZ}A_i$ (resp. $B=\oplus_{i\in \bbZ}B_i$; resp. $C=\oplus_{i\in \bbZ}C_i$) be the grading given by the $\bbG_m$-action. As $X\to Y$ is $\bbG_m$-equivariant, the maps $A\to B$ and $B\to B/J=C$ are $\bbZ$-graded. The equality $Y=Y^+$ (resp. $(X^+)^+=X^+$) means that $A_i=0$ (resp. $C_i=0$) for all $i<0$. Further, $\kappa(y)\simeq \kappa(x)$ means that there is an isomorphism $A/\frakm_y\simeq B/\frakm_x$ on residue fields. Let $(\bar{b}_1,\ldots,\bar{b}_d)$ be a homogeneous basis of $\frakm_x/(\frakm_x^2 + \frakm_y)$. Since the surjective map $\frakm_x\to \frakm_x/(\frakm_x^2 + \frakm_y)$ is $\bbZ$-graded, we can lift each $\bar{b}_i$ to an homogeneous element $b_{i}\in B_{n_i}$ of some degree $n_i\in \bbZ$. By [EGA IV, Prop 17.5.3 d'')] there is an isomorphism on completed local rings
\[
\hat{A}\pot{t_1,\ldots, t_d}\overset{\simeq}{\longto} \hat{B}, \;\; t_i\longmapsto b_i.
\] 
Recall that we arranged $Y$ to be Noetherian. After renumbering the $b_i$ we may assume that for some $r\geq 1$ we have $n_i<0$ for all $1\leq i\leq r-1$ and $n_i\geq 0$ for all $r\leq i\leq d$. As $B$ is Noetherian, we have $\hat{C}\simeq \hat{B}/J\hat{B}$, and further $J\hat{B}$ is the ideal generated by the $t_i$ for $1\leq i\leq r-1$. Thus, $\hat{C}\simeq \hat{A}\pot{t_r,\ldots, t_d}$ which implies that the map $A\to C$ is smooth at $x$.\smallskip\\
Using the ``diagonal trick'' as in \cite[Lem 3.2 ``General case'']{Mar15}, we reduce to the case $\kappa(y)\simeq \kappa(x)$ on residue fields, while preserving the property $Y = Y^+$. We have commutative diagrams
$$
\xymatrix{
X^+ \ar[r] \ar[d] & X \ar[dl] & & & X^+ \times_Y X \ar[r] \ar[d]_{\rm pr_2} & X \times_Y X \ar[dl] \\
Y &                       & & & X   &         
}
$$
where the second diagram arises from the first by base-change along $X \rightarrow Y$. The horizontal arrows are closed immersions. Since $X \rightarrow Y$ is smooth, as in \cite[Lem\,3.2]{Mar15} it is enough to prove ${\rm pr_2}$ is smooth at $(x,x)$.  But its image $x$ satisfies $\kappa(x)\simeq \kappa(x,x)$; hence we are reduced to the case $\kappa(y) \simeq \kappa(x)$.  However we need to do the reduction to $Y^+ = Y$ again, since the target $X$ of ${\rm pr_2}$ need not have this property. But in that reduction ${\rm pr_2}$ above is replaced by ${\rm pr_2}: X^+ \times_Y X^+ \rightarrow X^+$, and $(x,x)$ is still sent to $x$.  Therefore the reduction does not alter the property $\kappa(y)\simeq \kappa(x)$ when $x \in X^+$.  \smallskip\\
{\it The map $X^+\to Y$ is smooth.} Let $X^+_{\on{sm}}$ denote the open locus where the map $X^+\to Y$ is smooth which is $\bbG_m$-invariant and contains $X^0$ by the previous step. The $\bbG_m$-action on $X$ extends to a monoid action $\bbA^1_R\times X^+\to X^+$. Let $\calX^+_{\on{sm}}= (\bbA^1_R\times X^+)\times_{X^+}X^+_{\on{sm}}$ which is an open $\bbG_m\times\bbG_m$-invariant subscheme of $\bbA^1_R\times X^+$ which contains $\{0\}_R\times X^+$. Hence, $\calX^+_{\on{sm}}=\bbA^1_R\times X^+$ and thus $X^+_{\on{sm}}=X^+$. For details the reader may consult \cite[Claim 3.4]{Mar15}. 
The lemma follows.
\end{proof}

\begin{cor} \label{immersions}
Under the assumptions of Lemma \ref{0-perm}. If $X\to Y$ is a quasi-compact immersion \textup{(}resp. closed immersion; resp. open immersion\textup{)}, so are the maps $X^0\to Y^0$ and $X^\pm\to Y^\pm$.
\end{cor}
\begin{proof} As above we may assume that $X$ and $Y$ are schemes. If $X\to Y$ is a quasi-compact immersion, then by \cite[Tag 01RG]{StaPro} there is a factorization $X\to \bar{X}\to Y$ into an open immersion followed by a closed immersion. Here $\bar{X}$ denotes the scheme theoretic image of $X\to Y$ which is $\bbG_m$-invariant. Using the preservation of quasi-compactness from Lemma \ref{0-perm}, it is enough to treat the case of an open immersion and a closed immersion separately. Closed immersions where already treated in the proof of Theorem \ref{Gmthm} above. If $X\to Y$ is a monomorphism, so are $X^0\to Y^0$ and $X^\pm\to Y^\pm$. As being an open immersion is equivalent to being a smooth monomorphism (cf. \cite[Tag 025G]{StaPro}), the corollary follows for open immersions. 
\end{proof}

\section{Affine Grassmannians}\label{gctaffgrass}
We collect some facts on the geometry and cohomology of constant term maps on affine Grassmannians as considered in \cite{BD, BG02, MV07}. We include full proofs of those statements where we did not find a reference.

Let $F$ be any field, and let $G$ be a smooth affine $F$-group. The \emph{loop group} $LG=L_zG$ is the group functor on the category of $F$-algebras
\[
LG\co R\longmapsto G(R\rpot{z}),
\]
where $z$ denotes an additional formal variable. Then $LG$ is representable by an ind-affine ind-group scheme, and in particular defines a fpqc sheaf on the category of $F$-algebras. The \emph{positive loop group} $L^+G=L_z^+G$ is the group functor on the category of $F$-algebras
\[
L^+G\co R\longmapsto G(R\pot{z}),
\]
and we view $L^+G\subset LG$ as a subgroup functor. The \emph{affine Grassmannian} $\Gr_{G}=\Gr_{G,z}$ is the fpqc-sheaf on the category of $F$-algebras associated with the functor
\[
 R\longmapsto LG(R)/L^+G(R). 
\]
Then $\Gr_G$ is representable by a separated ind-scheme of ind-finite type over $F$, and is ind-proper (and then even ind-projective) if and only if the neutral component $G^\circ$ is reductive. The affine Grassmannian is equipped with a transitive action of the loop group
\begin{equation}\label{affineact}
LG\times \Gr_G\longto \Gr_G,
\end{equation}
i.e. a surjection of sheaves.

%The open cell
\subsection{The open cell}
The Beauville-Laszlo gluing lemma \cite{BL95} shows that the sheaf $\Gr_G$ represents the functor on the category $F$-algebras $R$ parametrizing isomorphism classes of tuples $(\calF,\al)$ with
\begin{equation}\label{bundleaffgrass}
\begin{cases}
\text{$\calF$ a $G$-torsor on $\bbP^1_R$};\\
\text{$\al\co \calF|_{\bbP^1_R\bslash \{0\}}\simeq \calF^0|_{\bbP^1_R\bslash \{0\}}$ a trivialization},
\end{cases}
\end{equation}
where $\calF^0$ denotes the trivial $G$-torsor. The variable $z$ is identified with a local coordinate of $\bbP^1_F$ at the origin, and we let $\bbP^1_F\bslash \{0\}=\Spec(F[z^{-1}])$. The \emph{negative loop group} $L^-G=L_z^-G$ is the functor on the category of $F$-algebras
\[
L^-G\co R\longmapsto G(R[z^{-1}]).
\]
Then $L^-G$ is representable by an ind-affine ind-group scheme of ind-finite type over $F$ (ind-finite type because the functor commutes with filtered colimits). Let $L^{--}G=\ker(L^-G\to G)$, $z^{-1}\mapsto 0$. 

\begin{lem}\label{opencell}
Let $G$ be a smooth affine $F$-group scheme, and let $e_0\in \Gr_G(F)$ denote the base point. The orbit map
\[
L^{--}G\to \Gr_G, \;\;\;\;\; g\mapsto g\cdot e_0
\]
is representable by an open immersion, and identifies $L^{--}G$ with those pairs $(\calF,\al)$ where $\calF$ is the trivial torsor.
\end{lem}

%{\cc To get Zariski local triviality one needs to show that $\Gr_G$ can be covered by $L^{--}G$. This is always true %\'etale locally on $F$, or even Zariski locally if $G$ is split. For a general (reductive) group one needs to show that %there are enough $G(F\rpot{z})$-points. I put some remark below.}

\begin{proof} For any $F$-algebra $R$, the loop group $LG(R)$ parametrizes isomorphism classes of triples $(\calF,\al,\be)$ where $(\calF,\al)\in \Gr_G(R)$ and $\be\co \calF^0|_{R\pot{z}}\simeq \calF|_{R\pot{z}}$. Hence, the multiplication map is given in the moduli description as 
\[
L^{--}G\times L^+G \longto LG,\;\;\; (g^-,g^+)\mapsto (\calF^0,g^-, g^+).
\]
Conversely, every triple $(\calF,\al,\be)$ with $\calF$ being the trivial torsor is isomorphic to a triple of the form $(\calF^0,g^-,g^+)$ for unique $g^-\in L^{--}G(R)$ and $g^+\in L^{+}G$: as $\calF\simeq \calF^0$ is trivial, the trivialization $\al$ defines an element in $\Aut(\calF^0|_{\bbP^1_R\bslash\{0\}})=L^-G(R)$. We extend the image $\bar{\al}\in G(R)$ under the reduction $z^{-1}\mapsto 0$ constantly to $\tilde{\al}\in G(\bbP^1_R)$. We put
\[
g^-= \al\circ (\tilde{\al}|_{\bbP^1_R\bslash\{0\}})^{-1}\;\;\;\; \text{and}\;\;\;\; g^+=(\tilde{\al}|_{R\pot{z}})\circ \be.
\]
Then $(\calF^0,g^-,g^+)$ and $(\calF,\al,\be)$ define the same element in $LG(R)$, and $g^-\in L^{--}G(R)$ (because $g^-\equiv 1 \mod z^{-1}$ by construction). The uniqueness of $g^-$ (and hence $g^+$) follows from $G(\bbP^1_R)=G(R)$. Thus, the map $L^{--}G\to \Gr_G$ identifies $L^{--}G$ with the pairs $(\calF,\al)$ where $\calF$ is the trivial torsor, and it is enough to show that being the trivial $G$-torsor on $\bbP^1_R$ is an open condition on $\Spec(R)$. \smallskip\\
As $L^{--}G$ and $\Gr_G$ commute with filtered colimits, we may assume $R$ to be a local ring: if we have a map ${\rm Spec}(R) \to \Gr_G$ which factors on a point $\frakp\in \Spec(R)$ through $L^{--}G$, then from the case of a local ring we would get a unique section $\Spec(R_\frakp)\to L^{--}G$. Since $L^{--}G$ is of ind-finite type, we get a section $\Spec(R_f)\to L^{--}G$ for some $f\in R\bslash \frakp$. These maps glue by uniqueness of the section, and there is some biggest open $U\subset \Spec(R)$ together with a section $U\to L^{--}G$. It remains to treat the case of a local ring $R$ with maximal ideal $\frakm$. Again as $L^{--}G$ and $\Gr_G$ commute with filtered colimits, we may further assume that $R$ is Noetherian.\smallskip\\
Let $(\calF,\al)\in \Gr_G(R)$ be a point. Assume that $\calF|_{\bbP^1_{R_0}}$ is trivial where $R_0=R/\frakm$ is the residue field. Being trivial is equivalent to the existence of a section $s_0\co\bbP^1_{R_0}\to \calF$. Our aim is to lift $s_0$ successively to a compatible family of sections $s_n\co \bbP^1_{R_n}\to \calF$ where $R_n=R/\frakm^{n+1}$ for $n\geq 0$. As $\calF$ is smooth (because $G$ is smooth), the obstruction of lifting $s_n$ to $s_{n+1}$ lives in 
\begin{equation}\label{obstruction}
H^1(\bbP^1_{R_0}, s_0^*(\frakg_{\calF/\bbP^1_{R_0}})\otimes_{\calO_{\bbP^1_{R_0}}}(\frakm^{n+1}\calO_{\bbP^1_R}/\frakm^{n+2}\calO_{\bbP^1_R})),
\end{equation}
where $\frakg_{\calF/\bbP^1_{R_0}}=(\Om^1_{\calF/\bbP^1_{R_0}})^*$, cf. \cite[Exp. III, Cor 5.4]{SGA1}. Now $\calF|_{\bbP^1_{R_0}}$ is trivial, and hence $s_0^*(\frakg_{\calF/\bbP^1_{R_0}})\simeq \frakg\otimes_{R_0}\calO_{\bbP^1_{R_0}}$ where $\frakg=e^*(\Om^1_{G/F})^*$ is the Lie algebra of $G$. On the other hand, it is clear that 
\[
\frakm^{n+1}\calO_{\bbP^1_R}/\frakm^{n+2}\calO_{\bbP^1_R}=(\frakm^{n+1}/\frakm^{n+2})\otimes_{R_0}\calO_{\bbP^1_{R_0}}.
\]
Since $H^1(\bbP^1_{R_0},\calO_{\bbP^1_{R_0}})=0$, we see that \eqref{obstruction} vanishes. Thus, we get a compatible family of sections $s_n\co \bbP^1_{R_n}\to \calF$. As $\calF$ is affine over $\bbP^1_{R}$, we get a section $\bbP^1_{\hat{R}}\to \calF$ where $\hat{R}=\lim_nR/\frakm^n$. Hence, we showed that $\calF|_{\bbP^1_{\hat{R}}}$ is trivial, and so $(\calF,\al)$ defines a point in $L^{--}G(\hat{R})$. As $R$ is noetherian, the map $R\to\hat{R}$ is faithfully flat, and we can use the sheaf property of $L^{--}G$ and $\Gr_G$ as follows: the map $L^{--}G\to \Gr_G$ is a monomorphism which implies that $(\calF,\al)$ lies in the equalizer of $L^{--}G(\hat{R})\rightrightarrows L^{--}G(\hat{R}\otimes_R\hat{R})$, i.e. defines a point of $L^{--}G(R)$. Therefore, $\calF|_{\bbP^1_R}$ needs to be trivial which is what we wanted to show. 
\end{proof}

The lemma shows that the map $LG\to\Gr_G$ has sections Zariski locally whenever $\Gr_G$ is covered by $L^{--}G$-translates, e.g. $G$ split connected reductive. The following corollary is an immediate consequence of Lemma \ref{opencell}, and is due to \cite[Prop 4.6]{LS97} for connected reductive groups (see also \cite{Fal03, dHL}).

\begin{cor}
The multiplication map $L^{--}G\times L^+G\to LG$ is representable by an open immersion.
\end{cor}

\subsection{Schubert varieties} \label{SchubertGrass}
Let $G$ be a connected reductive group over an arbitrary field $F$. By a Theorem of Grothendieck \cite[XIV, 1.1]{SGA3}, there exists a maximal $F$-torus $T\subset G$. The absolute Weyl group is
\[
W^{\rm abs}_0\defined \on{Norm}_G(T)(\sF)/T(\sF).
\]
The Weyl group $W^{\rm abs}_0$ acts on the $\sF$-cocharacter lattice $X_*(T)$. As all maximal $\sF$-tori are conjugate, the set $X_*(T)/W^{\rm abs}_0$ parametrizes the $G_\sF$-conjugacy classes of geometric cocharacters. Each class $\{\mu\}\in X_*(T)/W^{\rm abs}_0$ has a field of definition $E/F$ which is a finite separable extension. The $\{\mu\}$-Schubert variety is the reduced $L^+G_\sF$-orbit closure 
\begin{equation}\label{schubertgrass}
\Gr_G^{\leq \{\mu\}} \defined \overline{L^+G_\sF\cdot  z^\mu \cdot e_0},
\end{equation}
where $e_0\in \Gr_G$ denotes the base point. The scheme $\Gr_G^{\leq \{\mu\}}$ is a projective variety which is defined over $E$. The unique open $L^+G_E$-orbit
\begin{equation}\label{openorbit}
 \Gr_G^{ \{\mu\}} \hookto \Gr_G^{\leq \{\mu\}} 
\end{equation}
is a smooth dense open subvariety of $\Gr_G^{\leq \{\mu\}}$. If the class $\{\mu\}$ has an $E$-rational representative $\mu$, then we simply write $\Gr_G^\mu$ (resp. $\Gr_G^{\leq \mu}$). The Cartan decomposition $LG(\sF)=L^+G(\sF)\cdot LT(\sF)\cdot L^+G(\sF)$ implies that there is a presentation on reduced loci
\begin{equation}\label{presentgrass}
(\Gr_{G})_\red=\on{colim}_{\{\la\}} \bigcup_{\{\mu\} \in \Ga_F\cdot \{\la\} }\Gr_{G}^{\leq \{\mu\}},
\end{equation}
where $\{\la\}$ runs through the Galois orbits in $X_*(T)/W^{\rm abs}_0$, and each finite disjoint union of Schubert varieties is defined over $F$.

%Geometry of torus actions on affine Grassmannians
\subsection{Torus actions on affine Grassmannians} Let $G$ be a connected reductive group over an arbitrary field $F$. Let $\chi\co \bbG_{m,F}\to G$ be a $F$-rational cocharacter. The cocharacter $\chi$ induces via the composition
\begin{equation}\label{Gmaction}
\bbG_{m}\subset L^+\bbG_m\overset{L^+\!\chi\phantom{h}}{\longto}L^+G\subset LG
\end{equation}
a (left) $\bbG_m$-action on the affine Grassmannian $\Gr_G$. As in \eqref{hyperloc} we obtain maps of $F$-spaces
\begin{equation}\label{hyperlocgrass}
(\Gr_G)^0\leftarrow (\Gr_G)^\pm\to \Gr_G.
\end{equation}
Let us mention the following lemma which implies the ind-representability of the spaces \eqref{hyperlocgrass}, in light of Theorem \ref{Gmthm}.

\begin{lem}\label{loclinaffine}
The $\bbG_m$-action on $\Gr_G$ via \eqref{Gmaction} is Zariski locally linearizable. 
\end{lem}
\begin{proof} Let $ G\hookto \Gl_n$ be a faithful representation. The fppf quotient $\Gl_n/G$ is affine, and hence the map $ \Gr_G\hookto \Gr_{\Gl_n}$ is representable by a closed immersion (cf. \cite[Prop 1.2.6]{Zhu}) and equivariant for the $\bbG_m$-action via $\bbG_{m}\overset{\chi\;}{\to} G\to \Gl_n$. We reduce to the case $G=\GL_n$. After conjugation, we may assume that $\bbG_m\to \Gl_{n}$ factors through the diagonal matrices. Let $\La_0$ denote the standard $F\pot{z}$-lattice of $F\rpot{z}^n$ and let $\Lambda_{0,R} = R[\![z]\!] \otimes_{F\pot{z}} \Lambda_0$. We write $\Gr_G=\on{colim}_i\Gr_{G,i}$ where 
\[
\Gr_{G,i}(R) =\{\La\subset R\rpot{z}^n\;|\; z^i\La_{0,R}\subset\La\subset z^{-i}\La_{0,R}\},
\]
is the moduli space of $R\pot{z}$-lattices in $R\rpot{z}^n$ bounded by $i\geq 0$. The $F$-vector space $V_i=z^{-i}\La_0/z^i\La_0$ has a canonical basis and is equipped with a linear $\bbG_m$-action which preserves this basis. The projective $F$-scheme $\on{Quot}(V_i)$ which parametrizes quotients of $V_i$ is a finite disjoint union of the classical Grassmannians $\on{Grass}_d(V_i)$ for $0\leq d\leq \dim(V_i)$. Then the closed immersion
\[
p_i\co \Gr_{G,i}\hookto \on{Quot}(V_i), \;\; \La\mapsto  z^{-i}\La_0/\La
\]
is $\bbG_m$-equivariant with a linear action on the target. For varying $i$, the maps $p_i$ can be arranged into a system compatible with the standard affine opens of $\on{Quot}(V_i)$ given by the canonical basis of $V_i$. The lemma follows.
\end{proof}

%Comparison with group theoretical data
\subsubsection{Fixed points, attractors and repellers}
Our aim is to express \eqref{hyperlocgrass} in terms of group theoretical data related to the cocharacter $\chi$, cf. Proposition \ref{gctgeo} below. 

Let $\chi$ act on $G$ via conjugation $(\la,g)\mapsto \chi(\la)\cdot g\cdot\chi(g)^{-1}$. The fixed points $M=G^0$ is the centralizer of $\chi$ and defines a connected reductive subgroup of $G$. The attractor $P^+=G^+$ (resp. the repeller $P^-=G^-$) is a parabolic subgroup of $G$ with $P^+\cap P^-=M$. By \eqref{hyperloc} we have natural maps of $F$-groups
\begin{equation}\label{hyperlocgroup}
M\leftarrow P^\pm \to G,
\end{equation}
and the map $P^\pm\to M$ identifies $M$ as the maximal reductive quotient, cf. \cite[2.1]{CGP10}.

\begin{prop} \label{gctgeo}
The maps \eqref{hyperlocgroup} induce a commutative diagram of $F$-ind-schemes
\begin{equation}\label{gctgeo_diag}
\begin{tikzpicture}[baseline=(current  bounding  box.center)]
\matrix(a)[matrix of math nodes, 
row sep=1.5em, column sep=2em, 
text height=1.5ex, text depth=0.45ex] 
{\Gr_M & \Gr_{P^\pm} & \Gr_G \\ 
(\Gr_{G})^0& (\Gr_{G})^\pm& \Gr_{G}, \\}; 
\path[->](a-1-2) edge node[above] {}  (a-1-1);
\path[->](a-1-2) edge node[above] {}  (a-1-3);
\path[->](a-2-2) edge node[below] {}  (a-2-1);
\path[->](a-2-2) edge node[below] {} (a-2-3);
\path[->](a-1-1) edge node[left] {$\simeq$} (a-2-1);
\path[->](a-1-2) edge node[left] {$\simeq$} (a-2-2);
\path[->](a-1-3) edge node[left] {$\id$} (a-2-3);
\end{tikzpicture}
\end{equation}
where the vertical maps are isomorphisms. 
\end{prop}

\begin{rmk} The statement $\Gr_M\simeq (\Gr_G)^0$ appeared in \cite[Thm 1.3.4]{Zhu09}, but its proof contains a mistake: \cite[Lem 1.3.5]{Zhu09} only holds in the case that $R$ is a field and fails otherwise. The authors were told by the author of \cite{Zhu09} (private communication) that he was aware of the mistake and knew how to fix it. In view of Proposition \ref{gctgeo} the results of \cite{Zhu09} remain valid. 
\end{rmk}

Let us construct the diagram in Proposition \ref{gctgeo}. As the $\bbG_m$-action on $\Gr_M$ is trivial, the natural map $\Gr_M\to \Gr_G$ factors as $\Gr_M\to (\Gr_G)^0\to \Gr_G$. We use a construction explained in Heinloth \cite[1.6.2]{He} to define the map $\Gr_{P^+}\to (\Gr_G)^+$ in terms of the moduli description \eqref{bundleaffgrass} (the construction of $\Gr_{P^-}\to (\Gr_G)^-$ is given by inverting the $\bbG_m$-action). The $\bbG_m$-action $P^+\times \bbG_{m,F}\to P^+, (p,\la)\mapsto \chi(\la)\cdot p\cdot \chi(\la)^{-1}$ via conjugation extends via the monoid action of $\bbA^1$ on $(\bbA_F^1)^+$ in \eqref{flow} to a monoid action 
\begin{equation}\label{monoidaction}
m_\chi\co P^+\times \bbA^1_F\longto P^+
\end{equation}
such that $m_\chi(p,0)\in M$. We let $\on{gr}_\chi\co P^+\times \bbA^1_F\to P^+\times \bbA^1_F$, $(p, \la)\mapsto (m_\chi(p,\la),\la)$ viewed as an $\bbA^1_F$-group homomorphism. Then the restriction $\on{gr}_\chi|_{\{1\}}$ is the identity whereas $\on{gr}_\chi|_{\{0\}}$ is the composition $P^+\to M\to P^+$. For a point $(\calF^+,\al^+)\in \Gr_{P^+}(R)$, the Rees bundle is
\begin{equation}\label{Reesbundle}
\on{Rees}_\chi(\calF^+,\al^+)\defined \on{gr}_{\chi , *}(\calF^+_{ \bbA^1_R},\al^+_{ \bbA^1_R}) \in \Gr_{P^+}(\bbA^1_R),
\end{equation}
where $\on{gr}_{\chi , *}$ denotes the push forward under the $\bbA^1$-group homomorphism. The Rees bundle $\on{Rees}_\chi(\calF^+,\al^+)|_{\{1\}_R}$ is equal to $(\calF^+,\al^+)$ whereas $\on{Rees}_\chi(\calF^+,\al^+)|_{\{0\}_R}$ is the image of $(\calF^+,\al^+)$ under the composition $\Gr_{P^+}\to \Gr_M\hookto \Gr_{P^+}$. One checks that $\on{Rees}_\chi(\calF^+,\al^+)$ is $\bbG_m$-equivariant, and hence defines an $R$-point of $(\Gr_{P^+})^+$. As the Rees construction is functorial, we obtain a map of $F$-ind-schemes
\begin{equation}\label{Reesmap}
\on{Rees}_\chi\co \Gr_{P^+}\to (\Gr_{P^+})^+,
\end{equation}
which is inverse to the map $(\Gr_{P^+})^+\to \Gr_{P^+}$ given by evaluating at the unit section. We define the map $\Gr_{P^+}\to (\Gr_G)^+$ to be the composition $\Gr_{P^+}\simeq (\Gr_{P^+})^+\to (\Gr_G)^+$ where the latter map is deduced from the natural map $\Gr_{P^+}\to \Gr_G$. This constructs the commutative diagram in Proposition \ref{gctgeo}.

\begin{proof}[Proof of Proposition \ref{gctgeo}] We may assume $F$ to be algebraically closed. The Iwasawa decomposition $G(F\rpot{z})=P^\pm(F\rpot{z})\cdot G(F\pot{z})$ (which follows from the valuative criterion applied to the proper scheme $G/P^\pm$) implies that the vertical maps are bijections on $F$-points. It is enough to see that the maps are isomorphisms of ind-schemes in an open neighborhood of the base point. By Lemma \ref{opencell} the natural map
\begin{equation}\label{gctgeo:eq1}
L^{--}G\longto \Gr_G
\end{equation}
is representable by an open immersion, and likewise for $P^\pm$ (resp. $M$) replacing $G$. Further, the map \eqref{gctgeo:eq1} is $\bbG_m$-equivariant for the conjugation action on $L^{--}G$. Hence, we are reduced to proving that the natural closed immersions (cf.~Corollary \ref{immersions}) of ind-affine ind-schemes
\begin{align}\label{gctgeo:eq2}
 L^{-}M&\longto (L^{-}G)^0\\
\label{gctgeo:eq3}
 L^{-}P^\pm &\longto (L^{-}G)^\pm
\end{align}
are isomorphisms. For any $F$-algebra $R$, we have on points
\[
(L^{-}G)^0(R)=\{g\in G(R[z^{-1}])\;|\;  \forall S\in (\mbox{$R$-Alg}), \la\in S^\times\co\, \chi(\la)\cdot g\cdot \chi(\la)^{-1}=g\},
\]   
and $L^{-}M(R)$ is by definition (remember $M=G^0$) the subset of those $g\in G(R[z^{-1}])$ such that $\chi(\la)\cdot g \cdot\chi(\la)^{-1}=g$ holds for all $\la\in S[z^{-1}]^\times$ with $S\in \textup{(}R\on{-Alg)}$. 
Now for any $R$-algebra $S$, we may use the polynomial ring $S[t]$ as a test algebra in the definition of $(LG)^0$.
This gives a condition in $G(S[z^{-1},t])$.
Quotienting by the ideal $(z^{-1}-t)$ gives a condition in $G(S[z^{-1}])$ which is the condition defining $L^-M$ for the $R$-algebra $S$.
The reasoning in the case of \eqref{gctgeo:eq3} is similar. 
The proposition follows.
\end{proof}

The following lemma is the analogue of Proposition \ref{gctgeo} over a discrete valuation ring, and is needed in the proof of Theorem \ref{PZgctgeoBD} below.

\begin{lem} \label{gctgeo_over_Z}
Assume that $(G,\chi)$ are defined over a discrete valuation ring $\calO$, i.e., $G$ is a reductive group scheme over $\calO$ with geometrically connected fibers, and $\chi\co \bbG_{m,\calO}\to G$ a cocharacter. Then \eqref{gctgeo_diag} is defined over $\calO$, and the vertical maps are isomorphisms.  
\end{lem}
\begin{proof} The fixed point subgroup $M\subset G$, and the attractor (resp. repeller) subgroup $P^+\subset G$ (resp. $P^-\subset G$) are defined over $\calO$, and representable by smooth closed subgroups of $G$, cf.\,\cite{Mar15}. 
Then the functors $\Gr_M$, $\Gr_{P^\pm}$ and $\Gr_G$ are defined over $\calO$, and representable by separated $\calO$-ind-schemes of ind-finite type. 
Further, the diagram \eqref{gctgeo_diag} is defined over $\calO$ by the same construction as above. 

By Proposition \ref{gctgeo}, the vertical maps in \eqref{gctgeo_diag} are fiberwise isomorphisms, i.e., after passing to its fraction field $\on{Frac}(\calO)$ resp.~its residue field $k$. 
We do not know whether the $\calO$-ind-schemes are ind-flat, and hence we have to argue differently. 
Lemma \ref{opencell} holds for $F$ replaced with $\calO$ -- in fact for any ring -- by the same argument. Hence, the maps \eqref{gctgeo:eq2} are isomorphisms over $\calO$ by the same proof as in Proposition \ref{gctgeo}. 
By fpqc-descent, it is enough to prove that the vertical maps in \eqref{gctgeo_diag} are isomorphisms after passing to the strict Henselization $\breve{\calO}$. We consider the open subset 
\[
V_M\defined \bigcup_m\,m\cdot L^{--}M\cdot e_0\;\;\;\;\; \textup{(}\text{resp.}\,\; V_{P^\pm}\defined \bigcup_p\,p\cdot L^{--}P^\pm\cdot e_0\textup{)},
\]
of $\Gr_M$ (resp. $\Gr_{P^\pm}$), where the union runs over all $m\in LM(\breve{\calO})$ (resp. $p\in LP^\pm(\breve{\calO})$). By $LM$-equivariance (resp. $LP^\pm$-equivariance) of the map $\Gr_M\to (\Gr_G)^0$ (resp. of the map $\Gr_{P^\pm}\to (\Gr_G)^0$), it is an isomorphism restricted to $V_M$ (resp. $V_{P^\pm}$). As we already know that the maps are isomorphism over $\on{Frac}(\calO)$, it is enough to show that the map $V_M\coprod \Gr_{M,\on{Frac}(\calO)}\to \Gr_M$ (resp. $V_{P^\pm}\coprod \Gr_{P^\pm,\on{Frac}(\calO)}\to \Gr_{P^\pm}$) is an fpqc-cover. Flatness is immediate from the construction, and we need to show the surjectivity, i.e., that $V_M$ (resp. $V_{P^\pm}$) contains the special fiber $\Gr_M\otimes \bar{k}$ (resp. $\Gr_{P^\pm}\otimes \bar{k}$). 

Let us start with the case of $M$. As $\Gr_M$ is of ind-finite type, and $\Gr_M(\bar{k})=LM(\bar{k})/L^+M(\bar{k})$, it is enough to prove that the reduction map $LM(\breve{\calO})\to LM(\bar{k})=M(\bar{k}\rpot{z})$ is surjective. As $\bar{k}\rpot{z}$ is a field and $G_{\breve{\calO}}$ is split, we have the Bruhat decomposition
\[
M(\bar{k}\rpot{z})\,=\,\coprod_{w\in W_{0,M}}\left(U^w_M\cdot \dot{w}\cdot B_M\right)(\bar{k}\rpot{z}),
\]
where $U^w_M$ and $B_M$ are defined over $\breve{\calO}$ and where $W_{0,M}$ is a constant finite \'etale $\breve{\calO}$-group. Clearly, the elements $\dot{w}$ lift. Further, we have as $\breve{\calO}$-schemes $U^w_M\simeq \bbA^{l}_{\breve{\calO}}$ and $B_M\simeq \bbG_{m,\breve{\calO}}^m\times\bbA^n_{\breve{\calO}}$ for some $l,m,n\in\bbZ_{\geq 0}$. Hence, to show the surjectivity of $LM(\breve{\calO})\to LM(\bar{k})$, we reduce to the case of $\bbG_{m,\breve{\calO}}$ and $\bbA^1_{\breve{\calO}}$ (because the $L$-construction commutes with finite products). But the reduction maps $\breve{\calO}\rpot{t}^\times\to \bar{k}\rpot{t}^\times$ and $\breve{\calO}\rpot{t}\to \bar{k}\rpot{t}$ are clearly surjective. This finishes the case of $M$, and the case of $P^\pm$ is similar. The lemma follows.
\end{proof}

%Connectd components
\subsubsection{Connected components} We discuss connected components of $\Gr_M$ and $\Gr_{P^\pm}$.

\begin{lem} \label{basicgeo}
i\textup{)} The map $p^\pm\co \Gr_{P^\pm}\to \Gr_G$  is a schematic quasi-compact monomorphism, and the restriction to each connected component of $\Gr_{P^\pm}$ is a locally closed immersion. \smallskip\\
ii\textup{)} The map $q^\pm\co \Gr_{P^\pm}\to \Gr_M$ is ind-affine with geometrically connected fibers, and induces an isomorphism on the group of connected components $\pi_0(\Gr_{P^\pm_E})\simeq\pi_0(\Gr_{M_E})$ for any field extension $E/F$.
\end{lem}
\begin{proof} Use Proposition \ref{gctgeo} to identify $p$ resp. $q$ with the map on attractor resp. repeller schemes. The ``schematic'' assertion in part i) as well as part ii) follow from Theorem \ref{Gmthm} ii) using Lemma \ref{loclinaffine}, and the fact that $\Gr_G$ is of ind-finite type. It remains to explain why the restriction of $p^\pm$ to each connected component of $\Gr_{P^\pm}=(\Gr_G)^\pm$ is a locally closed immersion. By the proof of Lemma \ref{loclinaffine}, there is an $\bbG_m$-equivariant closed embedding $\Gr_G=\on{colim}_i\Gr_{G,i}\hookto \on{colim}_i\bbP(V_i)$ where $V_i$ are finite dimensional $F$-vector spaces equipped with a linear $\bbG_m$-action. Since $(\Gr_{G,i})^\pm=\Gr_{G,i}\times_{\bbP(V_i)}\bbP(V_i)^\pm$, it is enough to show that the restriction to each connected component of $\bbP(V_i)^\pm\to \bbP(V_i)$ is a locally closed immersion. This is easy to see, and left to the reader. The lemma follows.
\begin{comment}
If $E/F$ splits $G$, then the Grassmannian $\Gr_{G_E}$ can be covered by $L^{--}G_E$-orbits of $E$-points. The commutative diagram of $E$-ind-schemes
\[
\begin{tikzpicture}[baseline=(current  bounding  box.center)]
\matrix(a)[matrix of math nodes, 
row sep=1.5em, column sep=2em, 
text height=1.5ex, text depth=0.45ex] 
{L^{--}P_E&L^{--}G_E \\ 
\Gr_{P_E}&\Gr_{G_E} \\}; 
\path[->](a-1-1) edge node[above] {} (a-1-2);
\path[->](a-2-1) edge node[above] {} (a-2-2);
\path[->](a-1-1) edge node[right] {} (a-2-1);
\path[->](a-1-2) edge node[right] {} (a-2-2);
\end{tikzpicture}
\]
shows that the restriction of $p_E$ to each connected component of $\Gr_{P_E}$ is a locally closed immersion: the map $L^{--}P_E\hookto L^{--}G_E$ is a closed immersion (because $L^{--}$ preserves closed immersions), and by Lemma \ref{opencell} the vertical maps are open immersions. Now it is enough to observe that $p_E$ is $LP_E$-equivariant. 
Part i) follows from the effectivity of \'etale descent for (quasi-)projective schemes. For part ii), use Proposition \ref{gctgeo} together with the corresponding result for attractor (resp. repeller) schemes \cite[Cor 1.12]{Ri19}. 
\end{comment}
\end{proof}

Throughout this paper, for a connected reductive group $G$, we denote by $G_{\rm der}$ its derived group, and by $G_{\rm sc}$ the simply-connected cover of $G_{\rm der}$.

Let $T\subset G$ be a maximal (not necessarily split) $F$-torus. We may choose $T$ such that $\chi$ factors as $\bbG_m\to T\subset G$, in particular $T\subset M$. The cocharacter $\chi$ induces a natural $\bbZ$-grading on $\pi_0(\Gr_M)\simeq \pi_0(\Gr_P)$ as follows: We have $\pi_0(\Gr_{M_\sF})\simeq \pi_1(M)$ where $\pi_1(M)$ is the algebraic fundamental group in the sense of Borovoi \cite{Bo98}. The group $\pi_1(M)$ can be defined as the quotient of the Galois lattices
\begin{equation}\label{fundamentalgroup}
\pi_1(M)\;=\;X_*(T)/X_*(T_{M_\scon}),
\end{equation}
where $T_{M_\scon}$ is the preimage of $T\cap M_\der$ in $M_\scon$. Hence, there is a decomposition into connected components
\begin{equation}\label{conncomp}
\Gr_{M_\sF}\;=\;\coprod_{\nu\in \pi_1(M)}\Gr_{M_\sF,\nu},
\end{equation}
and likewise for $\Gr_{P_\sF}$ compatible with the map $q_\sF=\amalg_{\nu\in\pi_1(M)}q_{\sF,\nu}$, cf. Lemma \ref{basicgeo} ii). 

Let $P^\pm=M\ltimes N^\pm$ be the Levi decomposition. Let either $N=N^+$ or $N=N^-$, and denote by $\rho_N$ the half-sum of the roots in $N_\sF$ with respect to $T_\sF$. To every $\nu\in \pi_1(M_\sF)$, we attach the number 
\begin{equation}\label{number}
n_\nu=\lan2\rho_N,\dot{\nu}\ran,
\end{equation}
where $\dot{\nu}$ is any representative in $X_*(T)$, and $\lan\str,\str\ran\colon X^*(T)\times X_*(T)\to \bbZ$ is the natural pairing. Since $\lan\rho_N,\al^\vee\ran=0$ for all $\al^\vee\in X_*(T_{M_\scon})$, the number $n_\nu$ is well-defined. For every $m\in\bbZ$, let $\Gr_{P^\pm,m}$ (resp. $\Gr_{M,m}$) be the disjoint union of all $\Gr_{P^\pm_\sF,\nu}$ (resp. $\Gr_{M_\sF,\nu}$) with $n_\nu=m$. As $T$ and $N$ are defined over $F$, the function $\pi_1(M)\to \bbZ,\, \nu\mapsto n_\nu$ is constant on Galois orbits. Hence, $\Gr_{P^\pm,m}$ (resp. $\Gr_{M,m}$) is defined over $F$, and we get a decomposition into open and closed ind-subschemes
\begin{equation}\label{decomposition}
q^\pm=\coprod_{m\in\bbZ}q^\pm_m \co \Gr_{P^\pm}=\coprod_{m\in\bbZ}\Gr_{P^\pm,m}\longto \coprod_{m\in\bbZ}\Gr_{M,m}=\Gr_M.
\end{equation}
Likewise, we can write $p^\pm=\coprod_{m\in\bbZ}p^\pm_m$ where $p^\pm_m:=p^\pm|_{\Gr_{P^\pm,m}}$. One checks that the decomposition \eqref{decomposition} does not depend on the choice of $T$ as above. Further, the decomposition for $N=N^+$ differs by a sign from the decomposition for $N=N^-$. 

%Beilinson-Drinfeld Grassmannians
%\subsubsection{Beilinson-Drinfeld Grassmannians} We sketch

%Cohomology of contant terms on affine Grassmannians
\subsection{Cohomology of constant terms} \label{gctcoho}
Let $F$ be a field whose cyclotomic character $ \Ga_F\to \bbZ_\ell^{\times}$ composed with $\bbZ_\ell^\times \hookrightarrow \bar{\bbZ}_\ell^\times$ admits a square root. 
For non-archimedean local fields $F$ with residue characteristic $p \neq \ell$, the $\ell$-adic cyclotomic character is unramified and choosing such a square root is equivalent to choosing $q^{1/2} \in \bar{\bbZ}_\ell^\times$, where $q$ is the cardinality of the residue field of $F$.

For a separated ind-scheme $X=\on{colim}_iX_i$ of ind-finite type over $F$, we denote the bounded derived category of $\algQl$-complexes with constructible cohomology sheaves by
\[
D_c^b(X)\defined\on{colim}_i D_c^b(X_i,\algQl),
\]
where the transition maps are given by push forward along the closed immersions $X_i\hookto X_j$ for $j\geq i$. We let $\on{Perv}(X)=\on{colim}_i\on{Perv}(X_i)$ the full abelian subcategory of $D_c^b(X)$ given by the heart of the perverse $t$-structure. 

For any $\ell$-adic complex $\calA$ and any integer $n\in \bbZ$, we define the operator 
\[
\calA\langle n\rangle\defined \calA[n](\nicefrac{n}{2}),
\] 
where $(\nicefrac{1}{2})$ denotes the half twist using the square root of the cyclotomic character. We say that a sheaf on a smooth equidimensional $F$-scheme of dimension $n$ is constant if it is a direct sum of copies of $\algQl\langle n\rangle$. 

%The geometric Satake isomorphism
\subsubsection{The geometric Satake isomorphism}\label{GeoSatSec}
Let $G$ be a connected reductive $F$-group. The affine Grassmannian $\Gr_G$ admits a presentation $\Gr_G=\on{colim}_i\Gr_{G,i}$ by $L^+G$-stable projective subschemes $\Gr_{G,i}$. The group $L^+G$ is proalgebraic, and the action factors on each $\Gr_{G,i}$ through a smooth algebraic group. Hence, we define the category of $L^+G$-equivariant perverse sheaves on $\Gr_G$ as
\[
\on{Perv}_{L^+G}(\Gr_G)=\on{colim}_i\on{Perv}_{L^+G}(\Gr_{G,i}). 
\]
By definition the $L^+G$-equivariance is a condition on the perverse sheaves and not an additional datum: as $L^+G$ is connected both concepts give equivalent categories. The category $\on{Perv}_{L^+G}(\Gr_G)$ is a $\algQl$-linear abelian category.  

\begin{dfn}\label{Sat_cat_dfn}
i) The \emph{Satake category $\Sat_{G,\sF}$ over $\sF$} is the category $\on{Perv}_{L^+G_\sF}(\Gr_{G,\sF})$.\smallskip\\
ii) The \emph{Satake category $\Sat_G$ over $F$} is the full subcategory of $\on{Perv}_{L^+G}(\Gr_G)$ of semi-simple objects $\calA$ such that after passing to $\calA_E$, for a sufficiently big finite separable extension $E/F$ which splits $G$, the $0$-th perverse cohomology sheaves $^p\on{H}^0(\iota_\mu^*\calA_E)$ and $^p\on{H}^0(\iota_\mu^!\calA_E)$ are constant for all $L^+G_E$-orbits $\iota_{\mu}\co \Gr_{G_E}^{\{\mu\}}\hookto \Gr_{G_E}$, cf. \eqref{openorbit}.

%{\cg I am wondering whether this is really the most useful/natural definition for us -- couldn't we consider a category of objects in $\Sat_{G, \sF}$ endowed with a continuous compatible action of the Weil group $W_F$?}{\cc There are several options, it really depends on what we want: The category $\Sat_{G, \sF}$ endowed with a continuous compatible action of the Weil group $W_F$ is equivalent to continuous $\ell$-adic representations of $^LG(\algQl)$ whose restriction to $\hat{G}(\algQl)$ is algebraic. In particular, all $\ell$-adic representations of $W_F$ are in this category. Putting the condition of being constant along some orbit for $E/F$ sufficiently big, amounts to saying that we only look at representations of $W_F$ which are algebraic, i.e. with finite image when considered as a $W_F$ representation. All normalized $\IC$-complexes satisfy this condition. The Galois action we encounter on $\om(\IC)$ is trivial when the group $G$ is split and in general comes from the Galois action on the orbits, and in particular acts trough a finite group. Put it another way: the object of $\Sat_G$ are the objects $\calA$ which split as a direct sum of $\IC$-complexes after a sufficiently big finite extension $E/F$.} {\cm please add some language explaining the last sentence: it makes it clear that the category ${\rm Sat}_G$ defined here is closed under convolution...I am worried about the semi-simplicity of the convolution operation...}
\end{dfn}

Let us make Definition \ref{Sat_cat_dfn} ii) explicit. If $\Gr_G^{\leq \{\mu\}}$ is as in \eqref{schubertgrass} defined over $E/F$, then for the \emph{normalized intersection complex}
\begin{equation}\label{normalized_IC}
\IC_{\{\mu\}}\defined j_{!*}\algQl\lan n\ran\in \Sat_{G_E},
 \end{equation}
where $j\co \Gr_G^{ \{\mu\}}\hookto \Gr_G^{\leq \{\mu\}}$ is the inclusion, and $n=\dim(\Gr_G^{\{\mu\}})$ is the dimension. 
Hence, summing over the Galois orbit of $\{\mu\}$ as in \eqref{presentgrass}, the complex
\[
\bigoplus_{\{\la\}\in\Ga_{F}\cdot \{\mu\}}\IC_{\{\la\}}
\] 
descends to $F$, and defines an object of $\Sat_G$. Since $\Sat_{G,\sF}$ is semi-simple (cf.\,\cite[Prop 1]{Ga01} and \cite[Prop 3.1]{Ri14a} for details), every object in $\on{Perv}_{L^+G}(\Gr_G)$ is a direct sum of 
\begin{equation}\label{satobjects}
(\oplus_{\{\la\}}\IC_{\{\la\}})\otimes \calL,
\end{equation}
where $\calL$ is a local system on $\Spec(F)$. The objects in $\Sat_G$ are those objects of ${\rm Perv}_{L^+G}(\Gr_G)$ where the local systems $\calL$ in \eqref{satobjects} are trivial after some finite separable extension of $F$. We have a natural pullback functor $(\str)_\sF\co \Sat_{G}\to \Sat_{G,\sF}$. We view $\Ga_F$ as a pro-algebraic group, and we let $\Rep_{\algQl}(\Ga_F)$ be the category of algebraic representations of $\Ga_F$ on finite dimensional $\algQl$-vector spaces, i.e., representations which factor through a finite quotient of $\Ga_F$. There is the Tate twisted global cohomology functor
\begin{equation}\label{galoistwist}
\begin{aligned}
\om\co & \on{Perv}_{L^+G}(\Gr_G) \longto \Rep_{\algQl}(\Ga_F) \\
&\hspace{.15in} \calA  \longmapsto \bigoplus_{i\in \bbZ}\on{H}^i(\Gr_{G,\sF}, \calA_\sF)(\nicefrac{i}{2}).
\end{aligned}
\end{equation}

\begin{lem}\label{galoislem} 
Let $E/F$ be a finite separable extension which splits $G$, and let $\calA\in \on{Perv}_{L^+G}(\Gr_G)$. Then the $\Ga_E$-Galois action on $\om(\calA)$ is trivial if and only if $\calA_E$ is a direct sum of normalized intersection complexes. In this case, $\calA\in \Sat_{G}$.
\end{lem}
\begin{proof} First, let $\calA=\IC_{\{\mu\}}$ be a normalized intersection complex. Choose a Chevalley $\bbZ$-group scheme $H$ together with an isomorphism $H\otimes_\bbZ E\simeq G_E$. Then under this isomorphism, there is an identification of $E$-ind-schemes $\Gr_H\otimes_\bbZ E=\Gr_{G_E}$. The $\ell$-adic \'etale cohomology does not depend on the choice of a separable closure. In particular, if $E$ is of characteristic $p$, then \cite[Thm 3.1]{NP01} shows that the $\Ga_E$-Galois action on $\om(\IC_{\{\mu\}})$ is trivial, cf.\,the twist in \eqref{galoistwist}. If $E$ is of characteristic $0$, then the inertia group $I_E$ acts trivially on $\om(\IC_{\{\mu\}})$ by \cite[Prop 10.12]{PZ13}. Hence, the claim follows by proper base change applied to $\Gr_H\otimes_\bbZ \calO_E$ from the previous case. Conversely, let $\calA\in \on{Perv}_{L^+G}(\Gr_G)$. Then $\calA_E$ is a direct sum of $\IC_{\{\mu\}}\otimes \calL$ where $\calL$ is a local system on $\Spec(E)$ (because $G_E$ is split the class $\{\mu\}$ is defined over $E$). Further, $\IC_{\{\mu\}}\otimes \calL=\IC_{\{\mu\}}\star \calL$ by definition of convolution. Hence, if the $\Ga_E$-action on $\om(\IC_{\{\mu\}}\otimes \calL)=\om(\IC_{\{\mu\}})\otimes \calL$ is trivial, then $\calL$ must be trivial. Clearly, we have $\calA\in \Sat_G$.
\end{proof}

By the geometric Satake equivalence \cite{Gi, Lu81, BD, MV07, Ri14a, Zhu}, the category $\Sat_{G,\sF}$ admits a unique structure of a neutral Tannakian category such that taking global cohomology is an equivalence of Tannakian categories
\begin{equation}\label{Satake}
\om\co \Sat_{G,\sF} \overset{\simeq}{\longto} \Rep_\algQl(\widehat{G}),
\end{equation}
where $\Rep_\algQl(\widehat{G})$ is the category of algebraic representations of the Langlands dual group $\widehat{G}$ on finite dimensional $\algQl$-vector spaces. The tensor structure on $\Sat_{G,\sF}$ is given by the convolution of perverse sheaves, cf. \cite{Ga01}. Let us recall from \cite{RZ15} why $\Sat_G$ is stable under convolution as well. If $G$ is split, then every $L^+G$-orbit is defined over $F$, and we have $\IC_{\{\mu\}}\in \Sat_G$ for all $\{\mu\}\in X_*(T)/W^{\rm abs}_0$. Thus by Lemma \ref{galoislem} the Galois action on $\om(\IC_{\{\mu_1\}}\star\IC_{\{\mu_2\}})=\om(\IC_{\{\mu_1\}})\otimes\om(\IC_{\{\mu_2\}})$ is trivial, and hence $\IC_{\{\mu_1\}}\star\IC_{\{\mu_2\}}$ is a direct sum of normalized intersection complexes. The general case follows from this observation by Galois descent applied to $\Sat_{G_E}$ where $E/F$ is a Galois extension which splits $G$, cf. \cite[Prop A.10]{RZ15}.

Further, the Satake equivalence \eqref{Satake}  gives a canonical way of constructing the dual group $\widehat{G}$ together with a canonical pinning $(\widehat{G},\widehat{B},\widehat{T},\widehat{X})$ which is fixed by the action of the Galois group $\Ga_F$, cf. \cite[\S 4]{Zhu15}, \cite[Rmk 4.7 ii)]{Ri16a}. Then $\Ga_F$ acts through a finite quotient on $\widehat{G}$ via outer automorphisms, and we form the dual group $^LG=\widehat{G}\rtimes\Ga_F$ viewed as a pro-algebraic group over $\algQl$. The following result is derived from \cite{RZ15} (cf. also \cite[\S 5]{Ri14a} and \cite[\S 5.5]{Zhu}).

\begin{thm}\label{GeoSat1} 
 The functor $\om\co \Sat_G\to \Rep_\algQl(\Ga_F)$ can be upgraded to an equivalence of abelian tensor categories such that the diagram
\[
\begin{tikzpicture}[baseline=(current  bounding  box.center)]
\matrix(a)[matrix of math nodes, 
row sep=1.5em, column sep=2em, 
text height=1.5ex, text depth=0.45ex] 
{\Sat_G&\Sat_{G,\sF} \\ 
\Rep_{\algQl}(\LG)&\Rep_{\algQl}(\widehat{G}) \\}; 
\path[->](a-1-1) edge node[above] {$(\str)_{\sF}$} (a-1-2);
\path[->](a-2-1) edge node[above] {$\res$} (a-2-2);
\path[->](a-1-1) edge node[right] {$\om$} (a-2-1);
\path[->](a-1-2) edge node[right] {$\om$} (a-2-2);
\end{tikzpicture}
\]
is commutative up to natural isomorphism, where $\res$ denotes the restriction of representations along $\widehat{G}\hookto \LG$.
\end{thm}

\begin{rmk} As $\on{R\Ga}(\bbP^1_\sF,\algQl)=\algQl\oplus\algQl[-2](-1)$ everything is normalized such that for $\Gl_2$ and the minuscule Schubert cell, the representation  $\om(\algQl\langle 1\rangle)=\algQl^2$ is the standard representation with the trivial Galois action.
\end{rmk}

\begin{cor} \label{GeoSat1_cor} Let $\{\mu\}$ be a conjugacy class of a geometric cocharacters defined over $E/F$. Then the normalized intersection complex $\IC_{\{\mu\}}$ on $\Gr_G^{\leq \{\mu\}}$ is an object of $\Sat_{G_E}$, and the cohomology $\om(\IC_{\{\mu\}})$ is under Theorem \ref{GeoSat1} the $^LG_E$-representation $V_{\{\mu\}}$ of highest weight $\{\mu\}$ defined in \cite[6.1]{Hai14}. 
\end{cor}
\begin{proof}
It is enough to check that $\Ga_E$ acts trivially on the highest weight subspace of $\om(\IC_{\{\mu\}})$ attached to any $E$-rational Borel subgroup $B$, cf. \cite[\S 6.1]{Hai14}. Passing to $G_E$ we may assume that $E=F$. The Galois action on $\om(\IC_{\{\mu\}})$ only depends on the quasi-split inner form $G^*$ of $G$: the cocycle defining $G$ is of the form $c\co \Ga_F\to \Aut(G^*_{\rm ad, \sF}), \ga\mapsto \on{Int}(g_\ga)$ where $g_\ga\in (G^*)(\sF)$, and $\on{Int}(g_\ga)$ denotes conjugation by $g_\ga$. The formation of the affine Grassmannian is functorial in the group, and $\Gr_G$ is constructed from $\Gr_{G^*}$ by twisting against the cocycle $c$. Hence, for any $\calA\in\Sat_G$, the action of $\ga\in\Ga_F$ under the identification $\om_{G}(\calA)\simeq \om_{G^*}(\calA)$ is given by $\on{Int}(g_\ga)\cdot \ga^*$. Since $\on{Int}(g_\ga)$ belongs to the smooth connected algebraic group $G^*_\ad$ which acts on $\Gr_{G^*}$, the induced action on $\om_{G^*}(\calA)$ is trivial, cf. Lemma \ref{homotopy_lem} below. Thus, $\om_{G}(\calA)\simeq\om_{G^*}(\calA)$ as $^LG$-representations, and hence, we may assume $G=G^*$ is quasi-split. 

By \cite[Lem 1.1.3]{Ko84} the class $\{\mu\}$ admits an $F$-rational representative $\mu\co \bbG_{m}\to T$ where $T$ is the centralizer of a maximal $F$-split torus in $G$. Given an $F$-rational Borel subgroup $B\subset G$, containing $T$, we may choose the representative such that $\mu$ is $B$-antidominant. As in \cite[Eq (3.6)]{MV07} (or \cite[Eq (5.3.11)]{Zhu}, or also Lemma \ref{intersectionlem} below), one has
\[
\Gr_{G,\sF}^{\leq\{\mu\}}\cap (\Gr_{B,\sF})_\mu = \{z^\mu\},
\]
which is an $F$-rational point by construction. We have $\IC_{\{\mu\}}|_{\{z^\mu\}}=\algQl\lan\lan2\rho_B,-\mu\ran\ran$ by our choice of normalization. The cohomology
\[
\bbH_c^{\lan2\rho_B,\mu\ran}((\Gr_{B})_\mu,\IC_{\{\mu\}})=\bbH^{\lan2\rho_B,\mu\ran}((\Gr_{B})_\mu,\IC_{\{\mu\}})=\algQl(\lan\rho_B,-\mu\ran),
\]
is a direct summand of $\om(\IC_{\{\mu\}})$, and identifies with the subspace of weight $\mu\in X_*(T)=X^*(\widehat{T})$, cf. \cite[Thm 3.6]{MV07} (or \cite[Thm 5.3.9]{Zhu}). Taking the twists in \eqref{galoistwist} into account, we conclude that $\Ga_F$ acts trivially on the highest weight space.
\end{proof}

\begin{lem} \label{homotopy_lem}
Let $X$ be an $\sF$-scheme acted on by a smooth geometrically connected $\sF$-group $H$. Then, for each $i\in\bbZ$, the induced action of $H(\sF)$ on the intersection cohomology groups $\bbH^i(X,\IC_X)$ is trivial.   
\end{lem}
\begin{proof} 
Let $f\co X\to \Spec(\sF)$ be the structure morphism. The argument to show that $H(\sF)$ acts trivially on $\bbH^i(X,\IC_X)=\calH^i(f_*\IC_X)$ is the same as in \cite[Lem 3.2.3]{LN08}.
\end{proof}

%Induction and restriction
\subsubsection{Induction and restriction} The geometric Satake equivalence in Theorem \ref{GeoSat1} is compatible with induction and restriction of representations in the following sense. For a finite separable extension $E/F$, let $^LG_E=\widehat{G}\rtimes \Ga_E$ considered as a closed algebraic subgroup of $^LG$. Then we have the induction and restriction of representations
\begin{align*}
I(\str)=\on{Ind}_{^LG_E}^{^LG}(\str)\co &\Rep_{\algQl}(^LG_E)\to \Rep_{\algQl}(^LG); \\
R(\str)=(\str)|_{^LG_E}\co &\Rep_{\algQl}(^LG)\to \Rep_{\algQl}(^LG_E), 
\end{align*}
which form a pair of adjoint functors $(R,I)$. The projection onto the first factor
\[
\pi\co \Gr_{G_E}=\Gr_G\times_{\Spec(F)}\Spec(E)\to \Gr_G
\]
is finite \'etale. Hence, we have the push-forward and pull-back on Satake categories
\begin{align*}
\pi_*\co &\Sat_{G_E}\to \Sat_{G}; \\
\pi^*\co &\Sat_{G}\to \Sat_{G_E}, 
\end{align*}
which form an adjoint pair of functors $(\pi^*,\pi_*)$.

\begin{prop} \label{GeoSat1_IndRes}
There are commutative diagrams of neutral Tannakian categories 
\[
\begin{tikzpicture}[baseline=(current  bounding  box.center)]
\matrix(a)[matrix of math nodes, 
row sep=1.5em, column sep=2em, 
text height=1.5ex, text depth=0.45ex] 
{\Sat_{G_E}&\Sat_{G}  & \text{and} & \Sat_{G}&\Sat_{G_E} \\ 
\Rep_{\algQl}(^LG_E)&\Rep_{\algQl}(^LG) && \Rep_{\algQl}(^LG)&\Rep_{\algQl}(^LG_E), \\}; 
\path[->](a-1-1) edge node[above] {$\pi_*$} (a-1-2);
\path[->](a-2-1) edge node[above] {$I$} (a-2-2);
\path[->](a-1-1) edge node[right] {$\om_{G_E}$} (a-2-1);
\path[->](a-1-2) edge node[right] {$\om_{G}$} (a-2-2);

\path[->](a-1-4) edge node[above] {$\pi^*$} (a-1-5);
\path[->](a-2-4) edge node[above] {$R$} (a-2-5);
\path[->](a-1-4) edge node[right] {$\om_{G}$} (a-2-4);
\path[->](a-1-5) edge node[right] {$\om_{G_E}$} (a-2-5);
\end{tikzpicture}
\]
the vertical arrows are given by the equivalence in Theorem \ref{GeoSat1}. 
\end{prop}
\begin{proof} As $\pi$ is finite \'etale, it is easy to see from proper (resp. smooth) base change that $\pi_*$ (resp. $\pi^*$) admits a symmetric monoidal structure with respect to convolution. Clearly, the operation $\pi^*$ corresponds to $R$. As the pairs $(\pi^*,\pi_*)$ and $(R,I)$ are adjoint, the Yoneda lemma implies that $\pi_*$ corresponds to $I$. 
\end{proof}

%Constant terms
\subsubsection{Constant terms}
Let $\chi\co \bbG_{m,F}\to G$ be a cocharacter. Let $M$ be its centralizer and let $P^\pm$ be the associated parabolic subgroups as in \eqref{hyperlocgroup}. The natural maps $M\leftarrow P^\pm \to G$ give maps of $F$-ind-schemes
\[
\Gr_M \overset{\;\;\,q^\pm}{\leftarrow} \Gr_{P^\pm} \overset{p^\pm}{\to} \Gr_G,
\]
and identify with the maps on the attractor, resp.~repeller by Proposition \ref{gctgeo}. The positive parabolic $P^+$ induces as in \eqref{decomposition} a decomposition into open and closed $F$-ind-subschemes
\begin{equation}\label{normalizedecom}
q^+=\coprod_{m\in \bbZ}q^+_m\co \Gr_{P^+}=\coprod_{m\in\bbZ} \Gr_{P^+,m}\longto \coprod_{m\in\bbZ} \Gr_{M,m}=\Gr_M.
\end{equation}
We write $q^-=\amalg_{m\in\bbZ}q_m^-\co \Gr_{P^-}=\amalg_{m\in\bbZ} \Gr_{P^-,m}\to \amalg_{m\in\bbZ} \Gr_{M,m}=\Gr_M$ according to \eqref{normalizedecom}, i.e. the ind-scheme $\Gr_{M,m}$ is contained in $\Gr_{P^-,m}\cap \Gr_{P^+,m}$ for any $m\in\bbZ$.

\begin{dfn}\label{NormConsTermGrass}
 The \emph{\textup{(}normalized\textup{)} geometric constant term} is the functor $\on{CT}_\chi^+\co D_c^b(\Gr_G)\to D_c^b(\Gr_M)$ (resp. $\on{CT}_\chi^-\co D_c^b(\Gr_G)\to D_c^b(\Gr_M)$) defined as the shifted pull-push functor
\[
\on{CT}_\chi^+\defined \bigoplus_{m\in\bbZ}(q^+_m)_!(p^+)^*\langle m\rangle \;\;\;\text{(resp. $\on{CT}_\chi^-\defined \bigoplus_{m\in\bbZ}(q^-_m)_*(p^-)^!\langle m\rangle$).}
\]
\end{dfn}

As in \cite{Br03, DG15, Ri19}, there is a natural transformation of functors
\begin{equation}\label{Bradentrafo}
\on{CT}_\chi^-\longto \on{CT}_\chi^+,
\end{equation}
which is an isomorphism for $\bbG_m$-equivariant complexes. As the $\bbG_m$-action on $\Gr_G$ factors through the $L^+G$-action, the transformation \eqref{Bradentrafo} is an isomorphism of functors $\on{CT}_\chi^-\simeq \on{CT}_\chi^+$ when restricted to $\Sat_G$. We define the functor $\on{CT}_M^G\co \Sat_G\to D_c^b(\Gr_M)$ as
\[
\on{CT}_M^G\defined \on{CT}^+_\chi|_{\Sat_G}.
\]
 We also denote $\on{CT}^G_M=\on{CT}_M$ if $G$ is understood. We derive the following result from \cite{BD, MV07, RZ15}.

\begin{thm} \label{GeoSat2}
i\textup{)} For each $\calA\in \Sat_G$, the complex $\on{CT}_M(\calA)$ is an object in $\Sat_M$ and does not depend on the choice of $\chi$ such that $Z_G(\chi)=M$.\smallskip\\
ii\textup{)}  There is commutative diagram of neutral Tannakian categories 
\[
\begin{tikzpicture}[baseline=(current  bounding  box.center)]
\matrix(a)[matrix of math nodes, 
row sep=1.5em, column sep=2em, 
text height=1.5ex, text depth=0.45ex] 
{\Sat_G&\Sat_{M} \\ 
\Rep_{\algQl}(^LG)&\Rep_{\algQl}(^LM) \\}; 
\path[->](a-1-1) edge node[above] {$\on{CT}_M$} (a-1-2);
\path[->](a-2-1) edge node[above] {$\res$} (a-2-2);
\path[->](a-1-1) edge node[right] {$\om_G$} (a-2-1);
\path[->](a-1-2) edge node[right] {$\om_M$} (a-2-2);
\end{tikzpicture}
\]
where $\on{res}\co V\mapsto V|_{^LM}$ is the restriction of representations, and the vertical arrows are given by the equivalence in Theorem \ref{GeoSat1}. 
\end{thm}
\begin{proof}
Over $\bar{F}$, there is a canonical isomorphism $\om_G\simeq \om_M\circ \on{CT}_M^G$ \cite[Thm 3.6]{MV07}, \cite[Thm 5.3.9 (3); Rmk 5.3.16]{Zhu}, and we have to show that it is Galois equivariant (bear in mind \cite[Cor A.8]{RZ15}, \cite[Lem 5.5.7]{Zhu}, which allows us to pass between the ``geometric'' and ``algebraic'' notions of $^LG$ here). We need to unravel its construction. Let $\pi_G\co \Gr_G\to \Spec(F)$ (resp. $\pi_M\co \Gr_M\to \Spec(F)$) denote the structure morphism.\smallskip\\

{\it First case.} Let $G$ be quasi-split. First assume that $B^\pm=P^\pm$ is a Borel subgroup, and hence $T=M$ a maximal torus, i.e., the cocharacter $\chi$ is regular. Let $\nu$ be a Galois orbit in $X_*(T)$, and denote by $\Gr_{B^\pm,\nu}$ the corresponding connected component of $\Gr_{B^\pm}$. The map $p^\pm_\nu:=p^\pm|_{\Gr_{B^\pm,\nu}}$ factors by Lemma \ref{basicgeo} i) as
\[
\Gr_{B^\pm,\nu}\,\overset{j_\nu^\pm}{\longto}\,\overline{\Gr}_{B^\pm,\nu}\,\overset{i_\nu^\pm}{\longto}\,\Gr_G,
\]
where $j_\nu^\pm$ is a quasi-compact open immersion and $i_\nu^\pm$ is a closed immersion. Further, let $i_{<\nu^\pm}:=i_\nu^\pm|_{\overline{\Gr}_{B^\pm,<\nu}}$ where $\overline{\Gr}_{B^\pm,<\nu}:=\overline{\Gr}_{B^\pm,\nu}\bslash \Gr_{B^\pm,\nu}$ is the complement. This induces two Galois stable filtrations on the global (unshifted) cohomology functor $\bbH^*_G:= \left(\oplus_{i\in \bbZ}\calH^i\right)\circ\pi_{G_\sF,*}$: one is given for any $\calA\in\Sat_G$ by the kernels of
\begin{equation}\label{Fil1}
\on{Fil}_{\geq \nu}\bbH^*_{G}(\calA)\defined \ker(\bbH^*_G(\calA)\to \bbH^*_{G}((i_{<\nu}^+)^*\calA)),
\end{equation}
and one given by the images of 
\begin{equation}\label{Fil2}
\on{Fil}_{< \nu}'\bbH^*_{G}(\calA)\defined \on{im}(\bbH^*_G((i_{<\nu}^-)^!\calA)\to \bbH^*_{G}(\calA)).
\end{equation}
Here $\nu$ runs through the Galois orbits $X_*(T)/\Ga_F$, which is partially ordered by the requirement that $\nu\leq \nu'$ if one (and then any) representative of $\nu'-\nu$ in $ X_*(T)$ is a sum of coroots for roots appearing in $N^+_\sF$ with non-negative integer valued coefficients. That \eqref{Fil1} and \eqref{Fil2} are indeed filtrations on $\bbH^*_G$ indexed by the partially ordered set $(X_*(T)/\Ga_F, \leq)$ follows immediately from \cite[Prop 3.1]{MV07}, \cite[Cor 5.3.8]{Zhu}. We claim that the Galois stable filtrations \eqref{Fil1} and \eqref{Fil2} split each other and (taking the shifts into account) induce the desired isomorphism. 

Let $\bar{\pi}^\pm$ denote the structure morphism of $\overline{\Gr}_{B^\pm}$. There is a diagram of natural transformations
\begin{equation}\label{Bradendiagram}
\begin{tikzpicture}[baseline=(current  bounding  box.center)]
\matrix(a)[matrix of math nodes, 
row sep=1.5em, column sep=2em, 
text height=1.5ex, text depth=0.45ex] 
{(\bar{\pi}^-_\sF)_*\circ (i^-_\nu)^!&\pi_{G_\sF,*}& (\bar{\pi}^+_\sF)_!\circ (i^+_\nu)^* \\ 
(\bar{\pi}^-_\sF)_*\circ (j^-_\nu)_*\circ (j^-_\nu)^!\circ (i^-_\nu)^!& &(\bar{\pi}^+_\sF)_!\circ(j^+_\nu)_!\circ(j^+_\nu)^*\circ (i^+_\nu)^*  \\
(\pi_{T_\sF})_*\circ (q_\nu^-)_*\circ (p_\nu^-)^!& & (\pi_{T_\sF})_!\circ (q_\nu^+)_!\circ (p_\nu^+)^*\\}; 
\path[->](a-1-1) edge node[above] {} (a-1-2);
\path[->](a-1-2) edge node[above] {} (a-1-3);
\path[->](a-1-1) edge node[left] {$\on{(1)}$} (a-2-1);
\path[->](a-2-3) edge node[right] {$\on{(2)}$} (a-1-3);
\path[->](a-2-1) edge node[left] {$\simeq$} (a-3-1);
\path[->](a-2-3) edge node[right] {$\simeq$} (a-3-3);
\path[->](a-3-1) edge node[above] {$\simeq$} (a-3-3);
\end{tikzpicture}
\end{equation}
where we have used that $\bar{\pi}^\pm$ (resp. $\pi_G$) is ind-proper, and $j^\pm_\nu$ is an open immersion. The bottom arrow is the isomorphism in Braden's theorem \cite{Br03} (cf. also \cite{DG15, Ri19}) which is justified by Proposition \ref{gctgeo}. One checks that \eqref{Bradendiagram} commutes up to natural isomorphism. As in the proof of \cite[Thm 5.3.9 (3)]{Zhu}, the main point to prove the claim is that the maps (1) and (2) in \eqref{Bradendiagram} are on the $i$-cohomology an isomorphism if $i=n_\nu$, and zero otherwise. Note that $n_{\dot{\nu}}=n_{\ddot{\nu}}$ for any $\dot{\nu}, \ddot{\nu}$ in the Galois orbit $\nu$, and that $i_{\nu,\sF}=\coprod_{\dot{\nu}\in \nu}i_{\dot{\nu}}$. We may without loss of generality assume that $F=\sF$ and $\nu\in X_*(T)$. Then the statement about the maps (1) and (2) for $\calA\in\Sat_G$ follows from the equality
\begin{equation}\label{trafoeq2}
\calH^i\circ (q_\nu^+)_!\circ (p_\nu^+)^*\calA\,=\,0, \;\;\; i\not = n_\nu,
\end{equation} 
which ultimately rests on the dimension formula for $\Gr_G^{\leq \{\mu\}}\cap (\Gr_{B^\pm})_{\nu}$ (cf. \cite[Thm 3.2]{MV07}). For general ground fields the dimension formula can be derived as in \cite[Prop 4.2]{Ri14a} from \cite{GHKR06} and \cite{NP01} using a flatness argument. This shows the claim.

We see that the isomorphism $\om_G\simeq \om_M\circ \on{CT}_M^G$ is $\Ga_F$-equivariant in the case where $M=T$ is a maximal torus. Note the construction also shows the compatibility with shifts and twists. The case of where $M$ is a general Levi follows from the base change identity $\on{CT}_T^M\circ \on{CT}_M^G\simeq \on{CT}_T^G$ whenever we choose a regular cocharacter $\chi'\co\bbG_m\to M$, cf. \cite[Prop 5.3.29]{BD}. It is also not difficult to see that this is independent of the auxiliary choice of $\chi'$.

This reasoning also implies that $\on{CT}_M(\calA)$ is perverse: Indeed, we may assume $F$ to be algebraically closed, and if $M=T$ is a maximal torus, this is just \eqref{trafoeq2}. For a general Levi $M$, we have $\on{CT}_M^G(\calA)\in \on{Perv}_{L^+M}(\Gr_M)$ if and only if $\on{CT}_T^M\circ \on{CT}_M^G(\calA)\simeq \on{CT}_T^G(\calA)\in \on{Perv}_{L^+T}(\Gr_T)$ (cf. \cite[Lem 3.9]{MV07}) which holds true.\smallskip\\
{\it General case.} Let $G$ be a general connected reductive group. By the proof of Corollary \ref{GeoSat1_cor}, for $\calA\in \Sat_G$ the Galois action on $\om_G(\calA)$ only depends on the quasi-split form of $G$. The functor $\on{CT}_M^G$ is defined over the ground field $F$, and since we already know that $\on{CT}_M^G(\calA)$ is perverse, the same holds true for the Galois action on $\om_M\circ \on{CT}_M^G(\calA)$, cf. Lemma \ref{homotopy_lem}. Thus, the isomorphism $\om_G\simeq \om_M\circ \on{CT}_M^G$ is $\Ga_F$-equivariant by the previous case.\smallskip\\
{\it Proof of i\textup{)}.} We claim that for $\calA\in\Sat_G$, one has $\on{CT}_M(\calA)\in \Sat_M$. We already know that $\on{CT}_M(\calA)$ is perverse. Further, if $E/F$ splits $G$, and $\calA_E$ is a direct sum of normalized intersection complexes, so is $\on{CT}_M(\calA_E)\in \on{Perv}_{L^+M_E}(\Gr_{M,E})$ by using the isomorphism $\om_G\simeq \om_M\circ \on{CT}_M^G$ and Lemma \ref{galoislem}. This easily implies $\on{CT}_M(\calA)\in \Sat_M$. 

Let $\chi'\co \bbG_m\to G$ be another cocharacter with $Z_G(\chi')=M=Z_G(\chi)$. Then the corresponding parabolic subgroups $(P')^+=(G)^{\chi',+}$ and $P^+$ are conjugate by an element $g\in G(F)$: say $g\cdot P^+\cdot g^{-1}=(P')^+$. The isomorphism $G\to G, h\mapsto ghg^{-1}$ gives by transport of structure an isomorphism $c_g\co \Gr_G\to \Gr_G$, and we have 
\[
\on{CT}_{\chi'}^+=\on{CT}_{\chi}^+\circ\; (c_g)^*.
\]
But on Satake categories $(c_g)^*\co \Sat_G\simeq \Sat_G$ is the identity since every $L^+G$-orbit is stable under conjugation by $g$, and hence $\on{CT}_{\chi'}^+|_{\Sat_G}\simeq \on{CT}_{\chi}^+|_{\Sat_G}$. This implies i).\smallskip\\
{\it Proof of ii\textup{)}.} We have to equip $\on{CT}_M$ with a (necessarily unique) monoidal structure such that the isomorphism $\om_G\simeq \om_M\circ \on{CT}_M$ is monoidal. If $F$ is separably closed, this follows from the arguments in \cite{BD}, \cite{MV07} (cf. \cite[Thm 1.7.4]{Xue17} for details). As all ind-schemes are defined over the ground field $F$, we deduce the general case by descent. For convenience we give a short proof of the definition of the monoidal structure on $\on{CT}_M$ which is based on \cite[Thm 3.1]{Ri19}. Let us denote by $\Gr_{G,\bbA^1_F}^{\on{BD}}$ the functor on the category of $F$-algebras $R$ parametrizing triples $(x,\calF,\al)$ consisting of a point $x\in \bbA^1_F(R)$, a $G$-torsor $\calF\to\bbP^1_R$ and a trivialization 
\[
\al\co \calF|_{\bbP^1_R\bslash (\{x\}\cup\{0\})}\simeq \calF^0|_{\bbP^1_R\bslash (\{x\}\cup\{0\})},
\]
where $\calF^0$ is the trivial torsor. The forgetful map $\Gr_{G,\bbA^1_F}^{\on{BD}}\to \bbA^1_F$ is representable by an ind-projective ind-scheme, cf. \cite[Thm 3.1.3]{Zhu}. Let $\Spec(F\pot{z})\to \bbA^1_F$ be the completed local ring at zero, and define
\[
\Gr_G^{\on{BD}}\defined \Gr_{G,\bbA^1_F}^{\on{BD}}\times_{\bbA^1_F}\Spec(F\pot{z}).
\]
Let $s$ (resp. $\eta$) denote the closed (resp. open) point in $\Spec(F\pot{z})$. We have as $F$-ind-schemes
\begin{equation}\label{ProdStr}
\Gr_{G,s}^{\on{BD}}= \Gr_G\;\;\;\;\text{and}\;\;\;\; \Gr_{G,\eta}^{\on{BD}}=\Gr_G\times \Gr_G \times \Spec(F\rpot{z}).
\end{equation}
The construction is functorial in $G$, and there are maps of $F\pot{z}$-ind-schemes
\begin{equation}\label{GroupDiag}
\Gr_{M}^{\on{BD}}\leftarrow \Gr_{P^\pm}^{\on{BD}}\to \Gr_{G}^{\on{BD}}.
\end{equation}
The cocharacter $\chi$ induces a $\bbG_m$-action on $\Gr_{G,\bbA^1_F}^{\on{BD}}$ trivial on $\bbA^1_F$, and hence a $\bbG_m$-action on $\Gr_G^{\on{BD}}$ which is the action \eqref{Gmaction} on each factor in \eqref{ProdStr}. Similar to the argument in Lemma \ref{loclinaffine}, one sees that the $\bbG_m$-action on $\Gr_G^{\on{BD}}$ is Zariski locally linearizable. In view of Theorem \ref{Gmthm}, there are maps of $F\pot{z}$-ind-schemes
\begin{equation}\label{HypLocBD}
(\Gr_{G}^{\on{BD}})^0\leftarrow (\Gr_{G}^{\on{BD}})^\pm\to \Gr_{G}^{\on{BD}},
\end{equation}
and it is possible to see that this diagram identifies with \eqref{GroupDiag}. Note that this is true fiberwise by Proposition \ref{gctgeo}. Further, as in \cite{Ga01} (cf. also \cite{Zhu14}), we obtain for $\calA,\calB\in\Sat_G$ a canonical isomorphism of complexes
\begin{equation}\label{Convnearby}
\Psi_G^{\on{BD}}(\calA\boxtimes\calB\boxtimes\algQl)\simeq\calA\star\calB,
\end{equation}
where $\Psi_G^{\on{BD}}\co D_c^b(\Gr_{G,\eta}^{\on{BD}},\algQl)\to D_c^b(\Gr_{G,s}^{\on{BD}}\times_s\eta,\algQl)$ is the functor of nearby cycles with the conventions as in the appendix to \cite{Il94} (same conventions in \cite{Ri19}). First, let us ignore the shifts in the definition of $\on{CT}_M$. By the functorial properties of nearby cycles, there is a canonical map in $D_c^b(\Gr_{M,s}^{\on{BD}}\times_s\eta,\algQl)$ as
\begin{equation}\label{CTnearby}
\on{CT}_M\circ\Psi_G^{\on{BD}}(\calA\boxtimes\calB\boxtimes\algQl) \longto \Psi^{\on{BD}}_M(\on{CT}_M(\calA)\boxtimes\on{CT}_M(\calB)\boxtimes\algQl),
\end{equation}
where $\Psi^{\on{BD}}_M$ denote the nearby cycles for the family $\Gr_M^{\on{BD}}$. As all objects in $\Sat_G$ are $\bbG_m$-equivariant, the result \cite[Thm 3.3]{Ri19} applies and \eqref{CTnearby} is an isomorphism. Combining \eqref{Convnearby} and \eqref{CTnearby} we obtain the desired monoidal structure on $\on{CT}_M$. Up to sign (cf. \cite[Prop 6.4]{MV07} or \cite{Xue17}) the commutativity constraint is given by switching $\calA$, $\calB$ on the left hand side of \eqref{Convnearby}. Hence, once we know the compatibility of \eqref{CTnearby} with shifts it follows that $\on{CT}_M$ is a tensor functor. That the shifts agree is implied by the decomposition into open and closed sub-ind-schemes
\[
\Gr_{M}^{\on{BD}}\;=\;\coprod_{m\in\bbZ}\Gr_{M,m}^{\on{BD}},
\]   
which is the decomposition \eqref{normalizedecom} in the special fiber and takes in the generic fiber the form
\[
\Gr_{M,m,\eta}^{\on{BD}}\;=\;\coprod_{(m_1,m_2)}\Gr_{M,m_1}\times \Gr_{M,m_2}\times \eta.
\]
The coproduct runs over all pairs $(m_1,m_2)\in \bbZ^2$ with $m_1+m_2=m$. The theorem follows.
\end{proof}

%Affine flag varieties
\section{Affine flag varieties} \label{gctaffflag}
Let $k$ be either a finite field or an algebraically closed field. Let $F=k\rpot{t}$ be the Laurent power series local field with ring of integers $\calO=k\pot{t}$. For a smooth affine (geometrically) connected $F$-group scheme $G$, the \emph{\textup{(}twisted\textup{)} loop group $LG$} is the group functor on the category of $k$-algebras
\[LG\co R\longmapsto G(R\rpot{t}).\]
The loop group $LG$ is representable by an ind-affine ind-group scheme, cf.\,\cite[\S 1]{PR08}. Let $\calG$ be a smooth affine $\calO$-group scheme of finite type with geometrically connected fibers and generic fiber $\calG_F=G$. The \emph{\textup{(}twisted\textup{)} positive loop group $L^+\calG$} is the group functor on the category of $k$-algebras 
\[L^+\calG\co  R\longmapsto \calG(R\pot{t}).\] 
The positive loop group $L^+\calG$ is representable by a reduced closed subgroup scheme of $LG$. The \emph{\textup{(}partial\textup{)} affine flag variety $\Fl_\calG$} is the fpqc-sheaf on the category of affine $k$-algebras associated with the functor
\[  R\longmapsto LG(R)/L^+\calG(R).\]
The affine flag variety $\Fl_\calG$ is a separated ind-scheme of ind-finite type over $k$, and the quotient map $LG\to \Fl_\calG$ admits sections \'etale locally, cf.\,\cite[Thm 1.4]{PR08}. The affine flag variety is equipped with a transitive action of the loop group
\begin{equation}\label{flagact}
LG\times \Fl_\calG\longto \Fl_\calG.
\end{equation}
As the group scheme $\calG$ is smooth, the ind-scheme $\Fl_\calG$ represents the functor which assigns to every $k$-algebra $R$ the isomorphism classes of pairs $(\calF,\al)$ with
\begin{equation}\label{flagvar}
\begin{cases}
\text{$\calF$ a $\calG_{R\pot{t}}$-torsor on $\Spec(R\pot{t})$};\\
\text{$\al\co \calF|_{R\rpot{t}}\simeq \calF^0|_{R\rpot{t}}$ a trivialization},
\end{cases}
\end{equation}
where $\calF^0$ denotes the trivial torsor (this can be extracted from the reference \cite{Zhu} by comparing the definition given in its equation (1.2.1) with its Proposition 1.3.6).

\subsection{Schubert varieties} \label{Sch_var_sec} Let $G$ be a connected reductive $F$-group. Then the affine flag variety $\Fl_\calG$ is ind-proper (and then even ind-projective) if and only if $\calG$ is parahoric in the sense of Bruhat-Tits, cf. \cite[Thm A]{Ri16a}. Recall that parahoric group schemes $\calG=\calG_\bbf$ correspond bijectively to facets $\bbf$ of the (enlarged) Bruhat-Tits building, i.e.~$\calG_\bbf$ is (by definition) the fiberwise neutral component of the unique smooth affine $\calO$-group scheme whose generic fiber is $G$, and whose $\calO$-points are the pointwise fixer of $\bbf$ in $G(F)$. In this case, we also write
\[
\Fl_\bbf=\Fl_\calG.
\]
Let $I$ be the inertia subgroup of $\Ga_F$, and let $\Sigma= \Gal({\breve{F}/F})\simeq \Gal(\bar{k}/k)$, cf. notation. Let $A\subset G$ be a maximal $F$-split torus such that $\bbf$ is contained in the corresponding apartment of the Bruhat-Tits building. Denote by $M\subset G$ the centralizer of $A$ which is a minimal Levi subgroup. By Bruhat-Tits there exists a maximal $\breve{F}$-split torus $S$ containing $A$ and defined over $F$. As $G_{\breve{F}}$ is quasi-split by Steinberg's Theorem, the centralizer $T$ of $S$ is a maximal torus. We obtain a chain of $F$-tori
\begin{equation}\label{groupchain}
A\subset S \subset T,
\end{equation}
which are all contained in the minimal Levi subgroup $M$.

\begin{dfn}\label{IwahoriWeyl} i) The \emph{Iwahori-Weyl group over $\breve{F}$} is the group
\[
\breve{W}\defined \on{Norm}_G(S)(\breve{F})/\breve{T}_1,
\]
where $\breve{T}_1$ denotes the unique parahoric subgroup of $T(\breve{F})$, cf. \cite{HR08}.\smallskip\\
ii) The \emph{Iwahori-Weyl group over $F$} is the group
\[
W\defined \on{Norm}_G(A)(F)/M_1,
\]
where $M_1$ denotes the unique parahoric subgroup of $M(F)$, cf. \cite{Ri16b}.
\end{dfn}

For each $w \in \breve{W}$, we choose a lift $\dot{w} \in LG(\bar{k})$; the choice is normalized by requiring that, if $w \in T(\breve{F})/\breve{T}_1 \subset \breve{W}$, then $\dot{w} \in T(\breve{F}) \subset LG(\bar{k})$ has $\kappa_T(\dot{w}) =  \kappa_T(w)$ for the Kottwitz homomorphism $\kappa_T: T(\breve{F})/\breve{T}_1 \rightarrow X_*(T)_{I_F}$.

By \cite[\S 1.2]{Ri16b}, there is an injective group morphism
\begin{equation}\label{IWembed}
W\hookto \breve{W},
\end{equation}
which identifies $W=(\breve{W})^\Sig$. For any $w\in \breve{W}$, the \emph{Schubert variety $\Fl^{\leq w}_\bbf$ associated with $w$} is the reduced $L^+\calG_{\bar{k}}$-orbit closure 
\begin{equation}\label{schubertflag}
\Fl^{\leq w}_\bbf\defined \overline{L^+\calG_{\bar{k}} \cdot \dot{w}\cdot e_0},
\end{equation}
where $\dot{w} \in LG(\bar{k})$ is the element associated above to $w$, and $e_0\in\Fl_\bbf$ denotes the base point. The Schubert variety $\Fl^{\leq w}_\bbf$ is a geometrically irreducible projective scheme which is defined over some finite extension $k_E/k$. The $L^+\calG_{k_E}$-orbit of $\dot{w}\cdot e_0$ is denoted $\Fl^{w}_\bbf$ and is a smooth geometrically connected open dense $k_E$-subscheme of $\Fl^{\leq w}_\bbf$.  Further, $\Fl^{w}_\bbf$ (and hence $\Fl^{\leq w}_\bbf$) is defined over $k_E=k$ if $w \in\breve{W}_\bbf \cdot W \cdot \breve{W}_\bbf$.

The Iwahori-Weyl group $W$ (resp.\,$\breve{W}$) acts on the (enlarged) apartment $\scrA=\scrA(G,A,F)$ (resp. $\breve{\scrA}=\breve{\scrA}(G,S,\breve{F})$) by affine transformations (this is normalized as follows: if the Kottwitz homomorphism \cite[$\S7$]{Ko97} takes $w \in T(\breve{F})/\breve{T}_1 \subset \breve{W}$ to $\lambda \in X_*(T)_{I_F}$, then $w$ acts on $\breve{\scrA}$ by translation by $\lambda^\flat$, where $\lambda^\flat$ is the image of $\lambda$ in $X_*(T)_{I_F} \otimes \mathbb Q$). There is a natural inclusion of simplicial complexes 
\[
\scrA\hookto \breve{\scrA},
\]
which identifies $\scrA=(\breve{\scrA})^\Sig$. Let $\bba$ be an alcove containing $\bbf$ in its closure. Then there is a unique alcove $\breve{\bba}$ (resp. facet $\breve{\bbf}$) in $\breve{\scrA}$ containing $\bba$ (resp. $\bbf$). The choice of $\bba$ (resp. $\breve{\bba}$) equips $W$ (resp. $\breve{W}$) with a quasi-Coxeter structure and thus a length function and Bruhat order $(l,\leq)$ (resp.\,$(\breve{l},\leq)$), i.e., the simple reflections are the reflections at the walls meeting the closure of $\bba$ (resp. $\breve{\bba}$). More precisely, if $\Omega_{\breve{\bf a}}$ is the subgroup of $\breve{W}$ stabilizing $\breve{\bba}$, and $\breve{W}_{\rm aff}$ is the Coxeter group generated by the simple reflections corresponding to the walls of $\breve{\bba}$, then we have the decomposition
\begin{equation}\label{stabilizer_alcove}
\breve{W} = \breve{W}_{\rm aff} \rtimes \Omega_{\breve{\bba}}
\end{equation}
and the decomposition of $W$ is obtained by taking $\Sig$-invariants. The subgroup of $W$ (resp. $\breve{W}$) associated with $\bbf$ (resp. $\breve{\bf f}$) is the subgroup 
\begin{equation} \label{W_f_defn}
W_\bbf\defined \on{Norm}_G(A)(F)\cap\calG_\bbf(\calO)/M_1 \;\;\;\; \text{(resp. $\breve{W}_{\bbf}\defined \on{Norm}_G(S)(\breve{F})\cap\calG_\bbf(\breve{\calO})/\breve{T}_1$)}.
\end{equation}
Note that $\calG_{\bbf, \breve{\calO}}=\calG_{\breve{\bbf}}$ because parahoric group schemes are compatible with unramified extensions. The group $W_\bbf$ (resp. $\breve{W}_\bbf$) is a finite group which identifies with the subgroup generated by the reflections at the walls passing through $\bbf$ (resp. $\breve{\bbf}$). Note that $W_\bbf=(\breve{W}_{\bbf})^\Sig$ under \eqref{IWembed} (\cite[Cor 1.7]{Ri16b}). 

\begin{lem} \label{Bruhatflag} The natural map 
\[
W_\bbf \bslash W/W_\bbf \overset{\simeq}{\longto} L^+\calG(k)\bslash \Fl_\bbf(k), \;\; [w]\longmapsto L^+\calG(k)\cdot \dot{w}\cdot  e_0
\]
is bijective. 
\end{lem}
\begin{proof} The group $L^+\calG$ is an inverse limit of smooth geometrically connected $k$-groups. As $k$ is finite (or algebraically closed), an approximation argument and Lang's lemma show that $H^1(\Sig, \calG(\breve{\calO}))$ vanishes, and hence $\Fl_\calG(k)=LG(k)/L^+\calG(k)$. The lemma reduces to \cite[Thm 1.4]{Ri16b}.
\end{proof}

If $k$ is algebraically closed, then $W=\breve{W}$ and the map $\breve{W}_\bbf \bslash \breve{W}/\breve{W}_\bbf \simeq L^+\calG(\bar{k})\bslash \Fl_\bbf(\bar{k})$ is bijective by Lemma \ref{Bruhatflag}. By \cite[Rmk 9]{HR08}, the natural map 
\[
W_\bbf \bslash W/W_\bbf \overset{\simeq}{\longto} (\breve{W}_\bbf \bslash \breve{W}/\breve{W}_\bbf)^\Sig 
\]
is bijective. The Bruhat order $\leq$ induces a partial order on the double quotient $\breve{W}_\bbf \bslash \breve{W}/\breve{W}_\bbf$ compatible with the order on $W_\bbf \bslash W/W_\bbf$, and we have
\[
\Fl^{v}_\bbf\;\subset\; \Fl^{\leq w}_\bbf
\] 
if and only if $[v]\leq [w]$ in the induced Bruhat order on $\breve{W}_\bbf \bslash \breve{W}/\breve{W}_\bbf$ (resp. equivalently on $W_\bbf \bslash W/W_\bbf$ if both classes are $\Sig$-fixed). In particular, there is a presentation on reduced loci
\[
(\Fl_\bbf)_\red\;=\;\text{colim}_{v}\bigcup_{w\in \Sig\cdot v}\Fl^{\leq w}_\bbf
\]
where $v$ runs through the $\Sig$-orbits in $\breve{W}_\bbf \bslash \breve{W}/\breve{W}_\bbf$. Each such union of Schubert varieties is defined over $k$, stable under the $L^+\calG$-action and the stabilizers are geometrically connected, cf. \cite[Cor 2.3]{Ri16a}.

%G_m-actions on affine flag varieties
\subsection{Torus actions on affine flag varieties} \label{torussubsec}
Let $\calG=\calG_\bbf$, and let $\chi\co \bbG_{m,\calO}\to \calG$ be a $\calO$-cocharacter. The composition 
\[
\bbG_{m,k}\subset L^+\bbG_{m,\calO}\overset{L^+\!\chi \phantom{h}}{\longto} L^+\calG
\]
defines a $\bbG_m$-action on the affine flag variety $\Fl_\calG=\Fl_\bbf$. As in \eqref{hyperloc} we have the hyperbolic localization diagram
\begin{equation}\label{hyperflag}
(\Fl_\calG)^0\leftarrow (\Fl_\calG)^\pm\to \Fl_\calG.
\end{equation}
For a proof of the following lemma, which implies $(\Fl_\calG)^\pm$ and $(\Fl_\calG)^0$ are representable as ind-schemes by Theorem \ref{Gmthm}, we refer to Lemma \ref{loclinBD} below. 

\begin{lem}\label{loclinflag}
The $\bbG_m$-action on $\Fl_\calG$ is Zariski locally linearizable.
\end{lem}

\begin{lem} \label{basicgeoflag}
i\textup{)} The map $(\Fl_\calG)^\pm\to \Fl_\calG$ is schematic. \smallskip\\
ii\textup{)} The map $ (\Fl_\calG)^\pm\to (\Fl_\calG)^0$ is ind-affine with geometrically connected fibers, and induces an isomorphism on the group of connected components $\pi_0((\Fl_{\calG_{\bar{k}}})^\pm)\simeq\pi_0((\Fl_{\calG_{\bar{k}}})^0)$.
\end{lem}
\begin{proof}
 Part i) and ii) follow from Theorem \ref{Gmthm} ii) using Lemma \ref{loclinflag}.

%By \cite[Prop 1.3 b)]{PR08} there exists a faithful representation $\calG\hookto \Gl_{n,\calO}$ such that $\Gl_{n,\calO}/\calG$ is quasi-affine. Hence, the map $\Fl_\calG\to \Fl_{\Gl_{n,\calO}}$ is representable by a quasi-compact immersion (cf. \cite[Prop 1.2.6]{Zhu}) which implies closed immersion because $\Fl_\calG$ is ind-proper. Part i) follows using the closed embedding $\Fl_\calG\hookto \Fl_{\Gl_n}$, the consequent equality of sheaves $\Fl_{\calG}^\pm = \Fl_{\calG} \times_{\Fl_{\Gl_n}} \Fl^\pm_{\Gl_n}$, and the corresponding statement for $\Fl_{\Gl_n}$ (as $\Fl_{\Gl_n}$ can be embedded into an increasing union of projective spaces with a linear $\bbG_m$-action), cf. \cite[App B]{Dr}. Part ii) is \cite[Cor 1.12]{Ri19}, comp.\,Theorem \ref{Gmthm}. %{\cg why do we only need to consider $N$ instead of $\calN$?}{\cc The loop group $LN$ only depends on $N$. If $LN$ is connected, then its image in $\Fl_\calN$ (which is everything) is connected as well.}
\end{proof}

%Fixed points
\subsubsection{Fixed points, attractors and repellers}
Our aim is to express \eqref{hyperflag} in terms of group theoretical data, cf. Proposition \ref{gctgeoflag} below. 

The cocharacter $\chi$ acts on $\calG$ via $(\la,g)\mapsto \chi(\la)\cdot g\cdot \chi(\la)^{-1}$. Let $\calM=\calG^0$ be the centralizer, and let $\calP^+=\calG^+$ (resp. $\calP^-=\calG^-$) be the attractor (resp. the repeller). Note that the definition of the fixed point (resp. attractor, resp. repeller) functors \eqref{flow} makes sense over any base ring (or base scheme), cf. \cite{Ri19}. As $\calG$ is affine, the $\calO_F$-group functors are representable by closed subgroup schemes of $\calG$ (cf.\,\cite[Lem 1.9]{Ri19}), and there are natural maps of $\calO_F$-groups
\begin{equation}\label{hyperlocflag}
\calM\leftarrow \calP^\pm\to \calG.
\end{equation}
The generic fiber $M=\calM_F$ is an $F$-Levi subgroup of $G$ and $P^\pm=\calP^\pm_F$ are parabolic subgroups with $P^+\cap P^-=M$. The following result is similar to \cite[Lem 3.4]{He}.

\begin{lem}\label{groups} i\textup{)} The group schemes $\calM$ and $\calP^\pm$ are smooth closed subgroup schemes of $\calG$ with geometrically connected fibers.\smallskip\\
ii\textup{)} The centralizer $\calM$ is a parahoric group scheme for $M$.\smallskip\\
iii\textup{)} There is a semidirect product decomposition $\calP^\pm=\calM\ltimes \calN^\pm$ where $\calN^\pm$ is a smooth affine group scheme with geometrically connected fibers.
\end{lem}
\begin{proof}
The groups $\calM$, $\calP^\pm$ and the map $\calP^\pm\to \calM$ are smooth by \cite[Rmk 1.2, Thm 1.1 \& Rmk 3.3]{Mar15}. In particular, $\calM$ (resp. $\calP^\pm$) agrees with the flat closure of $M$ (resp. $P^\pm$) inside $\calG$, and the group $\calM$ is parahoric by \cite[Lem A.1]{Ri16a}. In particular, $\calM$ has geometrically connected fibers which implies $\calP^\pm$ having geometrically connected fibers by \cite[Cor 1.12]{Ri19}. Part i) and ii) follow. The scheme $\calN^\pm$ is the kernel of $\calP^\pm\to \calM$, and hence smooth with geometrically connected fibers. The lemma follows.%{\cc I removed the part about the unipotence. The scheme $\calN$ is a product of $\calO$-extensions of root subgroups. But I found nowhere that these extensions are unipotent.}
\end{proof}

The maps \eqref{hyperlocflag} induce maps of $k$-ind-schemes
\begin{equation}\label{hyperlocflagmap}
\Fl_\calM\leftarrow \Fl_{\calP^\pm}\to \Fl_\calG.
\end{equation}

 \begin{lem}\label{repmap}
i\textup{)} The map $\Fl_\calM\to \Fl_\calG$ is representable by a closed immersion.\smallskip\\
ii\textup{)} The map $\Fl_{\calP^\pm}\to \Fl_\calG$ is schematic, and factors as a quasi-compact immersion $\Fl_{\calP^\pm}\to (\Fl_\calG)^\pm$.\smallskip\\
iii\textup{)} The map $\Fl_{\calP^\pm}\to \Fl_\calM$ has geometrically connected fibers, and induces an isomorphism on the group of connected components $\pi_0(\Fl_{\calP_{\bar{k}}^\pm})\simeq\pi_0(\Fl_{\calM_{\bar{k}}})$.
\end{lem}
\begin{proof}
By \cite[Thm 2.4.1]{Co14}, the quotient $\calG/\calM$ is quasi-affine, and hence the map $\Fl_\calM\to \Fl_\calG$ is representable by a quasi-compact immersion (cf. \cite[Prop 1.2.6]{Zhu}) which implies closed immersion because $\Fl_\calM$ is ind-proper. For ii) let either $\calP=\calP^+$ or $\calP=\calP^-$. Choose $\calG\hookto \Gl_{n,\calO}$ such that $\Gl_{n,\calO}/\calG$ is quasi-affine, cf. \cite[Prop 1.3 b)]{PR08}. Let $P_\calO\subset \Gl_{n,\calO}$ be defined by the cocharacter $\bbG_{m,\calO}\to \calG\to \Gl_{n,\calO}$. Then we have $\calP=P_\calO\times_{\Gl_{n,\calO}}\calG$. By Lemma \ref{groups} iii), the group $\calP$ has geometrically connected fibers, and the main result of \cite{Ana73} implies that the fppf-quotient $P_\calO/\calP$ is representable by a quasi-projective scheme. The map $P_\calO/\calP\hookto \Gl_{n,\calO}/\calG$ is a monomorphism of finite type, and hence separated and quasi-finite \cite[Tag 0463]{StaPro} (use that ``quasi-finite''=``locally quasi-finite''+``quasi-compact''). By Zariski's main theorem \cite[02LR]{StaPro} the map is hence quasi-affine, and as the composition of quasi-affine maps is quasi-affine \cite[Tag 01SN]{StaPro}, the quotient $P_\calO/\calP$ is quasi-affine as well. Now there is a commutative diagram of $k$-ind-schemes 
 \begin{equation}\label{commsquare}
\begin{tikzpicture}[baseline=(current  bounding  box.center)]
\matrix(a)[matrix of math nodes, 
row sep=1.5em, column sep=2em, 
text height=1.5ex, text depth=0.45ex] 
{\Fl_\calP& \Fl_\calG  \\ 
\Fl_{P_\calO}& \Fl_{\Gl_{n,\calO}}, \\}; 
\path[->](a-1-1) edge node[above] {}  (a-1-2);
\path[->](a-1-1) edge node[left] {}  (a-2-1);
\path[->](a-2-1) edge node[below] {}  (a-2-2);
\path[->](a-1-2) edge node[right] {}  (a-2-2);
\end{tikzpicture}
\end{equation}
where all maps are monomorphisms. As $\Gl_{n,\calO}/\calG$ (resp. $P_\calO/\calP$) is quasi-affine, the vertical maps are representable by quasi-compact immersions. By \cite[Prop 6.2.11]{Co14} any two maximal split tori in $\Gl_{n,\calO}$ are $\calO$-conjugate which implies that $\bbG_{m,\calO}\to \Gl_{n,\calO}$ is defined over $k$ after conjugation. We may apply Lemma  \ref{basicgeo} i) to see that the map $\Fl_{P_\calO}\simeq (\Fl_{\Gl_{n,\calO}})^\pm \to \Fl_{\Gl_{n,\calO}}$ is schematic. Hence, the composition $\Fl_\calP\to \Fl_\calG\to \Fl_{\Gl_n,\calO}$ is schematic which by (the proof of) \cite[Cor 1.6.2 (b) (iii)]{LMB00} implies that $\Fl_\calP\to \Fl_\calG$ is schematic. Further, observe that in \eqref{commsquare} we have $(\Fl_\calG)^\pm=\Fl_{P_\calO}\times_{\Fl_{\Gl_{n,\calO}}}\Fl_\calG$ because $\Fl_\calG\to \Fl_{\Gl_n}$ is a closed immersion. As the composition $\Fl_{\calP}\to (\Fl_\calG)^\pm\to \Fl_{P_\calO}$ is a quasi-compact immersion, it follows that $\Fl_{\calP}\to (\Fl_\calG)^\pm$ is a quasi-compact immersion as well. For iii), let $k=\bar{k}$ and write $\calP=\calM\ltimes\calN$ with generic fiber $P=M\ltimes N$ as in Lemma \ref{groups} iii). The fiber above the base point of $\Fl_{\calP}\to \Fl_\calM$ is $\Fl_\calN=LN/L^+\calN$ (which is enough to consider by transitivity of the $LM$-action on $\Fl_\calM$). As the map $LN\to \Fl_\calN$ is surjective, it suffices to show that $L N$ is connected. But $N$ is a successive $\bbG_a$-extension, and we reduce to the case that $N=\bbG_a$ which is obvious.  
\end{proof}

\begin{comment}
\begin{lem}\label{wewin}
Let $X\hookto Y\hookto Z$ be monomorphisms of ind-schemes. If the composition is schematic, then $X\hookto Y$ is schematic.
\end{lem}
\begin{proof} Let $T\to Y$ be a scheme. To show that $X\times_YT$ is a scheme we may assume that $T$ is affine, hence quasi-compact. In particular, the map $T\to Y=\on{colim}_jY_j$ factors over some $Y_j$, and it is enough to prove that $X\times_YY_j$ is a scheme. But as $Y_j\hookto Y$ is a monomorphism, we have as sheaves $X\times_YY_j=X\times_ZY_j$, and the latter is a scheme by assumption. 
\end{proof}
\end{comment}

\begin{prop} \label{gctgeoflag}
The maps \eqref{hyperflag} and \eqref{hyperlocflagmap} fit into a commutative diagram of $k$-ind-schemes
\[
\begin{tikzpicture}[baseline=(current  bounding  box.center)]
\matrix(a)[matrix of math nodes, 
row sep=1.5em, column sep=2em, 
text height=1.5ex, text depth=0.45ex] 
{\Fl_\calM & \Fl_{\calP^\pm} & \Fl_\calG \\ 
(\Fl_{\calG})^0& (\Fl_{\calG})^\pm& \Fl_{\calG}, \\}; 
\path[->](a-1-2) edge node[above] {}  (a-1-1);
\path[->](a-1-2) edge node[above] {}  (a-1-3);
\path[->](a-2-2) edge node[below] {}  (a-2-1);
\path[->](a-2-2) edge node[below] {} (a-2-3);
\path[->](a-1-1) edge node[left] {$\iota^0$} (a-2-1);
\path[->](a-1-2) edge node[left] {$\iota^\pm$} (a-2-2);
\path[->](a-1-3) edge node[left] {$\id$} (a-2-3);
\end{tikzpicture}
\]
where $\iota^0$ and $\iota^\pm$ are monomorphisms with the following properties:\smallskip\\
i\textup{)} The maps $\iota^0$ and $\iota^\pm$ are closed immersions which are open on reduced loci.\smallskip\\
ii\textup{)} If $G=G_0\otimes_k F$ is constant \textup{(}hence unramified\ over $F$\textup{)} then the maps $\iota^0$ and $\iota^\pm$ are open and closed immersions. \smallskip\\
iii\textup{)} If $\calG_{\breve{\calO}}$ is a special parahoric \textup{(}i.e. $\calG$ is very special\textup{)}, then the maps $\iota^0$ and $\iota^\pm$ are surjective on topological spaces.
\end{prop}

\begin{rmk} We conjecture that the maps $\iota^0$ and $\iota^\pm$ are always open and closed immersions. The method of proof below shows that this already follows from the case of a very special vertex. For tamely ramified groups, this can be done by taking inertia invariants. For general groups, we lack a sufficiently good theory of the open cell in twisted affine flag varieties in order to prove this.
\end{rmk}

The proof of Proposition \ref{gctgeoflag} is finished in \ref{endgctgeoflag} below. We first explain how to construct the diagram. As the $\bbG_m$-action on $\Fl_\calM$ is trivial, Lemma \ref{repmap} i) implies that we obtain a closed immersion $\iota^0\co\Fl_\calM\to (\Fl_\calG)^0$. The map $\iota^\pm\co \Fl_{\calP^\pm}\to (\Fl_\calG)^\pm$ is constructed in Lemma \ref{repmap} ii) and is a quasi-compact immersion. In terms of the moduli interpretation \eqref{flagvar} the map $\iota^\pm$ is given by a Rees construction as in \S 2 above: For a $k$-algebra $R$, and a point $(\calF^\pm,\al^\pm)\in \Fl_{\calP^\pm}(R)$, the pullback $(\calF^\pm_{ \bbA^1_R},\al^\pm_{ \bbA^1_R})\in \Fl_{\calP^\pm}(\bbA^1_R)$ is by definition a bundle 
\[
\calF^\pm\to \Spec(R[z]\pot{t}),
\]
where we identify $\bbA^1_R=\Spec(R[z])$. As in \eqref{monoidaction} there is a $\bbA^1_\calO$-group morphism
\begin{equation}\label{realmap1}
\on{gr}^\pm_\chi\co \calP^\pm\times \bbA^1_\calO\longto \calP^\pm\times\bbA^1_\calO.
\end{equation}
such that $\on{gr}^\pm_\chi|_{\{1\}}=\id$ and $\on{gr}^\pm_\chi|_{\{0\}}$ factors as $\calP^\pm\to \calM\to \calP^\pm$. The base change of \eqref{realmap1} along $\Spec(R[z]\pot{t})\to \Spec(\calO[z])=\bbA^1_\calO$ gives a morphism of $\Spec(R[z]\pot{t})$-groups, and we define
\[
\on{Rees}_\chi(\calF^\pm,\al^\pm)\defined \on{gr}^\pm_{\chi , *}(\calF^\pm_{ \bbA^1_R},\al^\pm_{ \bbA^1_R}) \in \Fl_{\calP^\pm}(\bbA^1_R).
\]
As in \eqref{Reesmap} the $\bbG_m$-equivariance follows from the construction, and one shows that this gives an isomorphism of $k$-schemes $\on{Rees}_\chi\co \Fl_{\calP^\pm}\to (\Fl_{\calP^\pm})^\pm$ which is inverse to the map given by evaluating at the unit section. This constructs the diagram in Proposition \ref{gctgeoflag}, and we need some preparation for its proof.

%Chnaging the facet
\subsubsection{Changing the facet} If ${\bf c}$ is any facet in the closure of ${\bf f}$, then we obtain a morphism $\calG_{\bf f} \rightarrow \calG_{\bf c}$ of $\calO$-groups which is the identity in the generic fiber, and gives a closed immersion $L^+\calG_{\bf f} \rightarrow L^+\calG_{\bf c}$. Hence, we obtain a $\mathbb G_m$-equivariant surjective map of $k$-spaces $\Fl_{\calG_{\bf f}} \rightarrow \Fl_{\calG_{\bf c}}$. If $G_0$ denotes the maximal reductive quotient of the special fiber $\calG_{{\bf c}, k}$, and if $P_0$ denotes the image of the map of $k$-groups $\calG_{{\bf f},k} \rightarrow \calG_{{\bf c},k}\to G_0$, then $P_0\subset G_0$ is a parabolic subgroup, and there is an isomorphism of $k$-schemes
\[
 L^+\calG_{\bf c}/ L^+\calG_{\bf f}\simeq G_0/P_0.
\]

\begin{lem}\label{0-fs}
i\textup{)}The map $\Fl_{\calG_{\bf f}} \rightarrow  \Fl_{\calG_{\bf c}}$ is \'etale locally trivial on the base with general fiber $G_0/P_0$. In particular, the map $\Fl_{\calG_{\bf f}} \rightarrow  \Fl_{\calG_{\bf c}}$ is schematic smooth proper and surjective. \smallskip\\
ii\textup{)} The induced morphism $(\Fl_{\calG_{\bf f}})^0 \rightarrow (\Fl_{\calG_{\bf c}})^0$ is smooth proper and surjective. \smallskip\\
iii\textup{)} The induced morphism $(\Fl_{\calG_{\bf f}})^\pm \rightarrow (\Fl_{\calG_{\bf c}})^\pm$ is smooth and surjective.
\end{lem}
\begin{proof}
The \'{e}tale local triviality of $LG \rightarrow LG/ L^+\calG_{\bf c}$ is proved in \cite[Thm 1.4]{PR08}. This also follows from \eqref{flagvar}: let $\calF\to \Spec(R\pot{t})$ be a $\calG_{{\bf c},R\pot{t}}$-torsor. If we denote $\calF_0=\calF\otimes_{R\pot{t}}R$ for $R\pot{t}\to R$, $t\mapsto 0$, then $\calF_0$ is an $\calG_{{\bf c}, R}$-torsor which has a section over some \'etale cover $R\to R'$ (because $\calG_{{\bf c}, R}$ is smooth). By the formal lifting criterion for smooth morphism, we obtain a section $\Spf(R'\pot{t})\to \calF$ which gives a section $\Spec(R'\pot{t})\to \calF$ (because $\calF$ is affine). 

Now let $Y\to \Fl_{\calG_{\bf c}}$ be a map from a scheme, and denote by $X= \Fl_{\calG_{\bf f}}\times_{ \Fl_{\calG_{\bf c}}}Y$ the base change. We want to prove that $X$ is a scheme, and that the map $X\to Y$ is smooth proper. The desired properties are Zariski local on the base, and hence we may assume that $Y=\Spec(R)$ is affine. By the discussion above, there exists an \'etale affine cover $U\to Y$, and a section
\[
\begin{tikzpicture}[baseline=(current  bounding  box.center)]
\matrix(a)[matrix of math nodes, 
row sep=1em, column sep=2em, 
text height=1.5ex, text depth=0.45ex] 
{&LG \\ 
U&\Fl_{\calG_{\bf c}}. \\}; 
\path[->](a-2-1) edge[densely dotted] node[above] {$g$}  (a-1-2);
\path[->](a-1-2) edge node[right] {} (a-2-2);
\path[->](a-2-1) edge node[below] {} (a-2-2);
\end{tikzpicture}
\] 
Consider the map $\pi\co (G_0/P_0)\times_k U\subset \Fl_{\calG_{\bf f},U}\overset{m_g}{\longto} \Fl_{\calG_{\bf f},U}\to \Fl_{\calG_{\bf f}}$ where $m_g$ denotes the operator which is induced from multiplication with $g$. By definition, we have a cartesian diagram of $k$-ind-schemes
\[
\begin{tikzpicture}[baseline=(current  bounding  box.center)]
\matrix(a)[matrix of math nodes, 
row sep=1.5em, column sep=2em, 
text height=1.5ex, text depth=0.45ex] 
{X\times_Y U& U  \\ 
\Fl_{\calG_{\bf f}}& \Fl_{\calG_{\bf c}}, \\}; 
\path[->](a-1-1) edge node[above] {}  (a-1-2);
\path[->](a-1-1) edge node[left] {}  (a-2-1);
\path[->](a-2-1) edge node[below] {}  (a-2-2);
\path[->](a-1-2) edge node[right] {}  (a-2-2);
\end{tikzpicture}
\]
which induces a $U$-map $(\pi,\id)\co (G_0/P_0)\times_k U\to X\times_Y U =: X_U$. The map is an isomorphism with inverse constructed similarly using $m_{g^{-1}}$.  Now $X_U \rightarrow X$ is a surjective morphism of ind-schemes, since surjectivity is stable under base change in the category of ind-schemes. As $X$ is an ind-scheme and $X_U$ is a quasi-compact scheme, the surjective morphism $X_U \rightarrow X$ factors through a closed subscheme $X_i \hookto X$, hence $X_i \cong X$, i.e.\,$X$ is a scheme. Further, as $(G_0/P_0)_U\simeq X_U\to U$ is the projection, the scheme $X$ is proper and smooth over $Y$ as these properties can be checked \'etale locally on the base. These remarks imply i). Lemma \ref{0-perm} now implies ii) and iii).
\end{proof}

\begin{rmk} Using the formal smoothness of $LG$ in a similar way, one can prove that $\Fl_{\calG_{\bf f}} \to {\rm Spec}(k)$ is formally smooth.
\end{rmk}

\subsubsection{End of proof} \label{endgctgeoflag}
Write $\mathcal P^{\pm}_{\bf f}$ (resp. $\calM_\bbf$) for $\mathcal P^{\pm}$ (resp. $\calM$) when $\mathcal G = \mathcal G_{\bf f}$. 

\begin{proof}[Proof of Proposition \ref{gctgeoflag}]
We may assume $k=\bar{k}$ is algebraically closed. Let $\calG=\calG_\bbf$ for a facet $\bbf$ of the Bruhat-Tits building. The proof proceeds in three steps (1) $\bbf=\bbf_0$ is a special vertex, (2) $\bbf=\bba$ is an alcove and (3) $\bbf$ is a general facet.\smallskip\\
\emph{Step (1).} Let $\bbf=\bbf_0$ be a special vertex, i.e.\,$\calG=\calG_{\bbf_0}$ is a special parahoric. The Iwasawa decomposition
\[
LG(k)=LP^\pm(k)\cdot L^+\calG(k)
\] 
implies that the maps $\iota^0$ and $\iota^\pm$ are bijections on $k$-points which shows part iii). As the map $\iota^0$ (resp. $\iota^\pm$) is a closed immersion (resp. locally closed immersion), it is an isomorphism on the underlying reduced subschemes which shows i). If $G=G_0\otimes_kF$, then any special parahoric is hyperspecial and part ii) reduces to Proposition \ref{gctgeo}. Step (1) follows.\smallskip\\
\emph{Step (2).} Let $\bbf=\bba$ be an alcove, i.e. $\calG=\calG_\bba$ is an Iwahori group scheme. Choose a special facet $\bbf_0$ contained in the closure of $\bba$. Then the morphism $\calG_\bba\to \calG_{\bbf_0}$ induces a $\bbG_m$-equivariant proper smooth map on affine flag varieties $\Fl_{\bba}\longto \Fl_{\bbf_0}$ by Lemma \ref{0-fs}. Hence, we obtain a commutative diagram of $k$-ind-schemes
\begin{equation}\label{gctgeoflag:eq1}
\begin{tikzpicture}[baseline=(current  bounding  box.center)]
\matrix(a)[matrix of math nodes, 
row sep=1.5em, column sep=2em, 
text height=1.5ex, text depth=0.45ex] 
{\Fl_{\calM_\bba}& \Fl_{\calM_{\bbf_0}}  \\ 
(\Fl_\bba)^{0}& (\Fl_{\bbf_0})^{0}, \\}; 
\path[->](a-1-1) edge node[above] {}  (a-1-2);
\path[->](a-1-1) edge node[left] {$(\ast)$}  (a-2-1);
\path[->](a-2-1) edge node[below] {}  (a-2-2);
\path[->](a-1-2) edge node[right] {}  (a-2-2);
\end{tikzpicture}
\end{equation}
where $\Fl_{\calM_{\bbf_0}}\to (\Fl_{\bbf_0})^{0}$ is an isomorphism on reduced loci. The morphism on fixed points $(\Fl_\bba)^{0}\to (\Fl_{\bbf_0})^{0}$ is proper surjective and smooth by Lemma \ref{0-fs}. After passing to reduced loci, we want to show $(\ast)$ is an immersion which is both open and closed. It is enough to check that the map $(\ast)$ is an open immersion on fibers over points in $(\Fl_{\calM_{\bf f_0}})_{\rm red}$. Let us check why this is enough. We invoke the {\em crit\`{e}re de platitude par fibres} of EGAIV, 11.3.10, which implies that a morphism of finitely presented flat $S$-schemes $f : X \rightarrow Y$ is flat  (resp.\,an open immersion) if and only if $f_{\bar s}: X_{\bar s} \rightarrow Y_{\bar s}$ is flat (resp.\,an open immersion) for all geometric points $\bar{s}$ of $S$ (cf.\,\cite[7.4]{DR73}). We may apply this to the diagram of ind-schemes above, since the horizontal arrows are smooth (hence flat) by Lemma \ref{0-fs}.  For ``closed immersion'', we use Lemma \ref{repmap} i).

So we need to prove that the map $(\ast)$ is an open immersion on fibers. By the transitivity of the $L\calM_{\bbf_0}$-action, it is enough to consider the fiber above the base point. Let $\calG_{\bbf_0,k}\to G_0$ be the maximal reductive quotient. The image of $\calM_{\bbf_0,k}$ in $G_0$ is the Levi subgroup $M_0\subset G_0$ given by the centralizer of the cocharacter
\begin{equation}\label{gctgeoflag:eq2}
\bbG_{m,k}\overset{\chi_k}{\longto} \calG_{\bbf_0,k}\longto G_0. 
\end{equation}
The image $B_0$ of the composition $\calG_{\bba,k}\to \calG_{\bbf_0,k}\to G_0$ is a Borel subgroup in $G_0$. The cocharacter \eqref{gctgeoflag:eq2} factors by definition through $B_0\subset G_0$, and the image of $\calM_{\bba,k}$ in $B_0$ is its centralizer which is $M_0\cap B_0$.
Thus, the map $(\ast)$ in \eqref{gctgeoflag:eq1} becomes on the fiber above the base point
\begin{equation}\label{gctgeoflag:eq3}
M_0/(M_0\cap B_0)\longto (G_0/B_0)^{0}.
\end{equation}
This map is easily seen to be an open immersion by using the big open cell in the split connected reductive group $G_0$. This implies part i) for $\iota^0$ in the case of an alcove. For $\iota^\pm$ we use Lemma \ref{0-fs} iii) to deduce, analogously to the case of $\iota^0$, that $\iota^\pm$ is an open immersion on reduced loci. 
Note that we already know that $\iota^\pm$ is a quasi-compact immersion (cf. Lemma \ref{repmap} ii)). 
Hence, to deduce that $\iota^\pm$ is a closed immersion it remains to show that $\iota^\pm$ maps $\Fl_{\calP_\bba^\pm}$ bijectively onto a union of connected components of $(\Fl_\bba)^\pm$.
Using the result for $\iota^0$, this is shown in Lemma \ref{gctflagpoints} below. 
This implies i) for $\iota^\pm$. If $G=G_0\otimes_kF$, then by step (1) we do not need to pass to the reduced loci in \eqref{gctgeoflag:eq1} which implies ii) and finishes the proof of step (2). \smallskip\\
\emph{Step (3).} Next let $\bbf$ be a general facet, and choose an alcove $\bba$ containing $\bbf$ in its closure. As in the previous case, we obtain now a commutative diagram of $k$-ind-schemes
\begin{equation}\label{fixedpointthm:eq4}
\begin{tikzpicture}[baseline=(current  bounding  box.center)]
\matrix(a)[matrix of math nodes, 
row sep=1.5em, column sep=2em, 
text height=1.5ex, text depth=0.45ex] 
{\Fl_{\calM_\bba}& \Fl_{\calM_{\bbf}}  \\ 
(\Fl_\bba)^{0}& (\Fl_{\bbf})^{0}, \\}; 
\path[->](a-1-1) edge node[above] {}  (a-1-2);
\path[->](a-1-1) edge node[left] {}  (a-2-1);
\path[dashed,->](a-1-1) edge node[left] {}  (a-2-2);
\path[->](a-2-1) edge node[below] {}  (a-2-2);
\path[->](a-1-2) edge node[right] {}  (a-2-2);
\end{tikzpicture}
\end{equation}
where the dashed arrow is smooth on reduced loci (as the composition of an open immersion with a smooth morphism, cf.\,Lemma \ref{0-fs}). The map $\Fl_{\calM_\bba}\to \Fl_{\calM_{\bbf}}$ is smooth surjective, and hence the closed immersion $\Fl_{\calM_{\bbf}}\to (\Fl_{\bbf})^{0}$ is smooth on reduced loci by \cite[Tag 02K5]{StaPro}. In particular, it is an open immersion as well. This finishes i) for $\iota^0$ and the argument for $\iota^\pm$ is analogous. Again for ii), we do not need to pass to reduced loci by virtue of step (2). This finishes step (3) and the proposition follows.
\end{proof}

The following lemma is used in the proof above, and implies that the fibers of the map $(\Fl_\calG)^\pm\to (\Fl_\calG)^0$ agree on $\bar{k}$-valued points with the fibers of the map $\Fl_{\calP^\pm}\to \Fl_\calM$ for all points in $ \Fl_\calM(\bar{k})$: Let
\[
C^0=(\Fl_{\calG})^0(k)\bslash \iota^0(\Fl_\calM(k)) \;\;\; \text{(resp. $C^\pm=(\Fl_{\calG})^\pm(k)\bslash \iota^\pm(\Fl_{\calP^\pm}(k))$)}.
\]

\begin{lem}\label{gctflagpoints}
Under $(\Fl_\calG)^0\simeq \Fl_\calM(k)\amalg C^0$ and $(\Fl_\calG)^\pm(k)\simeq \Fl_{\calP^\pm}(k)\amalg C^\pm$, the diagram in Proposition \ref{gctgeoflag} gives on $k$-points the commutative diagram of sets
\[
\begin{tikzpicture}[baseline=(current  bounding  box.center)]
\matrix(a)[matrix of math nodes, 
row sep=1.5em, column sep=2em, 
text height=1.5ex, text depth=0.45ex] 
{\Fl_\calM(k) & \Fl_{\calP^\pm}(k) & \Fl_\calG(k) \\ 
\Fl_\calM(k)\amalg C^0 & \Fl_{\calP^\pm}(k)\amalg C^\pm & \Fl_{\calG}(k). \\}; 
\path[->](a-1-2) edge node[above] {}  (a-1-1);
\path[->](a-1-2) edge node[above] {}  (a-1-3);
\path[->](a-2-2) edge node[below] {}  (a-2-1);
\path[->](a-2-2) edge node[below] {} (a-2-3);
\path[->](a-1-1) edge node[left] {$\iota^0$} (a-2-1);
\path[->](a-1-2) edge node[left] {$\iota^\pm$} (a-2-2);
\path[->](a-1-3) edge node[left] {$\id$} (a-2-3);
\end{tikzpicture}
\]
\end{lem}

\begin{proof} We may assume $\bar{k}=k$ (the assertion follows by taking Galois invariants). If $\calP^\pm=\calM\ltimes\calN^\pm$ as in Lemma \ref{groups} iii), then the fiber over the base point of $\Fl_{\calP^\pm}(k)\to\Fl_\calM(k)$ is $X:=\Fl_{\calN^\pm}(k)$. Let $Y$ denote the fiber over the base point of $(\Fl_{\calG})^\pm(k)\to(\Fl_\calG)^0(k)$. As the above diagram is $LM(k)$-equivariant, it is enough to show that the map of sets 
\begin{equation}\label{fibermap}
\iota^\pm(k)|_X\co X\to Y
\end{equation}
is a bijection. Since $\iota^\pm$ is a monomorphism, the map \eqref{fibermap} is clearly injective. Now if $\calG=\calG_{\bbf_0}$ is a special parahoric, then $\iota^\pm$ is bijective by Proposition \ref{gctgeoflag} iii), and hence \eqref{fibermap} is surjective as well in this case. If $\calG=\calG_\bba$ is an Iwahori, then we choose a special facet $\bbf_0$ contained in the closure of $\bba$. The diagram in the formulation of the lemma is functorial with respect to the map of $\calO$-groups $\calG_\bba\to \calG_{\bbf_0}$, and we consider the left square. For the respective fibers above the base points, we obtain a commutative diagram of sets
\begin{equation}\label{alcind}
\begin{tikzpicture}[baseline=(current  bounding  box.center)]
\matrix(a)[matrix of math nodes, 
row sep=1.5em, column sep=2em, 
text height=1.5ex, text depth=0.45ex] 
{X_{\bbf_0}&  X_\bba \\ 
Y_{\bbf_0} & Y_\bba , \\}; 
\path[->](a-1-2) edge node[above] {}  (a-1-1);
\path[->](a-1-1) edge node[left] {}  (a-2-1);
\path[->](a-2-2) edge node[below] {}  (a-2-1);
\path[->](a-1-2) edge node[right] {}  (a-2-2);
\end{tikzpicture}
\end{equation}
and one checks that the horizontal maps are surjective. A diagram chase together with consideration of the $LN$-action shows that it is enough to see that the fibers above the base points of $X_{\bbf_0}$ resp. $Y_{\bbf_0}$ map bijectively onto each other. These fibers are identified with the $k$-points of the horizontal fibers over the base points in the commutative diagram of $k$-schemes 
\[
\begin{tikzpicture}[baseline=(current  bounding  box.center)]
\matrix(a)[matrix of math nodes, 
row sep=1.5em, column sep=2em, 
text height=1.5ex, text depth=0.45ex] 
{M_0/M_0\cap B_0&  P_0^\pm/P_0^\pm\cap B_0  \\ 
(G_0/B_0)^0 & (G_0/B_0)^\pm, \\}; 
\path[->](a-1-2) edge node[above] {}  (a-1-1);
\path[->](a-1-1) edge node[left] {}  (a-2-1);
\path[->](a-2-2) edge node[below] {}  (a-2-1);
\path[->](a-1-2) edge node[right] {}  (a-2-2);
\end{tikzpicture}
\]
where $G_0$ is the maximal reductive quotient of $\calG_{\bbf_0,k}$ and $P^\pm_0$ are the parabolic subgroups given by the image of $\calP^\pm_{\bbf_0, k}\subset \calG_{\bbf_0, k}$ in $G_0$, cf. \eqref{gctgeoflag:eq3}. The classical Bruhat decomposition implies that these horizontal fibers agree. This implies the surjectivity of \eqref{fibermap} for an Iwahori. If $\calG=\calG_\bbf$ is a general parahoric, we choose an alcove $\bba$ containing $\bbf$ in its closure. As in \eqref{alcind} one checks that the map $Y_\bba\to Y_\bbf$ is surjective which implies the surjectivity of $X_\bbf\to Y_\bbf$ using the diagram analogous to \eqref{alcind}. This proves the lemma.
\end{proof}
 
 %Connected components
\subsubsection{Connected components} \label{conncompflagsec}
We fix a chain of $F$-tori $A\subset S\subset T$ as in \eqref{groupchain} such that the facet $\bbf$ is contained in the apartment $\scrA=\scrA(G,A,F)$. We assume that the cocharacter $\chi_F$ factors through $A\subset G$ (hence $A\subset M$), and that $T\subset M$ which can always be arranged. We use the maximal torus $T$ to form the algebraic fundamental group $\pi_1(M)=X_*(T)/X_*(T_{M_\scon})$, cf. \eqref{fundamentalgroup}. Let $I\subset \Ga_F$ be the inertia group, and let $\Sig=\Ga_F/I$ the Galois group of $k$. By \cite[\S 2.a.2]{PR08}, the Kottwitz morphism (defined in \cite[\S 7]{Ko97}) is a locally constant morphism of ind-group schemes
\begin{equation}\label{Kottwitzmap}
\kappa_M\co  LM_{\bar{k}}\longto \underline{\pi_1(M)}_I,
\end{equation}
where $\pi_1(M)_I$ denotes the coinvariants under the inertia group $I$. In particular, as $L^+\calM$ is geometrically connected, the map \eqref{Kottwitzmap} gives an isomorphism on the group of connected components 
\begin{equation}\label{conncompflag}
\pi_0(\Fl_{\calM_{\bar{k}}})\overset{\simeq}{\longto} \pi_1(M)_I.
\end{equation}
By Lemma \ref{basicgeoflag}, we have an inclusion on connected components 
\[
\pi_0(\Fl_{\calP^\pm_{\bar{k}}})=\pi_0(\Fl_{\calM_{\bar{k}}})\subset \pi_0((\Fl_{\calG_{\bar{k}}})^0)=\pi_0((\Fl_{\calG_{\bar{k}}})^\pm).
\]
For $\bnu\in \pi_1(M)_I$, denote by $(\Fl_{\calG_{\bar{k}}})^0_\bnu$ (resp. $(\Fl_{\calG_{\bar{k}}})^\pm_\bnu$) the corresponding connected component. Note that all maps in Proposition \ref{gctgeoflag} are compatible with the decomposition into connected components. The disjoint sum of connected components
\begin{equation}
(\Fl_{\calG_{\bar{k}}})^{0,c}=\coprod_{\bnu\in\pi_1(M)_I}(\Fl_{\calG_{\bar{k}}})_{\bnu}^0 \;\;\; \text{(resp. $(\Fl_{\calG_{\bar{k}}})^{\pm,c}=\coprod_{\bnu\in\pi_1(M)_I}(\Fl_{\calG_{\bar{k}}})_{\bnu}^\pm$)}.
\end{equation}
is $\Sig$-stable, and hence defined over $k$. The ind-scheme $(\Fl_{\calG_{\bar{k}}})^{0,c}$ agrees on reduced loci with $\Fl_\calM$ by Proposition \ref{gctgeoflag} i). Further, we have a monomorphism 
\[
\Fl_{\calP_{\bar{k}}^\pm}\hookto(\Fl_{\calG_{\bar{k}}})^{\pm,c}, 
\]
which is a bijection on $\bar{k}$-points by Lemma \ref{gctflagpoints}. %{\cg and I think we don't have a proof of this and maybe will remove it:}{\cc I agree. The method of proof does not apply without having smoothness of $X^+\to Y^+$.} {\cm and which is an isomorphism if $G = G_0 \otimes_k F$ by Proposition \ref{gctgeoflag} ii)}.

Let $\calP^\pm=\calM\ltimes \calN^\pm$ with generic fiber $P^\pm=M\ltimes N^\pm$. Let $N$ be either $N^+$ or $N^-$. Let $\rho_N$ denote the half-sum of the roots in $N_\sF$ with respect to $T_\sF$. To every $\bnu\in \pi_1(M)_I$, we attach the number 
\[
n_{\bar{\nu}}\defined\lan2\rho_N,\dot{\nu}\ran,
\] 
where $\dot{\nu}\in X_*(T)$ denotes a representative of $\bnu$. Since the pairing $\lan\str,\str\ran$ is $I$-invariant, and $\lan\rho_N,\al^\vee\ran=0$ for all $\al^\vee\in X_*(T_{ M_\scon})$, the number $n_\bnu$ is well-defined. As in \eqref{decomposition} above the function $\pi_1(M)_I\to \bbZ, \bnu\mapsto n_\bnu$ is constant of $\Sig$-orbits, and we obtain a decomposition
\begin{equation}\label{decompositionflag}
(\Fl_{\calG})^{0,c}=\coprod_{m\in\bbZ}(\Fl_{\calG})_{m}^0 \;\;\; \text{(resp. $(\Fl_{\calG})^{\pm,c}=\coprod_{m\in\bbZ}(\Fl_{\calG})_{m}^\pm$)},
\end{equation}
where $(\Fl_{\calG})_{m}^0$ (resp. $(\Fl_{\calG})_{m}^\pm$)) denotes the disjoint sum over all $(\Fl_{\calG_{\bar{k}}})_{\bnu}^0$ (resp. $(\Fl_{\calG_{\bar{k}}})_{\bnu}^\pm$) with $n_\bnu=m$. The diagram in Proposition \ref{gctgeoflag} restricts to a commutative diagram of $k$-ind-schemes
\begin{equation}\label{restrictflag}
\begin{tikzpicture}[baseline=(current  bounding  box.center)]
\matrix(a)[matrix of math nodes, 
row sep=1.5em, column sep=2em, 
text height=1.5ex, text depth=0.45ex] 
{\Fl_\calM & \Fl_{\calP^\pm} & \Fl_\calG \\ 
(\Fl_{\calG})^{0,c}& (\Fl_{\calG})^{\pm,c}& \Fl_{\calG}, \\}; 
\path[->](a-1-2) edge node[above] {$q^\pm$}  (a-1-1);
\path[->](a-1-2) edge node[above] {$p^\pm$}  (a-1-3);
\path[->](a-2-2) edge node[below] {}  (a-2-1);
\path[->](a-2-2) edge node[below] {} (a-2-3);
\path[->](a-1-1) edge node[left] {$\iota^{0,c}$} (a-2-1);
\path[->](a-1-2) edge node[left] {$\iota^{\pm,c}$} (a-2-2);
\path[->](a-1-3) edge node[left] {$\id$} (a-2-3);
\end{tikzpicture}
\end{equation}
where $\iota^{0,c}$ and $\iota^{\pm,c}$ are nilpotent thickenings, i.e., isomorphisms on reduced loci. The maps $q^\pm=\coprod_{m\in\bbZ}q_m^\pm$ and $p^\pm=\coprod_{m\in\bbZ}p_m^\pm$ are compatible with the disjoint union decomposition \eqref{decompositionflag}. If $G=G_0\otimes_kF$ is constant, then $\iota^{0,c}$ and $\iota^{\pm, c}$ are isomorphisms.

%Beilinson-Drinfeld-Grassmannians 
\section{Beilinson-Drinfeld Grassmannians}\label{gctBDgrass}
The Beilinson-Drinfeld Grassmannians $\Gr_\calG$ are $\calO_F$-ind-schemes which degenerate the affine Grassmannian into the twisted affine flag variety. If $F\simeq \bbF_q\rpot{t}$, then $\Gr_\calG$ is constructed in \cite{Zhu14} for tamely ramified groups and in \cite{Ri16a} in general. If $F/\bbQ_p$ is a finite extension, then $\Gr_\calG$ is constructed in \cite{PZ13} for tamely ramified groups. We are interested in the study of fiberwise $\bbG_m$-actions on $\Gr_\calG$.

\subsection{Torus actions in equal characteristic} Let $F=k\rpot{t}$ be a Laurent series field with ring of integers $\calO=k\pot{t}$. The field $k$ is either finite or algebraically closed. Let $G$ be a connected reductive $F$-group, and choose $(A,S,T)$ as in \eqref{groupchain} above. Let $\calG=\calG_{\bbf}$ be a parahoric group whose facet $\bbf$ is contained in the apartment of $A$. Hence, the lft N\'eron model $\calA$ (resp. $\calS$, $\calT$) of $A$ (resp. $S$, $T$) is a closed subgroup scheme of $\calG$. Note that as $A$ (resp. $S_\bF$) is split, the smooth group scheme $\calA$ (resp. $\calS_{\breve{\calO}}$) is a $\calO$-split (resp. $\breve{\calO}$-split) torus.

\subsubsection{Beilinson-Drinfeld Grassmannians}
A technical but necessary step in the construction of BD-Grassmannians from local data is the spreading of the $\calO$-group schemes $(\calG,\calA,\calS,\calT)$ over a curve $X$. 

\begin{prop}\label{spreadprop}
 There exists a smooth affine connected $k$-curve $X$ of finite type with a point $x_0\in X(k)$, an identification $\hat{\calO}_{X,x_0}= \calO$ on completed local rings, and a tuple of smooth affine $X$-group schemes $(\ucG, \ucA, \ucS, \ucT)$ of finite type together with an isomorphism of $\calO$-group schemes
\[
(\ucG, \ucA, \ucS, \ucT)\otimes_{X}\calO\simeq (\calG,\calA,\calS,\calT),
\]
with the following properties: \smallskip\\
i\textup{)} The group scheme $\ucG|_{X\bslash x_0}$ is connected reductive with maximal torus $\ucT|_{X\bslash\{x_0\}}$, and the group $\ucG|_{(X\bslash x_0)_{\bar{k}}}$ is quasi-split.\smallskip\\
ii\textup{)} The group $\ucA$ is a maximal $X$-split torus, $\ucS$ is a maximal $X_{\bar{k}}$-split torus, and $\ucT$ is the centralizer of $\ucS$ in $\ucG$.\smallskip\\
iii\textup{)} The group scheme $\ucG\otimes \calO_{X,x_0}^h$ over the Henselization of the algebraic local ring is uniquely determined \textup{(}up to non-unique isomorphism\textup{)} by the property $\ucG\otimes\calO \simeq \calG$.
\end{prop}
\begin{proof} We follow the argument given in \cite[Lem 3.1]{Ri16a} using Proposition \ref{extensionprop} below. Let us recall the major steps: Let $v$ denote the restriction of the valuation of $k\rpot{t}$ to $E:=k(t)$. Then $E_v=F$ on completions, and we let $F'$ denote the Henselization of $(E,v)$. The subfield $F'\subset F$ is a Henselian valued field with completion $F$, and the same residue field $k$. By Proposition \eqref{extensionprop}, there exists a tuple of $F'$-groups $(\uG,\uA,\uS,\uT)$ with the properties as in i) extending the tuple $(G,A,S,T)$. For clarity, let us replace the tuple $(G,A,S,T)$ by $(\uG_F,\uA_F,\uS_F,\uT_F)$. Using the Beauville-Laszlo gluing lemma \cite{BL95} (cf. also \cite[Lem 5]{He10} for another method) we can glue $\uG$ with $\calG$ using the identification $\uG_{F}=G_F=\calG_F$. As in \cite[Lem 3.1, Cor A.3]{Ri16a} we obtain a smooth affine group scheme $\ucG'$ of finite type over $\calO_{F'}$ which extends $\calG$. Since the Beauville-Laszlo construction is functorial, we obtain also a tuple of smooth closed $\calO_{F'}$-subgroup schemes $(\ucA',\ucS',\ucT')$ extending the tuple $(\calA,\calS,\calT)$. As we glued along the identity morphism, it follows that the group $\ucA'$ (resp. $\ucS'$) is a $\calO_{F'}$-split (resp. $\breve{\calO}_{F'}$-split) torus. Further, the centralizer $Z_{\ucG'}(\ucS')$ is a smooth affine group scheme of finite type by \cite[Lem 2.2.4]{Co14}, and contains the commutative closed subscheme $\ucT'$. Thus, we must have $\ucT'=Z_{\ucG'}(\ucS')$ as both agree on an fpqc cover. Recall that $\calO_{F'}$ is the colimit over all finite \'etale local $\calO_{E,v}$-algebras $(B,\frakm)$ with $B/\frakm=k$. As the group scheme $\ucG'$ is of finite type, it is defined over some $(B,\frakm)$. Hence, the tuple $(\ucG',\ucA',\ucS',\ucT')$ extends to a tuple $(\ucG,\ucA,\ucS,\ucT)$ with the desired properties i) and ii) defined over some pointed curve $(X,x_0)$ with algebraic local ring $\calO_{X,x_0}=B$ (again because of the finite type hypothesis). In light of Proposition \ref{extensionprop}, part iii) is immediate from the construction.   
\end{proof}

Now as in \cite[Def 3.3]{Ri16a}, we use the spreading $\ucG$ to define the BD-Grassmannian $\Gr_\calG$ which is a separated $\calO$-ind-scheme of ind-finite type together with a transitive action of the global loop group
\begin{equation}\label{globalact}
\calL\calG\times \Gr_\calG\longto \Gr_\calG,
\end{equation}
such that the generic fiber of \eqref{globalact} is identified with the usual affine Grassmannian \eqref{affineact} (formed using an additional formal parameter), and the special fiber is identified with the twisted affine flag variety \eqref{flagact}. The BD-Grassmannian $\Gr_\calG$ is ind-proper (and then even ind-projective) over $\calO$ if and only if $\calG$ is parahoric in the sense of Bruhat-Tits, cf. \cite[Thm A]{Ri16a}. The construction is as follows: Denote by $\Gr_{\ucG,X}$ the functor on the category of $k$-algebras $R$ given by the isomorphism classes of triples $(x, \calF,\al)$ with
\begin{equation}\label{BDcurve}
\begin{cases}
\text{$x\in X(R)$ is a point};\\
\text{$\calF$ a $\ucG_{X_R}$-torsor on $X_R$};\\
\text{$\al\co \calF|_{X_R\bslash\Ga_x}\simeq \calF^0|_{X_R\bslash\Ga_x}$ a trivialization},
\end{cases}
\end{equation}
where $\calF^0$ denotes the trivial torsor, and $\Ga_x\subset X_R$ is the graph of $x$. Denote by $R\pot{\Ga_x}$ the ring of regular functions on the formal affine $R$-scheme given by the completion of $X_R$ along $\Ga_x$. Then $\Ga_x\subset \Spec(R\pot{\Ga_x})$ defines a Cartier divisor (in particular locally principal), and hence its complement is an affine scheme with ring of regular functions denoted by $R\rpot{\Ga_x}$. The \emph{global loop group} is the functor on the category of $k$-algebras given by
\begin{equation}\label{globalloop}
\calL_X\ucG\co R\mapsto \{(x,g)\;|\; \text{$x\in X(R)$ and $g\in \calG(R\rpot{\Ga_x})$}\},
\end{equation}
which is representable by an ind-affine ind-group scheme over $X$. By replacing $R\rpot{\Ga_x}$ with $R\pot{\Ga_x}$ in \eqref{globalloop}, one defines the \emph{global positive loop group} $\calL_X^+\ucG$ which is a flat affine $X$-group scheme, cf. \cite{Ri16a}. Again by the Beauville-Laszlo gluing lemma \cite{BL95} there is a natural isomorphism $\Gr_{\ucG,X}\simeq \calL_X\ucG/\calL^+_X\ucG$, and we obtain a transitive action morphism
\begin{equation}\label{globalacthelp}
\calL_X\ucG\times \Gr_{\ucG,X}\longto \Gr_{\ucG,X}.
\end{equation}
The map \eqref{globalact} is the base change of \eqref{globalacthelp} along the map $\Spec(\calO)=\Spec(\hat{\calO}_{X,x_0})\to X$.\medskip

\begin{rmk}\label{BDrmk}
 i) Since the formation of $ \Gr_{\ucG,X}$ is compatible with \'etale localizations on $X$ (cf. \cite[Lem 3.2]{Zhu14}), Proposition \ref{spreadprop} iii) implies that the ind-scheme $\Gr_\calG$ together with the map \eqref{globalact} is uniquely determined up to unique isomorphism by the data $(\uG,\calG,t)$. Indeed, for different choices of $(X,x_0)$ the ind-schemes $\Gr_\calG$ are canonically isomorphic by the above mentioned lemma. \smallskip\\
 ii) Spreading out is necessary for the following reason. If one copies and pastes \eqref{BDcurve} by replacing $X$ with $\Spec(\calO)$, and one tries to compute the generic fiber of the resulting functor, then one runs into the problem of computing the completion of $k\rpot{t}\otimes_kk\rpot{t}$ along the diagonal - a huge power series ring in an infinite number of variables. If one instead computes the completion of $k\rpot{t}\otimes_kk(t)$ along the diagonal, then one obtains $k\rpot{t}\rpot{z-t}$ where $z$ is identified with $1\otimes t$. 
See however \cite[\S0.3]{Ri19b} for a partial remedy. 
\end{rmk}

\subsubsection{Torus actions}
Let $\chi\co \bbG_{m,\calO}\to \calG$ be any cocharacter whose generic fiber $\chi_F$ factors through $A$. Then $\chi$ factors through $\calA$ because it is a maximal split torus in $\calG$. As the curve $X$ is connected, the cocharacter $\chi$ spreads uniquely to a cocharacter $\underline{\chi}\co \bbG_{m,X}\to \ucA$. Hence, by functoriality of the loop group construction, we obtain via the composition
\begin{equation}\label{flowBD}
\bbG_{m,\calO}\subset\calL^+\bbG_{m,\calO}\overset{\calL^+\!\chi\phantom{h}}{\longto} \calL^+\calG \subset \calL\calG
\end{equation}
a fiberwise $\bbG_{m}$-action on $\Gr_\calG\to \Spec(\calO)$.

\begin{lem}\label{loclinBD}
The $\bbG_m$-action on $\Gr_\calG$ is Zariski locally linearizable.
\end{lem}
\begin{proof} We may replace $X$ by the spectrum of the algebraic local ring at $x$, and choose a faithful $X$-representation $\ucG\hookto  \Gl_{n,X}$ such that $\Gl_{n,X}/\ucG$ is quasi-affine (cf. e.g. \cite[\S 2 Ex (1)]{He10} for the existence of the representation). Then the induced map on BD-Grassmannians $i\co \Gr_\calG\to\Gr_{\Gl_n}$ is representable by a quasi-compact immersion (cf. \cite[Prop 1.2.6]{Zhu}), and $\bbG_m$-equivariant if we equip $\Gl_{n}$ with an action via $\bbG_{m,X}\overset{\chi}{\to} \ucG\to \Gl_{n,X}$. As both $\Gr_\calG$ and $\Gr_{\Gl_{n}}$ are ind-proper, the map $i$ is a closed immersion, and in particular affine. We are reduced to showing that the $\bbG_{m,\calO}$-action on $\Gr_{\Gl_n}$ is Zariski locally linearizable. As $X$ is local affine, any two maximal split tori of $\Gl_{n,X}$ are conjugate over $X$, cf. \cite[Prop 6.2.11]{Co14}. Hence, we may assume that the image of $\bbG_{m,X}$ in $\Gl_{n,X}$ lies in the diagonal matrices. Then $\chi$ is defined over $k$, and the $\bbG_{m,\calO}$-action on the ind-scheme 
\[
\Gr_{\Gl_{n,\calO}}=\Gr_{\Gl_{n,k}}\times_{\Spec(k)}\Spec(\calO)
\] 
comes from the ground field $k$. The lemma follows from Lemma \ref{loclinaffine}.  
\end{proof}

\begin{rmk}\label{locally_closed_stratum_rem}
The proof of Lemma \ref{loclinBD} shows that there is a $\bbG_m$-equivariant closed immersion $\Gr_\calG\to \Gr_{\Gl_n}$ where (up to conjugation) the $\bbG_m$-action on the target is induced from a cocharacter of the diagonal matrices. 
The determinant line bundle $\calL_{\on{det}}$ on $\Gr_{\Gl_n}$ (cf.~\cite[\S1.5]{Zhu}) is an ample $\bbG_m$-equivariant line bundle.
Hence, it induces an $\bbG_m$-equivariant closed immersion 
\[
\Gr_{\Gl_n}\longto \bbP H^0(\Gr_{\Gl_n},\calL_{\on{det}}), 
\]
--also called Pl\"ucker embedding--, into an infinite-dimensional projective space with a linear $\bbG_m$-action.
As in \cite[Thm~B.0.3 (iii)]{Dr} it is now easy to see that the restriction of the map $(\Gr_\calG)^+\to \Gr_\calG$ to each connected component on the source is a locally closed immersion. 
Indeed, \cite{Dr} explains how we are reduced to the projective space above, and then we further reduce to affine spaces over connected schemes, which are handled by \cite[$\S$1.3]{Ri19}.
\end{rmk}

In light of Theorem \ref{Gmthm}, Lemma \ref{loclinBD} implies that there are maps of separated $\calO$-ind-schemes of ind-finite type
\begin{equation} \label{hyperBD}
(\Gr_\calG)^0\leftarrow (\Gr_\calG)^\pm\to \Gr_\calG,
\end{equation}
where $(\Gr_\calG)^0$ are the fixed points and $(\Gr_\calG)^+$ (resp. $(\Gr_\calG)^-)$ is the attractor (resp. repeller), cf. \eqref{hyperloc}. As the cocharacter $\chi$ spreads, the $\calO$-groups 
\begin{equation}\label{hyperlocBD}
\calM\leftarrow \calP^\pm\to \calG
\end{equation}
defined in \eqref{hyperlocflag} together with the maps spread as well. The following lemma is proven analogously to Lemmas \ref{basicgeoflag} and \ref{repmap}.

\begin{lem} \label{basicgeoBD_1}
i\textup{)} The map $ (\Gr_\calG)^\pm\to \Gr_\calG$ is schematic.%{\cg remove in light of Lemma \ref{schematic_BD}?}{\cc See comment above.}
\smallskip\\
ii\textup{)} The map $(\Gr_\calG)^\pm\to (\Gr_\calG)^0$ is ind-affine with geometrically connected fibers, and induces an isomorphism on the group of connected components $\pi_0((\Gr_{\calG})^\pm)\simeq\pi_0((\Gr_{\calG})^0)$. \smallskip\\
iii\textup{)} The map $\Gr_{\calP^\pm}\to \Gr_\calM$ has geometrically connected fibers. 
\end{lem}

%Fixed points
\subsubsection{Fixed points, attractors and repellers}

\begin{thm}\label{gctgeoBD}
The maps \eqref{hyperlocBD} induce a commutative diagram of $\calO$-ind-schemes
\[
\begin{tikzpicture}[baseline=(current  bounding  box.center)]
\matrix(a)[matrix of math nodes, 
row sep=1.5em, column sep=2em, 
text height=1.5ex, text depth=0.45ex] 
{\Gr_\calM & \Gr_{\calP^\pm} & \Gr_\calG \\ 
(\Gr_{\calG})^0& (\Gr_{\calG})^\pm& \Gr_{\calG}, \\}; 
\path[->](a-1-2) edge node[above] {}  (a-1-1);
\path[->](a-1-2) edge node[above] {}  (a-1-3);
\path[->](a-2-2) edge node[below] {}  (a-2-1);
\path[->](a-2-2) edge node[below] {} (a-2-3);
\path[->](a-1-1) edge node[left] {$\iota^0$} (a-2-1);
\path[->](a-1-2) edge node[left] {$\iota^\pm$} (a-2-2);
\path[->](a-1-3) edge node[left] {$\id$} (a-2-3);
\end{tikzpicture}
\]
with the following properties.\smallskip\\
i\textup{)} The generic fiber \textup{(}resp. special fiber\textup{)} is the diagram constructed in Proposition \ref{gctgeo} \textup{(}resp.\,Proposition \ref{gctgeoflag}\textup{)}. \smallskip\\
ii\textup{)} The maps $\iota^0$ and $\iota^\pm$ are closed immersions which are open immersions on the underlying reduced loci.\smallskip\\
iii\textup{) }If $G=G_0\otimes_k F$ is constant, then $\iota^0$ and $\iota^\pm$ are open and closed immersions.
\end{thm}

As in Proposition \ref{gctgeoflag} above, the map $\Gr_\calM\to \Gr_\calG$ is representable by a closed immersion (because $\calG/\calM$ is quasi-affine and $\Gr_\calM$ is ind-proper). The $\bbG_m$-action on $\Gr_\calM$ is trivial, and hence we obtain the closed immersion $\iota^0\co \Gr_\calM\to (\Gr_\calG)^0$. The map $\iota^\pm$ can be constructed in terms of the moduli description using a Rees construction, cf. Proposition \ref{gctgeo},\,\ref{gctgeoflag}. Here we use the argument given in Lemma \ref{repmap} ii) to construct a quasi-compact immersion $\iota^\pm$ on the spreadings $\Gr_{\ucP^{\pm},X}\to (\Gr_{\ucG,X})^\pm$ as follows. As in the proof of Lemma \ref{loclinBD}, we choose $\ucG\hookto \Gl_{n,X}$ such that $\Gl_{n,X}/\ucG$ is quasi-affine. Let $P_X^+\subset \Gl_{n,X}$ (resp. $P_X^-\subset \Gl_{n,X}$) be the attractor (resp. repeller) subgroup defined by the cocharacter $\bbG_{m,X}\overset{\chi}{\to} \ucG\to \Gl_{n,X}$. Then we have $\ucP^\pm=P_X^\pm\times_{\Gl_{n,X}}\ucG$. By Lemma \ref{groups} iii), the group $\ucP^\pm$ has geometrically connected fibers, and the main result of \cite{Ana73} implies that the fppf-quotient $P_X^\pm/\ucP^\pm$ is representable by a quasi-projective scheme. The map $P_X^\pm/\ucP^\pm\hookto \Gl_{n,X}/\ucG$ is a monomorphism of finite type, and by Zariski's main theorem it is quasi-affine, cf. proof of Lemma \ref{repmap} ii). Hence, $P_X^\pm/\ucP^\pm$ is quasi-affine as well. Now there is a commutative diagram of $X$-ind-schemes 
 \begin{equation}\label{commsquareBD}
\begin{tikzpicture}[baseline=(current  bounding  box.center)]
\matrix(a)[matrix of math nodes, 
row sep=1.5em, column sep=2em, 
text height=1.5ex, text depth=0.45ex] 
{\Gr_{\ucP^\pm} & & \\
&(\Gr_{\ucG})^\pm& \Gr_\ucG  \\ 
\Gr_{P^\pm_X}&(\Gr_{\Gl_{n,X}})^\pm& \Gr_{\Gl_{n,X}}, \\}; 
\path[->](a-1-1) edge[dotted] node[above] {}  (a-2-2);
\path[->](a-1-1) edge[bend left=15] node[above] {}  (a-2-3);
\path[->](a-1-1) edge node[above] {}  (a-3-1);
\path[->](a-3-1) edge node[above] {$\simeq$}  (a-3-2);
\path[->](a-2-2) edge node[above] {}  (a-2-3);
\path[->](a-2-2) edge node[left] {}  (a-3-2);
\path[->](a-3-2) edge node[below] {}  (a-3-3);
\path[->](a-2-3) edge node[right] {}  (a-3-3);
\end{tikzpicture}
\end{equation}
constructed as follows. The map $\Gr_\ucG\to \Gr_{\Gl_{n,X}}$ is a closed immersion (cf. proof of Lemma \ref{loclinBD}), and hence the square is Cartesian by general properties of attractor resp. repeller ind-schemes. This also constructs the dotted arrow in \eqref{commsquareBD} whose base change along $\Spec(\calO)\simeq\Spec(\hat{\calO}_{X,x_0})\to X$ is the map $\iota^\pm$. We claim that the dotted arrow is representable by a quasi-compact immersion. The map $\Gr_{P^\pm_X}\to(\Gr_{\Gl_{n,X}})^\pm$ is an isomorphism by Proposition \ref{gctgeo} because the cocharacter $\bbG_{m,X}\to \Gl_{n,X}$ is defined over $k$ after conjugation, cf. proof of Lemma \ref{loclinBD}. The map $\Gr_{\ucP^\pm}\to\Gr_{P^\pm_X}$ is a quasi-compact immersion because $\ucP^\pm/P^\pm_X$ is quasi-affine, and since $(\Gr_{\ucG})^\pm\to (\Gr_{\Gl_{n,X}})^\pm$ is a closed immersion by Corollary \ref{immersions}, the claim follows. This constructs the diagram in Theorem \ref{gctgeoBD}, and shows that $\iota^0$ is a closed immersion and $\iota^\pm$ is a quasi-compact immersion.

%Rees construction

\begin{proof}[Proof of Theorem \ref{gctgeoBD} i\textup{)}] It is immediate from the construction that the generic fiber (resp. special fiber) of the diagram in Theorem \ref{gctgeoBD} gives the diagram in Proposition \ref{gctgeo} (resp. Proposition \ref{gctgeoflag}). 
\end{proof}

The following proposition decomposes the image of the maps $\iota^0$ and $\iota^\pm$ into connected components, and part i) below implies Theorem \ref{gctgeoBD} ii).

\begin{prop}\label{conn_comp_BD}
Let either $N=\calN^+\otimes F$ or $N=\calN^-\otimes F$ with $\calN^\pm$ as in Lemma \ref{groups} iii\textup{)}. There exists an open and closed $\calO$-ind-subscheme $(\Gr_\calG)^{0,c}$ \textup{(}resp. $(\Gr_\calG)^{\pm,c}$\textup{)} of $(\Gr_\calG)^0$ \textup{(}resp. $(\Gr_\calG)^{\pm}$\textup{)} together with a disjoint decomposition, depending up to sign on the choice of $N$, as $\calO$-ind-schemes
\[
(\Gr_\calG)^{0,c}=\coprod_{m\in \bbZ}(\Gr_\calG)^{0}_m  \;\;\;\;\text{\textup{(}resp. $(\Gr_\calG)^{\pm,c}=\coprod_{m\in \bbZ}(\Gr_\calG)^{\pm}_m$\textup{)}},
\]
which has the following properties.\smallskip\\
i\textup{)} The map $\iota^0\co \Gr_\calM\to (\Gr_\calG)^{0}$ \textup{(}resp. $\iota^\pm\co \Gr_{\calP^\pm}\to (\Gr_\calG)^{\pm}$\textup{)} factors through $(\Gr_\calG)^{0,c}$ \textup{(}resp. $(\Gr_\calG)^{\pm,c}$\textup{)} inducing a closed immersion $\iota^{0,c}\co \Gr_\calM\to (\Gr_\calG)^{0,c}$ \textup{(}resp. $\iota^{\pm,c}\co \Gr_{\calP^\pm}\to (\Gr_\calG)^{\pm,c}$\textup{)} which is an isomorphism on reduced loci. \smallskip\\
ii\textup{)} The decomposition gives in the generic fiber decomposition \eqref{decomposition} and in the special fiber decomposition \eqref{decompositionflag}.\smallskip\\
iii\textup{)} The complement $(\Gr_\calG)^{0}\bslash (\Gr_\calG)^{0,c}$ \textup{(}resp. $(\Gr_\calG)^{\pm}\bslash (\Gr_\calG)^{\pm,c}$ \textup{)} has empty generic fiber, i.e., is concentrated in the special fiber.
\end{prop}
\begin{proof}
Let us construct the decomposition. Let $\pi_1(M)=X_*(T)/X_*(T_{M_\scon})$ be the algebraic fundamental group of $M$, cf. \eqref{fundamentalgroup}. For $\nu\in \pi_1(M)$, denote by $\dot{\nu}\in X_*(T)$ a representative which gives rise to a map
\[
\dot{\nu}\co \Spec(\sF)\,\to\, \Gr_{M}=\Gr_{\calM,F}\,\hookto\, \Gr_\calM.
\]
By the ind-properness of $\Gr_\calM\to \Spec(\calO)$, the map $\dot{\nu}$ extends uniquely to a map (still denoted)
\[
\dot{\nu}\co \Spec(\bar{\calO})\to \Gr_\calM,
\]
where $\bar{\calO}\subset \sF$ is the valuation subring of integral elements. By \cite[Lem 2.21]{Ri16a}, the special fiber $\bnu$ of $\dot{\nu}$ is the image under the canonical projection $X_*(T)\to X_*(T)_I$. Furthermore, since $\Gr_{\calM,\breve{\calO}}\to \Spec(\breve{\calO})$ is ind-proper and $\breve{\calO}$ is Henselian, the natural map
\[
\pi_0(\Gr_{\calM,\breve{\calO}})\,\overset{\simeq}{\longto}\, \pi_0(\Fl_{\calM,\bar{k}})\,\overset{\eqref{conncompflag}}{=}\, \pi_1(M)_I
\]
is an isomorphism by \cite[Arcata; IV-2; Prop 2.1]{SGA4}. This shows that there is a decomposition into connected components
\[
\Gr_{\calM,{\breve{\calO}}} = \coprod_{\bnu\in \pi_1(M)_I} (\Gr_{\calM,{\breve{\calO}}})_\bnu 
\]
such that $(\Gr_{\calM,{\breve{\calO}}})_\bnu\otimes \bar{k}\simeq (\Fl_{\calM,{\bar{k}}})_\bnu$ and $(\Gr_{\calM,{\breve{\calO}}})_\bnu\otimes \sF\simeq\coprod_{\nu\mapsto\bnu}(\Gr_{M,\sF})_\nu$. Likewise, we have $\pi_0((\Gr_{\calG,\breve{\calO}})^0)\simeq \pi_{0}((\Fl_{\calG,\bar{k}})^0)$ on connected components. Using Lemma \ref{basicgeoBD_1} ii), we get an inclusion
\[
\pi_1(M)_I\,=\,\pi_0(\Gr_{\calM,\breve{\calO}})\,\subset\,\pi_0 \left((\Gr_{\calG,\breve{\calO}})^0\right)\,=\, \pi_0 \left((\Gr_{\calG,\breve{\calO}})^\pm\right).
\]
For $\bar{\nu}\in \pi_1(M_I)$, we denote the corresponding connected component of $(\Gr_{\calG,\breve{\calO}})^0$ (resp. $(\Gr_{\calG,\breve{\calO}})^\pm$) by $(\Gr_{\calG})_{\bar{\nu}}^0$ (resp. $(\Gr_{\calG})_{\bar{\nu}}^\pm$).

For our choice of $N$, let $\rho_N$ denote the half-sum of the roots in $N_\sF$ with respect to $T_\sF$ which was used to define the integer $n_\nu=\lan2\rho_N,\dot{\nu}\ran$ (resp. $n_\bnu=\lan2\rho_N,\dot{\nu}\ran$) in \eqref{decomposition} (resp. \eqref{decompositionflag}). Note that we have $n_\nu=n_{\bnu}$ for all $\nu\mapsto \bnu$ by definition of $n_{\bar{\nu}}$. As in \eqref{decompositionflag}, we consider
\[
(\Gr_\calG)^{0}_m\defined\coprod_{\bar{\nu}}(\Gr_\calG)^{0}_{\bar{\nu}}  \;\;\;\;\text{(resp. $(\Gr_\calG)^{\pm}_m\defined\coprod_{\bar{\nu}}(\Gr_\calG)^{\pm}_{\bar{\nu}}$)},
\]
where the disjoint sum is indexed by all $\bar{\nu}\in \pi_1(M)_I$ such that $n_{\bar{\nu}}=m$. As the Galois action preserves the integers $n_{\bar{\nu}}$, the ind-scheme $(\Gr_\calG)^{0}_m$ (resp. $(\Gr_\calG)^{\pm}_m$) is defined over $\calO$. Note that $(\Gr_\calG)^{\pm}_m$ is the preimage of $(\Gr_\calG)^{0}_m$ along $(\Gr_\calG)^{\pm}\to (\Gr_\calG)^{0}$. We obtain a decomposition as $\calO$-ind-schemes
\[
(\Gr_\calG)^{0,c}\defined\coprod_{m\in \bbZ}(\Gr_\calG)^{0}_m  \;\;\;\;\text{(resp. $(\Gr_\calG)^{\pm,c}\defined\coprod_{m\in \bbZ}(\Gr_\calG)^{\pm}_m$)}.
\]
For part i), note that we have the factorization $\iota^{0,c}\co\Gr_\calM\to (\Gr_\calG)^{0,c}$ (resp. $\iota^{\pm,c}\co\Gr_{\calP^\pm}\to (\Gr_\calG)^{\pm,c}$) by construction which is a closed immersion (resp. quasi-compact immersion) because $\iota^0$ (resp. $\iota^\pm$) is. Theorem \ref{gctgeoBD} i) implies that the maps are bijective on the underlying topological spaces, and i) follows from Lemma \ref{qcimm} below applied to $\iota^{0,c}$ and $\iota^{\pm,c}$. Part ii) and iii) are immediate from the construction.
\end{proof}

\begin{lem}\label{qcimm}
Let $\iota\co Y\hookto Z$ be a quasi-compact immersion of ind-schemes  which is bijective on the underlying topological spaces. Then $\iota$ is a closed immersion which is an isomorphism on reduced loci.  
\end{lem}
\begin{proof} Writing $Z=\on{colim}_iZ_i$ we reduce to the case where $Z$ and hence $Y$ are schemes. By \cite[Tag 01QV]{StaPro}, we can factor $\iota=i\circ j$ where $j\co Y\to \bar{Y}$ is an open immersion, and $i\co \bar{Y}\to Z$ is a closed immersion. Since $\iota$ is bijective, $j$ is bijective as well, and hence $j\co Y\simeq \bar{Y}$ is an isomorphism. Thus, $\iota$ is a bijective closed immersion. To see that $\iota$ is an isomorphism on reduced loci, we may assume that $Z=\Spec(A)$ and hence $Y=\Spec(B)$ is affine. Since the induced surjective map $\iota^\#\co A\to B$ is bijective on spectra, its kernel is contained in the nilradical of $A$ and hence is generated by nilpotent elements. This implies that $\iota^\#$ is an isomorphism on reduced loci, and the lemma follows.
\end{proof}

\begin{proof}[Proof of Theorem \ref{gctgeoBD} iii\textup{)}] 
If $G=G_0\otimes_kF$ is constant, then we claim that the maps $\iota^{0,c}$ and $\iota^{\pm,c}$ constructed in Proposition \ref{conn_comp_BD} are isomorphisms. Using Proposition \ref{gctgeo} and \ref{gctgeoflag} ii), we already know that they are fiberwise isomorphisms. By applying Lemma \ref{safetylem}, it is enough to prove that $\Gr_\calM$ (resp. $\Gr_{\calP^\pm}$) is ind-flat over $\calO$. We claim that this holds for any smooth affine group scheme $\ucG$ of finite type over a smooth curve $X$ with constant generic fiber.\footnote{Note that the generic fiber of $\ucM$ (resp. $\ucP^\pm$) is constant as well if the generic fiber of $\ucG$ is constant.} The map $\calL_X\ucG\to \Gr_{\ucG,X}$ is a torsor under the flat affine $X$-group scheme $\calL^+_X\ucG$, and it is enough to show that $\calL_X\ucG$ is ind-flat over $X$. Working locally at $x\in |X|$, we may assume that $X$ admits a global coordinate. Let us consider the functor $\calX$ on the category of $k$-algebras $R$ with
\[
\calX\co R\mapsto \{( x,\varphi)\;|\; x\in X(R);\text{ $\varphi\co R\pot{\Ga_x}\simeq R\pot{z}$ continuous} \},
\]
where $\varphi$ is a continuous isomorphism of $R$-algebras. The forgetful map $\calX\to X$ is a left $L^+\bbG_m$-torsor, and we have $LG_0\times^{L^+\bbG_m} \calX\simeq \calL_X\ucG$. In particular, $\calL_X\ucG$ is fpqc locally isomorphic to $LG_0\times X$. This shows the ind-flatness of $\calL_X\ucG\to X$ and Theorem \ref{gctgeoBD} follows. 
\end{proof}

\begin{lem} \label{safetylem}
Let $f\co Y\to Z$ be a map of $\calO$-schemes of finite type where $Y$ is flat. If $f_\eta$ and $f_s$ are isomorphisms, then $f$ is an isomorphism.
\end{lem}
\begin{proof}
By [SGA I, Exp I, Prop 5.7] it is enough to show that $Z$ is flat. As $Y$ is flat, the map $f$ factors as $Y\to Z^{\on{fl}}\subset Z$ where $Z^{\on{fl}}$ is the scheme theoretic closure of the generic fiber. As the map $f$ is fiberwise an isomorphism, this implies $Z^{\on{fl}}= Z$ as follows. Let $\calI$ be the ideal of definition of $Z^{\on{fl}}\subset Z$. Then $\calI_\eta=0$, and as $Z$ is of finite type this implies $\calI=\calI\otimes \calO/(t^N)$ for $N>\!\!>0$. Now the short exact sequence of $\calO_Z$-modules
\[
0\to \calI \to \calO_Z\to \calO_{Z^{\on{fl}}}\to 0.
\]
stays exact after applying $\str\otimes_\calO k$ because $\calO_{Z^{\on{fl}}}$ is $\calO$-flat. As the composition $\calO_{Z}\otimes k\to \calO_{Z^{\on{fl}}}\otimes k\to\calO_{Y}\otimes k$ is an isomorphism, it follows that $\calI\otimes k=0$, i.e., $\calI = t\calI$, and hence $\calI = t\calI = \ldots = t^N\calI = 0$.

%A version of Nakayama's lemma \cite[Tag 00DV; Lem 10.19.1 (9)]{StaPro} now implies that $\calI=0$. {\cg instead of citing the Stacks Project here, it might be kinder to the reader to just state that $\calI = t\calI = \cdots = t^N\calI = 0$.}
\end{proof}

%Pappas-Zhu Grassmannians
\subsection{Torus actions in unequal characteristic} 

We translate the arguments of the previous paragraph to the BD-Grassmannians of Pappas-Zhu \cite{PZ13}. Here we restrict our attention to the case of tamely ramified groups. To handle restriction of scalars along wildly ramified extensions as in \cite{Lev16} more effort is needed.

Let $F$ be a finite extension of $\bbQ_p$ with uniformizer denoted $\varpi$. Let $G$ be a connected reductive $F$-group, and choose $(A,S,T)$ as in \eqref{groupchain} above. As in \cite{PZ13}, we assume that $G$ splits over a \emph{tamely ramified} extension of $F$. 

\subsubsection{Pappas-Zhu-Beilinson-Drinfeld Grassmannians} \label{PZBDGrass_Sec}
In \cite[\S 3]{PZ13} a spreading $(\uG,\uA,\uS,\uT)$ over $\calO[t,t^{-1}]=\calO[t^{\pm1}]$ is constructed. The group $\uG$ is a connected reductive $\calO[t^{\pm1}]$-group, the groups $\uA\subset \uS\subset \uT$ are $\calO[t^{\pm1}]$-subtori of $\uG$ with the following properties. The torus $\uA$ is a maximal split $\calO[t^{\pm1}]$-torus, the torus $\uS$ is a maximal $\breve{\calO}[t^{\pm1}]$-split torus defined over $\calO[t^{\pm1}]$, and the torus $\uT$ is a maximal torus of $\uG$. If we take the base change along the specialization $\calO[t^{\pm1}]\to F, t\mapsto \varpi$, then as $F$-groups
\begin{equation}\label{unequalspread}
(\uG,\uA,\uS,\uT)\otimes_{\calO[t^{\pm1}]}F\;\simeq\; (G,A,S,T),
\end{equation}
cf. \cite[4.3]{PZ13}. Interestingly, we may also consider the specialization along $\calO[t^{\pm1}]\to k\rpot{t}, \varpi\mapsto 0$. Let us denote
\[
(G',A',S',T')\defined (\uG,\uA,\uS,\uT)\otimes_{\calO[t^{\pm1}]}k\rpot{t}.
\]
Then $G'$ is a connected reductive $F' := k\rpot{t}$-group, and $(A',S',T')$ is as in \eqref{groupchain} above, cf.\,the discussion in \cite[4.1.2; 4.1.3]{PZ13}. Further, we obtain a canonical identification of the apartments $\scrA(\uG_F,\uA_F)=\scrA(G',A')$ (cf. \cite[4.1.3]{PZ13}), and hence under \eqref{unequalspread} an identification
\begin{equation}\label{apartments}
\scrA(G,A)= \scrA(G',A').
\end{equation}
In fact by \cite[4.1.2; 4.1.3]{PZ13}, we have a canonical identification of apartments
\begin{equation} \label{apartments_kappa}
\scrA(G, A) = \scrA(\uG_{\kappa\rpot{t}}, \uA_{\kappa\rpot{t}})
\end{equation}
for both $\kappa = k, F$.
We shall use the following two results in \S \ref{testfunctions} below.

\begin{lem} \label{Iwahorilem}
There is an identification of Iwahori-Weyl groups $W(G,A)= W(G',A')$ which is compatible with the action on the apartments under the identification \eqref{apartments}.
\end{lem}
\begin{proof} Over $\bF$ we obtain a $\sig$-equivariant isomorphism according to \cite[4.1.2]{PZ13}. The general case follows by taking $\sig$-fixed points from \cite[\S 1.2]{Ri16b} (cf.\,also \eqref{IWembed} above and \cite[4.1.3]{PZ13}).
\end{proof}

Now let $\calG=\calG_{\bbf}$ be a parahoric $\calO$-group scheme of $G$ whose facet $\bbf$ is contained in $\scrA(G,A)$. Then under \eqref{apartments} we obtain a unique facet $\bbf'\in \scrA(G',A')$, and hence a parahoric $k\pot{t}$-group scheme $\calG'=\calG_{\bbf'}$ of $G'$.

\begin{lem} \label{hecke_identification}
There is a canonical identification $\calZ(G(F),\calG(\calO_F))=\calZ(G'(F'),\calG'(\calO_{F'}))$ of centers of parahoric Hecke algebras, where the Haar measures are normalized to give $\calG(\calO_F)$ \textup{(}resp. $\calG'(\calO_{F'})$\textup{)} volume $1$.
\end{lem}
\begin{proof} Let $M$ (resp. $M'$) be the centralizer of $A$ (resp. $A'$) in $G$ (resp. $G'$) which is a minimal Levi. 
Let $\uM$ denote the centralizer in $\uG$ of $\uA$. Then we have identifications of groups $M = \uM \otimes_{\calO [t^\pm 1]} F$ (resp.\,$M' = \uM \otimes_{\calO[t^\pm 1]} k\rpot{t}$). Applying the above discussion to $(\uM, \uA, \uS, \uT)$, we get the identification of apartments $\scrA(M,A) = \scrA(M',A')$. Then we may apply Lemma \ref{Iwahorilem} for $M$, and we obtain an identification of abelian groups
\[
\La_M:= M(F)/M_1=M'(k\rpot{t})/M'_1=:\La_{M'},
\] 
where $M_1$ (resp. $M'_1$) is the unique parahoric group scheme of $M(F)$ (resp. $M'(k\rpot{t})$).  The result now follows via the Bernstein isomorphisms \cite[Thm 11.10.1]{Hai14}
$$
\bar{\mathbb Q}_\ell[\Lambda_M]^{W_0(G,A)} \cong \calZ(G(F),\calG(\calO_F)),
$$
noting that the finite relative Weyl groups of $(G,A)$ and $(G',A')$ are isomorphic (compatible with the action on $\La_M=\La_{M'}$).
\end{proof}

Let us now return to the construction of torus actions. We cite the following theorem \cite[Thm 4.1; Cor 4.2]{PZ13}.

\begin{thm}\label{PZspreading}
There is a unique \textup{(}up to unique isomorphism\textup{)} smooth affine $\bbA^1_\calO$-group scheme $\ucG$ of finite type with connected fibers and with the following properties:\smallskip\\
i\textup{)} The group scheme $\ucG|_{\calO[t,t^{-1}]}$ is the group scheme $\uG$.\smallskip\\
ii\textup{)} The base change of $\ucG$ under $\Spec(\calO)\to \bbA^1_\calO$ given by $t\mapsto \varpi$ is the parahoric group $\calG=\calG_\bbf$.\smallskip\\
iii\textup{)} The base change of $\ucG$ under $\calO_{F}[t]\to \kappa\pot{t}$, $t\mapsto t$ for both $\kappa=F,k$ is the parahoric group scheme for $\uG_{\kappa\rpot{t}}$ attached to $\bbf$ under \eqref{apartments_kappa}.
\end{thm}

Let $X=\bbA^1_\calO$, and defined the BD-Grassmannian $\Gr_{\ucG,X}$ as the functor on the category of $\calO$-algebras $R$ given by the set isomorphism classes of triples $(x, \calF,\al)$ with
\begin{equation}\label{PZmoduli}
\begin{cases}
\text{$x\in X(R)$ is a point};\\
\text{$\calF$ a $\ucG_{X_R}$-torsor on $X_R$};\\
\text{$\al\co \calF|_{X_R\bslash\Ga_x}\simeq \calF^0|_{X_R\bslash\Ga_x}$ a trivialization},
\end{cases}
\end{equation}
where $\calF^0$ denotes the trivial torsor, and $\Ga_x\subset X_R$ is the graph of $x$. Note that the definition of $\Gr_{\ucG,X}$ makes sense for any smooth affine group scheme. We cite the following result \cite[Prop 6.5]{PZ13}. 

\begin{lem} \label{PZreplem}
Let $\ucG$ be a smooth affine group scheme with connected fibers. The BD-Grassmannian $\Gr_{\ucG,X}\to X$ is representable by a separated ind-scheme of ind-finite type. If $\ucG$ is as in Theorem \ref{PZspreading}, then $\Gr_{\ucG,X}\to X$ is ind-projective.  
\end{lem}
\begin{proof} The first statement follows as in the proof of \cite[Prop 6.5]{PZ13} from the existence of a closed embedding $\ucG\hookto \Gl_{n,X}$ such that $\Gl_{n,X}/ \ucG$ is quasi-affine, cf. \cite[Cor 11.7]{PZ13}. Then the map $\Gr_{\ucG,X}\to \Gr_{\Gl_n,X}$ is representable by a locally closed immersion. As $\Gr_{\Gl_n,X}$ is ind-projective the ind-scheme $\Gr_{\ucG,X}$ is separated of finite type. The rest is \cite[Prop 6.5]{PZ13}. 
\end{proof}

For any $\calO$-algebra, let $R\pot{\Ga_x}\simeq R\pot{t-x}$ be the ring of regular functions on the formal affine $R$-scheme given by the completion of $X_R=\bbA^1_R$ along $\Ga_x$. Then $\Ga_x\subset \Spec(R\pot{\Ga_x})$ defines a Cartier divisor (in particular locally principal), and hence its complement is an affine scheme with ring of regular functions denoted by $R\rpot{\Ga_x}\simeq R\pot{t-x}[(t-x)^{-1}]$. The \emph{global loop group} is the functor on the category of $\calO$-algebras given by
\begin{equation}\label{PZgloballoop}
\calL_X\ucG\co R\mapsto \{(x,g)\;|\; \text{$x\in X(R)$ and $g\in \ucG(R\rpot{\Ga_x})$}\},
\end{equation}
which is representable by an ind-affine ind-group scheme over $X$ (cf. \cite[6.2.4]{PZ13}). By replacing $R\rpot{\Ga_x}$ with $R\pot{\Ga_x}$ in \eqref{globalloop}, one defines the \emph{global positive loop group} $\calL_X^+\ucG$ which is a flat affine $X$-group scheme. Again by the Beauville-Laszlo gluing lemma \cite{BL95} (cf. also \cite[Lem 6.1]{PZ13}) there is a natural isomorphism $\Gr_{\ucG,X}\simeq \calL_X\ucG/\calL^+_X\ucG$, and we obtain a transitive action morphism
\begin{equation}\label{PZact}
\calL_X\ucG\times \Gr_{\ucG,X}\longto \Gr_{\ucG,X},
\end{equation}
cf. \cite[6.2.4]{PZ13}.

\begin{dfn} \label{PZBDGrass_Def}
The (Pappas-Zhu) BD-Grassmannian\footnote{In\cite[6.2.6]{PZ13} the ind-scheme $\Gr_{\calG}$ is denoted $\Gr_{\calG,\calO}$.} $\Gr_\calG$ together with the action map
\begin{equation}\label{PZbasechange}
\calL\calG\times\Gr_\calG\longto\Gr_\calG
\end{equation}
is the base change of \eqref{PZact} along the map $\calO[t]\to \calO, t\mapsto \varpi$.
\end{dfn}

\begin{rmk}\label{PZBDGrass_rmk}
Fix the spreading $\uG$ of $G$ in \eqref{unequalspread}, and a uniformizer $\varpi\in \calO$. Then the ind-scheme $\Gr_\calG$ together with the action map \eqref{PZbasechange} depends up to unique isomorphism on the data $(\underline{G},\calG, \varpi)$, cf. Theorem \ref{PZspreading}. We refer the reader to \cite[Rmk 3.2]{PZ13} for a discussion of the uniqueness of the spreading $\underline{G}$. It is likely that \eqref{PZbasechange} is independent of the choice of $\varpi$, but we will not address this question here. Let us point out that in \cite[Conj.\,21.4.1]{SW}, Scholze predicts the existence of local models which canonically depend only on the data $(G,\calG,\{\mu\})$. 
\end{rmk}

By \cite[Cor 6.6]{PZ13}, the generic fiber of \eqref{PZbasechange} is identified with the usual affine Grassmannian \eqref{affineact} over $F$ (formed using an additional formal parameter), and the special fiber is identified with the twisted affine flag variety \eqref{flagact} for the $k\pot{t}$-group $\calG'$.  

Note that the construction of $\ucG$ is compatible with the chain of $\calO[t^{\pm1}]$-tori $\uA\subset \uS\subset \uT$, and we obtain a chain of commutative smooth closed $\bbA^1_\calO$-subgroup schemes 
\[
\ucA\subset \ucS\subset \ucT
\]
of $\ucG$ where $\ucA$ is a split $\bbA^1_\calO$-torus, $\ucS$ a split $\bbA^1_{\breve{\calO}}$-torus defined over $\bbA^1_\calO$, and $\ucT$ is a smooth commutative group scheme whose base change $\ucT\otimes_{\calO[t]}\calO$ along $t\mapsto \varpi$ is the connected lft Ner\'on model of $T$. The base change $\calO[t]\to \calO, t\mapsto \varpi$ gives the chain of group schemes $\calA\subset \calS\subset\calT$ as above.

\subsubsection{Torus actions}\label{PZ_Torus_Sec}
Let $\chi\co \bbG_{m,\calO}\to \calG$ be any cocharacter whose generic fiber $\chi_F$ factors through $A$. Then $\chi$ factors through $\calA$ because it is a maximal split torus in $\calG$. As the curve $X$ is connected, the cocharacter $\chi$ spreads uniquely to a cocharacter $\underline{\chi}\co \bbG_{m,X}\to \ucA$. The cocharacter $\chi$ acts via conjugation on $\ucG$, and we denote by $\ucM$ the fixed points, and by $\ucP^+$ (resp. $\ucP^-$) the attractor (resp. repeller) subgroup scheme. If $M$ denotes the centralizer of $\chi_F$ (which is a Levi subgroup), then the group scheme $\ucM|_{\calO[t^{\pm1}]}$ is a spreading $\uM$ associated with $M$. 

\begin{lem}\label{PZgroups} i\textup{)} The group schemes $\ucM$ and $\ucP^\pm$ are smooth closed subgroup schemes of $\ucG$ with geometrically connected fibers.\smallskip\\
ii\textup{)} The centralizer $\ucM$ is a parahoric group scheme for $\uM$ in the sense of Theorem \ref{PZspreading}.\smallskip\\
iii\textup{)} There is a semidirect product decomposition $\ucP^\pm=\ucM\ltimes \ucN^\pm$ where $\ucN^\pm$ is a smooth affine group scheme with geometrically connected fibers.
\end{lem}
\begin{proof}
The groups $\ucM$, $\ucP^\pm$ and the map $\ucP^\pm\to \ucM$ are smooth by \cite[Rmk 1.2, Thm 1.1 \& Rmk 3.3]{Mar15}. For part ii), (in view of the uniqueness statement in Theorem \ref{PZspreading}) it suffices to check that $\ucM|_\calO$ (resp. $\ucM|_{k\pot{t}}$) is a parahoric group for $M$ (resp. $M'$). As $\ucM|_\calO$ (resp. $\ucM|_{k\pot{t}}$) is smooth, it agrees with the flat closure of $M$ (resp. $M'$) inside $\calG$ (resp. $\calG'$), and hence $\ucM|_\calO$ (resp. $\ucM|_{k\pot{t}}$) is parahoric by \cite[Lem A.1]{Ri16a}. In particular, $\ucM$ has geometrically connected fibers which implies $\ucP^\pm$ having geometrically connected fibers by \cite[Cor 1.12]{Ri19}. Part i) and ii) follow. The scheme $\ucN^\pm$ is the kernel of $\ucP^\pm\to \ucM$, and hence smooth with geometrically connected fibers. The lemma follows.
\end{proof}

Using the functoriality of the loop group construction, we obtain via the composition
\begin{equation}\label{flowBD}
\bbG_{m,\calO}\subset\calL^+\bbG_{m,\calO}\overset{\calL^+\!\chi\phantom{h}}{\longto} \calL^+\calA\subset \calL^+\calG \subset \calL\calG
\end{equation}
a fiberwise $\bbG_{m}$-action on $\Gr_\calG\to \Spec(\calO)$.

\begin{lem}\label{PZloclinBD}
The $\bbG_m$-action on $\Gr_\calG$ is Zariski locally linearizable.
\end{lem}
\begin{proof} We follow the proof of Lemma \ref{loclinBD}: First, we reduce to the case of $\ucG=\Gl_{n,X}$ as in the proof of Lemma \ref{PZreplem}. Since $X$ is affine and $\Pic(X)=0$, we may by \cite[Prop 6.2.11]{Co14} reduce to the case that the image of $\underline{\chi}\co \bbG_{m,X}\to \Gl_{n,X}$ lies in the diagonal matrices. Then the $\bbG_m$-action on $\Gr_{\Gl_{n,X}}$ is constant, i.e. comes from a $\bbG_m$-action on $\Gr_{\Gl_{n,\bbZ}}$ over the integers. Now the result follows from Lemma \ref{loclinaffine} noting that the argument for $\Gl_n$ given there works over any ring.
\end{proof}

In light of Theorem \ref{Gmthm}, Lemma \ref{loclinBD} implies that there are maps of separated $\calO$-ind-schemes of ind-finite type
\begin{equation}
(\Gr_\calG)^0\leftarrow (\Gr_\calG)^\pm\to \Gr_\calG,
\end{equation}
where $(\Gr_\calG)^0$ are the fixed points and $(\Gr_\calG)^+$ (resp. $(\Gr_\calG)^-)$ is the attractor (resp. repeller), cf. \S 1. We denote by $\Gr_\calM$ (resp. $\Gr_{\calP^\pm}$) the BD-Grassmannians associated with the $X$-group schemes $\ucM$ (resp. $\ucP^\pm$).

\begin{rmk} \label{locally_closed_stratum_rem2}
As in Remark \ref{locally_closed_stratum_rem}, $(\Gr_\calG)^\pm \rightarrow \Gr_{\calG}$ induces locally closed immersions when restricted to the connected components of the source.
\end{rmk}

\begin{thm}\label{PZgctgeoBD}
The maps \eqref{hyperlocBD} induce a commutative diagram of $\calO$-ind-schemes
\[
\begin{tikzpicture}[baseline=(current  bounding  box.center)]
\matrix(a)[matrix of math nodes, 
row sep=1.5em, column sep=2em, 
text height=1.5ex, text depth=0.45ex] 
{\Gr_\calM & \Gr_{\calP^\pm} & \Gr_\calG \\ 
(\Gr_{\calG})^0& (\Gr_{\calG})^\pm& \Gr_{\calG}, \\}; 
\path[->](a-1-2) edge node[above] {}  (a-1-1);
\path[->](a-1-2) edge node[above] {}  (a-1-3);
\path[->](a-2-2) edge node[below] {}  (a-2-1);
\path[->](a-2-2) edge node[below] {} (a-2-3);
\path[->](a-1-1) edge node[left] {$\iota^0$} (a-2-1);
\path[->](a-1-2) edge node[left] {$\iota^\pm$} (a-2-2);
\path[->](a-1-3) edge node[left] {$\id$} (a-2-3);
\end{tikzpicture}
\]
whose generic fiber \textup{(}resp. special fiber\textup{)} is the diagram constructed in Proposition \ref{gctgeo} \textup{(}resp.\,Proposition \ref{gctgeoflag} for $\calG'/k\pot{t}$\textup{)}. Further, the maps $\iota^0$ and $\iota^\pm$ are closed immersions which are open immersions on reduced loci.
\end{thm}
\begin{proof}
As in Theorem \ref{gctgeoBD}, the map $\Gr_\calM\to \Gr_\calG$ is representable by a closed immersion (because $\ucG/\ucM$ is quasi-affine by \cite[Thm 2.4.1]{Co14} and $\Gr_\calM$ is ind-proper). The $\bbG_m$-action on $\Gr_\calM$ is trivial, and hence we obtain the closed immersion $\iota^0\co \Gr_\calM\to (\Gr_\calG)^0$. The quasi-compact immersion $\iota^\pm$ is constructed on the spreadings $\Gr_{\ucP^{\pm},X}\to (\Gr_{\ucG,X})^\pm$ analogous to the construction below Theorem \ref{gctgeoBD}: choose $\ucG\hookto \Gl_{n,X}$ such that $\Gl_{n,X}/\ucG$ is quasi-affine, cf. \cite[Cor 11.7]{PZ13}. Let $P_X^+\subset \Gl_{n,X}$ (resp. $P_X^-\subset \Gl_{n,X}$) be the attractor (resp. repeller) subgroup defined by the cocharacter $\bbG_{m,X}\overset{\chi}{\simeq}\ucG\hookto \Gl_{n,X}$. Then $\ucP^\pm=\ucG\times_{\Gl_{n,X}}P_X^\pm$. Further, $P_X^\pm$ (resp. $\ucP^\pm$) is smooth affine with geometrically connected fibers by the proof of Lemma \ref{PZgroups} iii). Hence the fppf-quotient $P_X^\pm/\ucP^\pm$ is a quasi-projective scheme by \cite[Cor 11.5]{PZ13}, and again by Zariski's main theorem applied to the map $P_X^\pm/\ucP^\pm\to \Gl_{n,X}/\ucG$ is quasi-affine. We have the same diagram \eqref{commsquareBD} as above. The isomorphism $\Gr_{P_X^\pm}\simeq (\Gr_{\Gl_{n,X}})^\pm$ follows from Lemma \ref{gctgeo_over_Z} using that $\bbG_{m,X}\to \Gl_{n,X}$ is defined over $\bbZ_p$ after conjugation, cf. the proof of Lemma \ref{PZloclinBD}. The rest of the construction of $\iota^\pm$ is literally the same. We do not repeat the full construction here, but instead refer the reader to Theorem \ref{gctgeoBD}. The assertion on the fibers is \cite[Cor 6.6]{PZ13}. The rest of the proof is the same as in Theorem \ref{gctgeoBD} using Lemma \ref{basicgeoBD_2} and Proposition \ref{conn_comp_BDPZ} below.
\end{proof}

We have the following lemma which is proven analogously to Lemmas \ref{basicgeoflag} and \ref{repmap}.

\begin{lem} \label{basicgeoBD_2}
i\textup{)} The map $ (\Gr_\calG)^\pm\to \Gr_\calG$ is schematic.\smallskip\\
ii\textup{)} The map $(\Gr_\calG)^\pm\to (\Gr_\calG)^0$ is ind-affine with geometrically connected fibers, and induces an isomorphism on the group of connected components $\pi_0((\Gr_{\calG})^\pm)\simeq\pi_0((\Gr_{\calG})^0)$. \smallskip\\
iii\textup{)} The map $\Gr_{\calP^\pm}\to \Gr_\calM$ has geometrically connected fibers.
\end{lem}

The following proposition is the analogue of Proposition \ref{conn_comp_BD}.

\begin{prop}\label{conn_comp_BDPZ}
Let either $N=\ucN^+\otimes F$ or $N=\ucN^-\otimes F$ with $\ucN^\pm$ as in Lemma \ref{PZgroups} iii\textup{)}. There exists an open and closed $\calO$-ind-subscheme $(\Gr_\calG)^{0,c}$ \textup{(}resp. $(\Gr_\calG)^{\pm,c}$\textup{)} of $(\Gr_\calG)^0$ \textup{(}resp. $(\Gr_\calG)^{\pm}$\textup{)} together with a disjoint decomposition, depending up to sign on the choice of $N$, as $\calO$-ind-schemes
\[
(\Gr_\calG)^{0,c}=\coprod_{m\in \bbZ}(\Gr_\calG)^{0}_m  \;\;\;\;\text{\textup{(}resp. $(\Gr_\calG)^{\pm,c}=\coprod_{m\in \bbZ}(\Gr_\calG)^{\pm}_m$\textup{)}},
\]
which has the following properties.\smallskip\\
i\textup{)} The map $\iota^0\co \Gr_\calM\to (\Gr_\calG)^{0}$ \textup{(}resp. $\iota^\pm\co \Gr_{\calP^\pm}\to (\Gr_\calG)^{\pm}$\textup{)} factors through $(\Gr_\calG)^{0,c}$ \textup{(}resp. $(\Gr_\calG)^{\pm,c}$\textup{)} inducing a closed immersion $\iota^{0,c}\co \Gr_\calM\to (\Gr_\calG)^{0,c}$ \textup{(}resp. $\iota^{\pm,c}\co \Gr_{\calP^\pm}\to (\Gr_\calG)^{\pm,c}$\textup{)} which is an isomorphism on reduced loci. \smallskip\\
ii\textup{)} The decomposition gives in the generic fiber decomposition \eqref{decomposition} and in the special fiber decomposition \eqref{decompositionflag}.\smallskip\\
iii\textup{)} The complement $(\Gr_\calG)^{0}\bslash (\Gr_\calG)^{0,c}$ \textup{(}resp. $(\Gr_\calG)^{\pm}\bslash (\Gr_\calG)^{\pm,c}$ \textup{)} has empty generic fiber, i.e., is concentrated in the special fiber.
\end{prop}
\begin{proof} Let $\pi_1(M)=X_*(T)/X_*(T_{M_\scon})$ be the algebraic fundamental group of $M$, cf.\,\eqref{fundamentalgroup}. For $\nu\in \pi_1(M)$, denote by $\dot{\nu}\in X_*(T)$ a representative which gives rise to a map
\[
\dot{\nu}\co \Spec(\sF)\,\to\, \Gr_{M}=\Gr_{\calM,F}\,\hookto\, \Gr_\calM.
\]
By the ind-properness of $\Gr_\calM\to \Spec(\calO)$, which follows from Lemma \ref{PZgroups} ii) together with Lemma \ref{PZreplem}, the map $\dot{\nu}$ extends uniquely to a map (still denoted)
\[
\dot{\nu}\co \Spec(\bar{\calO})\,\to\, \Gr_\calM.
\]
Here $\bar{\calO}\subset \sF$ is the valuation subring of integral elements. Now by Lemma \cite[Lem 9.8]{PZ13}, the special fiber $\bnu$ of $\dot{\nu}$ is the image under the canonical projection $X_*(T)\to X_*(T)_I$. Arguing as in Proposition \ref{conn_comp_BD}, we obtain, for the choice of $\ucP$, a decomposition into open and closed $\calO$-sub-ind-schemes
\[
(\Gr_\calG)^{0,c}\defined\coprod_{m\in \bbZ}(\Gr_\calG)^{0}_m  \;\;\;\;\text{(resp. $(\Gr_\calG)^{\pm,c}\defined\coprod_{m\in \bbZ}(\Gr_\calG)^{\pm}_m$\,)},
\]
with the desired properties. The proof of i)-iii) is the same as in Proposition \ref{conn_comp_BD}. 
\end{proof}

\section{Constant terms on affine flag varieties}\label{constermflag}

\subsubsection{Nearby cycles} \label{NBCycSec}
Let us briefly recall some general facts about nearby cycles. Let $(S,s,\eta)$ be a Henselian trait, i.e.,\,$S$ is the spectrum of a Henselian discrete valuation ring, $s\in S$ (resp. $\eta\in S$) the closed (resp. open) point. Let $\bar{\eta}\to \eta$ be a geometric point, and denote by $\Ga=\Gal(\bar{\eta}/\eta)$ the Galois group. Let $\bar{S}$ be the normalization of $S$ in $\bar{\eta}$, and let $\bar{s}\in \bar{S}$ be the closed point. We obtain the seven-tuple $(S,s,\eta,\bar{S},\bar{s},\bar{\eta},\Ga)$.  

In the following all schemes are assumed to be separated and of finite type. As coefficients for the derived categories, we take $\algQl$ for a fixed prime $\ell$ which is invertible on $S$. For a scheme $X$ over $s$, we denote by $D_c^b(X\times_s \eta)$ as in \cite[Exp. XIII]{SGA7} the bounded derived category of $\algQl$-sheaves on $X_{\bar{s}}$ with constructible cohomologies, and with a continuous action of $\Ga$ compatible with the action on $X_{\bar{s}}$. If $X\hookto Y$ is a closed immersion of $k$-schemes which is an isomorphism on reduced loci, then by the topological invariance of \'etale cohomology \cite[Tag 03SI]{StaPro} there is an equivalence of categories $D_c^b(X\times_s\eta)\simeq D_c^b(Y\times_s\eta)$. 

Recall from \cite[Exp.\,XIII]{SGA7} (cf.\,also \cite[Appendix]{Il94}) that for a $S$-scheme $X$, there is the functor of nearby cycles
\begin{equation}\label{nearbycycles}
\Psi_X\co D_b^c(X_\eta)\longto D_b^c(X_s\times_s\eta).
\end{equation}
If $f\co X\to Y$ is a map of $S$-schemes, then there is a natural transformation of functors $D_b^c(X_\eta)\to D_b^c(Y_s\times_s\eta)$ (resp. $D_b^c(Y_\eta)\to D_b^c(X_s\times_s\eta)$) as
\begin{equation}\label{nearbycomp}
f_{\bar{s},!}\circ\Psi_X\longto \Psi_Y\circ f_{\eta,!} \;\;\;\text{(resp.\;} f_{\bar s}^*\circ \Psi_Y\longto  \Psi_X\circ f_\eta^*\;\text{)},
\end{equation}
which is an isomorphism if $f$ is proper (resp. if $f$ is smooth). Furthermore, by \cite[Thm 4.2, Thm 4.7]{Il94} nearby cycles commute with Verdier duality and box products
\begin{equation}\label{nearbycomp2}
D_{\bar{s}}\circ\Psi_X\simeq \Psi_X\circ D_{\eta} \;\;\;\text{and\;\;\;} \Psi_{X\times_S Y}\simeq \Psi_X\boxtimes \Psi_Y. 
\end{equation}
By \cite{BBD82}, nearby cycles preserve perversity, and restrict to a functor on perverse complexes
\[
\Psi_X\co \on{Perv}(X_\eta,\algQl)\longto \on{Perv}(X_s\times_s\eta).
\]
The construction of nearby cycles extends to separated $S$-ind-schemes of ind-finite type; see the discussion in \cite[10.1]{PZ13} for more details.

\subsubsection{Hyperbolic Localization}\label{hl_sec}  Let $R$ be a ring. Let $X$ be an $R$-ind-scheme locally of finite presentation with an \'etale locally linearizable $\bbG_m$-action, cf. \S \ref{CollectGm}. The fixed points $X^0$ and the attractor $X^+$ (resp. repeller $X^-$) are as in \eqref{hyperloc} related by the maps of $R$-ind-schemes
\[
X^0\,\overset{\;\; q^\pm}{\leftarrow}\, X^\pm\,\overset{\,p^\pm\,}{\to}\,X,
\] 
cf. Theorem \ref{Gmthm} for the representability properties of $X^0$ and $X^\pm$. As in \cite{Br03, DG15} (cf. also \cite[Cons 2.2]{Ri19}), there is a natural transformation of functors $D^+(X,\algQl)\to D(X^0,\algQl)$,
\begin{equation}\label{Bradentrafohl}
(q^-)_*\circ (p^-)^! \,\longto\, (q^+)_!\circ (p^+)^*,
\end{equation}
where $D^+(\str)$ (resp. $D(\str)$) is the category of bounded below complexes (resp. full derived category). We say that a complex $\calA\in D^+(X,\algQl)$ is (naively) $\bbG_m$-equivariant if there exists an isomorphism $p^*\calA\simeq a^*\calA$ in $D_c^b(\bbG_{m,S}\times_SX,\algQl)$, where $p$ (resp. $a$) denotes the projection (resp. action) $\bbG_{m,S}\times_SX\to X$. Following the method in \cite{Br03} for normal varieties over algebraically closed fields, it is shown in \cite[Thm 2.6]{Ri19} that the transformation \eqref{Bradentrafohl} is an isomorphism when restricted to the full subcategory of $\bbG_m$-equivariant complexes (the extension to ind-schemes is immediate). 

Now specialize to the case where $R$ is a Henselian trait, and set $S=\Spec(R)$. Let $(S,s,\eta,\bar{S},\bar{s},\bar{\eta},\Ga)$ be as in \S \ref{NBCycSec}. The following theorem is the analogue of \cite[Thm 3.3]{Ri19} for ind-schemes.

\begin{thm}\label{nearbyhlthm}
Let $X$ be a separated $S$-ind-scheme of ind-finite type with an \'etale locally linearizable $\bbG_m$-action. Then, for $\calA\in D^b_c(X_\eta,\algQl)$, there is a commutative diagram of arrows in $D_c^b(X_s^0\times_S\eta,\algQl)$ 
\begin{equation}\label{hldiagthm}
\begin{tikzpicture}[baseline=(current  bounding  box.center)]
\matrix(a)[matrix of math nodes, 
row sep=1.5em, column sep=2em, 
text height=1.5ex, text depth=0.45ex] 
{(q^-_{\bar{s}})_* \circ (p^-_{\bar{s}})^! \circ\Psi_X(\calA) & \Psi_{X^0}\circ (q^-_\eta)_*\circ (p^-_\eta)^!(\calA)\\ 
(q^+_{\bar{s}})_!\circ (p^+_{\bar{s}})^*\circ \Psi_X(\calA) & \Psi_{X^0}\circ (q^+_\eta)_!\circ (p^+_\eta)^*(\calA),\\ }; 
\path[->](a-1-2) edge (a-1-1); 
\path[->](a-1-1) edge  (a-2-1); 
\path[->](a-2-1) edge (a-2-2);
\path[->](a-1-2) edge  (a-2-2);
\end{tikzpicture}
\end{equation}
and all arrows are isomorphisms if $\calA$ is \textup{(}naively\textup{)} $\bbG_m$-equivariant. 
\end{thm}

\begin{rmk} 
More generally, the theorem holds when ``$\bbG_m$-equivariant'' is replaced by ``$\bbG_m$-monodromic'', cf. \cite{Ri19}. We do not need this more general statement in the paper.
\end{rmk}

\begin{proof}[Proof of Theorem \ref{nearbyhlthm}] The horizontal maps in \eqref{hldiagthm} are constructed from the usual functorialities of nearby cycles \eqref{nearbycomp}. The vertical maps in \eqref{hldiagthm} are given by \eqref{Bradentrafohl} in the generic (resp. special) fiber. When $X$ is a scheme, the theorem is \cite[Thm 3.3]{Ri19}. The case of ind-schemes is deduced as follows. Write $X=\on{colim}_iX_i$ where $X_i$ are separated $S$-schemes of finite type with an \'etale locally linearizable $\bbG_m$-action. By definition of $D^b_c(X_\eta,\algQl)$, there is an $X_i$ such that the support $\Supp(\calA)$ is contained in $X_{i,\eta}$, and all maps in \eqref{hldiagthm} are defined when using $X_i$ instead of $X$. Since nearby cycles (cf. \eqref{nearbycomp}) and the map \eqref{Bradentrafohl} (cf. \cite[Lem 2.22]{Ri19}) commute with push forward along closed immersions, the isomorphism is independent of the choice of $X_i$ with $i>\!\!>0$. The theorem for ind-schemes follows.
\end{proof}

\subsubsection{The data}\label{data_consterm}
Let us specialize to our set-up. Let $F$ be a non-archimedean local field, i.e., either $F/\bbQ_p$ a finite extension or $F\simeq\bbF_q\rpot{t}$. Take $S=\Spec(\calO_F)$, and the rest of the data $(S,s,\eta,\bar{S},\bar{s},\bar{\eta},\Ga_F)$ with the obvious meaning.

We fix a triple $(G,\calG,\chi)$ where $G$ is a connected reductive $F$-group, $\calG$ is a parahoric $\calO_F$-group scheme with generic fiber $G$, and $\chi\co\bbG_{m,\calO_F}\to \calG$ is a cocharacter defined over $\calO_F$. If $F/\bbQ_p$, we assume $G$ to split over a tamely ramified extension of $F$, fix a uniformizer $\varpi$ and a spreading $\ucG$ as in Theorem \ref{PZspreading}. If $F\simeq\bbF_q\rpot{t}$, we fix a spreading $\ucG$ over a pointed curve $(X,x_0)$ as in Proposition \ref{spreadprop}. Let $\calM$ and $\calP^\pm$ denote the smooth closed $\calO_F$-subgroup schemes of $\calG$ associated with $\chi$ by Lemma \ref{groups}, \ref{PZgroups}. 

In the following, we treat the case of $F/\bbQ_p$ and the case of $F\simeq\bbF_q\rpot{t}$ in complete analogy. If $F/\bbQ_p$, the notation $\Fl_\calG$, $L\calG$, $L^+\calG$, $\Fl_\calM$... means $\Fl_{\calG'}$, $L\calG'$, $L^+\calG'$, $\Fl_{\calM'}$... in the notation of \S \ref{PZBDGrass_Sec}.

\subsection{Geometric constant terms for affine flag varieties} 
The affine flag variety $\Fl_\calG$ is equipped with a left action of the pro-smooth affine $k$-group $L^+\calG$. As in \cite[10.1.3]{PZ13}, we make the following definition.

\begin{dfn}\label{EquiGaloisCat}
The category $\on{Perv}_{L^+\calG}(\Fl_\calG\times_s\eta)$ is the category of pairs $(\calA, \theta_\calA)$ with $\calA \in \on{Perv}(\Fl_\calG\times_s\eta)$ together with an isomorphism $\theta_\calA\co m^*(\calA)\simeq p^*(\calA)$ in $\on{Perv}(\Fl_\calG\times_s\eta)$ satisfying a cocycle condition. Here $m,p\co L^+\calG\times\Fl_{\calG}\to\Fl_\calG$ is the action, resp. the projection.
\end{dfn}

Recall from \eqref{globalact}, resp. Definition \ref{PZBDGrass_Def} that the Beilinson-Drinfeld Grassmannian $\Gr_\calG$ is an ind-proper $\calO_F$-ind-scheme with generic fiber $\Gr_{\calG,\eta}=\Gr_G$ and special fiber $\Gr_{\calG,s}=\Fl_\calG$.
The global positive loop group $\calL^+\calG$ is a flat affine $\calO_F$-group scheme which acts on $\Gr_\calG$. 
Its generic fiber is the $F$-group $L^+G$, and its special fiber is the $k$-group $L^+\calG$.  
Furthermore, $\calL^+\calG \simeq \on{lim}_{n\geq 0}\calL^+_n\calG$ is an inverse limit of smooth affine $\calO$-group schemes $\calL^+_n\calG$.
Here for $R/\calO$ one has $\calL^+_n\calG(R)=\ucG\big(R[t]/(t-\varpi)^{n+1}\big)$ if $F/\bbQ_p$ and $\calL^+_n\calG(R)=\ucG\big(R[z]/(z-t)^{n+1}\big)$ if $F\simeq k\rpot{t}$.

For each finite dimensional $\calL^+\calG$-invariant closed subscheme in $\Gr_\calG$, the $\calL^+\calG$-action factors over the smooth $\calO_F$-group scheme $\calL^+_n\calG$ for $n>\!\!> 0$.

As any object in $\on{Perv}_{L^+G}(\Gr_G)$ has by definition a finite dimensional support, smooth base change (cf. \eqref{nearbycomp}) shows that the nearby cycles 
\begin{equation}\label{equi_nearby}
\Psi_{\Gr_\calG}\co \on{Perv}_{L^+G}(\Gr_G)\;\longto\; \on{Perv}_{L^+\calG}(\Fl_\calG\times_s\eta).
\end{equation}
take values in $L^+\calG$-equivariant objects. 

In Theorem \ref{gctgeoBD}, Proposition \ref{conn_comp_BD} (for $F\simeq\bbF_q\rpot{t}$), and Theorem \ref{PZgctgeoBD}, Proposition \ref{conn_comp_BDPZ} (for $F/\bbQ_p$ finite), we constructed a commutative diagram of separated $\calO_F$-ind-schemes
\begin{equation}\label{restrictBD}
\begin{tikzpicture}[baseline=(current  bounding  box.center)]
\matrix(a)[matrix of math nodes, 
row sep=1.5em, column sep=2em, 
text height=1.5ex, text depth=0.45ex] 
{\Gr_\calM & \Gr_{\calP^\pm} & \Gr_\calG \\ 
(\Gr_{\calG})^{0,c}& (\Gr_{\calG})^{\pm,c}& \Gr_{\calG}, \\}; 
\path[->](a-1-2) edge node[above] {$q^\pm$}  (a-1-1);
\path[->](a-1-2) edge node[above] {$p^\pm$}  (a-1-3);
\path[->](a-2-2) edge node[below] {}  (a-2-1);
\path[->](a-2-2) edge node[below] {} (a-2-3);
\path[->](a-1-1) edge node[left] {$\iota^{0,c}$} (a-2-1);
\path[->](a-1-2) edge node[left] {$\iota^{\pm,c}$} (a-2-2);
\path[->](a-1-3) edge node[left] {$\id$} (a-2-3);
\end{tikzpicture}
\end{equation}
whose generic fiber is the diagram in Proposition \ref{gctgeo}, and whose special fiber is diagram \eqref{restrictflag} for $\calG$ (resp. for $\calG'$ if $F/\bbQ_p$). The maps $\iota^{0,c}$ and $\iota^{\pm,c}$ are nilpotent thickenings by Proposition \ref{conn_comp_BD}, \ref{conn_comp_BDPZ} i), and we may and do identify their derived categories of $\ell$-adic complexes in what follows. Then there is a natural isomorphism of functors $D_c^b(\Gr_M,\algQl)\to D_c^b(\Fl_{\calM}\times_s\eta,\algQl)$,
\begin{equation}\label{nearbyiso}
\Psi_{\Gr_\calM}\simeq \Psi_{(\Gr_{\calG})^{0,c}}.
\end{equation}
We write $\Psi_\calG=\Psi_{\Gr_\calG}$ (resp. $\Psi_\calM=\Psi_{\Gr_\calM}$) in the following. 

Since $\iota^{0,c}$ and $\iota^{\pm,c}$ are nilpotent thickenings, Proposition \ref{conn_comp_BD}, \ref{conn_comp_BDPZ} gives a decomposition 
\[
q^\pm=\coprod_{m\in\bbZ}q_m^\pm\co  \;\Gr_{\calP^\pm}\,=\,\coprod_{m\in\bbZ}\;\Gr_{\calP^\pm,m}\,\longto\, \coprod_{m\in\bbZ}\Gr_{\calM,m}\,=\,\Gr_\calM,
\]
according to the choice of the unipotent group $N:=\calN^\pm\otimes F$. We use the special fiber of diagram \eqref{restrictBD} to define the geometric constant term functors on affine flag varieties as follows.

\begin{dfn} \label{gctnormal}
The \emph{\textup{(}normalized\textup{)} geometric constant term} is the functor $\on{CT}_\chi^+\co D_c^b(\Fl_\calG\times_s\eta)\to D_c^b(\Fl_\calM\times_s\eta)$ (resp. $\on{CT}_\chi^-\co D_c^b(\Fl_\calG\times_s\eta)\to D_c^b(\Fl_\calM\times_s\eta)$) defined as the shifted pull-push functor
\[
\on{CT}_\chi^+\defined \bigoplus_{m\in\bbZ}(q^+_{m,s})_!(p^+_s)^*\langle m\rangle \;\;\;\text{(resp. $\on{CT}_\chi^-\defined \bigoplus_{m\in\bbZ}(q^-_{m,s})_*(p^-_s)^!\langle m\rangle$).}
\]
\end{dfn}

As in \eqref{Bradentrafo}, there is a natural transformation of functors
\begin{equation}\label{Bradentrafoflag}
\on{CT}_\chi^-\longto \on{CT}_\chi^+,
\end{equation}
which is an isomorphism for $\bbG_m$-equivariant complexes by Braden's theorem, cf. \S \ref{hl_sec}. In particular, if we restrict both functors to the category $\on{Perv}_{L^+\calG}(\Fl_\calG\times_s\eta)$ of $L^+\calG$-equivariant perverse sheaves (the $\bbG_m$-action factors through the $L^+\calG$-action), then \eqref{Bradentrafoflag} is an isomorphism of functors. The functor $\on{Perv}_{L^+\calG}(\Fl_\calG\times_s\eta)\to D_b^c(\Fl_\calM\times_s\eta)$ is defined as
\begin{equation}\label{consttermflag}
\on{CT}_\calM\defined \on{CT}^+_\chi|_{\on{Perv}_{L^+\calG}(\Fl_\calG\times_s\eta)}.
\end{equation}
As in Theorem \ref{nearbyhlthm}, the usual functorialities of nearby cycles \eqref{nearbycomp} give a natural transformation of functors $\Sat_G\to D_c^b(\Fl_\calM\times_s\eta)$ as
\begin{equation}\label{commctnearby}
\on{CT}_\calM\circ \Psi_{\calG}\longto \Psi_{\calM}\circ \on{CT}_M,
\end{equation}
where we use that the decomposition $q^+=\amalg_{m\in\bbZ}q_m^+$ is compatible with the decomposition $q^+_\eta=\amalg_{m\in\bbZ}q_{\eta,m}^+$ (resp. $q^+_s=\amalg_{m\in\bbZ}q_{s,m}^+$) in \eqref{decomposition} (resp. \eqref{decompositionflag}). 

\begin{thm}\label{commctnearbythm}
The transformation \eqref{commctnearby} is an isomorphism of functors $\Sat_G\to D_c^b(\Fl_\calM\times_s\eta)$. In particular, for every $\calA\in \Sat_G$, the complex $\on{CT}_\calM\circ \Psi_{\calG}(\calA)$ is naturally in $\on{Perv}_{L^+\calM}(\Fl_\calM\times_s\eta)$.
\end{thm}

\begin{rmk} It is possible to define the full subcategory $D_c^b(\Fl_\calG\times_s\eta)^{\bbG_m\on{-mon}}$ of $\bbG_m$-monodromic complexes in $D_c^b(\Fl_\calG\times_s\eta)$. If we restrict in \eqref{consttermflag} the functor $\on{CT}^+_\chi$ to this subcategory, then the transformation \eqref{commctnearby} is still an isomorphism (by the same proof). We do not need this more general statement in the paper.
\end{rmk}

\begin{proof} In light of \eqref{equi_nearby} combined with Theorem \ref{GeoSat2} i), we have $$\Psi_{\calM}\circ \on{CT}_M(\calA)\in \on{Perv}_{L^+\calM}(\Fl_\calM\times_s\eta),$$ and it is enough to show that \eqref{commctnearby} is an isomorphism.

Let $'q^+\co (\Gr_\calG)^{+}\to (\Gr_\calG)^0$ (resp. $'p^+\co (\Gr_\calG)^+\to \Gr_\calG$) which agrees with $q^+$ (resp. $p^+$) in the generic fiber. As each object in $\Sat_G$ is $\bbG_m$-equivariant, there is by Theorem \ref{nearbyhlthm} a natural isomorphism of functors $\Sat_G\to D_c^b((\Gr_\calG)^0\times_s\eta)$ as
\begin{equation}\label{commthm:eq1}
('q^+)_{s,!} \circ('p^+)_s^*\circ\Psi_\calG\overset{\simeq}{\longto} \Psi_{(\Gr_\calG)^0}\circ ('q^+)_{\eta,!} \circ('p^+)_\eta^*.
\end{equation}
The map $(\Gr_{\calG})^{0,c}\subset (\Gr_{\calG})^{0}$ is an open and closed immersion which is  which is an isomorphism on generic fibers, and we denote by $C^0$ its complement. Then $C^0$ has empty generic fiber by Proposition \ref{conn_comp_BD}, \ref{conn_comp_BDPZ} iii), and we obtain
\[
\Psi_{(\Gr_\calG)^0}|_{C^0}=0,
\]
by smooth base change \eqref{nearbycomp}. Together with \eqref{commthm:eq1}, we obtain $('q^+)_{s,!} \circ('p^+)_s^*\circ\Psi_\calG|_{C^0}=0$. The theorem follows using diagram (\ref{restrictBD}).
\end{proof}

In \S \ref{Iwahoricase}, we explain what Theorem \ref{commctnearbythm} means in a special case in terms of cohomology groups. Let us point out two applications: the construction of geometric constant terms for the Satake equivalence \S \ref{ConsTermRamGeoSat}, and applications to local models \S \ref{AppLocMod}.

%Ramified GeoSat
\subsection{Geometric constant terms for ramified groups} \label{ConsTermRamGeoSat}
We proceed with the data and notation as in \S \ref{data_consterm}. Let $\calG=\calG_\bbf$, and assume that the facet $\bbf_M\in\scrB(M,F)$ in the Bruhat-Tits building associated with $\calM=\calM_{\bbf_M}$ is very special, i.e. the image under $\scrB(M,F)\hookto \scrB(M,\bF)$ is special.  

Building upon the work \cite{Zhu15}, the second named author defined in \cite[Def 5.10]{Ri16a} a semi-simple full subcategory $\Sat_\calM$ of ${\rm Perv}_{L^+\calM}(\Fl_\calM\times_s\eta)$ which is stable under the convolution of perverse sheaves. It is shown that $\Sat_\calM$ is neutral Tannakian, and that the global cohomology functor
\begin{equation}\label{ramgaloistwist}
\begin{aligned}
\om_\calM\co & \Sat_\calM\longto \Rep_{\algQl}(\Ga_F) \\
&\calA  \longmapsto \bigoplus_{i\in \bbZ}\on{H}^i(\Fl_{\calM,\bar{k}}, \calA_{\bar{k}})(\nicefrac{i}{2}).
\end{aligned}
\end{equation}
defines a Tannakian equivalence $\Sat_\calM \simeq \Rep_{\algQl}({^LM}_r)$, cf. \cite[Thm 5.11]{Ri16a}. Here ${^LM}_r$ is the algebraic $\algQl$-group ${^LM}_r=\widehat{M}^{I_F}\rtimes \Ga_F$ where the inertia group $I_F$ acts on $\widehat{M}$ by pinning preserving automorphisms, cf. the discussion around Theorem \ref{GeoSat1}. Note that the group $\widehat{M}^{I_F}$ is reductive, but in general not connected. 

For any $\calA\in\Sat_M$, the nearby cycles $\Psi_\calM(\calA)$ belong to $\Sat_\calM$, and the functor $\Psi_\calM\co \Sat_M\to\Sat_\calM$ admits a unique structure of a tensor functor together with an isomorphism $\om_\calM\circ\Psi_\calM\simeq \om_M$. By \cite{Zhu14}, \cite[Thm 5.11 iii)]{Ri16a} (if $F\simeq\bbF_q\rpot{t}$) and \cite[Thm 10.18]{PZ13} (if $F/\bbQ_p$), there is a commutative diagram of neutral Tannakian categories
\begin{equation}\label{ramGeoSat}
\begin{tikzpicture}[baseline=(current  bounding  box.center)]
\matrix(a)[matrix of math nodes, 
row sep=1.5em, column sep=2em, 
text height=1.5ex, text depth=0.45ex] 
{\Sat_M&\Sat_{\calM} \\ 
\Rep_{\algQl}({^LM})&\Rep_{\algQl}({^LM}_r), \\}; 
\path[->](a-1-1) edge node[above] {$\Psi_\calM$} (a-1-2);
\path[->](a-2-1) edge node[above] {$\res$} (a-2-2);
\path[->](a-1-1) edge node[right] {$\om_M$} (a-2-1);
\path[->](a-1-2) edge node[right] {$\om_\calM$} (a-2-2);
\end{tikzpicture}
\end{equation}
where $\res\co V\mapsto V|_{{^LM}_r}$ denotes the restriction of representations along the inclusion ${^LM}_r\subset {^LM}$. The diagram is compatible with pullback $\Sat_M\to \Sat_{M,\sF}$ in the obvious sense. The following theorem generalizes \cite[Thm 4]{AB09} from the case of split reductive group to general reductive groups.

\begin{thm} \label{ramConsTerm1}
Assume $\calM$ is a very special parahoric group scheme.\smallskip\\
i\textup{)} For every $\calA\in \Sat_G$, one has $\on{CT}_\calM\circ \Psi_\calG(\calA)\in \Sat_\calM$.\smallskip\\
ii\textup{)} The functor $\on{CT}_\calM\circ \Psi_\calG\co \Sat_G\to\Sat_\calM$ admits a unique structure of a tensor functor together with an isomorphism $\om_\calM\circ \on{CT}_\calM \circ\Psi_\calG\simeq \om_G$. Under the geometric Satake equivalence, it corresponds to the restriction of representations $\res\co \Rep_{\algQl}({^LG})\to \Rep_{\algQl}({^LM}_r)$ along the inclusion ${^LM}_r\subset {^LG}$.  
\end{thm}
\begin{proof} The theorem follows from the canonical isomorphism $\on{CT}_\calM\circ \Psi_\calG\simeq \Psi_\calM\circ \on{CT}_M$ given by Theorem \ref{commctnearbythm}, and the corresponding statement for the latter functor, cf. Theorem \ref{GeoSat2} and \eqref{ramGeoSat}. 
\end{proof}

The following corollary was announced in \cite{Ri14b}.

\begin{cor}  \label{ramConsTerm2}
Let $\calG$ be very special, and consequently $\calM$ is very special as well. \smallskip\\
i\textup{)} For every $\calA\in\Sat_\calG$, one has $\on{CT}_\calM(\calA)\in\Sat_\calM$. \smallskip\\
ii\textup{)} The functor $\on{CT}_\calM\co \Sat_\calG\to\Sat_\calM$ admits a unique structure of a tensor functor together with an isomorphism $\om_{\calM}\circ \on{CT}_\calM\simeq \om_{\calG}$. Under the geometric Satake equivalence \eqref{ramGeoSat}, it corresponds to the restriction of representations $\res\co \Rep_{\algQl}({^LG}_r)\to \Rep_{\algQl}({^LM}_r)$ along the inclusion ${^LM}_r\subset {^LG}_r$.
\end{cor}
\begin{proof} As the nearby cycles $\Psi_\calG\co \Sat_G\to \Sat_\calG$ are a tensor functor between semi-simple Tannakian categories inducing the closed immersion ${^LG}_r\subset {^LG}$ on Tannakian groups, every object $\calA\in \Sat_\calG$ is a direct summand of some object of the form $\Psi_\calG(\calB)$ with $\calB \in \Sat_G$, cf. \cite[Prop 2.21 (b)]{DM82}. The corollary is an immediate consequence of Theorem \ref{ramConsTerm1}.
\end{proof}

In \cite{HR10}, the Satake isomorphism for special parahoric Hecke algebras is constructed. If $\calM=\calT$ is the connected lft N\'eron model of a torus, Corollary \ref{ramConsTerm2} is a geometrization of the Satake isomorphism for a very special parahoric subgroup, cf. Lemma \ref{CT_RPsi_comp} below. 

%Iwahori case
\subsubsection{The case of an Iwahori} \label{Iwahoricase}
Let us make explicit what Theorem \ref{ramConsTerm1} means in a special case: let $k=\bar{k}$ be algebraically closed, i.e. $F=\bF$, and assume that $\calG$ is an Iwahori group scheme and that the cocharacter $\chi\co \bbG_{m,\calO}\to \calG$ is regular, i.e. $M=T$ is a maximal torus. Then the parahoric $\calT=\calM$ is the lft N\'eron model of $T$ which is very special. The parabolic subgroups $B^\pm=P^\pm$ are $F$-rational Borel subgroups, and we denote $\calB^\pm=\calP^\pm$. The Iwahori-Weyl group $W=W(G,A)$ (cf. Definition \ref{IwahoriWeyl}) sits in a short exact sequence
\[
1\to \La_T\to W\to W_0\to 1,
\]
where $\La_T=\Fl_\calT(k)=T(F)/\calT(\calO_F)$ is the subgroup of translation elements in $W$. Note that $X_*(T)_I\simeq \La_T$, $\bar{\la}\mapsto t^{\bar{\la}}$ under the Kottwitz isomorphism. The fixed point scheme is on reduced loci
\[
(\Fl_\calG)^0_\red=\underline{W},
\]
where $\underline{W}$ denotes the constant $k$-scheme. Hence, there is a decomposition into connected components
\[
(\Fl_{\calG})^\pm=\coprod_{w\in W} (\Fl_\calG)^{\pm}_w,
\]
and the image of the map $\Fl_{\calB^\pm}\hookto (\Fl_{\calG})^\pm$ identifies (on reduced loci) with the sum of the connected components $(\Fl_{\calG})^{\pm}_{t^{\bar{\la}}}$ for $\bar{\la}\in X_*(T)_I$. Let $2\rho\in X^*(T)$ denote the sum of the $B^+$-positive roots. For $\bar{\la}\in X_*(T)_I$, let $\la\in X_*(T)$ be any lift and define the integer 
\[
\lan2\rho,\bar{\la} \ran \defined\lan2\rho, \la\ran \in \bbZ,
\]
which is well-defined since $B^+$ is $F$-rational, and hence $\ga\cdot 2\rho=2\rho$ for all $\ga\in I$.

\begin{cor} \label{ABcor}
Let $V\in \Rep_{\algQl}({\widehat{G}})$, and denote by $\calA_V\in\Sat_{G,\sF}$ the object with $\om_\sF(\calA_V)=V$. For the compact cohomology group as $\algQl$-vector spaces
\[
\bbH_c^i((\Fl_{\calG})_{w}^+,\Psi_\calG(\calA_V))= \begin{cases} V({\bar{\la}}) &\text{if $w=t^{\bar{\la}}$ and $i=\langle2\rho,\bar{\la}\rangle$;}\\ 0 &\text{else,}\end{cases}
\]
where $V({\bar{\la}})$ is the $\bar{\la}$-weight space in $V|_{\widehat{T}^I}$.
\end{cor}
\hfill{\ensuremath{\Box}}

The following lemma is used in the proof of Theorem \ref{specialfiber} below. 

\begin{lem}\label{intersectionlem}
For $w\in W$, there is an inclusion of non-empty sets 
\[
(\Fl_\calG^{\leq w})^\pm_w(k)=(\Fl_\calG^{\leq w})(k)\cap(\Fl_\calG)^\pm_w(k)\subset \Fl_\calG^w(k).
\]
\end{lem}
\begin{proof} As the class of $w$ is contained in all sets, these are non-empty. The first equality follows from the equality $(\Fl_\calG^{\leq w})^\pm=\Fl_\calG^{\leq w}\times_{\Fl_\calG}(\Fl_\calG)^\pm$ which holds since $\Fl_\calG^{\leq w}\hookto \Fl_\calG$ is a closed immersion. The case of repellers follows by inverting the $\bbG_m$-action from the case of attractors. We proceed by induction on the length $l(w)$. If $l(w)=1$, then $w$ is a simple reflection and $\Fl_\calG^{\leq w}\simeq \bbP^1_k$. In this case, either $(\Fl_\calG^{\leq w})^+_w=\{w\}$ or $(\Fl_\calG^{\leq w})^+_w=\Fl_\calG^w$ because $\chi$ is regular (hence the $\bbG_m$-action is non-trivial). 

Now let $l(w)\geq 2$, and write $w=w'\cdot s$ with $l(w)=l(w')+1$, and $s$ a simple reflection. The partial Demazure resolution embeds as a closed $\bbG_m$-invariant\footnote{For the diagonal $\bbG_m$-action on the target.} subscheme 
\[
(p,m)\co \Fl_\calG^{\leq w'}\tilde{\times}\Fl_\calG^{\leq s}\;{\hookto}\;\Fl_\calG^{\leq w'}\times \Fl_\calG^{\leq w},
\]
where $p\co (x,y)\mapsto x$ denotes the projection on the first factor, and $m\co (x,y)\mapsto  x\cdot y$ the ``multiplication'' map.
Identifying $\Fl_\calG^{\leq w'}\tilde{\times}\Fl_\calG^{\leq s}$ with its image under $(p,m)$, the multiplication map is given by projection onto the second factor. This implies the following description on fixed points
\[
(\Fl_\calG^{\leq w'}\tilde{\times}\Fl_\calG^{\leq s})^0(k)=\{(v_1,v_2)\in W^2\; ;\; v_1\leq w',\;\; v_1^{-1}\cdot v_2\leq s\}=:\mathscr{S}.
\]
Hence, the connected components of $(\Fl_\calG^{\leq w'}\tilde{\times}\Fl_\calG^{\leq s})^+$ are enumerated by the set $\mathscr{S}$ by general properties of attractors, cf. \cite[Prop 1.17 ii)]{Ri19}. The map $m$ is given on fixed points by $m^0\co \mathscr{S}\to W, \; (v_1,v_2)\mapsto v_2$, and since $s$ is a simple reflection we must have $m^{0,-1}(w)=\{(w',w)\}$. This implies for the preimage
\begin{equation}\label{intersectionlem:eq1}
m^{-1}\big((\Fl_\calG^{\leq w})^+_w\big) = (\Fl_\calG^{\leq w'}\tilde{\times}\Fl_\calG^{\leq s})^+_{(w',w)}.
\end{equation}
By the induction hypothesis, we have an inclusion $(\Fl_{\calG}^{\leq w'})^+_{w'}\subset \Fl_\calG^{w'}$ which implies the inclusion
\[
(\Fl_\calG^{\leq w'}\tilde{\times}\Fl_\calG^{\leq s})^+_{(w',w)}\subset p^{-1}\big(\Fl_\calG^{w'}\big)=\Fl_\calG^{w'}\tilde{\times}\Fl_\calG^{\leq s}.
\]
Hence, $(\Fl_\calG^{\leq w'}\tilde{\times}\Fl_\calG^{\leq s})^+_{(w',w)}$ identifies with a connected component of $(\Fl_\calG^{w'}\tilde{\times}\Fl_\calG^{\leq s})^+$, and we claim that $(\Fl_\calG^{\leq w'}\tilde{\times}\Fl_\calG^{\leq s})^+_{(w',w)}\subset \Fl_\calG^{w'}\tilde{\times}\Fl_\calG^{s}$. Indeed, if $(x,y)\in (\Fl_\calG^{w'}\tilde{\times}\Fl_\calG^{\leq s})^+_{(w',w)}(k)$, then $x\in (\Fl_\calG^{w'})^+_{w'}(k)$ (by the induction hypothesis), and either $y\in \Fl_\calG^{w}(k)$ or $y\in \Fl_\calG^{w'}(k)$ because $s$ is simple. In the second case, $(x,y)\in (\Fl_\calG^{\leq w'}\tilde{\times}\Fl_\calG^{\leq s})^+_{(w',w')}(k)$, and hence we must have $y\in \Fl_\calG^{w}(k)$ for $(x,y)\in (\Fl_\calG^{\leq w'}\tilde{\times}\Fl_\calG^{\leq s})^+_{(w',w)}(k)$. Together with \eqref{intersectionlem:eq1} it follows that 
\[
(\Fl_\calG^{\leq w})^+_w=m\big((\Fl_\calG^{\leq w'}\tilde{\times}\Fl_\calG^{\leq s})^+_{(w',w)}\big)\subset m\big(\Fl_\calG^{w'}\tilde{\times}\Fl_\calG^{ s}\big)=\Fl_\calG^{w},
\]
which is what we wanted to show.
\end{proof}

%Local models
\subsection{Applications to local models} \label{AppLocMod}
We continue with the data and notation as in \S \ref{data_consterm}. Let $\{\mu\}$ be a $G(\sF)$-conjugacy class of geometric cocharacters with reflex field $E/F$. The following definition is \cite{PZ13} when $F/\bbQ_p$ and \cite{Zhu14, Ri14a} when $F\simeq\bbF_q\rpot{t}$.

\begin{dfn}\label{localmodel} The local model $M_{\{\mu\}}=M_{G,\calG,\{\mu\}}$ is the scheme theoretic closure of the locally closed subscheme
\[
\Gr_{G}^{\leq\{\mu\}}\hookto (\Gr_{G}\otimes_F E)_{\red}\hookto (\Gr_{\calG}\otimes_{\calO_F}\calO_E)_{\red},
\]
where $\Gr_G^{\leq \{\mu\}}$ is as in \eqref{schubertgrass}.
\end{dfn} 

The local model $M_{\{\mu\}}$ is a flat projective $\calO_E$-subscheme of $(\Gr_{\calG}\otimes_{\calO_F}\calO_E)_{\red}$ which is uniquely determined up to unique isomorphism by the data $(\uG,\calG,\{\mu\})$ and the choice of a uniformizer in $F$; for a further discussion see Theorem \ref{PZspreading} and Remark \ref{PZBDGrass_rmk} (if $F/\bbQ_p$), and Remark \ref{BDrmk} i) (if $F\simeq k\rpot{t}$). Its generic fiber $M_{\{\mu\}}\otimes E= \Gr_{G,E}^{\leq\{\mu\}}$ is a (geometrically irreducible) variety, and the special fiber $M_{\{\mu\}}\otimes k_E$ is equidimensional, cf. \cite[Thm 14.114]{GW10}. There is a closed embedding into the flag variety
\[
M_{\{\mu\}}\otimes k_E\hookto \Gr_{\calG}\otimes_{\calO_F}k_E=\Fl_{\calG, k_E},
\]
which identifies the reduced locus $(M_{\{\mu\}}\otimes k_E)_\red$ with a union of Schubert varieties in $\Fl_{\calG, k_E}$. We show how Theorem \ref{commctnearbythm} implies that these Schubert varieties are enumerated by the $\{\mu\}$-admissible set in the sense of Kottwitz-Rapoport.

\subsubsection{The special fiber of Local Models} Let us recall the definition of the $\{\mu\}$-admissible set. For the rest of this section, we assume that $k=k_E$ is algebraically closed, i.e. $F=\nF$. Let $\calG=\calG_\bbf$, and fix $A=S\subset T$ with notations as in \eqref{groupchain} above such that $\bbf$ is contained in the apartment $\scrA(G,A)$. Let $\bba\subset \scrA(G,A)$ be an alcove containing $\bbf$ in its closure. Let $W=W(G,A)$ be the Iwahori-Weyl group (cf. Definition \ref{IwahoriWeyl}), $W_0=W_0(G,A)$ be the relative Weyl group, and let $W_{0}^{\text{abs}}=W_0(G_\sF,T_\sF)$ be the absolute Weyl group. The class $\{\mu\}$ gives a well defined $W_{0}^{\text{abs}}$-orbit 
\[
W_{0}^{\text{abs}}\cdot \mu,
\]
where $\mu\in X_*(T)$ is any representative of $\{\mu\}$. Denote by $\tilde{\La}_{\{\mu\}}$ the set of elements $\la\in W_0^{\text{abs}}\cdot \mu$ such that $\la$ is dominant with respect to some $\breve{F}$-rational Borel subgroup of $G$ containing $T$. Let $\La_{\{\mu\}}$ be the image of $\tilde{\La}_{\{\mu\}}$ under the canonical projection $X_*(T)\to X_*(T)_I$. The $\{\mu\}$-admissible set $\Adm_{\{\mu\}}$ (relative to $\bba$) is the partially ordered subset of the Iwahori-Weyl group
\begin{equation}\label{muadmset}
\Adm_{\{\mu\}}\defined \big\{w\in W\;|\;\exists \bar{\la} \in \La_{\{\mu\}}: \;w\leq t^{\bar{\la}}\big\},
\end{equation}
where $\leq$ is the Bruhat order of $W$. Let $W_{\bbf}\subset W$ be the subgroup associated with $\bbf$, cf.\,(\ref{W_f_defn}). The $\{\mu\}$-admissible set $\Adm^\bbf_{\{\mu\}}$ relative to $\bbf$ is the partially ordered subset
\begin{equation}\label{fmuadmset}
\Adm_{\{\mu\}}^\bbf\defined W_\bbf\bslash \Adm_{\{\mu\}} / W_\bbf\subset W_\bbf\bslash W/ W_\bbf.
\end{equation}
This does not depend on the choice of the alcove $\bba\subset \scrA(G,A)$ containing $\bbf$ in its closure.

If $G$ splits over a tamely ramified extension and $p\nmid |\pi_1(G_\der)|$, then the following theorem is a weaker form of \cite[Thm 9.3]{PZ13} (if $F/\bbQ_p$) and \cite[Thm 3.8]{Zhu14} (if $F\simeq\bbF_q\rpot{t}$). Hence, the result is new when either $p \mid |\pi_1(G_\der)|$ or $F\simeq\bbF_q\rpot{t}$ and $G$ splits over a wildly ramified extension.

\begin{thm}\label{specialfiber}
The smooth locus $(M_{\{\mu\}})^{\on{sm}}$ is fiberwise dense in $M_{\{\mu\}}$, and on reduced subschemes
\[
(M_{\{\mu\},k})_\red=\bigcup_{w\in\Adm_{\{\mu\}}^\bbf}\Fl_\calG^{\leq w}.
\]
In particular, the special fiber $M_{\{\mu\},k}$ is generically reduced.
\end{thm}
\begin{proof} Once we have determined $(M_{\{\mu\},k})_\red$, the method of \cite[Cor 3.14]{Ri16a} shows in both cases (i.e. $F/\breve{\bbQ}_p$ or $F=\bar{\bbF}_p\rpot{t}$) that the special fiber of $(M_{\{\mu\}})^{\on{sm}}$ is dense in $M_{\{\mu\},k}$: each $\la\in \tilde{\La}_{\{\mu\}}$ determines a point $\Spec(\sF)\to \Gr_\calG$ which extends uniquely (by ind-properness) to a map $\tilde{\la}\co \Spec(\calO_{\sF})\to \Gr_\calG$. The $\calL^+\calG_{\calO_\sF}$-orbit of $\tilde{\la}$ is then representable by an open and smooth subscheme of $M_{\{\mu\}}\otimes\calO_{\sF}$, and the union of the orbits for $\la\in\tilde{\La}_{\{\mu\}}$ is open, smooth and fiberwise dense in $M_{\{\mu\}}$.

Now let $\IC_{\{\mu\}}$ be the normalized intersection complex on the generic fiber $M_{\{\mu\},E}=\Gr_{G,E}^{\leq \{\mu\}}$. The support of the nearby cycles $\Psi_\calG(\IC_{\{\mu\}})$ is a $L^+\calG$-equivariant reduced closed subscheme of $\Fl_\calG$ and as such
\[
(M_{\{\mu\},k})_\red=\Supp\Psi_\calG(\IC_{\{\mu\}})\subset \Fl_\calG,
\]
by \cite[Lem 7.1]{Zhu14}. In particular, the support is a union of Schubert varieties in $\Fl_\calG$, and we let $\Supp_{\{\mu\}}^\bbf$ denote the subset of the classes in $W_\bbf\bslash W/ W_\bbf$ belonging to these Schubert varieties. We have to show 
\[
\Adm_{\{\mu\}}^\bbf\overset{!}{=}\Supp_{\{\mu\}}^\bbf,
\] 
as subsets of $W_\bbf\bslash W/W_\bbf$. By \cite[Lem 3.12]{Ri16a}, we already know $\Adm_{\{\mu\}}^\bbf\subset \Supp_{\{\mu\}}^\bbf$. We proceed in two steps.\smallskip\\
\emph{Reduction to the case $\bbf=\bba$ is an alcove.} Let $\bba\subset \scrA(G,A)$ be an alcove containing $\bbf$ in its closure. The map of group schemes $\calG_\bba\to \calG_\bbf$ induces a proper map of $\calO_E$-schemes
\[
f\co M_{(G,\calG_\bba,\{\mu\})}\to M_{(G,\calG_\bbf,\{\mu\})}. 
\]
The compatibility of nearby cycles with proper push forward implies that $f_*\circ \Psi_{\calG_\bba}(\IC_{\{\mu\}})=\Psi_{\calG_\bbf}(\IC_{\{\mu\}})$, and hence the map $\Supp_{\{\mu\}}^\bba\to \Supp_{\{\mu\}}^\bbf$ is surjective. We obtain a commutative diagram of sets
\[
\begin{tikzpicture}[baseline=(current  bounding  box.center)]
\matrix(a)[matrix of math nodes, 
row sep=1.5em, column sep=2em, 
text height=1.5ex, text depth=0.45ex] 
{\Adm_{\{\mu\}}^\bba&\Supp_{\{\mu\}}^\bba \\ 
\Adm_{\{\mu\}}^\bbf&\Supp_{\{\mu\}}^\bbf, \\}; 
\path[right hook->](a-1-1) edge node[above] {} (a-1-2);
\path[right hook->](a-2-1) edge node[above] {} (a-2-2);
\path[->](a-1-1) edge node[right] {} (a-2-1);
\path[->](a-1-2) edge node[right] {} (a-2-2);
\end{tikzpicture}
\]
with the vertical maps being surjective. Thus, the equality $\Adm_{\{\mu\}}^\bba=\Supp_{\{\mu\}}^\bba$ implies the equality $\Adm_{\{\mu\}}^\bbf=\Supp_{\{\mu\}}^\bbf$.\smallskip\\
\emph{Proof in the case of an alcove.} Let $\bbf=\bba$ be an alcove, and drop the superscript from the notation. We show that the maximal elements of $\Supp_{\{\mu\}}$ are precisely the $t^{\bar{\la}}$ for $\bar{\la}\in \La_{\{\mu\}}$ (cf. \eqref{muadmset} above) which proves the theorem.

We choose a regular cocharacter $\chi\co \bbG_{m,\calO}\to \calG$ as in \S \ref{Iwahoricase}, and use the notation introduced there. Let $\bar{\mu}$ be the $B^+$-dominant element in $\La_{\{\mu\}}$. Now let $w\in \Supp_{\{\mu\}}$ be maximal, i.e. $\Fl_\calG^{\leq w}$ is an irreducible component of $(M_{\{\mu\},k})_\red$. By the equidimensionality, we have for the length
\[
l(w)=\dim(\Fl_\calG^{\leq w})=\dim(\Gr_{G,E}^{\{\mu\}})=\lan2\rho,\bar{\mu}\ran.
\]
Further, as $\algQl$-vector space
\[
\bbH^*_c((\Fl_\calG)_{w}^+, \Psi_\calG(\IC_{\{\mu\}}))\not = 0,
\]
because $\Fl_\calG^{\leq w}\cap (\Fl_\calG)_{w}^+\subset\Fl_\calG^w$ is non-empty by Lemma \ref{intersectionlem}, and $\Psi_\calG(\IC_{\{\mu\}})|_{\Fl_\calG^w}=\algQl\lan l(w)\ran^m$ with $m>0$. Now Corollary \ref{ABcor} implies that $w=t^{\bar{\la}}$ for some $\bar{\la}\in X_*(T)_I$ which is also a weight in $V_{\{\mu\}}|_{\widehat{G}^I}$. We conclude $\bar{\la}\in \La_{\{\mu\}}$ by citing \cite[Thm 4.2 and (7.11-12)]{Hai}.
\end{proof}

\begin{rmk} \label{geometry_minuscule_rmk}
The proof of Theorem \ref{specialfiber} can be used to obtain additional information on the $\bbG_m$-stratification $(M_{\{\mu\}})^+\to M_{\{\mu\}}$ if $\{\mu\}$ is minuscule, $G$ is split and $\calG$ is an Iwahori, cf.~also \S\ref{local_model_intro_sec}.
In this case, the generic fibre $M_{\{\mu\}}\otimes E$ is the variety of type $\{\mu\}$ parabolic subgroups in $G_E$, which is smooth.
In particular, it is contained in the smooth locus $(M_{\{\mu\}})^{\on{sm}}$ which is an $\bbG_m$-invariant open subscheme of $M_{\{\mu\}}$.
By Corollary \ref{immersions}, the induced map on the attractor 
\begin{equation}\label{open_immersion_minuscule_eq}
(M_{\{\mu\}})^{\on{sm},+}\longto (M_{\{\mu\}})^+ 
\end{equation}
is an open immersion. 
We claim that \eqref{open_immersion_minuscule_eq} induces an isomorphism on each connected component with non-empty generic fiber.
Since $(M_{\{\mu\}})^{\on{sm},+}\to \Spec(\calO_E)$ is smooth by Lemma \ref{0-perm} ii), this implies that every connected component of $(M_{\{\mu\}})^+$ with non-empty generic fiber is smooth as well.
To prove the claim, observe that every $\la\in \tilde{\La}_{\{\mu\}}=W_0\cdot \mu$ determines an $\bbG_m$-fixed point $\Spec(\sF)\to M_{\{\mu\}}$ which extends uniquely (by properness) to an $\bbG_m$-fixed point $\tilde{\la}\co \Spec(\calO_{\sF})\to M_{\{\mu\}}$. 
The $\calL^+\calG_{\calO_\sF}$-orbit of $\tilde{\la}$ is an open smooth subscheme of $M_{\{\mu\}}\otimes \calO_{\sF}$ with generic fibre $M_{\{\mu\}}\otimes \sF$ (because $\{\mu\}$ is minuscule) and special fiber $\Fl_\calG^{t^{\bar\la}}$.
At least set theoretically, this orbit contains by Lemma \ref{intersectionlem} the unique connected component $(M_{\{\mu\}})^+_\la$ of $(M_{\{\mu\}})^+$ containing $\la$.
Therefore, $(M_{\{\mu\}})^+_\la$ is set theoretically contained in $(M_{\{\mu\}})^{\on{sm}}$.
Hence, it identifies with a connected component of $(M_{\{\mu\}})^{\on{sm},+}$ because \eqref{open_immersion_minuscule_eq} is an open immersion. 
As every connected component of $(M_{\{\mu\}}\otimes E)^+$ passes through some $\la\in W_0\cdot \mu$, this implies the claim. 
\end{rmk}

%Central sheaves
\subsection{Central sheaves}
 We continue with the data and notation as in \S \ref{data_consterm}. As in Definition \ref{EquiGaloisCat}, let $\on{Perv}_{L^+\calG}(\Fl_\calG\times_s\eta)$ the category of $L^+\calG$-equivariant perverse sheaves compatible with a continuous Galois action. 

Recall that for objects in $\on{Perv}_{L^+\calG}(\Fl_\calG\times_s\eta)$ there is the convolution product defined by Lusztig \cite{Lu81}. Consider the convolution diagram
\[
\Fl_\calG\times\Fl_\calG\overset{\;q}{\leftarrow} L\calG\times \Fl_\calG \overset{p}{\to} L\calG\times^{L^+\calG}\Fl_\calG=:\Fl_\calG\tilde{\times}\Fl_\calG\overset{m\;}{\to} \Fl_\calG. 
\]
For $\calA,\calB\in \on{Perv}_{L^+\calG}(\Fl_\calG\times_s\eta)$, let $\calA\tilde{\times}\calB$ be the (unique up to canonical isomorphism) complex on $\Fl_\calG\tilde{\times}\Fl_\calG$ such that $q^*(\calA\boxtimes \calB)\simeq p^*(\calA\tilde{\times} \calB)$. By definition
\begin{equation}\label{convol_sheaf}
\calA\star \calB\defined m_*(\calA\tilde{\times} \calB)\in D_{c}^b(\Fl_\calG\times_s\eta, \algQl).
\end{equation}
In the following, we consider ${\rm Perv}_{L^+\calG}(\Fl_\calG)$ as a full subcategory of $ {\rm Perv}_{L^+\calG}(\Fl_\calG\times_s\eta)$. 

Fix a chain of tori $A\subset S\subset T$ as in \eqref{groupchain}, and let $W=W(G,A)$ be the associated Iwahori-Weyl group over $F$, cf. Definition \ref{IwahoriWeyl} ii). For each $w\in W$, the associated Schubert variety $\Fl_\calG^{\leq w}\subset \Fl_\calG$ is defined over $k$. Let $j\co \Fl_\calG^{w}\hookto \Fl_\calG^{\leq w}$, and and denote by $\IC_w=j_{!*}(\algQl[\dim(\Fl_\calG^{w})])$ the intersection complex. As in \eqref{equi_nearby}, we have the functor of nearby cycles
\[
\Psi_{\calG}\co \on{Perv}_{L^+G}(\Gr_{G})\longto \on{Perv}_{L^+\calG}(\Fl_{\calG}\times_{s}\eta).
\]
The next theorem follows from \cite[Thm 10.5]{PZ13} (if $F/\bbQ_p$) and \cite[Thm 7.3]{Zhu14} (if $F\simeq\bbF_q\rpot{t}$) which are both built upon ideas of \cite{Ga01}:

\begin{thm}[Gaitsgory, Zhu, Pappas-Zhu] \label{CentralNB}
For each $\calA\in \on{Perv}_{L^+G}(\Gr_{G})$, and $w\in W$, both convolutions $\Psi_{\calG}(\calA)\star\IC_w$, $\IC_w\star\Psi_{\calG}(\calA)$ are objects in $P_{L^+\calG}(\Fl_\calG\times_s\eta)$, and as such there is a canonical isomorphism
\[
\Psi_{\calG}(\calA)\star\IC_w\simeq \IC_w\star\Psi_{\calG}(\calA).
\]
\end{thm}
\begin{proof} {\it Mixed Characteristic.} Let $F/\bbQ_p$, and hence we assumed $G$ to be tamely ramified. If $\calA= \IC_{\{\mu\}}$ where $\{\mu\}$ is a class which is defined over $F$, then the theorem is a special case of \cite[Thm 10.5]{PZ13}. However, the proof given there works for general objects $\calA\in \on{Perv}_{L^+G}(\Gr_{G})$, and only uses that the support $\Supp(\calA)$ is finite dimensional and defined over $F$. \smallskip\\
{\it Equal Characteristic.} Let $F\simeq k\rpot{t}$. If $G$ is tamely ramified, and if $\calA= \IC_{\{\mu\}}$, then the theorem is a special case of \cite[Thm 7.3]{Zhu14}. However, the arguments given in \cite[\S 7.2]{Zhu14} suffice to treat the case of a general (possibly wildly ramified) connected reductive group $G$. Here we use \cite[Thm 2.19]{Ri16a} to justify the ind-properness of the Beilinson-Drinfeld and Convolution Grassmannians which are used in the proof. We do not repeat the arguments here.    
\end{proof}

%Test functions
\section{Application to the test function conjecture}\label{testfunctions}

\subsection{From sheaves to functions}
\subsubsection{The semi-simple trace} 
Let us collect some facts about the sheaf function dictionary for semi-simple traces. For further details, we refer to \cite[3.1]{HN02} (cf.\,also \cite[10.4]{PZ13}). The notion of semi-simple trace is due to Rapoport. 

For a separated $k$-scheme $X$ of finite type, we have $D_c^b(X\times_s \eta)$ as in \S \ref{constermflag} above. For a complex $\calA\in D_c^b(X\times_s\eta)$, we consider the semi-simple trace of geometric Frobenius function
\[
\tau_\calA^{\on{ss}}\co X(k)\to \algQl, \;\; x\mapsto \sum_{i\in \bbZ}(-1)^i\tr^{\on{ss}}(\Phi\,|\,\calH^i(\calA)_{\bar{x}}),
\]
where $\tr^{\on{ss}}$ is the trace on the inertia-fixed vectors in the associated graded of a Galois stable filtration on which the inertia group acts via a finite quotient. If $f\co X\to Y$ is a map of separated $k$-schemes of finite type, then there are the identities
\begin{equation}\label{funcid}
\tau^{\on{ss}}_{f_!\calA}(y)=\sum_{x\in f^{-1}(y)}\tau_\calA^{\on{ss}}(x) \;\;\;\text{and}\;\;\; \tau^{\on{ss}}_{f^*\calA}(x)=\tau_\calA^{\on{ss}}(f(x)).
\end{equation}
For shifts and twists one has
\begin{equation}\label{shiftandtwist}
\tau^{\on{ss}}_{\calA[m]}=(-1)^m\tau^{\on{ss}}_{\calA}\;\;\;\text{and}\;\;\;\tau^{\on{ss}}_{\calA(\nicefrac{m}{2})}=q^{-\nicefrac{m}{2}}\tau^{\on{ss}}_{\calA}.
\end{equation}
The construction carries over to the case of separated $k$-ind-schemes of finite type, cf. \cite[10.4]{PZ13} for details.

\subsubsection{The Hecke algebra} We proceed with the data and notation from \S \eqref{data_consterm}. Let $A\subset S\subset T$ be a chain of tori as in \eqref{groupchain}. Fix a Haar measure on $G(F)$ giving the compact open subgroup $\calG(\calO)$ volume $1$. The parahoric Hecke algebra is the $\algQl$-algebra
\[
\calH(G(F),\calG(\calO_F))\defined C_c(\calG(\calO)\bslash G(F)/\calG(\calO); \algQl)
\]
of bi-$\calG(\calO_F)$-invariant compactly supported, locally constant functions on $G(F)$. The algebra structure is given by convolution of functions, and is for $f_1,f_2 \in \calH(G(F),\calG(\calO_F))$ given by the formula
\begin{equation}\label{convol_function}
(f_1\star f_2)(x)=\int_{G(F)}f_1(g)f_2(g^{-1}x)dg.
\end{equation}
We write $\calZ(G(F),\calG(\calO_F))$ for the center of $\calH(G(F),\calG(\calO_F))$.

\begin{rmk} \label{equal_equals_mixed}
If $F/\bbQ_p$ is of mixed characteristic, then we fix a spreading $\uG$ of $G$ as in \eqref{unequalspread}. With the notation of Lemma \ref{hecke_identification} we can identify $\calZ(G(F),\calG(\calO_F))=\calZ(G'(k\rpot{t}),\calG'(k\pot{t}))$ as algebras in a way compatible with an identification on Iwahori-Weyl groups $W(G,A)=W(G',A')$. We will use this identification freely in what follows.
\end{rmk}

\subsubsection{Constant terms} \label{HA_CT_sec} In \cite[\S 11.11]{Hai14}, the first named author constructed the constant term map $c^G_M\co \calZ(G(F),\calG(\calO_F))\to \calZ(M(F),\calM(\calO_F))$ abstractly using the Bernstein center, and then showed it is given by the formula
\begin{equation} \label{C_defn}
c^G_M(f)(m)= \delta_{P^+}^{1/2}(m)\cdot \int_{N^+(F)}f(mn)dn,
\end{equation}
where $dn$ is normalized such that $\calN^+(\calO_F)$ gets volume $1$, and in $\calZ(M(F),\calM(\calO_F))$ we give $\calM(\calO_F)=\calG(\calO_F)\cap M(F)$ volume $1$. 

It will be convenient to work with another normalization which matches better with the geometric constant term.  Denote by $\bar{\nu}_m$ the image of $m \in M(F)$ under the Kottwitz homomorphism $M(F) \rightarrow \pi_1(M)_I^\Phi$, and let $\dot{\nu}_m \in X_*(T)$ be any lift of $\bar{\nu}_m$. The integer $\langle 2\rho_{N^+}, \bar{\nu}_m \rangle := \langle 2\rho_{N^+}, \dot{\nu}_m \rangle$ does not depend on the choice of $\dot{\nu}_m$, cf. \S \ref{conncompflagsec}. For $f \in \calZ(G, \calG)$, define $^pc^G_M(f) \in \calZ(M, \calM)$ by the formula
\begin{equation} \label{pC_defn}
^pc^G_M(f)(m) = (-1)^{\langle 2\rho_{N^+}, \bar{\nu}_m\rangle} \, \delta_{P^+}^{1/2}(m) \cdot \int_{N^+(F)} f(mn) dn.
\end{equation}

\begin{lem} \label{CT_RPsi_comp} Let $\calA \in \on{Perv}_{L^+\calG}(\Fl_\calG\times_s\eta)$ be an equivariant perverse sheaf, cf.\,Definition \ref{EquiGaloisCat}.\smallskip\\
i\textup{)} The function $\tau^{\on{ss}}_\calA$ is an element in the Hecke algebra $\calH(G,\calG)$. \smallskip\\ 
ii\textup{)} If $\tau_\calA^{\on{ss}}\in \calZ(G(F),\calG(\calO_F))$, then as functions in $\calZ(M(F),\calM(\calO_F))$ there is an equality
\[
^pc^G_M(\tau_\calA^{\on{ss}})=\tau^{\on{ss}}_{\on{CT}_\calM(\calA)}.
\]

\end{lem}
\begin{proof} As an element of $\calH(G(F),\calG(\calO_F))$ is the same as a finitely supported function on the double coset $\calG(\calO)\bslash G(F)/\calG(\calO)$, part i) follows from Lemma \ref{Bruhatflag} together with \eqref{funcid} (to check that equivariance as a sheaf translates to equivariance as a function). For part ii), we use Corollary \ref{gctflagpoints} and the functorialities \eqref{funcid}, \eqref{shiftandtwist}. It remains to explain, why $\delta_{P^+}^{1/2}$ agrees with the normalization in Definition \ref{gctnormal}. For $m\in M(F)$, by definition
\[
\delta_{P^+}(m)= |\on{det}\big(\on{Ad}(m)\,|\,\on{Lie}(N^+(F))\big)|_F
\]
where $|x|_F=q^{-\on{val}_F(x)}$ with $\on{val}_F(t)=1$ and $\on{Ad}(m)(\frakn)= m\cdot \frakn\cdot m^{-1}$ for $\frakn\in {\rm Lie}(N^+(F))$. The Kottwitz map gives an isomorphism
\begin{equation}\label{Kotthelp}
M(F)/M(F)^\circ \simeq \pi_1(M)_I^\Phi
\end{equation}
where $M(F)^\circ=(LM)^\circ(F)$ is the neutral component. Note that the classes in $\pi_1(M)_I$ where $\Phi$ acts non-trivially do not contribute. Consider the character
\[
\chi^{\vee} \co M\overset{\on{Ad}\;}{\longto} \underline{\Aut}({\rm Lie}(N^+))\overset{\on{det}\;}{\longto} \bbG_m.
\]
This gives the formula 
\[
\delta_{P^+}(m)=|\chi^\vee(m)|_F=q^{-\lan\chi^\vee|_T,\bnu_m\ran},
\]
where the class $[m]\in M(F)/M(F)^\circ$ corresponds to $\bnu_m \in \pi_1(M)_I$ under \eqref{Kotthelp}. On the other hand $\chi^\vee|_T=2\rho_{N^+}$, and we obtain 
\[
\delta_{P^+}^{1/2}(m)=q^{-\lan\rho_{N^+},\bnu_m\ran}. 
\]
Hence, the normalizations match, and the lemma follows.
\end{proof}

\subsubsection{Central functions}
As in \eqref{equi_nearby}, we have the functor of nearby cycles
\[
\Psi_{\calG}\co \on{Perv}_{L^+G}(\Gr_{G})\longto \on{Perv}_{L^+\calG}(\Fl_{\calG}\times_{s}\eta).
\]
The following theorem is an immediate consequence of Theorem \ref{CentralNB}.

\begin{thm}\label{centralfunction}
For each $\calA\in \on{Perv}_{L^+G}(\Gr_{G})$, the function $\tau_{\Psi_{\calG}(\calA)}^{\on{ss}}$ naturally is an element in the center $\calZ(G(F),\calG(\calO_F))$ in the following sense: \smallskip\\
i\textup{)} If $F\simeq k\rpot{t}$, then the function $\tau_{\Psi_{\calG}(\calA)}^{\on{ss}}$ depends canonically only on the data $(\uG,\calG,\calA)$.\smallskip\\
ii\textup{)} Let $F/\bbQ_p$, and recall that we fixed a spreading $\uG$ in Remark \ref{equal_equals_mixed}. Then the function $\tau_{\Psi_{\calG}(\calA)}^{\on{ss}}$ depends canonically only on the data $(\uG,\calG,\calA)$.
\end{thm}
\begin{proof} If $F/\bbQ_p$, then we use Remark \ref{equal_equals_mixed} to identify $\calZ(G(F),\calG(\calO_F))=\calZ(G'(k\rpot{t}),\calG'(k\pot{t}))$. Part i) (resp. ii)) follows from Remark \ref{BDrmk} i) (resp. Remark \ref{PZBDGrass_rmk}). Note that different choices of uniformizers in $\calO_F$ differ by elements in $\calO_F^\times$ which induced the identity automorphism on Hecke algebras. Let us show that $\tau_{\Psi_{\calG}(\calA)}^{\on{ss}}$ defines a central function. Using the sheaf function dictionary, especially \eqref{funcid}, and the definition of the convolution product \eqref{convol_function}, resp. \eqref{convol_sheaf}, we obtain that
\[
\tau_{\Psi_{\calG}(\calA)\star \calB}^{\on{ss}}= \tau_{\Psi_{\calG}(\calA)}^{\on{ss}}\star \tau_{\calB}^{\on{ss}},
\]
for every $\calB\in \on{Perv}_{L^+\calG}(\Fl_{\calG})$\footnote{More generally, the formula holds when the inertia action on $\calB$ is unipotent. We will not need this fact.}. In particular, by Theorem \ref{CentralNB}, the function $\tau_{\Psi_{\calG}(\calA)}^{\on{ss}}$ commutes with all the functions $\tau_w:=\tau_{\IC_w}^{\on{ss}}$ for $w\in W=W(G,A)$ in the Iwahori-Weyl group. But it is easy to see from Lemma \ref{Bruhatflag} that the algebra $\calH(G(F),\calG(\calO_F))$ is generated by the functions $\tau_w$ for $w\in W$: if $\calG$ is an Iwahori for example, we can write 
\[
\tau_w=(-1)^{l(w)}\left(\mathds{1}_{w} \;+\; \sum_{v<w}c_{v,w}\cdot \mathds{1}_{v}\right),
\]
where, for $v\in W$, the function $\mathds{1}_{v}$ is the characteristic function on the double coset $\calG(\calO_F)\cdot \dot{v}\cdot\calG(\calO_F)$, and $l\co W\to \bbZ_{\geq 0}$ denotes the length function. The fact that $\mathds{1}_{w}$ appears with multiplicity $1$, follows from the identity $\IC_w=j_{!*}(\algQl[\dim(\Fl_\calG^{w})])|_{\Fl_\calG^w}=\algQl[\dim(\Fl_\calG^{w})]$. Thus, by induction on $l(w)$, we get that all functions $\mathds{1}_{w}$ are contained in the $\algQl$-vector subspace generated by the $\tau_w$'s, which is of course the full Hecke algebra $\calH(G(F),\calG(\calO_F))$. The general parahoric case is similar using Lemma \ref{Bruhatflag}. 
\end{proof}

\subsection{Review of Satake parameters and definition of $z^{\rm ss}_{\mathcal G, I(V)}$} \label{Sat_trans_sec}

We review the construction of the Satake parameter of a representation with parahoric fixed vectors \cite{Hai15, Hai17}. 

Let $E/F$ be a finite extension field, and let $E_0/F$ be the maximal unramified subextension. Let $G$ be a connected reductive group over $E_0$, with usual data $A, S, T, M$ as in \eqref{groupchain}.  Let $G^*$ denote the $E_0$-quasisplit inner form of $G$, with corresponding data $A^*, S^*, T^* = M^*$. Let $W^*_0$ (resp.\,$W^*_{0,E}$) denote the relative Weyl group of $G^*$ (resp.\,$G^*_E$). Then $W^*_0 = (W^*_{0,E})^{{\rm Gal}(E/E_0)}$. The geometric Frobenius is insensitive to the extension $E/E_0$: $\Phi_E = \Phi_{E_0}$; abbreviate it by $\Phi$. We have $G_{\bar{E}_0} = G^*_{\bar{E}_0}$ but the Galois actions differ by a 1-cocycle in $G^*_{\ad, \breve{E}_0}$. Thus we use $\Phi^*$ and $\Phi$  to distinguish the actions of $\Phi$ related to $G^*$ and $G$ (although on dual groups there is no difference, and $^LG_E = \,^LG^*_E$). There is a canonical finite morphism of algebraic varieties over $\bar{\mathbb Q}_\ell$
\begin{equation} \label{s_defn}
s : (Z(\widehat{M})^{I_{E_0}})_{\Phi}/W_0 \longrightarrow (\widehat{T^*}^{I_{E_0}})_{\Phi^*}/W^*_0 \longrightarrow  (\widehat{T^*}^{I_E})_{\Phi^*}/W^*_{0,E} \,\cong \, [\widehat{G^*}^{I_{E}} \rtimes \Phi^* ]_{\rm ss}/\widehat{G^*}^{I_{E}}
\end{equation}
(see \cite[(9.1), Prop 6.1]{Hai15} and \cite[Lem 8.2']{Hai17}).  

By \cite[Thm 11.10.1]{Hai14}, the central algebra $\mathcal Z(G(E_0), \mathcal G(\mathcal O_{E_0}))$ is the ring of regular functions on the variety $(Z(\widehat{M})^{I_{E_0}})_{\Phi}/W_0$, which is the union of the components of the Bernstein variety corresponding to representations of $G(E_0)$ with parahoric-fixed vectors.

Let $V \in {\rm Rep}(\,^LG_E)$ be an algebraic representation. Let $I(V) = {\rm Ind}^{\,^LG_{E}}_{\,^LG_{E_0}}(V)$. For $t^* \in \widehat{T^*}^{I_E}$, the map $t^* \mapsto {\rm tr}(t^* \rtimes \Phi^* \, | \, V^{1 \rtimes I_E})$ gives a regular function on $(\widehat{T^*}^{I_E})_{\Phi^*}/W^*_{0,E}$.  Pulling back along $s$, we obtain the regular function $\chi \mapsto {\rm tr}(s(\chi) \, | \, V^{1 \rtimes I_E})$.  This is precisely the definition of $z^{\rm ss}_{\mathcal G, I(V)}$. In other words, $z^{\rm ss}_{\mathcal G, I(V)}$ is the element of $\mathcal Z(G(E_0), \mathcal G(\mathcal O_{E_0}))$ which, for any irreducible smooth representation $\pi$ of $G(E_0)$ with supercuspidal support $(M(E_0), \chi)_{G}$, acts on $\pi^{\mathcal G(\mathcal O_{E_0})}$ by the scalar $${\rm tr}(s(\chi) \, | \, V^{1 \rtimes I_E}) = {\rm tr}(s(\chi) \, | \, I(V)^{1 \rtimes I_{E_0}}).$$ 
(See \cite[Lem 8.1]{Hai}.) The Satake parameter of $\pi$ is by definition $s(\pi) := s(\chi)$. 

\begin{rmk}
One can construct unconditionally an element $Z_{I(V)}$ of the stable Bernstein center of $G(E_0)$ as in \cite[5.7]{Hai14}. If one accepts the enhanced local Langlands conjecture LLC+, then there is a map from the stable Bernstein center to the usual Bernstein center of $G(E_0)$ (cf.\,\cite[Cor 5.5.2]{Hai14}). Denote also by $Z_{I(V)}$ the resulting $G(E_0)$-invariant distribution on $G(E_0)$. We obtain a function $Z_{I(V)} \star 1_J \in \mathcal Z(G(E_0), J)$ for any compact open subgroup $J \subset G(E_0)$. If $J = \mathcal G(\mathcal O_{E_0})$ is parahoric, then $z^{\rm ss}_{\mathcal G, I(V)}$ is an unconditional version of $Z_{I(V)} \star 1_J$.
\end{rmk}

\subsection{Statement of the test function conjecture for local models} \label{TFC_stmt_sec}

\subsubsection{The data}

We consider the fields $E, E_0, F$ as before (we discuss equal and mixed characteristic settings uniformly), and the data $(G, \calG, \{\mu\}, {\rm Gr}_{\mathcal G})$. Instead of requiring $E$ to be the field of definition of $\{\mu\}$, it is enough for us to assume that $E$ is a finite unramified extension thereof. Let $k_{E} = k_{E_0}$ be the common residue field of $E$ and $E_0$, and let $\Phi_E = \Phi_{E_0} = \Phi_F^{[E_0:F]}$ be the common geometric Frobenius element in $\Gamma_E \subset \Gamma_{E_0}$. Let $I_E \subset \Gamma_E$ (resp.\,$I_{E_0} \subset \Gamma_{E_0}$) be the inertia subgroup. 

\subsubsection{The representation side}

Let $(V_{\{\mu\}}, r_{\{ \mu \}})$ be the representation of $^LG_E = \widehat{G} \rtimes \Gamma_E$ constructed as in \cite[$\S6.1$]{Hai14}. We write $(I(V_{\{\mu\}}), i_{\{\mu\}})$ for the induced representation 
$$
I(V_{\{\mu\}}) = {\rm Ind}^{\widehat{G} \rtimes \Gamma_{E_0}}_{\widehat{G} \rtimes \Gamma_E}(V_{\{\mu\}}).
$$
Then we define $z^{\rm ss}_{\calG, \{ \mu\}}$ (written $z^{\rm ss}_{\{\mu\}}$ when $\calG$ is understood) to be the function $z^{\rm ss}_{\calG, I(V_{\{\mu\}})}$.

\subsubsection{Nearby cycles side}

The conjugacy class $\{ \mu \}$ gives rise as usual to a Schubert variety in ${\rm Gr}_{G, \bar{F}}$; it is a finite dimensional projective scheme to which we give the reduced structure. It is stable under the action of $\Gamma_E$, hence is defined over $E$.  Let $M_{\{\mu\},E}$ denote the resulting $E$-variety in ${\rm Gr}_{G, E}$. We let $M_{\{\mu\}}$ denote the flat closure of $M_{\{\mu\}, E}$ in ${\rm Gr}_{\mathcal G, \mathcal O_E}$, with reduced structure. Let $d_\mu$ be the dimension of the generic fiber  $M_{\{\mu\},E}$ over $E$.

Therefore we have a closed embedding
$$
M_{\{\mu\}} \, \hookrightarrow \, \Big({\rm Gr}_{\mathcal G} \otimes_{\mathcal O} \mathcal O_{E}\Big)_{\rm red} = \Big({\rm Gr}_{\mathcal G, \mathcal O_{E_0}} \otimes_{\mathcal O_{E_0}} \mathcal O_E \Big)_{\rm red}.
$$
Write the base change projection as
$$
f\co {\rm Gr}_{\mathcal G, \mathcal O_{E_0}} \otimes_{\mathcal O_{E_0}} \mathcal O_E \, \longrightarrow \, {\rm Gr}_{\mathcal G, \mathcal O_{E_0}}.
$$
Recall ${\rm IC}_{\{\mu\}}$ denotes the intersection complex on $M_{\{\mu\}, E}$, normalized as in \eqref{normalized_IC} so that it is perverse and weight zero. Then we have 
$$
{\rm IC}_{\{\mu\}} \in {\rm Perv}_{L^+G_E}({\rm Gr}_{G, E}, \overline{\mathbb Q}_\ell).
$$
Our goal is to understand the function 
\begin{equation} \label{starting_pt}
{\rm tr}^{\rm ss}(\Phi_E \, | \, \Psi_{{\rm Gr}_{\calG, \mathcal O_{E}}}({\rm IC}_{\{\mu\}}))\co M_{\{\mu\}}(k_E) \longrightarrow \bar{\mathbb Q}_\ell.
\end{equation}

The operation $f_{\bar s, *}$ corresponds to induction of Galois representations, cf. \cite[Exp XIII 1.2.7 b)]{SGA7}. Therefore by the analogue of \cite[Lem 8.1]{Hai} and the equality $\Phi_{E_0} = \Phi_E$, we can rewrite (\ref{starting_pt}) as
$$
{\rm tr}^{\rm ss}(\Phi_E \, | \, \Psi_{{\rm Gr}_{\calG, \mathcal O_{E}}}({\rm IC}_{\{\mu\}})) = {\rm tr}^{\rm ss}\big(\Phi_{E_0} \, | \, f_{\bar s, *} \, \Psi_{{\rm Gr}_{\calG, \mathcal O_{E}}}({\rm IC}_{\{\mu\}})). 
$$
On the generic fiber we have the sheaf $f_{\eta, *}{\rm IC}_{\{\mu\}} \in {\rm Perv}_{L^+_zG_{E_0}}({\rm Gr}_{G, E_0}, \overline{\mathbb Q}_\ell)$.  Since $f$ is proper and defined over $E_0$, there is an isomorphism
$$
f_{\bar{s}, *} \,  \Psi_{{\rm Gr}_{\calG, \mathcal O_{E}}}({\rm IC}_{\{\mu\}})  =  \Psi_{{\rm Gr}_{\mathcal G, \mathcal O_{E_0}}}(f_{\eta, *}{\rm IC}_{\{\mu\}}) 
$$
in the category ${\rm Perv}_{L^+\mathcal G_{k_{E_0}}}( {\rm Gr}_{\calG, \mathcal O_{E_0}} \times_{s_{E_0}} \eta_{E_0}, \, \overline{\mathbb Q}_\ell)$. This yields
\begin{equation}\label{starting_pt2}
{\rm tr}^{\rm ss}(\Phi_E \, | \, \Psi_{{\rm Gr}_{\calG, \mathcal O_{E}}}({\rm IC}_{\{\mu\}})) = {\rm tr}^{\rm ss}(\Phi_{E_0} \, | \, \Psi_{{\rm Gr}_{\calG, \mathcal O_{E_0}}}( f_{\eta, *} {\rm IC}_{\{\mu\}})).
\end{equation}
Let ${\rm Sat}(V_{\{\mu\}})$ denote the perverse sheaf corresponding to $V_{\{\mu\}}$ under the geometric Satake equivalence given by Theorem \ref{GeoSat1}. Recall ${\rm Sat}(V_{\{\mu\}}) = {\rm IC}_{\{\mu\}}$ by Corollary \ref{GeoSat1_cor}.
Also, $f_{\eta, *} {\rm Sat}(V_{\{\mu\}}) \simeq {\rm Sat}(I(V_{\{\mu\}}))$ by Proposition \ref{GeoSat1_IndRes}. Therefore as functions on $M_{\{\mu\}}(k_E)$,
\begin{equation}\label{starting_pt3}
{\rm tr}^{\rm ss}(\Phi_E \, | \, \Psi_{{\rm Gr}_{\calG, \mathcal O_{E}}}({\rm IC}_{\{\mu\}})) = {\rm tr}^{\rm ss}(\Phi_{E_0} \, | \, \Psi_{{\rm Gr}_{\calG, \mathcal O_{E_0}}}({\rm Sat}(I(V_{\{\mu\}}))).
\end{equation}
By Theorem \ref{centralfunction}, the right hand side belongs to $\mathcal Z(G(E_0), \calG(\mathcal O_{E_0}))$. This explains how we view the left hand side also as an element of that algebra.  We remark that we really needed to pass to ${\rm Gr}_{\calG, \mathcal O_{E_0}}$ and not just to ${\rm Gr}_{\calG, \mathcal O_E}$, since $\calG_{\mathcal O_{E_0}}$ is a parahoric group scheme, whereas ${\calG}_{\mathcal O_E}$ might not be.

\subsubsection{The conjecture}

 By the above discussion, the Test Function Conjecture  announced in \cite{Hai14} (more precisely, the local model version) can be rephrased in the case of parahoric level as follows.

\begin{conject} \label{TFC_mu} Recall $d_\mu = {\rm dim} \, M_{\{\mu\}, E}$. As elements of $\mathcal Z(G(E_0), \mathcal G(\mathcal O_{E_0}))$,
$$
{\rm tr}^{\rm ss}\big(\Phi_E\, | \, \Psi_{M_{\{\mu\}}}({\rm IC}_{\{\mu\}})\big) \;=\;   (-1)^{d_\mu} \cdot z^{\rm ss}_{\{\mu\}}.
$$
\end{conject} 

Because of the nature of the proof which goes via a reduction to minimal Levi subgroups, we also need a more flexible version of this statement as follows.  

Now suppose $V$ is any irreducible algebraic representation of $^LG=\widehat{G}\rtimes \Ga_F$. Let $V_0$ be an irreducible constituent of $V|_{\widehat{G}}$. Then as $^LG$-representations $V = \sum_{\gamma \in \Gamma_F} \gamma(V_0)$. This means that the $\widehat{B}$-highest $\widehat{T}$-weights $\lambda$ appearing in $V|_{\widehat{G}}$ are all $\Gamma_F$-conjugate.  We will use the weights $\lambda$ to define a sign attached to $V$; but it will be convenient to phrase the definition in terms of $\lambda$ viewed as cocharacters of $G$.  To this end, we choose an $F$-rational maximal torus $T \subset G$ which is the centralizer of an $F$-rational maximal $\breve{F}$-split torus $S \subset G$ as in \eqref{groupchain}. We have abstractly an identification $X_*(T) = X^*(\widehat{T})$, which in general does not respect $\Gamma_F$-actions.  However, because of the careful choice of $T$, the two natural actions of $\gamma \in \Gamma_F$ differ from each other by the action of an element $w_\gamma \in W(G,S)(\breve{F})$ (this follows from the discussion around \cite[(11.2.2)]{HR10}).  Therefore, transporting the $\Gamma_F$-action on $\widehat{T}$ yields a twisted $\Gamma_F$-action on $X^*(T)$ which still permutes the sets of simple roots corresponding to the various Borel subgroups $B \supset T$.  Now we may view $\lambda \in X_*(T)$, and if $\rho_B$ denotes the half-sum of the $B$-positive $T$-roots, we define the parity
\begin{equation} \label{dV_defn}
d_V \defined \langle 2\rho_B, \lambda \rangle \,\,\mbox{mod $2$}.
\end{equation}
This is well-defined independent of the choice of $\lambda$  and $B$ (if $B'$ is another Borel subgroup containing $T$, then $\rho_{B}-\rho_{B'} \in X^*(T)$ (a sum of roots), hence $\langle \rho_B - \rho_{B'}, \lambda \rangle \in \mathbb Z$; if $\gamma \in \Gamma_F$, then $\langle 2\rho_B, \gamma\lambda \rangle = \langle 2\rho_{\gamma^{-1}B}, \lambda \rangle$), where $\rho_B \mapsto \rho_{\gamma^{-1}{B}}$ refers to the twisted action of $\gamma \in \Gamma_F$ mentioned above.

If $V$ is any algebraic (hence semi-simple) representation of $^LG$ , then we can write $V=V^+\oplus V^-$, where all irreducible constituents of $V^+$ (resp. $V^-$) have even (resp. odd) parity. Let ${\rm Sat}(V)=\Sat(V^+)\oplus\Sat(V^-)$ be the object in ${\rm Perv}_{L^+G}({\rm Gr}_{G}, \bar{\mathbb Q}_\ell)$ which corresponds to $V$ under the geometric Satake equivalence of Theorem \ref{GeoSat1}. We define as a function on $\Fl_\calG(k_F)$,
\begin{equation}\label{tau_dfn}
\tau_{\calG,V}^{\on{ss}}\defined \tr^{\on{ss}}(\Phi_F\,|\, \Psi_{\Gr_\calG}(\Sat(V^+)))- \tr^{\on{ss}}(\Phi_F\,|\, \Psi_{\Gr_\calG}(\Sat(V^-))).
\end{equation}
Note that if $V$ is irreducible, then
\begin{equation} \label{tau_sign_fact}
(-1)^{d_V} \tau^{\rm ss}_{\calG, V} = {\rm tr}^{\rm ss}(\Phi\, | \, \Psi_{\Gr_\calG}(\Sat(V))).
\end{equation}
We define $z^{\rm ss}_{\calG, V}$ to be the unique function in $\mathcal Z(G(F), \calG(\mathcal O_F))$ such that, if $\pi^{\calG(\mathcal O_F)} \neq 0$, then $z^{\rm ss}_{\calG, V}$ acts on $\pi^{\calG(\mathcal O_F)}$ by the scalar ${\rm tr}^{\rm ss}(s(\pi) \, | \, V^{1 \rtimes I_F})$. Note that this is consistent with our earlier notation: if $V \in \,^LG_E$, then the above sense of $z^{\rm ss}_{\mathcal G_{\mathcal O_{E_0}}, I(V)} \in \mathcal Z(G(E_0), \calG({\mathcal O}_{E_0}))$ coincides with what we defined in section \ref{Sat_trans_sec}.

We have the following theorem.

\begin{thm} \label{TFC_V} Let $V$ be an algebraic representation of $^LG$. Then as elements of $\mathcal Z(G(F), \mathcal G(\mathcal O_{F}))$,
\[
\tau^{\rm ss}_{\mathcal G, V}\; = \;  z^{\rm ss}_{\mathcal G, V}.
\]
\end{thm}

Theorem \ref{TFC_V} implies Conjecture \ref{TFC_mu} as follows. The statement of Conjecture \ref{TFC_mu} depends only on the data $G_{E_0}, \calG_{\mathcal O_{E_0}}$ and $V_{\{\mu\}} \in {\rm Rep}(\,^LG_E)$. Therefore, we may replace $E_0$ with $F$, i.e., we may assume $E_0=F$. By \eqref{starting_pt3} and the very definition of $z^{\rm ss}_{\mathcal G, V}$, the conjecture for $V_{\{\mu\}} \in {\rm Rep}(\,^LG_E)$ follows from Theorem \ref{TFC_V} for the induced representation $I(V_{\{\mu\}}) \in {\rm Rep}(\,^LG)$. It remains to discuss the sign given by the parity. Note that all irreducible constituents of $I(V_{\{\mu\}})$ have the same parity as $V_{\{\mu\}}$. Now if $V = V_{\{\mu\}}$, then $d_V \equiv d_{\mu}$. To prove this note that ${\rm Sat}(V_{\{\mu\}})$ is supported on ${\rm Gr}_{G_{\bar{F}}}^{\leq \{\mu\}}$, whose dimension is $\langle 2\rho_B, \mu\rangle$ by \cite[Prop 2.2, Cor 3.10]{Ri16a}. Hence, Conjecture \ref{TFC_mu} follows from Theorem \ref{TFC_V}.

Let us note that if $V\oplus W$ is a direct sum of algebraic representations of $^LG$, then $\tau^{\rm ss}_{\mathcal G, V\oplus W}=\tau^{\rm ss}_{\mathcal G, V}+\tau^{\rm ss}_{\mathcal G, W}$ and likewise, $z^{\rm ss}_{\mathcal G, V\oplus W}=z^{\rm ss}_{\mathcal G, V}+z^{\rm ss}_{\mathcal G, W}$. So the theorem for general $V$ follows from the theorem for irreducible $V$ which we will prove in the next sections.

\subsubsection{Notation} \label{Notation_TFC_sec} As the formulation of Theorem \ref{TFC_V} only makes reference to the field $F$, we will drop the subscript $F$ from the notation. In particular, $\calO=\calO_F$ is the ring of integers with residue field $k=k_F$. The Galois group is denoted $\Ga=\Ga_F$ with inertia subgroup $I=I_F$, and fixed geometric Frobenius $\Phi=\Phi_F$ etc. 

\subsubsection{More on irreducible representations of $^LG$}
The following lemma will be useful.

\begin{lem} \label{vanishing}
Suppose $V$ is an irreducible algebraic representation of $^LG$. Suppose that $V|_{\widehat{G} \rtimes I}$ is {\em not} irreducible.  Then $\tau^{\rm ss}_{\mathcal G, V} = z^{\rm ss}_{\mathcal G, V} = 0$.
\end{lem}

\begin{proof}
Suppose $V_0 \subsetneq V|_{\widehat{G} \rtimes {I}}$ is an irreducible constituent. Let $r$ be the order of $\Phi$ acting as an automorphism of $V$. By the irreducibility of $V$, $\sum_{i \in \mathbb Z} \Phi^iV_0 = V$.  Let $d \geq 1$ be minimal such that $\Phi^dV_0 = V_0$. Then $d > 1$ and $V = \oplus_{i=0}^{d-1} \Phi^iV_0$ as $^LG$-modules. Moreover $V^{I} = \oplus_{i=0}^{d-1} \Phi^i(V_0^{I})$.  Since $\Phi$ permutes these direct summands without fixed points, the trace of any $s(\pi) \in \widehat{G}^{I} \rtimes \Phi$ \ on $V^{I}$ is zero.  This proves that $z^{\rm ss}_{\mathcal G, V} = 0$. Similar reasoning applied to ${\rm Sat}(V)$ proves $\tau^{\rm ss}_{\mathcal G, V} = 0$.
\end{proof}

%%%% MAJOR COMMENT OUT BELOW

\subsection{Reduction to minimal $F$-Levi subgroups}  \label{Levi_red_subsec}
Let $M$ be a minimal\footnote{Everything is valid for general $F$-Levi subgroups, but we do not need it in the manuscript.} $F$-Levi subgroup of $G$. There is a choice of embedding $\widehat{M} \hookrightarrow \widehat{G}$ such that the canonical $\Gamma$-action on $\widehat{M}$ is inherited from the $\Gamma$-action on $\widehat{G}$ (\cite[Lem 2.1]{Hai17}). Fix this choice from now on.  The embedding extends canonically to an $L$-embedding $^LM = \widehat{M} \rtimes \Gamma \hookrightarrow \,^LG = \widehat{G} \rtimes \Gamma$.

The group $M(F) \cap \mathcal G(\mathcal O)$ is a parahoric subgroup of $M(F)$ (cf.\,\cite[Lem 4.1.1]{HR10}); let $\mathcal M$ be the associated parahoric group scheme over $\mathcal O$, which is endowed with a closed immersion of $\mathcal O$-group schemes $\mathcal M \hookto \mathcal G$.

Fix an irreducible algebraic representation $V$ of $^LG$, and let $V_M := V|_{\,^LM}$ denote its restriction to $^LM$.  Write $V_M = \oplus_W W^{\oplus m_W}$ where $W$ ranges over the irreducible representations of $^LM$ appearing in $V_M$, with multiplicity $m_W \in \bbZ_{>0}$.

\begin{lem} \label{C_lem} Recall the constant term homomorphism
$$
c^G_M \, : \, \mathcal Z(G(F), \calG(\calO)) \, \rightarrow \, \mathcal Z(M(F), \calM(\calO))
$$
defined in \cite[11.11]{Hai14}, cf.\,\textup{(}\ref{C_defn}\textup{)}. Then $c^G_M(z^{\rm ss}_{\mathcal G, V}) = z^{\rm ss}_{\mathcal M, V_M} = \sum_W m_W \, z^{\rm ss}_{\mathcal M, W}$.
\end{lem}

\begin{proof}
Let $\chi$ be any weakly unramified character of $M(F)$. By definition, $z:= z^{\rm ss}_{\mathcal G, V}$ acts on $i^G_P(\chi)^{\calG(\calO)}$ by the scalar ${\rm tr}(s(\chi) \, | \, V^{1 \rtimes I})$, where $s(\chi) \in [\widehat{G}^{I} \rtimes \Phi]_{\rm ss}/\widehat{G}^{I}$ is the Satake parameter associated to an irreducible representation of $G(F)$ with supercuspidal support $(M(F), \chi)_{G(F)}$; cf.\,\cite{Hai15}. By a property of the Bernstein isomorphism $ S \co \mathcal Z(G(F), \mathcal G(\mathcal O)) \overset{\sim}{\rightarrow} \bar{\mathbb Q}_\ell[\Lambda_M]^{W_{0}}$, $z$ acts by $S(z)(\chi)$, where $S(z)$ is viewed as a regular function on the quotient variety $Z(\widehat{M})^{I}_{\Phi}/W_{0}$ (cf.\,\cite[$\S$11.8]{Hai14}).  By the definition of $c^G_M$ (cf.\,\cite[(11.11.1]{Hai14}), $S(z)(\chi) = S(c^G_M(z))(\chi)$.  As above, this is also the scalar by which $c^G_M(z)$ acts on the 1-dimensional representation $\chi$, this time by a property of the Bernstein isomorphism $S$ for $M$ applied to $c^G_M(z)$.  

On the other hand, by definition $z^{\rm ss}_{\mathcal M, V_M}$ acts by the same scalar on $\chi$, namely by
$$
{\rm tr}(s^M(\chi) \, |\, V_M^{1 \rtimes I}) = {\rm tr}(s(\chi) \, |\, V^{1 \rtimes I}).
$$
Here $s^M(\chi)$ denotes the Satake parameter for the group $M$ instead of $G$, but clearly by the construction of Satake parameters in \cite[(9.1)]{Hai15}, $s^M(\chi) = s(\chi)$. This justifies the equality displayed above, and thus the fact that $c^G_M(z)$ and $z^{\rm ss}_{\mathcal M, V_M}$ act by the same scalar on $\chi$. These remarks imply the lemma.
 \end{proof}

\begin{lem} \label{pC_lem}
Recall the normalized variant $^pc^G_M$ of the constant term homomorphism $c^G_M$, defined in \textup{(}\ref{pC_defn}\textup{)}. Then
\begin{equation} \label{pC_on_tau}
^pc^G_M((-1)^{d_V}\tau^{\rm ss}_{\calG, V}) = \sum_{W} m_W \cdot (-1)^{d_W} \cdot \tau^{\rm ss}_{\calM, W}.
\end{equation}
\end{lem}
\begin{proof}
By definition \eqref{tau_dfn}, we have
\[
(-1)^{d_V}\cdot\tau^{\rm ss}_{\calG, V}\;=\; \tr^{\on{ss}}(\Phi\,|\, \Psi_{\Gr_\calG}(\Sat(V))).
\]
The lemma follows immediately from Lemma \ref{CT_RPsi_comp} ii), Theorem \ref{commctnearbythm} and Theorem \ref{GeoSat2}. 
\end{proof}

Next we must unwind what (\ref{pC_on_tau}) means. By Lemma \ref{vanishing}, we may assume $W$ ranges only over the $W$ such that $W|_{\widehat{M} \rtimes I}$ is irreducible. Suppose for such a $W$ that $W|_{\widehat{M}}$ has $\widehat{B}$-highest $\widehat{T}$-weights $\lambda_1, \dots, \lambda_n$; view $\lambda_i \in X^*(\widehat{T})$ and give the latter the $\Gamma$-action coming from that on $(\widehat{G}, \widehat{B}, \widehat{T})$.  Because $W|_{\widehat{M} \rtimes I}$ is irreducible, the $\lambda_i$ are $I$-conjugate. 

Now we view $\lambda_i \in X_*(T)$, for an $F$-rational maximal torus $T \subset M$ chosen carefully as above (\ref{dV_defn}). We will consider the images $\bar{\lambda}_i \in \pi_1(M)_{I}$. The natural $\Gamma$-action on $X_*(T)$ is not compatible with that on $X^*(\widehat{T})$; rather, the latter is compatible with the Galois action on $X_*(T^*)$, for an $F$-rational maximal torus $T^* \subset M^*$ in a quasi-split $F$-inner form $M^*$ of $M$. However, $\pi_1(M) = \pi_1(M^*)$ as $\Gamma$-modules, so there is no ambiguity and we can conclude that the $\bar{\lambda}_i \in \pi_1(M)_{I}$ are $I$-conjugate. Therefore they all coincide. Since $\Phi$ permutes the set $\{ \lambda_i \}$, this common image belongs to $(\pi_1(M)_{I})^{\Phi}$.

Using the surjectivity of the Kottwitz homomorphism $\kappa_{M}(F) : M(F) \twoheadrightarrow \pi_1(M)_{I}^{\Phi}$, we may choose $m \in M(F)$ mapping to the element $\bar{\lambda}_i$. In previous notation, we let $\bar{v}_m \in \pi_1(M)_{I}^{\Phi}$ denote that image and we let $\dot{v}_m \in X_*(T)$ be an arbitrary lift of $\bar{v}_m$.

 As $\widehat{M}$-representations, $W|_{\widehat{M}} = \oplus_i W_{\lambda_i}^{\oplus m_i}$ for certain multiplicities $m_i \in \mathbb N$. Therefore the sheaf ${\rm Sat}(W)$ is supported on the union of the generic fibers of the local models $M_{\{\lambda_i\}, \sF}$. Since nearby cycles are supported on the closure of the generic fiber, the function $\tau^{\rm ss}_{\calM, W}$ is supported on the fiber of the Kottwitz homomorphism $M(F) \rightarrow \pi_1(M)_{I}^{\Phi}$ over the common image of the elements $\bar{\lambda}_i \in \pi_1(M)_{I}^{\Phi}$, in other words, on the fiber containing $m \in M(F)$ above.

Finally we can relate $d_V, d_W,$ and $\dot{v}_m$ as follows. Let $B$ be any $\sF$-rational Borel subgroup in $P^+ = MN^+$ containing $N^+$ and $T$, and let $B_M = B \cap M$. 

\begin{lem}\label{vm} Assume $W|_{\widehat{M} \rtimes I}$ is irreducible and let $\lambda_i$ and $m$ be constructed as above. We have for each $i$,
$$\langle 2\rho_{N^+}, \dot{v}_m \rangle \equiv  \langle 2\rho_{N^+}, \lambda_i \rangle \equiv d_V + d_W \,\, \mbox{{\rm mod} $2$}.$$
\end{lem}

\begin{proof}
By construction, each $\lambda_i \in X_*(T)$ is a lift of $\bar{v}_m \in \pi_1(M)_{I}^{\Phi}$ so can be taken to be $\dot{v}_m$.  But then we also have
\begin{align*}
\langle 2\rho_{N^+}, \lambda_i \rangle &= \langle 2\rho_B, \lambda_i \rangle - \langle 2\rho_{B_M}, \lambda_i \rangle \\
&\equiv \langle 2\rho_B, \mu \rangle + \langle 2\rho_{B_M}, \lambda_i \rangle \\
&\equiv d_V + d_W \,\,\, \mbox{mod $2$}.
\end{align*}
Here $\mu$ denotes one of the highest weights appearing in $V|_{\widehat{G}}$, and we can assume that $\lambda_i$ lies in the weight space for the corresponding representation of $\widehat{G}$. We have used that $\mu - \lambda_i$ is a sum of coroots for $G$, hence $\langle \rho_B, \mu - \lambda_i \rangle \in \mathbb Z$.
\end{proof}

For each $W$ such that $W|_{\widehat{M} \rtimes I}$ is irreducible (which we may assume by Lemma \ref{vanishing}), let $m \in M(F)$ be chosen as above. Using Lemma \ref{vm}, we deduce from (\ref{pC_on_tau}) that 
\begin{align}
c^G_M((-1)^{d_V}\tau^{\rm ss}_{\mathcal G, V}) &= \sum_{W} m_W \, \tau^{\rm ss}_{\mathcal M, W} \cdot (-1)^{d_W} \cdot (-1)^{\langle 2 \rho_{N^+}, \dot{v}_m \rangle} \notag\\
&= \sum_W m_W \, \tau^{\rm ss}_{\mathcal M, W} \cdot (-1)^{d_W} \cdot (-1)^{d_V + d_W}. \label{dVdW}
\end{align}
We have used for the first equality that $\tau^{\rm ss}_{\mathcal M, W}$ is supported on the fiber containing $m$.

Now suppose Theorem \ref{TFC_V} holds for every $W$. Then using Lemma \ref{C_lem} and (\ref{dVdW}), we see
\[
c^G_M\big(\tau^{\rm ss}_{\mathcal G, V}\big) \,=\,  \sum_W m_W \, \tau^{\rm ss}_{\mathcal M, W} \,=\,  \sum_W m_W \, z^{\rm ss}_{\mathcal M, W} \,=\, c^G_M\big( z^{\rm ss}_{\mathcal G, V} \big).
\]
Since $c^G_M$ is injective, we conclude $\tau^{\rm ss}_{\mathcal G, V} = z^{\rm ss}_{\mathcal G, V}$. Therefore to prove Theorem \ref{TFC_V} for $G$ it is enough to prove it for a minimal $F$-Levi subgroup of $G$.  

We close this section with a definition made possible by the above arguments. Note that when $V|_{\widehat{G} \rtimes I}$ is irreducible, then the $\widehat{B}$-highest $\widehat{T}$-weights appearing in $V|_{\widehat{G}}$ are $I$-conjugate.

\begin{dfn} \label{omega_V_def}
For $V \in {\rm Rep}(\,^LG)$ such that $V|_{\widehat{G} \rtimes I}$ is irreducible, let $\omega_V \in \Omega_{\bf a} \cong \pi_1(G)_{I}^{\Phi}$ be the common image of the $\lambda_i \in X_*(T)$ appearing as $\widehat{B}$-highest $\widehat{T}$-weights in $V|_{\widehat{G}}$. 
\end{dfn}

\subsection{Reduction from anisotropic mod center groups to quasi-split groups}\label{reduction_aniso_mod_center}

By section \ref{Levi_red_subsec}, Theorem \ref{TFC_V} holds for $(\mathcal G, V)$ if it holds for $(\mathcal M, V_{M})$.  Therefore, it is enough to prove Theorem \ref{TFC_V} when $G$ is $F$-anisotropic mod center. Assume this. Let $\mathcal G$ be the unique parahoric $\mathcal O_F$-group scheme with generic fiber $G$.

Let $G^*$ be a quasi-split inner form of $G$ over $F$. Let $\calG^*$ be any parahoric $\mathcal O_F$-group scheme for $G^*$. We fix once and for all an inner twisting $G \rightarrow G^*$ as in \cite{Hai14} which is needed to define the normalized transfer homomorphism \cite[11.12]{Hai14} 
$$\tilde{t}: \mathcal Z(G^*(F),\mathcal G^*(\mathcal O)) \rightarrow \mathcal Z(G(F), \calG(\mathcal O)).$$ This special choice of twist induces an isomorphism $G_{\breve F} \overset{\sim}{\rightarrow} G^*_{\breve F}$; hence we may assume $G$ and $G^*$ are the same group over $\breve{F}$, endowed with different actions $\Phi$ and $\Phi^*$ of the geometric Frobenius element. Similar comments apply to ${\rm Gr}_{G}$ and ${\rm Gr}_{G^*}$. 

\begin{lem} \label{connected_trivial}
For any representation $V$ of $^LG = \,^LG^*$, we have
$$
{\rm tr}^{\rm ss}(\Phi \,  | \, {\mathbb H}^*({\rm Gr}_{G, \bar{F}}, {\rm Sat}(V))) = {\rm tr}^{\rm ss}(\Phi^* \,  | \, {\mathbb H}^*({\rm Gr}_{G^*, \bar{F}}, {\rm Sat}(V))).
$$
\end{lem}

\begin{proof}
This follows from Lemma \ref{homotopy_lem} as in the proof of Corollary \ref{GeoSat1_cor}. Note that $\Sat(V)_\sF$ is a direct sum of intersection complexes. 
\end{proof}

Fix an irreducible $^LG$-representation $V$ such that $V|_{\widehat{G}\rtimes I}$ is irreducible, and define $\omega_V$ as in Definition \ref{omega_V_def}. Let $M_V$ be the closed $\calO$-subscheme of $\Gr_\calG$ which is given by the scheme theoretic closure of the support of $\Sat(V)$ considered as a sheaf on $\Gr_G$. Likewise, denote by $M^*_V$ the closed $\calO$-subscheme of $\Gr_{\calG^*}$ given by $\Sat(V)$ considered as a sheaf on $\Gr_{G^*}$. Since $V|_{\widehat{G} \rtimes I}$ is irreducible, the special fiber $M_{V,k}$ (resp. $M_{V,k}^*$) is geometrically connected. To check this we use the fact that ${\rm Sat}(V|_{\widehat{G} \rtimes I}) = (\oplus_{\gamma \in I/I_{\{\lambda\}}}{\rm IC}_{\gamma \cdot \{\lambda\}}) \otimes \mathcal L$ for some $\lambda \in X_*(T)$ and an irreducible local system $\mathcal L$ on ${\rm Spec}({\breve F})$ (cf.\,(\ref{satobjects})). This reflects the fact that the highest $\widehat{T}$-weights appearing in $V|_{\widehat{G}}$ are $I$-conjugate. Then we have on reduced loci $M_{V, \bar{k}} = {\rm Supp}(\Psi({\rm Sat}(V|_{\widehat{G} \rtimes I}))) = M_{\gamma \cdot \{\lambda\}, \bar{k}}$, for any $\gamma \in I$.

\begin{lem} \label{single_point}
If $k$ denotes the residue field, then $M_V(k) = \{x_V\}$, i.e., the $\calO$-scheme $M_V$ has a single $k$-rational point.
\end{lem}
\begin{proof}
For convenience, in the mixed characteristic case, write $(G, \calG)$ in place of the function field analogues $(G', \calG')$ of section \ref{PZBDGrass_Sec}, and assume $F =  k(\!(t)\!)$. Suppose $y \in \Fl_\calG(\bar{k}) = G(\breve{F})/\calG(\breve \calO)$ is a $\Phi$-fixed point in $M_V(\bar{k})$. Then it belongs to an $L^+\calG(\bar{k})$-orbit fixed by $\Phi$, which by Lemma \ref{Bruhatflag} is indexed by an element of the set $W_{\bf f}\backslash W/W_{\bf f}$.  But $W = \breve{W}^\Phi = [\breve{W}_{\rm sc} \rtimes \breve{\Omega}_{\breve {\bf a}}]^\Phi = \Omega_{\bf a}$, by e.g.\,\cite[Lem 3.0.1 (III)]{HR10}. On the other hand since $M_{V, k}$ is geometrically connected, it meets only the connected component containing the image $x_V\in \Fl_{\calG}(k)$ of the element $\omega_V  \in \Omega_{\bf a}$. 
Therefore $y=x_V$ is that image. 
Conversely, since $\{x_V\}$ is the unique closed Iwahori orbit in this connected component, it must be contained in $M_{V,k}$. 
This shows $M_V(k)=M_V(\bar k)^\Phi = \{x_V\}$.
\end{proof}

From the Grothendieck-Lefschetz fixed point formula combined with Lemmas \ref{connected_trivial}, \ref{single_point}, we obtain the following equalities
\begin{align*}
\tau^{\rm ss}_{\mathcal G, V}(x_V) &:= (-1)^{d_V}{\rm tr}^{\rm ss}\left(\Phi \, | \, \Psi_{{\rm Gr}_{\calG}}({\rm Sat}(V))_{\bar{x}_V}\right) \\
&= (-1)^{d_V} \tr^{\on{ss}}\left(\Phi\, |\, \bbH^*(\Fl_{\calG,\bar{k}},\Psi_{\Gr_\calG}(\Sat(V)))\right)\\
&=   (-1)^{d_V}{\rm tr}^{\rm ss}\left(\Phi\, |\, {\mathbb H}^*({\rm Gr}_{G, \bar{F}}, {\rm Sat}(V))\right)  \\
&=  (-1)^{d_V} {\rm tr}^{\rm ss}\left(\Phi^* \, | \, {\mathbb H}^*({\rm Gr}_{G^*, \bar{F}}, {\rm Sat}(V))\right),
\end{align*}

Now we turn to the function $z^{\rm ss}_{\mathcal G, V}$. By construction of the functions and of $\tilde{t}$, we have the identity
\begin{equation} \label{t(z)}
\tilde{t}(z^{\rm ss}_{\mathcal G^*, V}) = z^{\rm ss}_{\mathcal G, V}.
\end{equation}
Since $G$ is anisotropic mod center, a basic property of $\tilde{t}$ is that $\tilde{t}(z^*)(\omega)$ for $\omega \in \Omega_{\bf a}$ is calculated by summing the values of $z^*$ over the preimage of $\omega$ under the Kottwitz homomorphism for $G^*$ (see \cite[Prop 11.12.6]{Hai14}). We are assuming that $z^{\rm ss}_{\calG^*, V} =  \tau^{\rm ss}_{\calG^*, V}$, and we know that this function is supported on the connected component indexed by $\omega_V$ (as above Lemma \ref{single_point}, $M^*_{V, k}$ is geometrically connected). Thus as $G$ is anisotropic mod center
$$z^{\rm ss}_{\mathcal G, V} = C \cdot \mathds{1}_{\omega_V},$$ a function supported on the single (double) coset indexed by $\omega_V$. Therefore, assuming Theorem \ref{TFC_V} holds for $(\calG^*, V)$, we obtain
$$
C = \sum_{x\in M_V^*(k)} z^{\rm ss}_{\mathcal G^*, V}(x) = (-1)^{d_V} \sum_{x \in M_V^*(k)} {\rm tr}^{\rm ss}(\Phi^* \, | \, \Psi_{{\rm Gr}_{\calG^*}}({\rm Sat}(V))_{\bar{x}}).
$$
Again by the Grothendieck-Lefschetz fixed point formula, we see 
$$
C =   (-1)^{d_V} {\rm tr}^{\rm ss}(\Phi^*\, |\, {\mathbb H}^*({\rm Gr}_{G^*, \bar{F}}, {\rm Sat}(V))) = \tau^{\rm ss}_{\mathcal G, V}(x_V)
$$
Of course $x_V$ is the point corresponding to $\omega_V$. This yields Theorem \ref{TFC_V} for $(\calG, V)$.

\begin{ex}
Let $G=D^\times$ where $D$ is a central division algebra over $F$ as in the proof of Lemma \ref{single_point}. Let $V=V_{\{\mu\}}$ with $\mu=(1,0,\ldots,0)$. Then $M^*_{\{\mu\},F}=\bbP^{n-1}_F$, and $M_{\{\mu\},F}$ is the Severi-Brauer form associated with $D$. Then
\[
\tau_{\{\mu\}}^{\on{ss}}(x_V)\;=\;(-1)^{n-1}\;\tr(\Phi\;|\; \bbH^*(\bbP^{n-1}_\sF,\algQl\lan n-1\ran))\;=\;q^{-\nicefrac{(n-1)}{2}}(1+q+\ldots +q^{n-1}),
\]
which is the trace of the Satake parameter of the trivial representation $\pi=1_{D^\times}$ on $V$. 
\end{ex}

\subsection{Proof in the quasi-split case}\label{quasi_split_sec}

Now we assume $G$ is quasi-split over $F$, so that its minimal $F$-Levi subgroup is an $F$-torus $T$. We may run our reduction steps again.  By section \ref{Levi_red_subsec}, to prove Theorem \ref{TFC_V} for $(G, \mathcal G)$ (and any irreducible representation $V$ of $^LG$), it is enough to prove it for $(T, \mathcal T)$, where $\mathcal T$ is the unique parahoric $\mathcal O$-group scheme with generic fiber $T$. Let $V$ be a representation of $^LT$ such that $V|_{\widehat{T}\rtimes I}$ is irreducible. Tori are anisotropic modulo center, and by the reasoning of \S \ref{reduction_aniso_mod_center}, it remains to show that
\begin{equation}\label{proof_tori}
z^{\on{ss}}_{\mathcal T,V}(\om_V)\,=\, \tr^{\on{ss}}(\Phi\,|\, \bbH^*(\Gr_{T,\sF},\Sat(V))),
\end{equation}
where we use $d_V=0$ because $\dim(\Gr_{T,\sF})=0$. We have $\bbH^*(\Gr_{T,\sF},\Sat(V))=\bbH^0(\Gr_{T,\sF},\Sat(V))= V$ as $^LT$-representations under the geometric Satake isomorphism. This gives
\[
\tr^{\on{ss}}(\Phi\,|\, \bbH^*(\Gr_{T,\sF},\Sat(V)))\,=\,\tr(\Phi\,|\, V^{1 \rtimes I}),
\] 
which is $z^{\on{ss}}_{\mathcal T, V}(\om_V)$ by the definition of $z^{\on{ss}}_{\mathcal T, V}$. Explicitly, the representation $V^{1\rtimes I}$ has a single $\widehat{T}^I$-weight $\bar{\lambda}$, which identifies with $\omega_V \in X^*(\widehat{T}^I)^\Phi = X_*(T)_I^\Phi \cong T(F)/T(F)_1$. 
The function $z^{\rm ss}_{\mathcal T, V}$ acts on a weakly unramified character $\chi: T(F)/T(F)_1 \rightarrow \bar{\mathbb Q}_\ell^\times$ (i.e.,\,$\chi \in \widehat{T}$) by the scalar
$$
{\rm tr}(s(\chi) \, | \, V^{1 \rtimes I}) = {\rm tr}(\chi \rtimes \Phi \, | \, V^{1 \rtimes I}) = \chi(\omega_V) \, {\rm tr}(\Phi \, | \, V^{1 \rtimes I}),
$$
hence $z^{\rm ss}_{\mathcal T, V} = {\rm tr}(\Phi \, | \, V^{1 \rtimes I})  \, \mathds{1}_{\omega_V}$.

This implies the Main Theorem from the introduction in the case of quasi-split groups, and by the preceding reductions in full generality.

\subsection{On values of the test functions} \label{values_sec}

Recall $q = p^m$ is the cardinality of the residue field of $F$ and $q_{E_0} = q^{[E_0: F]}$. Recall the quasi-split inner form $G^*$ with its usual data $A^*, S^* = T^*, M^*, B^*, W^*_0, \calG^*$ parallel to the data $A, S, T, M, P, W_0,\calG$ for $G$. We may assume $\calG^*$ is an Iwahori group scheme.

The objects $\widehat{G}$ and $V_{\{\mu\}}|_{\widehat{G}}$ can be defined over $\mathbb Q$. In addition, the $\Gamma$-action on $\widehat{G}$ can be defined over $\mathbb Q$, as can the full representation $V_{\{\mu\}}$ of $^LG$.

\begin{thm} \label{indep_thm}
The function $q_{E_0}^{d_\mu/2}z^{\rm ss}_{\calG, \{\mu\}}$ takes values in $\mathbb Z$, and it is independent of the choice of $\ell \neq p$ and $q^{1/2} \in \bar{\mathbb Q}_\ell$.
\end{thm}

\begin{proof} We may reduce to $E_0 = F$, so $E/F$ is a totally ramified extension. It is enough to consider $\calG$ an Iwahori group scheme. In the following we will use freely the notation of \cite{Hai}. Write $V_\mu = V_{\{\mu\}}$ for a representative $\mu \in X_*(T^*)$.

\begin{lem} The function $z^{\rm ss}_{\calG^*, \{\mu\}} \in \calZ(G^*(F), \calG^*(\mathcal O))$ is a $\mathbb Z$-linear combination of Bernstein elements  $z_{\bar{\nu}^*}$ where $\bar{\nu}^* \in {\mathcal Wt}(\bar{\mu})^{+,\Phi^*}$,
\begin{equation} \label{qs_exp}
z^{\rm ss}_{\calG^*, \{\mu\}} = \sum_{\bar{\nu}^*} a^*_{\bar{\nu}^*, \{\mu\}} \, z_{\bar{\nu}^*}.
\end{equation}
\end{lem} 
\begin{proof} First consider the case where $E = F$, so that $I(V_{\{\mu\}}) = V_{\{\mu\}}$.
By \cite[Thm 7.5, 7.11]{Hai}, we may write
\begin{equation} \label{E_0=E_case}
z^{\rm ss}_{\calG^*, \{\mu\}}  = z^{\rm ss}_{\calG^*, V_{\{\mu\}}} = \sum_{\bar{\lambda} \in {\mathcal Wt}(\bar{\mu})^{+,\Phi^*}} {\rm tr}(\Phi \, | \, {\mathbb H}_{\mu}(\bar{\lambda})) \, \sum_{\bar{\nu} \in {\mathcal Wt}(\bar{\lambda})^{+,\Phi^*}} P_{w_{\bar{\nu}}, w_{\bar{\lambda}}}(1) \, z_{\bar{\nu}},
\end{equation}
where ${\mathbb H}_\mu(\bar{\lambda})$ is the space of ``vectors with highest weight $\bar{\lambda}$'' in $V^{1 \rtimes I}_\mu|_{\widehat{G}^I}$. By construction $P_{w_{\bar{\nu}}, w_{\bar{\lambda}}}(1) \in \mathbb Z$. Since $\Phi$ has finite order $n$ in $\widehat{G} \rtimes \Gamma_{F'/F}$ (if $F'/F$ is a finite extension splitting $G$) and stabilizes ${\mathbb H}_\mu(\bar{\lambda})$, we see ${\rm tr}(\Phi \, | \, {\mathbb H}_\mu(\bar{\lambda})) \in \mathbb Z[\zeta_n]$, for a primitive $n$-th root of unity $\zeta_n \in \bar{\mathbb Q}_\ell$. On the other hand this trace belongs to $\mathbb Q$ because $V^{1 \rtimes I}_\mu$ is defined over $\mathbb Q$; hence ${\rm tr}(\Phi \, | \, {\mathbb H}_\mu(\bar{\lambda})) \in \mathbb Z$. 

Now consider the general case where $E \supseteq F$ is totally ramified.  Now $z^{\rm ss}_{\calG^*, \{\mu\}}$ is the function in $\mathcal Z(G^*(F), \calG^*(\mathcal O_{F}))$ which is the regular function on the variety $(\widehat{T^*}^{I_{F}})_{\Phi^*}/W^*_{F}$ sending the weakly unramified character $\chi \in (\widehat{T^*}^{I_{F}})_{\Phi^*}$ to 
$$
{\rm tr}( \chi \rtimes \Phi^* \, | \, I(V_{\{\mu\}})^{1 \rtimes I_{F}}) = {\rm tr}(\chi \rtimes  \Phi^* \, | \, V_{\{\mu\}}^{1 \rtimes I_E}).
$$
The same argument which produced (\ref{E_0=E_case}) shows that this function takes the form of (\ref{E_0=E_case}), except that ${\mathbb H}_\mu(\bar{\lambda})$, ${\mathcal Wt}(\bar{\lambda})$ and ${\mathcal Wt}(\bar{\mu})$ are replaced by their analogues ${\mathbb H}_{\mu, E}(\bar{\lambda})$, ${\mathcal Wt}(\bar{\lambda})_E$ and ${\mathcal Wt}(\bar{\mu})_E$ for the group $G_E$, and $z_{\bar{\nu}}$ is replaced with the sum
$$
\sum_{\bar{\nu}'} z_{\bar{\nu}^*}.
$$
Here $\bar{\nu}'$ ranges over a set of representatives for the $W^*_{F}$-orbits contained in $W^*_{E} \cdot \bar{\nu} \subset X^*(\widehat{T^*}^{I_{E}})$, and $\bar{\nu}^*$ denotes the image of $\bar{\nu}'$ in $X^*(\widehat{T}^{I_{F}})$. Also, note that the restriction map $X^*(\widehat{T^*}^{I_{E}}) \rightarrow X^*(\widehat{T^*}^{I_{F}})$ sends the \'echelonnage coroots for $G_{\breve{E}}$ to those for $G_{\breve{F}}$, by \cite[Thm 6.1]{Hai}. Therefore restriction sends ${\mathcal Wt}(\bar{\lambda})_E$ to ${\mathcal Wt}(\bar{\lambda})$. This concludes the proof of the lemma.
\end{proof}

To pass to general groups $G$, we use the normalized transfer homomorphism $\tilde{t}$. Recall \cite[$\S11.12$]{Hai14} that $\tilde{t}$ is canonical but its construction uses a choice of a triple $(B^*, P, \psi)$ where $\psi: G \rightarrow G^*$ is an inner twisting compatible with $B^*$ and $P$ in a certain sense.

We know that $\tilde{t}(z^{\rm ss}_{\calG^*, \{\mu\}}) = z^{\rm ss}_{\calG, \{\mu\}}$. We need to express 
$\tilde{t}(z_{\bar{\nu}^*})$ in an explicit way. To this end, recall following \cite[(11.12.1)]{Hai14} that our choice of inner twisting $G \rightarrow G^*$ induces an inner twisting $M \rightarrow M^*$ and a surjective homomorphism $t_{A^*, A}$:
$$
\xymatrix{
T^*(F)/T^*(F)_1 \ar@{=}[d] \ar[r] & M^*(F)/M^*(F)_1 \ar@{=}[d] \ar[r]^{\sim} & M(F)/M(F)_1 \ar@{=}[d] \\
X^*(\widehat{T^*})^{\Phi^*}_{I} \ar[r] & X^*(Z(\widehat{M^*}))^{\Phi^*}_I \ar[r]^{\sim} & X^*(Z(\widehat{M}))^\Phi_I
}
$$
We also have the normalized version on the level of group algebras defined in \cite[Lem 11.12.4]{Hai14}
$$
\tilde{t}_{A^*, A} \, : \, \bar{\mathbb Q}_\ell[\Lambda_{T^*}]^{W^*_0} \, \longrightarrow \, \bar{\mathbb Q}_\ell[\Lambda_M]^{W_0},
$$
which induces $\tilde{t}: \mathcal Z(G^*(F), \mathcal G^*(\mathcal O)) \rightarrow \mathcal Z(G(F), \calG(\mathcal O))$ via the Bernstein isomorphisms.

We fix $\bar{\nu}^* \in X_*(T^*)_I^{+, \Phi^*} = T^*(F)/T^*(F)_1 = \Lambda_{T^*}$. We form the ``monomial'' sum 
$$
{\rm mon}_{\bar{\nu}^*} = \sum_{t^* \in W^*_0(\bar{\nu}^*)} \mathds{1}_{t^*} \in \bar{\mathbb Q}_\ell[\Lambda_{T^*}]^{W^*_0}.
$$
We need to compute $\tilde{t}_{A^*, A}({\rm mon}_{\bar{\nu}^*}) \in \bar{\mathbb Q}_\ell[\Lambda_M]^{W_0}$. By definition,
$$
\tilde{t}_{A^*, A}({\rm mon}_{\bar{\nu}^*}) = \sum_{m} \Big(\sum_{t^* \mapsto m} \delta_{B^*}^{-1/2}(t^*) \, \delta^{1/2}_P(m)\Big) \, \mathds{1}_m  
$$
where $m \in M(F)/M(F)_1 = \Lambda_M$, and $t^* \mapsto m$ means $t^* \in W^*_0(\bar{\nu}^*)$ and $t_{A^*,A}(t^*) = m$. The proof of \cite[Lem 11.12.4]{Hai14} shows the set of $m$ in the support of $\tilde{t}_{A^*,A}({\rm mon}_{\bar{\nu}^*})$ is $W_0$-stable, and also the terms
$$
c_m = \sum_{t^* \mapsto m} \delta_{B^*}^{-1/2}(t^*) \, \delta^{1/2}_P(m) \hspace{.3in}(\in \mathbb Z[q^{1/2}, q^{-1/2}])
$$
are independent of $m \in W_0(m_0)$ for $m_0$ fixed, and are even independent of the choice of compatible triple $(B^*, P, \psi)$. Let $m_0$ range over a set of representatives for the $W_0$-action on the set of $m$ appearing above. We obtain:
\begin{equation}\label{tilde_1}
\tilde{t}_{A^*, A}({\rm mon}_{\bar{\nu}^*}) = \sum_{m_0} c_{m_0} \sum_{m \in W_0(m_0)}  \mathds{1}_m  =: \sum_{m_0} c_{m_0} \, {\rm mon}_{m_0}~.
\end{equation}
Let $\theta_P : \bar{\mathbb Q}_\ell[\Lambda_M]^{W_0} \overset{\sim}{\rightarrow} \mathcal Z(G(F), \calG(\mathcal O))$ be the Bernstein isomorphism described in \cite[Thm 11.10.1]{Hai14}. By definition, $z_{\bar{\nu}^*} := \theta_{B^*}({\rm mon}_{\bar{\nu}^*})$ and $z_{m_0} := \theta_P({\rm mon}_{m_0})$.  Therefore as  $\tilde{t}$ is induced by $\tilde{t}_{A^*, A}$ via $\theta_P$ and $\theta_{B^*}$, we have an explicit understanding of how $\tilde{t}$ behaves:

\begin{lem} \label{tilde{t}_on_Bern}
In the notation above, $\tilde{t}(z_{\bar{\nu}^*}) = \sum_{m_0} c_{m_0} \, z_{m_0}$.
\end{lem}
Therefore, applying $\tilde{t}$ to (\ref{qs_exp}), we need only to prove the following lemma.
\begin{lem} \label{final_lem} For each $m_0$ as above, the element $q^{d_\mu/2} c_{m_0} \, z_{m_0}$ takes values in $\mathbb Z$, independent of $\ell$ and the choice of $q^{1/2} \in \bar{\mathbb Q}_\ell$.
\end{lem}

\begin{proof}
We fix $P, B^*$ used to define the inner twisting $G \rightarrow G^*$, in such a way that $B^*$- and $P$-positive roots in $X^*(T^*)$ take negative values on the alcove $\breve{\bf a}$ which determines $\calG$ (recall our conventions for the Bruhat order in section \ref{Sch_var_sec}). The explicit formula for $\theta_{P}$ (cf.\,\cite[11.8+refs therein]{Hai14}) is the following:
\begin{equation} \label{BK_defn}
\theta_P({\rm mon}_{m_0}) =\sum_{m \in W_0(m_0)} \theta_{P}(\mathds{1}_m) = \sum_{m \in W_0(m_0)} \delta_P^{1/2}(m) \, T_{m_1} T_{m_2}^{-1},
\end{equation}
where $m = m_1 m_2^{-1}$ for any choice of $P$-dominant elements $m_i \in M(F)$. Here, we say $m \in M(F)$ is $P$-dominant if this property holds for the corresponding element $v_m \in \Lambda_M \subset W = \Lambda_M \rtimes W_0$.

By the multiplication in the Hecke algebra $\mathcal H(G(F), \calG(\mathcal O))$, the function $T_{m_1}T_{m_2}^{-1}$ takes values in $\mathbb Z[q,q^{-1}]$ and is independent of $\ell$.  Hence it suffices to show that negative powers of $q^{1/2}$ do not appear when $q^{d_\mu/2} c_{m_0}$ times (\ref{BK_defn}) is expressed as a linear combination of the basis elements $T_w$.

\medskip

\noindent {\bf Claim 1:} Recall $\mathcal H(G(F), \calG(\mathcal O))$ is the specialization at $v = q^{1/2}$ of an affine Hecke algebra with parameter system $L$, coming from the based root data $(\Lambda_M, \Sigma_0, \Pi_0)$, where $\Sigma_0$ is the \'echelonnage root system for $G$ (cf.\,e.g.\,\cite{Ro15}). For a simple affine reflection $s \in W_{\rm aff}$, write $q_s := q^{L(s)}$. If $v \in W$ has reduced expession $v = s_1 \cdots s_l \omega_v$ ($\omega_v \in \Omega_{\bf a}$), then define $l_L(v)$ by the equality
$$
l_L(v) = \sum_{i=1}^l L(s_i).
$$
(This is well-defined by the fact that one can get from a reduced expression for $v$ to any other by a sequence of braid relations, cf. \cite{Ti69}.) Then for $v \in W$, $l_L(v) = \breve{l}(v)$, where $\breve{l}: \breve{W} \rightarrow \mathbb Z_{\geq 0}$ is the length function on $\breve{W}$. 

\medskip

{\em Proof of Claim 1:}  
In light of the equality $q^{L(s)} = \#I\dot{s}I/I$ (cf.\,e.g.\,\cite{Ro15}), Claim 1 is \cite[Prop 1.11; Rmk 1.13 ii)]{Ri16b}.

\medskip

\noindent {\bf Claim 2:}
Then for $m  = \dot{v}_m \in M(F)$ corresponding to $v_m \in \Lambda_M$ which is $P$-dominant, we have 
$$
\delta_P(m) = q^{-l_L(v_m)}.
$$

\medskip

{\em Proof of Claim 2:} The element $v_m$ is straight: we have
$$
(I \dot{v}_m I)^c= I \dot{v}_m^c I,
$$
and consequently $l_L(v_m^c) = c \cdot l_L(v_m)$, for all $c \in \mathbb Z_{\geq 0}$. This can be proved by the same method as in \cite[Prop 13.1.3(ii)]{GHKR10}. Therefore, letting $c$ be sufficiently divisible so that $v_m^c = t_\lambda$ is a translation element in $\Lambda_T$ (which is automatically central with respect to the roots in ${\rm Lie}(M)$), it enough to prove the result for $m = t \in T(F)$ with $\kappa_T(t) = \lambda \in X_*(T)_I^{\Phi}$. But then 
$$
\delta_P(t) = q^{-\langle 2\rho_N, \lambda \rangle} = q^{-\langle 2\rho_B, \lambda \rangle} = q^{-\breve{l}(t_\lambda)} = q^{-l_L(t_\lambda)}
$$
by Claim 1. This proves Claim 2.

\medskip

\noindent {\bf Claim 3:} If $t^* \in T^*(F)$ and $m \in M(F)$ are related by $t_{A^*, A}(t^*) = m$, then $l_L(v_m) = |\langle 2\rho_{N^*}, v_{t^*} \rangle|$.

\medskip

{\em Proof of Claim 3:} By a suitable $W_0$-conjugation we easily reduce to the case where $m$ is $P$-dominant and then use the same argument as in Claim 2. \qed

\bigskip
From the claims it follows that $l_L(v_{m_i}) = \langle 2\rho_N, v_{m_i} \rangle$, and therefore
$$
\delta_P^{1/2}(m) = q^{-\langle \rho_N, v_m \rangle} =  q^{-l_L(v_{m_1})/2 \, +\, l_L(v_{m_2})/2}.
$$
It follows that $\theta_P(\mathds{1}_{m})$ coincides with the definition of $\theta_{v_m}$ as given, e.g., in \cite[\S14]{HR12}. Thus by \cite[14.2]{HR12} or \cite{Goe07} we have a {\em minimal expression} of the form
$$
\theta_P(\mathds{1}_m) = \tilde{T}_{s_1}^{\epsilon_1} \cdots \tilde{T}_{s_l}^{\epsilon_l} T_{\omega_{v_m}}
$$
where $\tilde{T}_s = q^{-1/2}_s T_s$ and $\epsilon_{s} \in \{ \pm1\}$ for each $s$. We deduce that
$$
q^{l_L(v_m)/2}\theta_P(\mathds{1}_m) \,\,\, \mbox{takes values in $\mathbb Z$}.
$$

Now we recall that for $m \in W_0(m_0)$,
$$
\sum_{t^* \mapsto m} \delta^{-1/2}_{B^*}(t^*) \, \delta_P^{1/2}(m) = \sum_{t^* \mapsto m} q^{\langle \rho_{B^*}, v_{t^*} \rangle \, - \, \langle \rho_N, v_m \rangle} = \sum_{t^* \mapsto m} q^{\langle \rho_{B^*_{M^*}} , v_{t^*} \rangle}.
$$
To prove the lemma, it is therefore enough to show the following inequality:
\begin{equation*} %\label{final_equality}
\langle \rho_{B^*}, \mu \rangle + \langle \rho_{B^*_{M^*}}, v_{t^*} \rangle-  l_L(v_m)/2 \geq 0,
\end{equation*}
or equivalently
\begin{equation} \label{final_equality}
\langle \rho_{B^*}, \mu \rangle + \langle \rho_{B^*_{M^*}}, v_{t^*} \rangle-  |\langle \rho_{N^*}, v_{t^*} \rangle| \geq 0
\end{equation}
But it is easy to see that 
$$
\langle \rho_{B^*}, \mu \rangle + \langle \rho_{B^*_{M^*}}, v_{t^*} \rangle-  |\langle \rho_{N^*}, v_{t^*} \rangle| \geq \langle \rho_{B^*} , \mu \rangle - |\langle \rho_{B^*}, v_{t^*} \rangle| \geq 0.
$$
The final inequality follows because $v_{t^*} = \bar{\nu}$, the image of an element $\nu \in {\mathcal Wt}(\mu)$ under $X_*(T^*) \rightarrow X_*(T^*)_I$; cf.\,\cite[(7.12)]{Hai}.
\end{proof}
This completes the proof of Lemma \ref{final_lem}, and thus Theorem \ref{indep_thm}.
\end{proof}

%Adjoint groups

%%% START LARGE COMMENTED OUT MATERIAL

\begin{appendix}
\section{Spreading of connected reductive groups}
Let $F'$ be a discretely valued Henselian field with perfect residue field $k_{F'}$. The completion $F$ is a complete discretely valued field, and we fix a separable closure $\sF$. Let $\sF'\subset \sF$ be the separable closure of $F$. By the equivalence of Galois theories for $F'$ and $F$, their Galois groups are naturally isomorphic. Denote by $\bF'$ the maximal unramified subextension of $\sF'/F'$, and denote by $\bF$ its completion. Let $\sig\in {\rm Aut}(\bF/F)={\rm Aut}(\bF'/F')$ denote the Frobenius.

Let $G$ be a connected reductive $F$-group, and  fix a maximal $F$-split torus $A$, a maximal $\bF$-split torus $S$ containing $A$ and defined over $F$. Let $M=Z_G(A)$ denote the centralizer of $A$ which is a minimal Levi, and let $T=Z_G(S)$ be the centralizer of $S$. Then $T$ is a maximal torus because $G_\bF$ is quasi-split by Steinberg's theorem. Further fix a parabolic $F$-subgroup $P$ containing $M$.

\begin{prop} \label{extensionprop}
i\textup{)} There exists a connected reductive $F'$-group $\uG$ together with a tuple of closed $F'$-subgroups $(\uA,\uS,\uT,\uM,\uP)$ and an isomorphism of $F$-groups
\[
(\uG, \uA,\uS,\uT,\uM,\uP)\otimes_{F'}F\simeq (G,A,S,T,M,P),
\]
where $\uA$ is a maximal $F'$-split torus, $\uS$ a maximal $\bF'$-split torus defined over $F'$, $\uT$ its centralizer \textup{(}a maximal $F'$-torus\textup{)}, $\uM$ the centralizer of $\uA$ \textup{(}a minimal Levi\textup{)}, and $\uP$ a parabolic $F'$-subgroup with Levi $\uM$.\smallskip\\
ii\textup{)} The group $\uG$ is uniquely determined up to isomorphism, and the base change $\uG_{\bF'}$ is quasi-split. \smallskip\\
iii\textup{)} The isomorphism in i\textup{)} is compatible with the following constructions: the quasi-split outer form, restriction of scalars, passing to the adjoint \textup{(}resp. derived; resp. simply connected\textup{)} group. 
\end{prop}

Let $\tF/F$ be a finite Galois extension which splits $G$, and denote the corresponding extension by $\tF'/F'$. Let $\Ga=\Gal(\tF/F)=\Gal(\tF'/F')$ be the Galois group. Choose a Chevalley group scheme $G_0/F'$ with $G_0\otimes_{F'}\tF\simeq G\otimes_F\tF$. The Galois (or \'etale\footnote{There are two equivalent points of view on these cohomology groups: $\Ga$ cochains with values in the abstract group $\Aut(G_{0,\tF})$; \'etale torsors for the sheaf of groups $\Aut(G_0)$ which have an $\tF$-section.}) cohomology set
\[
H^1(\tF/F,\Aut(G_0)) \;\;\;\text{(resp. $H^1(\tF'/F',\Aut(G_0))$)}
\]
classifies isomorphism classes of $F$-groups (resp. $F'$-groups) which become over $\tF$ (resp. $\tF'$) isomorphic to $G_{0,\tF}$ (resp. $G_{0,\tF'}$). In \cite[App 2]{Ri16a} it is shown that the canonical map of pointed sets
\[
H^1(\tF'/F', \Aut(G_0))\to H^1(\tF/F, \Aut(G_0)).
\]
is a bijection for general Henselian valued fields. This already implies the existence of $\uG$. Let us denote by $[c]\in H^1(\tF/F, \Aut(G_0))$ (resp. $[\underline{c}]\in H^1(\tF'/F', \Aut(G_0))$) the class corresponding to the isomorphism class of $G$ (resp. $\uG$). The aim of this appendix is to show the extra compatibilities claimed in Proposition \ref{extensionprop}. 

By \cite[7.1.9]{Co14}, the group functor $\Aut(G_0)$ is representable by a smooth separated $F'$-group, and there is a short exact sequence of $F'$-groups
\begin{equation}\label{exactsequence}
1\to G_{0,\ad}\to \Aut(G_0)\to \on{Out}(G_0)\to 1.
\end{equation}
The proof of Proposition \ref{extensionprop} follows \cite[\S 2]{PZ13} using the results from \cite[App 2]{Ri16a} and proceeds in two steps corresponding to the outer terms of \eqref{exactsequence}. We construct the quasi-split $F'$-form $\uG^*$ first, and then define $\uG$ by inner twisting from $\uG^*$.

\subsection{Outer twisting.} We choose a pinning $(T_0,B_0,X_0)$ of $G_0$, and hence the map of $F'$-groups 
\begin{equation}\label{splitiso}
\Aut(G_0,T_0,B_0,X_0)\to \Aut(G_0)\to \on{Out}(G_0)
\end{equation}
is an isomorphism. The isomorphism \eqref{splitiso} splits the extension \eqref{exactsequence}, and there is a semi-direct product decomposition
\begin{equation}\label{semidirect}
\Aut(G_0)=G_{0,\ad}\rtimes \on{Out}(G_0).
\end{equation}
Let $[c^*]\in H^1(\tF/F,\on{Out}(G_0))$ denote the image of the class $[c]$. Under the semi-direct product decomposition we may view $[c^*]$ as a class in $H^1(\tF/F,\Aut(G_0))$, and the (unique up to isomorphism) associated group $G^*$ is the quasi-split $F$-form of $G$. Let us construct a representative $c^*$ of the class $[c^*]$.

Let $(R_0,\Delta)$ denote the based root datum of $G_0$. Then we have an isomorphism of $F'$-groups 
\[
\Aut(G_0,T_0,B_0,X_0)\simeq \Aut(R_0,\Delta),
\] 
where the latter denotes the constant $F'$-group associated with the automorphisms of the based root datum. This gives via \eqref{splitiso} an identification
\[
H^1(\tF/F,\on{Out}(G_0))\,\simeq\, H^1(\tF/F,\Aut(R_0,\Delta))\,=\,\Hom(\Ga, \Aut(R_0,\Delta)), 
\]
and we denote the image of $[c^*]$ by
\begin{equation}\label{outeract}
c^*\co \Ga\to \Aut(R_0,\Delta).
\end{equation}
%After modifying the isomorphism $\varphi\co G_{0,\tF}\simeq G_\tF$ if necessary, we may assume that $\varphi(T_{0,\tF})=T_\tF$ and $\varphi(B_{0,\tF})\subset P_\tF$. {\cm The pair $(T_\tF,\varphi(B_{0,\tF}))$ gives us a $\Ga$-action on the based root datum $(R_0,\Delta)$, i.e. a morphism of abstract groups
But now as $\Aut(R_0,\Delta)$ is an abstract group and as $\Ga=\Gal(\tF/F)=\Gal(\tF'/F')$, the natural map of pointed sets
\[
H^1(\tF'/F',\Aut(R_0,\Delta))\to H^1(\tF/F,\Aut(R_0,\Delta))
\]
is a bijection. Hence, we may view the class $[c^*]$ via \eqref{semidirect} as a class $[\underline{c}^*]$ in $H^1(\tF'/F',\Aut(G_0))$. We obtain a (unique up to isomorphism) quasi-split connected reductive $F'$-group scheme $\underline{G}^*$ extending $G^*$ such that as $\tF$-groups $\underline{G}^*_\tF\simeq G_{0,\tF}$. Concretely, Galois descent\footnote{Whenever we have a $\Ga$-torsor $\pi\co \tilde{X}\to X$, sheaves on $X$ are the same as $\Ga$-equivariant sheaves on $\tilde{X}$: a sheaf $\calF/X$ maps to the $\Ga$-equivariant sheaf $\pi^*\calF$; a $\Ga$-equivariant sheaf $\tilde{F}/\tilde{X}$ maps to $(\pi_*\tilde{F})^\Ga$. Note that it does not matter whether the order of $\Ga$ is divisible by the characteristic or not.} shows that the $F'$-group scheme $\underline{G}^*$ (resp. $G^*$) is given by 
\begin{equation}\label{explicit}
\underline{G}^*=\Res_{\tF'/F'}(G_0\otimes_{F'}\tF')^\Ga\;\;\;\text{(resp. $G^*=\Res_{\tF/F}(G_0\otimes_{F'}\tF)^\Ga$)},
\end{equation}
where $\Ga$ acts diagonally via $c^*(\ga)\otimes \ga$. After modifying the isomorphism $\varphi\co G_{0,\tF}\simeq G_\tF$ if necessary, we may assume that $\varphi(T_{0,\tF})=T_\tF$ and $\varphi(B_{0,\tF})\subset P_\tF$. Then the pair $\uT^*\subset \uB^*$ (resp. $T^*\subset B^*$) is constructed from the pair $T_0\subset B_0$ by the same formula \eqref{explicit}, and defines a maximal torus and a Borel subgroup. The minimal Levi $M$ defines a $\Ga$-invariant subset of simple roots in $\Delta$, and we denote by $M_0$ the Levi in $G_0$. Let $P_0=M_0\cdot B_0$ be the corresponding standard parabolic. By altering the inner twisting $G \rightarrow G^*$ if necessary, we may arrange that $c^*(\gamma) \otimes \gamma$ preserves $A_{0,\tF},  M_{0,\tF} , P_{0, \tF}$ (cf.\,\cite[11.12.1]{Hai14}). Again we construct the pair $\uM^*\subset \uP^*$ (resp. $M^*\subset P^*$) from the pair $M_0\subset P_0$ by \eqref{explicit}. Further, the chain of $F$-tori $A\subset S\subset T$ in $G$ gives rise to a chain of $F'$-tori $A_0\subset S_0\subset T_0$ in $G_0$ together with $\Ga$-actions on their scalar extensions to $\tF'$ (resp. $\tF$). Again this defines a chain of $F'$-tori $\uA^*\subset\uS^*\subset\uT^*$ (resp. $F$-tori $A^*\subset S^*\subset T^*$) with the following properties: the torus $\uS^*$ (resp. $S^*$) is a maximal $\bF'$-split (resp. $\bF$-split), and the torus $\uA^*$ (resp. $A^*$) is $F'$-split (resp. $F$-split). Note that the latter torus is \emph{not} maximal split in general. All in all, we obtain tuples such that as $F$-groups
\begin{equation}\label{quasisplit}
(\uG^*, \uA^*,\uS^*,\uT^*,\uM^*,\uP^*)\otimes_{F'}F\simeq (G^*,A^*,S^*,T^*, M^*,P^*), 
\end{equation}
where $G^*$ is the quasi-split $F$-form of $G$. Because Galois descent is compatible with taking centralizers (resp. normalizers), we have the relations $Z_{\uG^*}(\uS^*)=\uT^*$, and $Z_{\uG^*}(\uA^*)=\uM^*$.

\begin{rmk} 
Note that the spreading construction for tori is completely solved by \eqref{explicit} (because there are no inner twists). However, the spreading $\uA\subset \uS\subset \uT$ differs from $\uA^*\subset \uS^*\subset \uT^*$ in general: the first chain is constructed from $\Aut(T)=\on{Out}(T)$; the second chain is constructed from $\on{Out}(G)$.
\end{rmk}

The compatibilities claimed in Proposition \ref{extensionprop} are evident from the construction \eqref{explicit}. 

\subsection{Inner twisting.} Let us explain how to reconstruct the tuple $(G,A,S,T, M, P)$ from \eqref{quasisplit} via inner twisting. By construction, there is an isomorphism of $\tF$-groups
\begin{equation}\label{innertwist}
\psi\co G_\tF\overset{\varphi^{-1}}{\simeq}  G_{0,\tF}\simeq G^*_{\tF},
\end{equation}
where the last isomorphism comes from descent (the Galois action on $G_{0,\tF}$ is the outer action via $c^*$). If, for $\ga\in\Ga$, we denote $\ga(\psi)=\ga\circ \psi \circ \ga^{-1}$, then the image of $\psi\circ \ga(\psi)^{-1}$ in $\on{Out}(G^*)$ is trivial. Hence, for every $\ga\in \Ga$ there is an element $g_\ga\in G^*_\ad(\tF)$ with
\[
\psi\circ \ga(\psi)^{-1}= \on{Int}({g_\ga})\in \Aut(G^*_\tF),
\] 
where $\on{Int}(g_\ga)$ denotes the automorphism given by conjugation. The function $c^{\on{rig}}\co \ga\mapsto \psi\circ\ga(\psi)^{-1}$ is a $1$-cocycle, and its class defines an element $[c^{\on{rig}}]\in H^1(\tF/F,G^*_\ad)$. Conversely, the class $[c^{\on{rig}}]$ corresponds to a $\Ga$-stable $G^*_\ad(\tF)$-orbit of isomorphisms of $\tF$-groups $G_\tF\simeq G^*_{\tF}$. 

\begin{rmk}
Note that we can also consider $c^{\on{rig}}$ as a $1$-cocycle with values in $G_{0,\ad}$. If we do so, then under the semi-direct product decomposition \eqref{semidirect} the class of the cocycle $c^{\on{rig}}\rtimes c^*$ is the class $[c]$ we started with.
\end{rmk}

By construction, we have
\begin{equation}\label{compatible}
\psi((A,S,T,M,P))=(A^*,S^*,T^*,M^*,P^*).
\end{equation}
Let us denote by ${S^*}'$ (resp. ${M^*}'$) the image of $S^*$ (resp. $M^*$) in the adjoint group $G^*_\ad$.

\begin{lem} The element $g_\sig$ is contained in the $\tF$-points of the subgroup
\[
{N^*}'\defined \on{Norm}_{{M^*}'}({S^*}').
\] 
\end{lem}
\begin{proof} By \eqref{compatible}, the element $g_\sig$ normalizes $P^*$, and hence is contained in the parabolic ${P^*}'$ (the image of $P^*$ in $G^*_\ad$). As $g_\sig$ also normalizes $M^*$, it must be contained in ${M^*}'$. Finally, as $g_\sig$ also normalizes $S^*$ the lemma follows. 
\end{proof}

Let us further denote the $F'$-group by 
\[
{\uN^*}'=\on{Norm}_{{\uM^*}'}({\uS^*}'),
\]
where ${\uM^*}'$ (resp. ${\uS^*}'$) is the image of $\uM^*$ (resp. ${\uS^*}$) in $\uG^*_\ad$. Then ${\uN^*}'$ is a smooth affine group scheme, and we have ${\uN^*}'\otimes_{F'}F={N^*}'$. Hence, by the result of \cite[Prop 3.5.3 (2)]{GGM14} (see also \cite[Thm A.3]{Ri16a}) the natural map
\[
H^1(\tF'/F', {\uN^*}')\to H^1(\tF/F, {N^*}')
\]
is a bijection. In particular, the cocycle $[c^{\on{rig}}]$ corresponds to a unique cocycle $[\underline{c}^{\on{rig}}]\in H^1(\tF'/F', {\uN^*}')$ which defines via inner twisting of \eqref{quasisplit} the desired tuple
\[
(\uG,\uA,\uS,\uT,\uM,\uP).
\]
Concretely, if for $\ga\in \Ga$ the element $g_\ga^{\on{rig}}\in {\uN^*}'(\tF')$ is the value of $\underline{c}^{\on{rig}}$ at $\ga$, then we have for an $F'$-algebra $R$ the inner twisting
\[
\uG(R)=\uG^*(\tF'\otimes_{F'}R)^\Ga,
\]
where $\Ga$ acts via $\ga\mapsto \on{Int}({g_\ga^{\on{rig}}})\cdot \ga$. The same formulas hold for $(\uA,\uS,\uT,\uM,\uP)$ as every element in ${\uN^*}'$ preserves each of these groups. The compatibilities claimed in Proposition \ref{extensionprop} iii) are immediate from the descent construction.

\end{appendix}

\end{document}